\documentclass[11pt]{preprint}
\usepackage{difftrees} 
\usepackage[full]{textcomp}
\usepackage[osf]{newtxtext} 
\usepackage[cal=boondoxo]{mathalfa}
\usepackage{colortbl}

\usepackage{tikz-cd} 
\usetikzlibrary{cd} 
\usepackage[all,cmtip]{xy}
\usepackage{comment}

\usepackage{amssymb}
\usepackage{mathtools}
\usepackage{hyperref}
\usepackage{breakurl}
\usepackage{mhenvs}
\usepackage{mhequ} 
\usepackage{mhsymb}
\usepackage{booktabs}
\usepackage{tikz}
\usepackage{tcolorbox}
\usepackage{mathrsfs}
\usepackage[utf8]{inputenc}
\usepackage{longtable}
\usepackage{wrapfig}
\usepackage{rotating} 
\usepackage{subcaption}
\usepackage{mathrsfs}
\usepackage{epsfig}
\usepackage{microtype}
\usepackage{comment}
\usepackage{wasysym}
\usepackage{centernot}
\usepackage{enumitem}
\usepackage{bm}
\usepackage{stackrel}
\usepackage{graphicx}
\usepackage{axodraw}
\usepackage{xspace}
\usepackage{subcaption} 
\usepackage{epsfig} 
\usepackage{axodraw} 
\usepackage{xspace} 
\usepackage[toc,page]{appendix} 
\usepackage[all,cmtip]{xy} 
\usepackage{relsize} 
\usepackage{shuffle} 

\makeatletter
\newcommand{\globalcolor}[1]{%
	\color{#1}\global\let\default@color\current@color
}
\makeatother

\usetikzlibrary{calc}
\usetikzlibrary{decorations}
\usetikzlibrary{positioning}
\usetikzlibrary{shapes}
\usetikzlibrary{external}

\definecolor{blush}{rgb}{0.87, 0.36, 0.51}
\definecolor{brightcerulean}{rgb}{0.11, 0.67, 0.84}
\definecolor{greenryb}{rgb}{0.4, 0.69, 0.2}

\newif\ifdark
\darkfalse

\ifdark
\definecolor{darkred}{rgb}{0.9,0.2,0.2}
\definecolor{darkblue}{rgb}{0.7,0.3,1}
\definecolor{darkgreen}{rgb}{0.1,0.9,0.1}
\definecolor{franck}{rgb}{0,0.8,1}
\definecolor{pagebackground}{rgb}{.15,.21,.18}
\definecolor{pageforeground}{rgb}{.84,.84,.85}
\pagecolor{pagebackground}
\AtBeginDocument{\globalcolor{pageforeground}}
\definecolor{symbols}{rgb}{0,0.7,1}
\colorlet{connection}{red!80!black}
\colorlet{boxcolor}{blue!50}

\else

\definecolor{darkred}{rgb}{0.7,0.1,0.1}
\definecolor{darkblue}{rgb}{0.4,0.1,0.8}
\definecolor{darkgreen}{rgb}{0.1,0.7,0.1}
\definecolor{franck}{rgb}{0,0,1}
\definecolor{pagebackground}{rgb}{1,1,1}
\definecolor{pageforeground}{rgb}{0,0,0}
\colorlet{symbols}{blue!90!black}
\colorlet{connection}{red!30!black}
\colorlet{boxcolor}{blue!50!black}

\fi

\def\slash{\leavevmode\unskip\kern0.18em/\penalty\exhyphenpenalty\kern0.18em}
\def\dash{\leavevmode\unskip\kern0.18em--\penalty\exhyphenpenalty\kern0.18em}

\DeclareMathAlphabet{\mathbbm}{U}{bbm}{m}{n}

\DeclareFontFamily{U}{BOONDOX-calo}{\skewchar\font=45 }
\DeclareFontShape{U}{BOONDOX-calo}{m}{n}{
	<-> s*[1.05] BOONDOX-r-calo}{}
\DeclareFontShape{U}{BOONDOX-calo}{b}{n}{
	<-> s*[1.05] BOONDOX-b-calo}{}
\DeclareMathAlphabet{\mcb}{U}{BOONDOX-calo}{m}{n}
\SetMathAlphabet{\mcb}{bold}{U}{BOONDOX-calo}{b}{n}

\setlist{noitemsep,topsep=4pt,leftmargin=1.5em}

\DeclareMathAlphabet{\mathbbm}{U}{bbm}{m}{n}

\DeclareMathAlphabet{\mcb}{U}{BOONDOX-calo}{m}{n}
\SetMathAlphabet{\mcb}{bold}{U}{BOONDOX-calo}{b}{n}
\DeclareFontFamily{U}{mathx}{\hyphenchar\font45}
\DeclareFontShape{U}{mathx}{m}{n}{
	<5> <6> <7> <8> <9> <10>
	<10.95> <12> <14.4> <17.28> <20.74> <24.88>
	mathx10
}{}
\DeclareSymbolFont{mathx}{U}{mathx}{m}{n}
\DeclareMathSymbol{\bigtimes}{1}{mathx}{"91}

\def\s{\mathfrak{s}}

\def\reg{\mathord{\mathrm{reg}}}

\providecommand{\figures}{false}
{ \ifthenelse{\equal{\figures}{false}} {#1}{\[ {\rm Figure \ missing !} \]} }{}
\def\id{\mathrm{id}}

\def\reg{\mathop{\mathrm{reg}}}

\def\CA{\mathcal{A}}

\def\CB{\mathcal{B}}

\def\CT{\mathcal{T}}

\tikzstyle{tinydots}=[dash pattern=on \pgflinewidth off \pgflinewidth]
\tikzstyle{superdense}=[dash pattern=on 4pt off 1pt]

\newcommand{\mcM}{\mathcal{M}}

\newcommand{\mcA}{\mathcal{A}}

\newcommand{\mcC}{\mathcal{C}}
\newcommand{\mcB}{\mathcal{B}}

\newcommand{\mcS}{\CS}
\newcommand{\mcU}{\mathcal{U}}
\newcommand{\mcL}{\mathcal{L}}
\newcommand{\mcI}{\mathcal{I}}

\newcommand{\mcY}{\mathcal{Y}}

\newcommand{\mcD}{\mathcal{D}}

\newcommand{\mcP}{\mathcal{P}}

\newcommand{\mcX}{\mathcal{X}}

\newcommand{\mcZ}{\mathcal{Z}}
\newcommand{\mcT}{\mathcal{T}}

\newcommand{\beq}{\begin{equation}}
	\newcommand{\eeq}{\end{equation}}

\usepackage{empheq}

\newcommand{\T}{\mathbf{T}}

\def\Labn{\mathfrak{n}}

\def\${|\!|\!|}

\newcounter{theorem}

\newtheorem{ex}[theorem]{Example}

\newenvironment{DIFnomarkup}{}{} 

\newtheorem{assumption}{Assumption}

\theorembodyfont{\rmfamily}

\newcommand{\rrightarrow}{{\to\hskip -4.9mm\raise 1pt\hbox{$\to$}}}

\newfont{\indic}{bbmss12}

\def\PPi{\boldsymbol{\Pi}}

\def\Nabla_#1{\nabla_{\!#1}}

\makeatletter
\pgfdeclareshape{crosscircle}
{
	\inheritsavedanchors[from=circle] 
	\inheritanchorborder[from=circle]
	\inheritanchor[from=circle]{north}
	\inheritanchor[from=circle]{north west}
	\inheritanchor[from=circle]{north east}
	\inheritanchor[from=circle]{center}
	\inheritanchor[from=circle]{west}
	\inheritanchor[from=circle]{east}
	\inheritanchor[from=circle]{mid}
	\inheritanchor[from=circle]{mid west}
	\inheritanchor[from=circle]{mid east}
	\inheritanchor[from=circle]{base}
	\inheritanchor[from=circle]{base west}
	\inheritanchor[from=circle]{base east}
	\inheritanchor[from=circle]{south}
	\inheritanchor[from=circle]{south west}
	\inheritanchor[from=circle]{south east}
	\inheritbackgroundpath[from=circle]
	\foregroundpath{
		\centerpoint%
		\pgf@xc=\pgf@x%
		\pgf@yc=\pgf@y%
		\pgfutil@tempdima=\radius%
		\pgfmathsetlength{\pgf@xb}{\pgfkeysvalueof{/pgf/outer xsep}}%
		\pgfmathsetlength{\pgf@yb}{\pgfkeysvalueof{/pgf/outer ysep}}%
		\ifdim\pgf@xb<\pgf@yb%
		\advance\pgfutil@tempdima by-\pgf@yb%
		\else%
		\advance\pgfutil@tempdima by-\pgf@xb%
		\fi%
		\pgfpathmoveto{\pgfpointadd{\pgfqpoint{\pgf@xc}{\pgf@yc}}{\pgfqpoint{-0.707107\pgfutil@tempdima}{-0.707107\pgfutil@tempdima}}}
		\pgfpathlineto{\pgfpointadd{\pgfqpoint{\pgf@xc}{\pgf@yc}}{\pgfqpoint{0.707107\pgfutil@tempdima}{0.707107\pgfutil@tempdima}}}
		\pgfpathmoveto{\pgfpointadd{\pgfqpoint{\pgf@xc}{\pgf@yc}}{\pgfqpoint{-0.707107\pgfutil@tempdima}{0.707107\pgfutil@tempdima}}}
		\pgfpathlineto{\pgfpointadd{\pgfqpoint{\pgf@xc}{\pgf@yc}}{\pgfqpoint{0.707107\pgfutil@tempdima}{-0.707107\pgfutil@tempdima}}}
	}
}
\makeatother

\def\symbol#1{\textcolor{symbols}{#1}}

\def\decorate#1#2{
	\ifnum#2>0
	\foreach \count in {1,...,#2}{
		let
		\p1 = (sourcenode.center),
		\p2 = (sourcenode.east),
		\n1 = {\x2-\x1},
		\n2 = {1mm},
		\n3 = {(1.3+0.6*(\count-1))*\n1},
		\n4 = {0.7*\n1}
		in 
		node[rectangle,fill=symbols,rotate=30,inner sep=0pt,minimum width=0.2*\n2,minimum height=\n2] at ($(sourcenode.center) + (\n3,\n4)$) {}
	}
	\fi
	\ifnum#1>0
	\foreach \count in {1,...,#1}{
		let
		\p1 = (sourcenode.center),
		\p2 = (sourcenode.east),
		\n1 = {\x2-\x1},
		\n2 = {1mm},
		\n3 = {(1.3+0.6*(\count-1))*\n1},
		\n4 = {0.7*\n1}
		in 
		node[rectangle,fill=symbols,rotate=-30,inner sep=0pt,minimum width=0.2*\n2,minimum height=\n2] at ($(sourcenode.center) + (-\n3,\n4)$) {}
	}
	\fi
}

\tikzset{
	dectriangle/.style 2 args={
		triangle,
		alias=sourcenode,
		append after command={\decorate{#1}{#2}}
	},
	dectriangle/.default={0}{0},
}
\definecolor{colzA}{RGB}{220,50,47}   
\definecolor{colzB}{RGB}{38,139,210}  
\definecolor{colzC}{RGB}{133,153,0}   
\definecolor{colzD}{RGB}{203,75,22}   
\definecolor{colzE}{RGB}{108,113,196} 
\definecolor{colzF}{RGB}{181,137,0}   
\definecolor{colzG}{RGB}{42,161,152}  
\definecolor{colzH}{RGB}{0,128,128}   
\definecolor{colzI}{RGB}{163,206,39}  
\definecolor{colzJ}{RGB}{255,105,180} 
\definecolor{colzK}{RGB}{120,120,120} 
\definecolor{colzL}{RGB}{38,71,178}   
\definecolor{colzM}{RGB}{189,55,255}  

\tikzset{
	ann/.style={circle, draw=colzA, line width=1.2pt, fill=white, inner sep=0pt, minimum size=3mm},
	ful/.style={circle, draw=pageforeground, fill=colzA, inner sep=0pt, minimum size=3mm},
	ann1/.style={circle, draw=colzA, line width=1.2pt, fill=white, inner sep=0pt, minimum size=3mm},
	ful1/.style={circle, draw=pageforeground, fill=colzA, inner sep=0pt, minimum size=3mm},
	ann2/.style={circle, draw=colzB, line width=1.2pt, fill=white, inner sep=0pt, minimum size=3mm},
	ful2/.style={circle, draw=pageforeground, fill=colzB, inner sep=0pt, minimum size=3mm},
	ann3/.style={circle, draw=colzC, line width=1.2pt, fill=white, inner sep=0pt, minimum size=3mm},
	ful3/.style={circle, draw=pageforeground, fill=colzC, inner sep=0pt, minimum size=3mm},
	ann4/.style={circle, draw=colzD, line width=1.2pt, fill=white, inner sep=0pt, minimum size=3mm},
	ful4/.style={circle, draw=pageforeground, fill=colzD, inner sep=0pt, minimum size=3mm},
	ann5/.style={circle, draw=colzE, line width=1.2pt, fill=white, inner sep=0pt, minimum size=3mm},
	ful5/.style={circle, draw=pageforeground, fill=colzE, inner sep=0pt, minimum size=3mm},
	ann6/.style={circle, draw=colzF, line width=1.2pt, fill=white, inner sep=0pt, minimum size=3mm},
	ful6/.style={circle, draw=pageforeground, fill=colzF, inner sep=0pt, minimum size=3mm},
	ann7/.style={circle, draw=colzG, line width=1.2pt, fill=white, inner sep=0pt, minimum size=3mm},
	ful7/.style={circle, draw=pageforeground, fill=colzG, inner sep=0pt, minimum size=3mm},
	ann8/.style={circle, draw=colzH, line width=1.2pt, fill=white, inner sep=0pt, minimum size=3mm},
	ful8/.style={circle, draw=pageforeground, fill=colzH, inner sep=0pt, minimum size=3mm},
	ann9/.style={circle, draw=colzI, line width=1.2pt, fill=white, inner sep=0pt, minimum size=3mm},
	ful9/.style={circle, draw=pageforeground, fill=colzI, inner sep=0pt, minimum size=3mm},
	ann10/.style={circle, draw=colzJ, line width=1.2pt, fill=white, inner sep=0pt, minimum size=3mm},
	ful10/.style={circle, draw=pageforeground, fill=colzJ, inner sep=0pt, minimum size=3mm},
	ann11/.style={circle, draw=colzK, line width=1.2pt, fill=white, inner sep=0pt, minimum size=3mm},
	ful11/.style={circle, draw=pageforeground, fill=colzK, inner sep=0pt, minimum size=3mm},
	ann12/.style={circle, draw=colzL, line width=1.2pt, fill=white, inner sep=0pt, minimum size=3mm},
	ful12/.style={circle, draw=pageforeground, fill=colzL, inner sep=0pt, minimum size=3mm},
	ann13/.style={circle, draw=colzM, line width=1.2pt, fill=white, inner sep=0pt, minimum size=3mm},
	ful13/.style={circle, draw=pageforeground, fill=colzM, inner sep=0pt, minimum size=3mm}
}
\tikzset{
	ring/.style={
		circle,
		draw=pageforeground,     
		fill=white,              
		line width=0.4pt,        
		inner sep=0pt,
		minimum size=3mm
	},
	ring1/.style={circle, draw=red!80!black,   fill=white, line width=0.4pt, inner sep=0pt, minimum size=3mm},
	ring2/.style={circle, draw=darkgreen!70!black, fill=white, line width=0.4pt, inner sep=0pt, minimum size=3mm},
	ring3/.style={circle, draw=brown!60!black, fill=white, line width=0.4pt, inner sep=0pt, minimum size=3mm},
	ring4/.style={circle, draw=purple!60!black, fill=white, line width=0.4pt, inner sep=0pt, minimum size=3mm},
	ring5/.style={circle, draw=blue!60!black,   fill=white, line width=0.4pt, inner sep=0pt, minimum size=3mm},
	ring6/.style={circle, draw=darkblue!60!black, fill=white, line width=0.4pt, inner sep=0pt, minimum size=3mm}
}

\tikzset{
	ann/.style={ring, draw=colzA, line width=1.2pt},
	ful/.style={circle, draw=pageforeground, fill=colzA, inner sep=0pt, minimum size=3mm},
	ann1/.style={ring, draw=colzB, line width=1.2pt},
	ful1/.style={circle, draw=pageforeground, fill=colzB, inner sep=0pt, minimum size=3mm},
	ann2/.style={ring, draw=colzC, line width=1.2pt},
	ful2/.style={circle, draw=pageforeground, fill=colzC, inner sep=0pt, minimum size=3mm},
	ann3/.style={ring, draw=colzD, line width=1.2pt},
	ful3/.style={circle, draw=pageforeground, fill=colzD, inner sep=0pt, minimum size=3mm},
	ann4/.style={ring, draw=colzE, line width=1.2pt},
	ful4/.style={circle, draw=pageforeground, fill=colzE, inner sep=0pt, minimum size=3mm},
	ann5/.style={ring, draw=colzF, line width=1.2pt},
	ful5/.style={circle, draw=pageforeground, fill=colzF, inner sep=0pt, minimum size=3mm},
	ann6/.style={ring, draw=colzG, line width=1.2pt},
	ful6/.style={circle, draw=pageforeground, fill=colzG, inner sep=0pt, minimum size=3mm},
	ann7/.style={ring, draw=colzH, line width=1.2pt},
	ful7/.style={circle, draw=pageforeground, fill=colzH, inner sep=0pt, minimum size=3mm},
	ann8/.style={ring, draw=colzI, line width=1.2pt},
	ful8/.style={circle, draw=pageforeground, fill=colzI, inner sep=0pt, minimum size=3mm},
	ann9/.style={ring, draw=colzJ, line width=1.2pt},
	ful9/.style={circle, draw=pageforeground, fill=colzJ, inner sep=0pt, minimum size=3mm},
	ann10/.style={ring, draw=colzK, line width=1.2pt},
	ful10/.style={circle, draw=pageforeground, fill=colzK, inner sep=0pt, minimum size=3mm},
	ann11/.style={ring, draw=colzL, line width=1.2pt},
	ful11/.style={circle, draw=pageforeground, fill=colzL, inner sep=0pt, minimum size=3mm},
	ann12/.style={ring, draw=colzM, line width=1.2pt},
	ful12/.style={circle, draw=pageforeground, fill=colzM, inner sep=0pt, minimum size=3mm}
}
\tikzset{
	cross/.style={path picture={ 
			\draw[symbols]
			(path picture bounding box.south east) -- (path picture bounding box.north west) (path picture bounding box.south west) -- (path picture bounding box.north east);
	}},
	root/.style={circle,fill=green!50!black,inner sep=0pt, minimum size=1.2mm},
	dot/.style={circle,fill=pageforeground,inner sep=0pt, minimum size=1mm},
	dotred/.style={circle,fill=pageforeground!50!pagebackground,inner sep=0pt, minimum size=2mm},
	var/.style={circle,fill=pageforeground!10!pagebackground,draw=pageforeground,inner sep=0pt, minimum size=3mm},
	var2/.style={circle,fill=darkgreen,draw=pageforeground,inner sep=0pt, minimum size=3mm},
	var3/.style={circle,fill=brown,draw=pageforeground,inner sep=0pt, minimum size=3mm},
	var4/.style={circle,fill=purple,draw=pageforeground,inner sep=0pt, minimum size=3mm},
	var5/.style={circle,fill=blue,draw=pageforeground,inner sep=0pt, minimum size=3mm},
	var6/.style={circle,fill=darkblue,draw=pageforeground,inner sep=0pt, minimum size=3mm},
	kernel/.style={semithick,draw=green,shorten >=2pt,shorten <=2pt},
	kernels/.style={snake=zigzag,shorten >=2pt,shorten <=2pt,segment amplitude=1pt,segment length=4pt,line before snake=2pt,line after snake=5pt,},
	rho/.style={densely dashed,semithick,shorten >=2pt,shorten <=2pt},
	testfcn/.style={dotted,semithick,shorten >=2pt,shorten <=2pt},
	renorm/.style={shape=circle,fill=pagebackground,inner sep=1pt},
	labl/.style={shape=rectangle,fill=pagebackground,inner sep=1pt},
	xic/.style={very thin,circle,draw=symbols,fill=symbols,inner sep=0pt,minimum size=1.2mm},
	g/.style={very thin,rectangle,draw=symbols,fill=symbols!10!pagebackground,inner sep=0pt,minimum width=2.5mm,minimum height=1.2mm},
	xi/.style={very thin,circle,draw=symbols,fill=symbols!10!pagebackground,inner sep=0pt,minimum size=1.2mm},
	xide/.style={very thin,circle,fill=darkgreen, draw=pageforeground, inner sep=0pt,minimum size=1.2mm},
		xide1/.style={very thin,circle,fill=purple, draw=pageforeground, inner sep=0pt,minimum size=1.2mm},
	xies/.style={very thin,rectangle,fill=green!50!black!25,draw=symbols,inner sep=0pt,minimum size=1.1mm},
	xiesf/.style={very thin,rectangle,fill=green!50!black,draw=symbols,inner sep=0pt,minimum size=1.1mm},
	xix/.style={very thin,crosscircle,fill=symbols!10!pagebackground,draw=symbols,inner sep=0pt,minimum size=1.2mm},
	X/.style={very thin,cross,rectangle,fill=pagebackground,draw=symbols,inner sep=0pt,minimum size=1.2mm},
	xib/.style={thin,circle,fill=symbols!10!pagebackground,draw=symbols,inner sep=0pt,minimum size=1.6mm},
	xie/.style={thin,circle,fill=green!50!black,draw=symbols,inner sep=0pt,minimum size=1.6mm},
	xid/.style={thin,circle,fill=symbols,draw=symbols,inner sep=0pt,minimum size=1.6mm},
	xibx/.style={thin,crosscircle,fill=symbols!10!pagebackground,draw=symbols,inner sep=0pt,minimum size=1.6mm},
	kernels2/.style={very thick,draw=connection,segment length=12pt},
	keps/.style={thin,draw=symbols,->},
	kepspr/.style={thick,draw=connection,->},
	krho/.style={thin,draw=symbols,superdense,->},
	krhopr/.style={thick,draw=connection,superdense},
	triangle/.style = { regular polygon, regular polygon sides=3},
	not/.style={thin,circle,draw=connection,fill=connection,inner sep=0pt,minimum size=0.5mm},
	diff/.style = {very thin,draw=symbols,triangle,fill=red!50!black,inner sep=0pt,minimum size=1.6mm},
	diff1/.style = {very thin,dectriangle={1}{0},fill=red!50!black,draw=symbols,inner sep=0pt,minimum size=1.6mm},
	diff2/.style = {very thin,dectriangle={1}{1},fill=red!50!black,draw=symbols,inner sep=0pt,minimum size=1.6mm},
	diffmini/.style = {very thin,rectangle,fill=black,draw=black,inner sep=0pt,minimum size=0.75mm},
	kernelsmod/.style={very thick,draw=connection,segment length=12pt},
	rec/.style = {very thin,rectangle,fill=black,draw=black,inner sep=0pt,minimum size=2mm},
	cerc/.style={very thin,circle,draw=black,fill=symbols,inner sep=0pt,minimum size=2mm},
	stars/.style={very thin,star,star points=6,star point ratio=0.5, draw=black,fill=red,inner sep=0pt,minimum size=0.7mm},
	>=stealth,
}
\tikzset{
	root/.style={circle,fill=black!50,inner sep=0pt, minimum size=3mm},
	circ/.style={circle,fill=white,draw=black,very thin,inner sep=.5pt, minimum size=1.2mm},
	round1/.style={fill=white,outer sep = 0,inner sep=2pt,rounded corners=1mm,draw,text=black,thin,minimum size=1.2mm},
	circ1/.style={circle,fill=red!10,draw=red,very thin,inner sep=.5pt, minimum size=1.2mm},
	rect/.style={fill=white,outer sep = 0,inner sep=2pt,rectangle,draw,text=black,thin,minimum size=1.2mm},
	rect1/.style={fill=white,outer sep = 0,inner sep=2pt,rectangle,draw,text=black,thin,minimum size=1.2mm},
	round2/.style={fill=red!10,outer sep = 0,inner sep=2pt,rounded corners=1mm,draw,text=black,thin,minimum size=1.2mm},
	round3/.style={fill=blue!10,outer sep = 0,inner sep=2pt,rounded corners=1mm,draw,text=black,thin,minimum size=1.2mm}, 
	rect2/.style={fill=black!10,outer sep = 0,inner sep=2pt,rectangle,draw,text=black,thin,minimum size=1.2mm},
	dot/.style={circle,fill=black,inner sep=0pt, minimum size=1.2mm},
	dotred/.style={circle,fill=black!50,inner sep=0pt, minimum size=2mm},
	var/.style={circle,fill=black!10,draw=black,inner sep=0pt, minimum size=3mm},
	kernel/.style={semithick,draw=darkgreen},
	diag/.style={thin,shorten >=4pt,shorten <=4pt},
	kernel1/.style={thick},
	kernels/.style={snake=zigzag,shorten >=2pt,shorten <=2pt,segment amplitude=1pt,segment length=4pt,line before snake=2pt,line after snake=5pt},
	kernels1/.style={snake=zigzag,segment amplitude=0.5pt,segment length=2pt},
	rho1/.style={densely dotted,semithick},
	rho/.style={densely dashed,semithick,shorten >=2pt,shorten <=2pt},
	testfcn/.style={dotted,semithick,shorten >=2pt,shorten <=2pt},
	visible/.style={draw, circle, fill, inner sep=0.25ex},
	renorm/.style={shape=circle,fill=white,inner sep=1pt},
	labl/.style={shape=rectangle,fill=white,inner sep=1pt},
	xic/.style={very thin,circle,fill=symbols,draw=black,inner sep=0pt,minimum size=1.2mm},
	xi/.style={very thin,circle,fill=blue!10,draw=black,inner sep=0pt,minimum size=1.2mm},
	xib/.style={very thin,circle,fill=blue!10,draw=black,inner sep=0pt,minimum size=1.6mm},
	xie/.style={very thin,circle,fill=green!50!black,draw=black,inner sep=0pt,minimum size=1mm},
	xid/.style={very thin,circle,fill=symbols,draw=black,inner sep=0pt,minimum size=1.6mm},
	edgetype/.style={very thin,circle,draw=black,inner sep=0pt,minimum size=5mm},
	nodetype/.style={very thick,circle,draw=black,inner sep=0pt,minimum size=5mm},
	kernels2/.style={very thick,draw=connection,segment length=12pt},
	clean/.style={thin,circle,fill=black,inner sep=0pt,minimum size=1mm},	not/.style={thin,circle,fill=symbols,draw=connection,fill=connection,inner sep=0pt,minimum size=0.8mm},
	>=stealth,
}

\makeatletter
\def\DeclareSymbol#1#2#3{%
	\expandafter\gdef\csname MH@symb@#1\endcsname{\tikzsetnextfilename{symbol#1}%
		\tikz[baseline=#2,scale=0.15,draw=symbols,line join=round]{#3}}%
	\expandafter\gdef\csname MH@symb@#1s\endcsname{\scalebox{0.75}{\tikzsetnextfilename{symbol#1}%
			\tikz[baseline=#2,scale=0.15,draw=symbols,line join=round]{#3}}}%
	\expandafter\gdef\csname MH@symb@#1ss\endcsname{\scalebox{0.65}{\tikzsetnextfilename{symbol#1}%
			\tikz[baseline=#2,scale=0.15,draw=symbols,line join=round]{#3}}}%
}
\def\<#1>{\ifthenelse{\boolean{mmode}}{\mathchoice{\csname MH@symb@#1\endcsname}{\csname MH@symb@#1\endcsname}{\csname MH@symb@#1s\endcsname}{\csname MH@symb@#1ss\endcsname}}{\csname MH@symb@#1\endcsname}}
\makeatother

\DeclareSymbol{Xi22}{0.5}{\draw (0,0) node[xi] {} -- (-1,1) node[not] {} -- (0,2) node[xi] {};} 
\DeclareSymbol{Xi2}{-2}{\draw (-1,-0.25) node[xi] {} -- (0,1) node[xi] {};} 
\DeclareSymbol{Xi2b}{-2}{\draw (-1,-0.25) node[xic] {} -- (0,1) node[xic] {};} 
\DeclareSymbol{Xi2g}{-2}{\draw (-1,-0.25) node[xies] {} -- (0,1) node[xi] {};} 
\DeclareSymbol{Xi2g2}{-2}{\draw (-1,-0.25) node[xi] {} -- (0,1) node[xies] {};} 
\DeclareSymbol{cXi2}{-2}{\draw (0,-0.25) node[xi] {} -- (-1,1) node[xic] {};}
\DeclareSymbol{Xi3}{0}{\draw (0,0) node[xi] {} -- (-1,1) node[xi] {} -- (0,2) node[xi] {};}
\DeclareSymbol{XiIIXi}{0}{\draw (0,0) node[xi] {} -- (-1,1); \draw[kernels2] (-1,1) node[not] {} -- (0,2) node[xi] {};}

\DeclareSymbol{Xi4}{2}{\draw (0,0) node[xi] {} -- (-1,1) node[xi] {} -- (0,2) node[xi] {} -- (-1,3) node[xi] {};}

\DeclareSymbol{Xi4_11}{2}{\draw (0,0) node[not] {} -- (-1,1) node[xi] {} -- (0,2) node[xi] {} -- (-1,3) node[xi] {};}

\DeclareSymbol{Xi4_12}{2}{\draw[kernels2] (0,0) node[not] {} -- (-1,1); \draw (-1,1) node[xi] {} -- (0,2) node[xi] {} -- (-1,3) node[xi] {};}

\DeclareSymbol{Xi4_1}{2}{\draw (0,0) node[xic] {} -- (-1,1) node[xic] {} -- (0,2) node[xi] {} -- (-1,3) node[xi] {};}
\DeclareSymbol{Xi4_2}{2}{\draw (0,0) node[xic] {} -- (-1,1) node[xi] {} -- (0,2) node[xi] {} -- (-1,3) node[xic] {};}
\DeclareSymbol{Xi2X}{-2}{\draw (0,-0.25) node[xi] {} -- (-1,1) node[xix] {};}
\DeclareSymbol{XXi2}{-2}{\draw (0,-0.25) node[xix] {} -- (-1,1) node[xi] {};}
\DeclareSymbol{IIXi}{0}{\draw (0,-0.25) node[not] {} -- (-1,1) node[xi] {} -- (0,2) node[xi] {};}
\DeclareSymbol{IXi^2}{-1}{\draw (-1,1) node[xi] {} -- (0,0) node[not] {} -- (1,1) node[xi] {};}
\DeclareSymbol{IIXi^2}{-4}{\draw (0,-1.5) node[not] {} -- (0,0);
	\draw[kernels2] (-1,1) node[xi] {} -- (0,0) node[not] {} -- (1,1) node[xi] {};}
\DeclareSymbol{XiX}{-2.8}{\node[xibx] {};}
\DeclareSymbol{tauX}{-2.8}{ \node[X] {};}
\DeclareSymbol{Xi}{-2.8}{\node[xib] {};}

\DeclareSymbol{IXiX}{-1}{\draw (0,-0.25) node[not] {} -- (1,1) node[xix] {};}

\DeclareSymbol{IiXiX}{-1}{\draw[kernels2] (0,-0.25) node[not] {} -- (1,1) node[xix] {};}

\DeclareSymbol{IXi3}{2}{\draw (0,-0.25) node[not] {} -- (-1,1) node[xi] {} -- (0,2) node[xi] {} -- (-1,3) node[xi] {};}
\DeclareSymbol{IXi}{-2}{\draw (0,-0.25) node[not] {} -- (1,1) node[xi] {};}
\DeclareSymbol{IiiXi}{-2}{\draw[kernels2] (0,-0.25) node[not] {} -- (1,1) node[xi] {};}

\DeclareSymbol{IXippp}{-2}{\draw[kernels2] (0,-0.25) node[not] {} -- (1,1) node[xide1] {};}

\DeclareSymbol{IXig}{-2}{\draw (0,-0.25) node[not] {} -- (1,1) node[xide] {};}

\DeclareSymbol{XiI}{-2}{\draw (0,-0.25) node[xi] {} -- (-1,1) node[not] {};}

\DeclareSymbol{Xi4b}{0}{\draw(0,1.5) node[xi] {} -- (0,0); \draw (-1,1) node[xi] {} -- (0,0) node[xi] {} -- (1,1) node[xi] {};}

\DeclareSymbol{Xi4b11}{0}{\draw(0,1.5) node[xi] {} -- (0,0); \draw (-1,1) node[xi] {} -- (0,0) node[not] {} -- (1,1) node[xi] {};}

\DeclareSymbol{Xi4b111}{0}{\draw(0,1.5) node[not] {} -- (0,0); \draw (-1,1) node[xi] {} -- (0,0) node[not] {} -- (1,1) node[xi] {};
\draw(0,3) node[xi] {} -- (0,1.5);
\draw (-1,2.5) node[xi] {} -- (0,1.5) node[not] {} -- (1,2.5) node[xi] {};
}

\DeclareSymbol{Xi4b111a}{0}{\draw(0,1.5) node[not] {} -- (0,0); \draw (-1,1) node[xi] {} -- (0,0) node[not] {} -- (1,1) node[xi] {};
	\draw (-1,2.5) node[xi] {} -- (0,1.5) node[not] {} -- (1,2.5) node[xi] {};
}

\DeclareSymbol{Xi4b111b}{0}{\draw(0,1.5) node[not] {} -- (0,0); \draw (-1,1) node[xi] {} -- (0,0) node[not] {};
	\draw(0,3) node[xi] {} -- (0,1.5);
	\draw (-1,2.5) node[xi] {} -- (0,1.5) node[not] {} -- (1,2.5) node[xi] {};
}

\DeclareSymbol{Xi4b111b1}{0}{\draw(0,1.5) node[not] {} -- (0,0); \draw  (0,0) node[not] {};
	\draw(0,3) node[xi] {} -- (0,1.5);
	\draw (-1,2.5) node[xi] {} -- (0,1.5) node[not] {} -- (1,2.5) node[xi] {};
}

\DeclareSymbol{Xi4b111b2}{0}{\draw(0,1.5) node[not] {} -- (0,0); \draw  (0,0) node[not] {};
	\draw (-1,2.5) node[xi] {} -- (0,1.5) node[not] {} -- (1,2.5) node[xi] {};
}

\DeclareSymbol{Xi4b'}{0}{\draw(0,1.5) node[xi] {} -- (0,-0.2); \draw (-1,1) node[xi] {} -- (0,-0.2) node[not] {} -- (1,1) node[xi] {};}

\DeclareSymbol{Xi4c}{0}{\draw (0,1) -- (0.8,2.2) node[xi] {};\draw (0,-0.25) node[xi] {} -- (0,1) node[xi] {} -- (-0.8,2.2) node[xi] {};}

\DeclareSymbol{Xi4ccccc}{0}{\draw (1,1) node[xi] {} -- (0,0) node[xi] {} -- (-1,1) node[xi] {} -- (0,2) node[xi] {};
}

\DeclareSymbol{Xi4ccccb}{0}{\draw (1,1) node[xi] {} -- (0,0) node[not] {} -- (-1,1) node[xi] {} -- (0,2) node[xi] {};
}

\DeclareSymbol{Xi4c11}{0}{\draw (0,1) -- (0.8,2.2) node[xi] {};\draw (0,-0.25) node[not] {} -- (0,1) node[xi] {} -- (-0.8,2.2) node[xi] {};}

\DeclareSymbol{Xi4c12}{0}{\draw (0,1) -- (0.8,2.2) node[xi] {}; \draw (0,1) -- (-0.8,2.2) node[xi] {}; \draw[kernels2] (0,-0.25) node[not] {} -- (0,1) node[xi] {} ; }

\DeclareSymbol{Xi4d}{-4.5}{\draw (0,-1.5) node[not] {} -- (0,0); \draw (-1,1) node[xi] {} -- (0,0) node[xi] {} -- (1,1) node[xi] {};}
\DeclareSymbol{Xi4e}{0}{\draw (0,2) node[xi] {} -- (-1,1) node[xi] {} -- (0,0) node[xi] {} -- (1,1) node[xi] {};}
\DeclareSymbol{Xi4e'}{0}{\draw (0,2) node[xi] {} -- (-1,1) node[xi] {} -- (0,-0.2) node[not] {} -- (1,1) node[xi] {};}

\DeclareSymbol{Xitwo}
{0}{\draw[kernels2] (0,0) node[not] {} -- (-1,1) node[not] {}
	-- (-2,2) node[not]{} -- (-3,3) node[xi]  {};
	\draw[kernels2] (0,0) -- (1,1) node[xi] {};
	\draw[kernels2] (-1,1) -- (0,2) node[xi] {};
	\draw[kernels2] (-2,2) -- (-1,3) node[xi] {};}

\DeclareSymbol{IXitwo}
{0}{\draw (-.7,1.2) node[xi] {} -- (0,-0.2) -- (.7,1.2) node[xi] {};}
\DeclareSymbol{I1Xitwo}
{0}{\draw[kernels2] (0,0) node[not] {} -- (-1,1) node[xi] {};
	\draw[kernels2] (0,0) -- (1,1) node[xi] {};}

\DeclareSymbol{I1Xitwobis}
{0}{\draw[kernels2] (0,0) node[not] {} -- (-1,1) node[xies] {};
	\draw[kernels2] (0,0) -- (1,1) node[xies] {};}

\DeclareSymbol{I1Xitwog}
{0}{\draw[kernels2] (0,0) node[not] {} -- (-1,1) node[xies] {};
	\draw[kernels2] (0,0) -- (1,1) node[xi] {};}

\DeclareSymbol{cI1Xitwo}
{0}{\draw[kernels2] (0,0) node[not] {} -- (-1,1) node[xic] {};
	\draw[kernels2] (0,0) -- (1,1) node[xi] {};}

\DeclareSymbol{I1IXi3}{0}{\draw (0,0) node[xi] {} -- (-1,1) ; 
	\draw[kernels2] (-1,1) node[not] {} -- (0,2) node[xi] {};
	\draw[kernels2] (-1,1) node[not] {} -- (-2,2) node[xi] {};}

\DeclareSymbol{I1Xi3c}{-1}{\draw[kernels2](0,1.5) node[xi] {} -- (0,0) node[not] {}; \draw (-1,1) node[xi] {} -- (0,0) ; \draw[kernels2] (0,0) -- (1,1) node[xi] {};}

\DeclareSymbol{I1Xi3cbis}{-1}{\draw[kernels2](0,1.5) node[xies] {} -- (0,0) node[not] {}; \draw (-1,1) node[xies] {} -- (0,0) ; \draw[kernels2] (0,0) -- (1,1) node[xies] {};}

\DeclareSymbol{I1IXi3b}{0}{\draw[kernels2] (0,0) node[not] {} -- (-1,1) ; \draw[kernels2] (0,0)   -- (1,1) node[xi] {} ;
	\draw (-1,1) node[xi] {} -- (0,2) node[xi] {};
}

\DeclareSymbol{I1IXi3c}{0}{\draw[kernels2] (0,0) node[not] {} -- (-1,1) ; \draw[kernels2] (0,0)   -- (1,1) node[xi] {} ;
	\draw[kernels2] (-1,1) node[not] {} -- (0,2) node[xi] {};
	\draw[kernels2] (-1,1) node[not] {} -- (-2,2) node[xi] {};}

\DeclareSymbol{I1IXi3cbis}{0}{\draw[kernels2] (0,0) node[not] {} -- (-1,1) ; \draw[kernels2] (0,0)   -- (1,1) node[xies] {} ;
	\draw[kernels2] (-1,1) node[not] {} -- (0,2) node[xies] {};
	\draw[kernels2] (-1,1) node[not] {} -- (-2,2) node[xies] {};}

\DeclareSymbol{I1Xi}{0}{\draw[kernels2] (0,0) node[not] {} -- (-1,1)  node[xi] {} ;}

\DeclareSymbol{I1Xi4a}{2}{\draw[kernels2] (0,0) node[not] {} -- (-1,1) ; \draw[kernels2] (0,0) node[not] {} -- (1,1) node[xi] {} ;
	\draw (-1,1) node[xi] {} -- (0,2) node[xi] {} -- (-1,3) node[xi] {};}

\DeclareSymbol{cI1Xi4a}{2}{\draw[kernels2] (0,0) node[not] {} -- (-1,1) ; \draw[kernels2] (0,0) node[not] {} -- (1,1) node[xic] {} ;
	\draw (-1,1) node[xic] {} -- (0,2) node[xi] {} -- (-1,3) node[xi] {};}

\DeclareSymbol{I1Xi4b}{2}{\draw (0,0) node[xi] {} -- (-1,1) node[xi] {} -- (0,2) ; \draw[kernels2] (0,2) node[not] {} -- (-1,3) node[xi] {};\draw[kernels2] (0,2)  -- (1,3) node[xi] {};
}

\DeclareSymbol{cI1Xi4b}{2}{\draw (0,0) node[xic] {} -- (-1,1) node[xic] {} -- (0,2) ; \draw[kernels2] (0,2) node[not] {} -- (-1,3) node[xi] {};\draw[kernels2] (0,2)  -- (1,3) node[xi] {};
}

\DeclareSymbol{Xii}{0}{ \draw (-1,0.25) node[xi] {} -- (0,1.5) node[xi] {};
}

\DeclareSymbol{Xiig1}{0}{ \draw (-1,0.25) node[xide] {} -- (0,1.5) node[xi] {};
}

\DeclareSymbol{Xiig2}{0}{ \draw (-1,0.25) node[xi] {} -- (0,1.5) node[xide] {};
}

\DeclareSymbol{Xiip1}{0}{ \draw (-1,0.25) node[xide1] {} -- (0,1.5) node[xi] {};
}

\DeclareSymbol{IXip}{0}{ \draw (-1,0.25) node[not] {} -- (0,1.5) node[xide1] {};
}

\DeclareSymbol{Xiip2}{0}{ \draw (-1,0.25) node[xi] {} -- (0,1.5) node[xide1] {};
}

\DeclareSymbol{Xii1}{0}{ \draw (-1,0.25) node[xi] {} -- (0,1.5) node[xix] {};
}

\DeclareSymbol{Xii2}{0}{ \draw (-1,0.25) node[xix] {} -- (0,1.5) node[xi] {};
}

\DeclareSymbol{I1Xi4c}{2}{\draw (0,0) node[xi] {} -- (-1,1) node[not] {}; \draw[kernels2] (-1,1) -- (0,2) ; 
	\draw[kernels2] (-1,1) -- (-2,2) node[xi] {} ;
	\draw (0,2) node[xi] {} -- (-1,3) node[xi] {};}

\DeclareSymbol{cI1Xi4c}{2}{\draw (0,0) node[xic] {} -- (-1,1) node[not] {}; \draw[kernels2] (-1,1) -- (0,2) ; 
	\draw[kernels2] (-1,1) -- (-2,2) node[xic] {} ;
	\draw (0,2) node[xi] {} -- (-1,3) node[xi] {};}

\DeclareSymbol{I1Xi4ab}{2}{\draw[kernels2] (0,0) node[not] {} -- (-1,1) ; \draw[kernels2] (0,0) node[not] {} -- (1,1) node[xi] {};\draw (-1,1) node[xi] {} -- (0,2) ; \draw[kernels2] (0,2) node[not] {} -- (-1,3) node[xi] {};\draw[kernels2] (0,2)  -- (1,3) node[xi] {}; }

\DeclareSymbol{cI1Xi4ab}{2}{\draw[kernels2] (0,0) node[not] {} -- (-1,1) ; \draw[kernels2] (0,0) node[not] {} -- (1,1) node[xic] {};\draw (-1,1) node[xic] {} -- (0,2) ; \draw[kernels2] (0,2) node[not] {} -- (-1,3) node[xi] {};\draw[kernels2] (0,2)  -- (1,3) node[xi] {}; }

\DeclareSymbol{I1Xi4bc}{2}{\draw (0,0) node[xi] {} -- (-1,1) node[not] {}; \draw[kernels2] (-1,1) -- (0,2) ; 
	\draw[kernels2] (-1,1) -- (-2,2) node[xi] {} ; \draw[kernels2] (0,2) node[not] {} -- (-1,3) node[xi] {};\draw[kernels2] (0,2)  -- (1,3) node[xi] {};
}

\DeclareSymbol{cI1Xi4bc}{2}{\draw (0,0) node[xic] {} -- (-1,1) node[not] {}; \draw[kernels2] (-1,1) -- (0,2) ; 
	\draw[kernels2] (-1,1) -- (-2,2) node[xic] {} ; \draw[kernels2] (0,2) node[not] {} -- (-1,3) node[xi] {};\draw[kernels2] (0,2)  -- (1,3) node[xi] {};
}

\DeclareSymbol{I1Xi4abcc1}{2}{\draw[kernels2] (0,0) node[not] {} -- (-1,1) node[not] {}
	-- (-2,2) node[not]{} -- (-3,3) node[xic]  {};
	\draw[kernels2] (0,0) -- (1,1) node[xic] {};
	\draw[kernels2] (-1,1) -- (0,2) node[xi] {};
	\draw[kernels2] (-2,2) -- (-1,3) node[xi] {};
}

\DeclareSymbol{I1Xi4abcc1b}{2}{\draw[kernels2] (0,0) node[not] {} -- (-1,1) node[not] {}
	-- (-2,2) node[not]{} -- (-3,3) node[xi]  {};
	\draw[kernels2] (0,0) -- (1,1) node[xic] {};
	\draw[kernels2] (-1,1) -- (0,2) node[xic] {};
	\draw[kernels2] (-2,2) -- (-1,3) node[xi] {};
}

\DeclareSymbol{I1Xi4abcc2}{2}{\draw[kernels2] (0,0) node[not] {} -- (-1,1) node[not] {}
	-- (-2,2) node[not]{} -- (-3,3) node[xic]  {};
	\draw[kernels2] (0,0) -- (1,1) node[xi] {};
	\draw[kernels2] (-1,1) -- (0,2) node[xi] {};
	\draw[kernels2] (-2,2) -- (-1,3) node[xic] {};
}

\DeclareSymbol{I1Xi4ac}{2}{\draw[kernels2] (0,0) node[not] {} -- (-1,1) ; \draw[kernels2] (0,0) node[not] {} -- (1,1) node[xi] {}; 
	\draw[kernels2] (-1,1) node[not] {} -- (0,2) ; 
	\draw[kernels2] (-1,1) -- (-2,2) node[xi] {} ;
	\draw (0,2) node[xi] {} -- (-1,3) node[xi] {};}

\DeclareSymbol{cI1Xi4ac}{2}{\draw[kernels2] (0,0) node[not] {} -- (-1,1) ; \draw[kernels2] (0,0) node[not] {} -- (1,1) node[xic] {}; 
	\draw[kernels2] (-1,1) node[not] {} -- (0,2) ; 
	\draw[kernels2] (-1,1) -- (-2,2) node[xic] {} ;
	\draw (0,2) node[xi] {} -- (-1,3) node[xi] {};}

\DeclareSymbol{I1Xi4acc1}{2}{\draw[kernels2] (0,0) node[not] {} -- (-1,1) ; \draw[kernels2] (0,0) node[not] {} -- (1,1) node[xic] {}; 
	\draw[kernels2] (-1,1) node[not] {} -- (0,2) ; 
	\draw[kernels2] (-1,1) -- (-2,2) node[xi] {} ;
	\draw (0,2) node[xic] {} -- (-1,3) node[xi] {};}

\DeclareSymbol{I1Xi4acc2}{2}{\draw[kernels2] (0,0) node[not] {} -- (-1,1) ; \draw[kernels2] (0,0) node[not] {} -- (1,1) node[xic] {}; 
	\draw[kernels2] (-1,1) node[not] {} -- (0,2) ; 
	\draw[kernels2] (-1,1) -- (-2,2) node[xi] {} ;
	\draw (0,2) node[xi] {} -- (-1,3) node[xic] {};}

\DeclareSymbol{2I1Xi4}{2}{\draw[kernels2] (0,0) node[not] {} -- (-1,1) node[not] {};
	\draw[kernels2] (0,0) -- (1,1) node[not] {};
	\draw[kernels2] (-1,1) -- (-1.5,2.5) node[xi] {};
	\draw[kernels2] (-1,1) -- (-0.5,2.5) node[xi] {};
	\draw[kernels2] (1,1) -- (0.5,2.5) node[xi] {};
	\draw[kernels2] (1,1) -- (1.5,2.5) node[xi] {};
}

\DeclareSymbol{2I1Xi4dis}{2}{\draw[kernels2] (0,0) node[not] {} -- (-1,1) node[not] {};
	\draw[kernels2] (0,0) -- (1,1) node[not] {};
	\draw[kernels2] (-1,1) -- (-1.5,2.5) node[xies] {};
	\draw[kernels2] (-1,1) -- (-0.5,2.5) node[xies] {};
	\draw[kernels2] (1,1) -- (0.5,2.5) node[xies] {};
	\draw[kernels2] (1,1) -- (1.5,2.5) node[xies] {};
}

\DeclareSymbol{2I1Xi4c1}{2}{\draw[kernels2] (0,0) node[not] {} -- (-1,1) node[not] {};
	\draw[kernels2] (0,0) -- (1,1) node[not] {};
	\draw[kernels2] (-1,1) -- (-1.5,2.5) node[xic] {};
	\draw[kernels2] (-1,1) -- (-0.5,2.5) node[xi] {};
	\draw[kernels2] (1,1) -- (0.5,2.5) node[xic] {};
	\draw[kernels2] (1,1) -- (1.5,2.5) node[xi] {};
}

\DeclareSymbol{2I1Xi4c2}{2}{\draw[kernels2] (0,0) node[not] {} -- (-1,1) node[not] {};
	\draw[kernels2] (0,0) -- (1,1) node[not] {};
	\draw[kernels2] (-1,1) -- (-1.5,2.5) node[xic] {};
	\draw[kernels2] (-1,1) -- (-0.5,2.5) node[xic] {};
	\draw[kernels2] (1,1) -- (0.5,2.5) node[xi] {};
	\draw[kernels2] (1,1) -- (1.5,2.5) node[xi] {};
}

\DeclareSymbol{2I1Xi4b}{2}{\draw[kernels2] (0,0) node[not] {} -- (-1,1) ;
	\draw[kernels2] (0,0) -- (1,1);
	\draw (-1,1) node[xi] {} -- (-1,2.5) node[xi] {};
	\draw (1,1)  node[xi] {} -- (1,2.5) node[xi] {};
}

\DeclareSymbol{2I1Xi4bb}{2}{\draw[kernels2] (0,0) node[not] {} -- (-1,1) ;
	\draw[kernels2] (0,0) -- (1,1);
	\draw (-1,1) node[xi] {} -- (-1,2.5) node[xiesf] {};
	\draw (1,1)  node[xi] {} -- (1,2.5) node[xic] {};
}

\DeclareSymbol{2I1Xi4c}{2}{\draw[kernels2] (0,0) node[not] {} -- (-1,1);
	\draw[kernels2] (0,0) -- (1,1) node[not] {};
	\draw (-1,1)  node[xi] {} -- (-1,2.5) node[xi] {};
	\draw[kernels2] (1,1) -- (0.4,2.5) node[xi] {};
	\draw[kernels2] (1,1) -- (1.6,2.5) node[xi] {};
}

\DeclareSymbol{2I1Xi4cc1}{2}{\draw[kernels2] (0,0) node[not] {} -- (-1,1);
	\draw[kernels2] (0,0) -- (1,1) node[not] {};
	\draw (-1,1)  node[xic] {} -- (-1,2.5) node[xi] {};
	\draw[kernels2] (1,1) -- (0.4,2.5) node[xic] {};
	\draw[kernels2] (1,1) -- (1.6,2.5) node[xi] {};
}

\DeclareSymbol{2I1Xi4cc2}{2}{\draw[kernels2] (0,0) node[not] {} -- (-1,1);
	\draw[kernels2] (0,0) -- (1,1) node[not] {};
	\draw (-1,1)  node[xic] {} -- (-1,2.5) node[xic] {};
	\draw[kernels2] (1,1) -- (0.4,2.5) node[xi] {};
	\draw[kernels2] (1,1) -- (1.6,2.5) node[xi] {};
}

\DeclareSymbol{Xi4ba}{0}{\draw(-0.5,1.5) node[xi] {} -- (0,0); \draw (-1.5,1) node[xi] {} -- (0,0) node[not] {}; \draw[kernels2] (0,0) -- (1.5,1) node[xi] {};
	\draw[kernels2] (0,0) -- (0.5,1.5) node[xi] {} ;}

\DeclareSymbol{Xi4badis}{0}{\draw(-0.5,1.5) node[xies] {} -- (0,0); \draw (-1.5,1) node[xies] {} -- (0,0) node[not] {}; \draw[kernels2] (0,0) -- (1.5,1) node[xies] {};
	\draw[kernels2] (0,0) -- (0.5,1.5) node[xies] {} ;}

\DeclareSymbol{Xi4ba1}{0}{\draw(-0.5,1.5) node[xi] {} -- (0,0); \draw (-1.5,1) node[xi] {} -- (0,0) node[not] {}; \draw[kernels2] (0,0) -- (1.5,1) node[xic] {};
	\draw[kernels2] (0,0) -- (0.5,1.5) node[xic] {} ;}

\DeclareSymbol{Xi4ba1b}{0}{\draw(-0.5,1.5) node[xic] {} -- (0,0); \draw (-1.5,1) node[xic] {} -- (0,0) node[not] {}; \draw[kernels2] (0,0) -- (1.5,1) node[xi] {};
	\draw[kernels2] (0,0) -- (0.5,1.5) node[xi] {} ;}

\DeclareSymbol{Xi4ba1bdiff}{0}{\draw(-0.5,1.5) node[xic] {} -- (0,0); \draw (-1.5,1) node[xic] {} -- (0,0) node[not] {}; \draw (0,0) -- (1.5,1) node[xi] {};
	\draw (0,0) -- (0.5,1.5) node[xi] {};
	\draw(0,0) node[diff] {};}

\DeclareSymbol{Xi4ba1bb}{0}{\draw(-0.5,1.5) node[xic] {} -- (0,0); \draw (-1.5,1) node[xiesf] {} -- (0,0) node[not] {}; \draw[kernels2] (0,0) -- (1.5,1) node[xi] {};
	\draw[kernels2] (0,0) -- (0.5,1.5) node[xi] {} ;}

\DeclareSymbol{Xi4ba2}{0}{\draw(-0.5,1.5) node[xi] {} -- (0,0); \draw (-1.5,1) node[xic] {} -- (0,0) node[not] {}; \draw[kernels2] (0,0) -- (1.5,1) node[xi] {};
	\draw[kernels2] (0,0) -- (0.5,1.5) node[xic] {} ;}

\DeclareSymbol{Xi4ba2b}{0}{\draw(-0.5,1.5) node[xi] {} -- (0,0); \draw (-1.5,1) node[xic] {} -- (0,0) node[not] {}; \draw[kernels2] (0,0) -- (1.5,1) node[xi] {};
	\draw[kernels2] (0,0) -- (0.5,1.5) node[xiesf] {} ;}

\DeclareSymbol{Xi4ca}{0}{\draw (0,1) -- (-1,2.2) node[xi] {};\draw (0,-0.25) node[xi] {} -- (0,1) ; \draw[kernels2] (0,1) node[not] {} -- (1,2.2) node[xi] {};
	\draw[kernels2] (0,1) {} -- (0,2.7) node[xi] {};
}

\DeclareSymbol{Xi4cb}{0}{\draw (-1,1) -- (-2,2) node[xi] {};\draw[kernels2] (0,0)  -- (-1,1) node[xi] {} ; \draw[kernels2] (0,0) node[not] {} -- (1,1) node[xi] {} ; 
	\draw (-1,1) node[xi] {} -- (0,2) node[xi] {};}

\DeclareSymbol{Xi4cbb}{0}{\draw (-1,1) -- (-2,2) node[xiesf] {};\draw[kernels2] (0,0)  -- (-1,1) node[xi] {} ; \draw[kernels2] (0,0) node[not] {} -- (1,1) node[xi] {} ; 
	\draw (-1,1) node[xi] {} -- (0,2) node[xic] {};}

\DeclareSymbol{Xi4cbc1}{0}{\draw (-1,1) -- (-2,2) node[xic] {};\draw[kernels2] (0,0)  -- (-1,1) node[xic] {} ; \draw[kernels2] (0,0) node[not] {} -- (1,1) node[xi] {} ; 
	\draw (-1,1) node[xic] {} -- (0,2) node[xi] {};}

\DeclareSymbol{Xi4cbc2}{0}{\draw (-1,1) -- (-2,2) node[xi] {};\draw[kernels2] (0,0)  -- (-1,1) node[xi] {} ; \draw[kernels2] (0,0) node[not] {} -- (1,1) node[xic] {} ; 
	\draw (-1,1) node[xic] {} -- (0,2) node[xi] {};}

\DeclareSymbol{Xi4cab}{0}{\draw (-1,1) -- (-2,2) node[xi] {};\draw[kernels2] (0,0)  -- (-1,1); \draw[kernels2] (0,0) node[not] {} -- (1,1) node[xi] {} ; 
	\draw[kernels2] (-1,1)  {} -- (0,2) node[xi] {};
	\draw[kernels2] (-1,1) node[not] {} -- (-1,2.5) node[xi] {};
}

\DeclareSymbol{Xi4cabdis}{0}{\draw (-1,1) -- (-2,2) node[xies] {};\draw[kernels2] (0,0)  -- (-1,1); \draw[kernels2] (0,0) node[not] {} -- (1,1) node[xies] {} ; 
	\draw[kernels2] (-1,1)  {} -- (0,2) node[xies] {};
	\draw[kernels2] (-1,1) node[not] {} -- (-1,2.5) node[xies] {};
}

\DeclareSymbol{Xi4cabc1}{0}{\draw (-1,1) -- (-2,2) node[xi] {};\draw[kernels2] (0,0)  -- (-1,1); \draw[kernels2] (0,0) node[not] {} -- (1,1) node[xic] {} ; 
	\draw[kernels2] (-1,1)  {} -- (0,2) node[xic] {};
	\draw[kernels2] (-1,1) node[not] {} -- (-1,2.5) node[xi] {};
}

\DeclareSymbol{Xi4cabc2}{0}{\draw (-1,1) -- (-2,2) node[xic] {};\draw[kernels2] (0,0)  -- (-1,1); \draw[kernels2] (0,0) node[not] {} -- (1,1) node[xic] {} ; 
	\draw[kernels2] (-1,1)  {} -- (0,2) node[xi] {};
	\draw[kernels2] (-1,1) node[not] {} -- (-1,2.5) node[xi] {};
}

\DeclareSymbol{Xi4ea}{1.5}{\draw (-1,2.5) node[xi] {} -- (-1,1) node[xi] {} -- (0,0); 
	\draw[kernels2] (0,0)  -- (1,1) node[xi] {};
	\draw[kernels2] (0,0) node[not] {} -- (0,1.5) node[xi] {}; }

\DeclareSymbol{Xi4eac1}{1.5}{\draw (-1,2.5) node[xic] {} -- (-1,1) node[xi] {} -- (0,0); 
	\draw[kernels2] (0,0)  -- (1,1) node[xic] {};
	\draw[kernels2] (0,0) node[not] {} -- (0,1.5) node[xi] {}; }

\DeclareSymbol{Xi4eac1b}{1.5}{\draw (-1,2.5) node[xic] {} -- (-1,1) node[xi] {} -- (0,0); 
	\draw[kernels2] (0,0)  -- (1,1) node[xiesf] {};
	\draw[kernels2] (0,0) node[not] {} -- (0,1.5) node[xi] {}; }

\DeclareSymbol{Xi4eac2}{1.5}{\draw (-1,2.5) node[xic] {} -- (-1,1) node[xic] {} -- (0,0); 
	\draw[kernels2] (0,0)  -- (1,1) node[xi] {};
	\draw[kernels2] (0,0) node[not] {} -- (0,1.5) node[xi] {}; }

\DeclareSymbol{Xi4eact1}{1.5}{\draw (-1,2.5) node[xic] {} -- (-1,1) node[xi] {} -- (0,0); 
	\draw (0,0)  -- (1,1) node[xic] {};
	\draw[rho] (0,0) node[not] {} -- (0,1.5) node[xi] {}; }

\DeclareSymbol{Xi4eact2}{1.5}{\draw[rho] (-1,2.5) node[xic] {} -- (-1,1) node[xi] {} -- (0,0); 
	\draw (0,0)  -- (1,1) node[xic] {};
	\draw (0,0) node[not] {} -- (0,1.5) node[xi] {}; }

\DeclareSymbol{Xi4eabis}{1.5}{\draw (-1,2.5) node[xi] {} -- (-1,1) ; \draw[kernels2] (-1,1) node[xi] {} -- (0,0); 
	\draw (0,0)  -- (1,1) node[xi] {};
	\draw[kernels2] (0,0) node[not] {} -- (0,1.5) node[xi] {}; }

\DeclareSymbol{Xi4eabisc1}{1.5}{\draw (-1,2.5) node[xic] {} -- (-1,1) ; \draw[kernels2] (-1,1) node[xi] {} -- (0,0); 
	\draw (0,0)  -- (1,1) node[xi] {};
	\draw[kernels2] (0,0) node[not] {} -- (0,1.5) node[xic] {}; }

\DeclareSymbol{Xi4eabisc1b}{1.5}{\draw (-1,2.5) node[xic] {} -- (-1,1) ; \draw[kernels2] (-1,1) node[xi] {} -- (0,0); 
	\draw (0,0)  -- (1,1) node[xi] {};
	\draw[kernels2] (0,0) node[not] {} -- (0,1.5) node[xiesf] {}; }

\DeclareSymbol{Xi4eabisc1bis}{1.5}{\draw (-1,2.5) node[xi] {} -- (-1,1) ; \draw[kernels2] (-1,1) node[xi] {} -- (0,0); 
	\draw (0,0)  -- (1,1) node[xi] {};
	\draw[kernels2] (0,0) node[not] {} -- (0,1.5) node[xi] {};
	\draw (-2,1) node[] {\tiny{$i$}};
	\draw (-2,2.5) node[] {\tiny{$\ell$}};
	\draw (2,1) node[] {\tiny{$k$}};
	\draw (0,2.5) node[] {\tiny{$j$}};
}

\DeclareSymbol{Xi4eabisc1tris}{1.5}{\draw (-1,2.5) node[xi] {} -- (-1,1) ; \draw[kernels2] (-1,1) node[xi] {} -- (0,0); 
	\draw (0,0)  -- (1,1) node[xi] {};
	\draw[kernels2] (0,0) node[not] {} -- (0,1.5) node[xi] {};
	\draw (-2,1) node[] {\tiny{i}};
	\draw (-2,2.5) node[] {\tiny{j}};
	\draw (2,1) node[] {\tiny{j}};
	\draw (0,2.5) node[] {\tiny{i}};
}

\DeclareSymbol{Xi4eabisc1quater}{1.5}{\draw (-1,2.5) node[xic] {} -- (-1,1) ; \draw[kernels2] (-1,1) node[xi] {} -- (0,0); 
	\draw (0,0)  -- (1,1) node[xic] {};
	\draw[kernels2] (0,0) node[not] {} -- (0,1.5) node[xi] {};
}

\DeclareSymbol{Xi4eabisc2}{1.5}{\draw (-1,2.5) node[xic] {} -- (-1,1) ; \draw[kernels2] (-1,1) node[xi] {} -- (0,0); 
	\draw (0,0)  -- (1,1) node[xic] {};
	\draw[kernels2] (0,0) node[not] {} -- (0,1.5) node[xi] {}; }

\DeclareSymbol{Xi4eabisc2l}{1.5}{\draw (-1,2.5) node[xiesf] {} -- (-1,1) ; \draw[kernels2] (-1,1) node[xi] {} -- (0,0); 
	\draw (0,0)  -- (1,1) node[xic] {};
	\draw[kernels2] (0,0) node[not] {} -- (0,1.5) node[xi] {}; }

\DeclareSymbol{Xi4eabisc2r}{1.5}{\draw (-1,2.5) node[xic] {} -- (-1,1) ; \draw[kernels2] (-1,1) node[xi] {} -- (0,0); 
	\draw (0,0)  -- (1,1) node[xiesf] {};
	\draw[kernels2] (0,0) node[not] {} -- (0,1.5) node[xi] {}; }
	
	\DeclareSymbol{Xiii}{0}{\draw (0,0) node[xi] {} -- (-1,1) node[xi] {} -- (0,2) node[xi] {};
	}
	
	\DeclareSymbol{Xiiibis}{0}{\draw (0,0) node[not] {} -- (-1,1) node[xi] {} -- (0,2) node[xi] {};
	}
		\DeclareSymbol{iXiiibis}{0}{\draw[kernels2] (0,0) node[not] {} -- (-1,1) {}; \draw (-1,1) node[xi] {} -- (0,2) node[xi] {};
 xi] {};
	}
	
	\DeclareSymbol{Xiiiaa}{0}{
		\draw (-1,1)  node[xi] {} -- (0,0); 
		\draw (0,0) node[not] {}  -- (1,1) node[xi] {};
	}
	
	\DeclareSymbol{Xiiia}{0}{
		\draw (-1,1)  node[xi] {} -- (0,0); 
		\draw (0,0) node[xi] {}  -- (1,1) node[xi] {};
	}
	
		\DeclareSymbol{Xiiiag1}{0}{
		\draw (-1,1)  node[xide1] {} -- (0,0); 
		\draw (0,0) node[xi] {}  -- (1,1) node[xi] {};
	}
	
		\DeclareSymbol{Xiiiag2}{0}{
		\draw (-1,1)  node[xi] {} -- (0,0); 
		\draw (0,0) node[xi] {}  -- (1,1) node[xide1] {};
	}
	
		\DeclareSymbol{Xiiiagg2}{0}{
		\draw (-1,1)  node[xi] {} -- (0,0); 
		\draw (0,0) node[xi] {}  -- (1,1) node[xide] {};
	}
	
		\DeclareSymbol{Xiiiag3}{0}{
		\draw (-1,1)  node[xi] {} -- (0,0); 
		\draw (0,0) node[xide1] {}  -- (1,1) node[xi] {};
	}

\DeclareSymbol{Xi4eabisc3}{1.5}{\draw (-1,2.5) node[xic] {} -- (-1,1) ; \draw[kernels2] (-1,1) node[xic] {} -- (0,0); 
	\draw (0,0)  -- (1,1) node[xi] {};
	\draw[kernels2] (0,0) node[not] {} -- (0,1.5) node[xi] {}; }

\DeclareSymbol{Xi4eb}{0}{
	\draw[kernels2] (0,2) node[xi] {} -- (-1,1) ; \draw[kernels2] (-2,2)  node[xi] {} -- (-1,1) ; \draw (-1,1)  node[not] {} -- (0,0); 
	\draw (0,0) node[xi] {}  -- (1,1) node[xi] {};
}

\DeclareSymbol{Xi4eab}{1.5}{\draw[kernels2] (-1,2.5) node[xi] {} -- (-1,1) ; \draw[kernels2] (-2,2)  node[xi] {} -- (-1,1) ; \draw (-1,1)  node[not] {} -- (0,0); 
	\draw[kernels2] (0,0)  -- (1,1) node[xi] {};
	\draw[kernels2] (0,0) node[not] {} -- (0,1.5) node[xi] {}; 
}

\DeclareSymbol{Xi4eabdis}{1.5}{\draw[kernels2] (-1,2.5) node[xies] {} -- (-1,1) ; \draw[kernels2] (-2,2)  node[xies] {} -- (-1,1) ; \draw (-1,1)  node[not] {} -- (0,0); 
	\draw[kernels2] (0,0)  -- (1,1) node[xies] {};
	\draw[kernels2] (0,0) node[not] {} -- (0,1.5) node[xies] {}; 
}

\DeclareSymbol{Xi4eabc1}{1.5}{\draw[kernels2] (-1,2.5) node[xic] {} -- (-1,1) ; \draw[kernels2] (-2,2)  node[xi] {} -- (-1,1) ; \draw (-1,1)  node[not] {} -- (0,0); 
	\draw[kernels2] (0,0)  -- (1,1) node[xic] {};
	\draw[kernels2] (0,0) node[not] {} -- (0,1.5) node[xi] {}; 
}

\DeclareSymbol{Xi4eabc2}{1.5}{\draw[kernels2] (-1,2.5) node[xi] {} -- (-1,1) ; \draw[kernels2] (-2,2)  node[xi] {} -- (-1,1) ; \draw (-1,1)  node[not] {} -- (0,0); 
	\draw[kernels2] (0,0)  -- (1,1) node[xic] {};
	\draw[kernels2] (0,0) node[not] {} -- (0,1.5) node[xic] {}; 
}

\DeclareSymbol{Xi4eabbis}{1.5}{\draw[kernels2] (-1,2.5) node[xi] {} -- (-1,1) ; \draw[kernels2] (-2,2)  node[xi] {} -- (-1,1) ; \draw[kernels2] (-1,1)  node[not] {} -- (0,0); 
	\draw (0,0)  -- (1,1) node[xi] {};
	\draw[kernels2] (0,0) node[not] {} -- (0,1.5) node[xi] {}; 
}

\DeclareSymbol{Xi4eabbisc1}{1.5}{\draw[kernels2] (-1,2.5) node[xic] {} -- (-1,1) ; \draw[kernels2] (-2,2)  node[xi] {} -- (-1,1) ; \draw[kernels2] (-1,1)  node[not] {} -- (0,0); 
	\draw (0,0)  -- (1,1) node[xic] {};
	\draw[kernels2] (0,0) node[not] {} -- (0,1.5) node[xi] {}; 
}

\DeclareSymbol{Xi4eabbisc1perm}{1.5}{\draw[kernels2] (-1,2.5) node[xi] {} -- (-1,1) ; \draw[kernels2] (-2,2)  node[xic] {} -- (-1,1) ; \draw[kernels2] (-1,1)  node[not] {} -- (0,0); 
	\draw (0,0)  -- (1,1) node[xic] {};
	\draw[kernels2] (0,0) node[not] {} -- (0,1.5) node[xi] {}; 
}

\DeclareSymbol{Xi4eabbisc2}{1.5}{\draw[kernels2] (-1,2.5) node[xi] {} -- (-1,1) ; \draw[kernels2] (-2,2)  node[xi] {} -- (-1,1) ; \draw[kernels2] (-1,1)  node[not] {} -- (0,0); 
	\draw (0,0)  -- (1,1) node[xic] {};
	\draw[kernels2] (0,0) node[not] {} -- (0,1.5) node[xic] {}; 
}

\DeclareSymbol{Xi2cbis}{0}{\draw[kernels2] (0,1) -- (0.8,2.2) node[xi] {};\draw[kernels2] (0,-0.25) node[not] {} -- (0,1); \draw[kernels2] (0,1) node[not] {} -- (-0.8,2.2) node[xi] {};}

\DeclareSymbol{Xi2cbis1}{0}{\draw (0,1) -- (-0.8,2.2) node[xi] {};\draw[kernels2] (0,-0.25) node[not] {} -- (0,1) node[xi] {}; }

\DeclareSymbol{Xi2Xbis}{-2}{\draw[kernels2] (0,-0.25)  -- (-1,1) ; \draw (-1,1) node[xix] {};
	\draw[kernels2] (0,-0.25) node[not] {} -- (1,1) node[xi] {};}

\DeclareSymbol{XXi2bis}{-2}{\draw[kernels2] (0,-0.25) -- (-1,1) node[xi] {};
	\draw[kernels2] (0,-0.25) node[X] {} -- (1,1) node[xi] {};}

\DeclareSymbol{I1XiIXi}{0}{\draw[kernels2] (0,-0.25) -- (1,1) node[xi] {};
	\draw (0,-0.25) node[not] {} -- (-1,1) node[xi] {};}

\DeclareSymbol{I1XiIXib}{0}{\draw  (0,-0.25) node[xi] {} -- (0,1) node[not] {};
	\draw[kernels2] (0,1) -- (0,2.25) ; \draw (0,2.25) node[xi]{}; }

\DeclareSymbol{I1XiIXic}{0}{
	\draw[kernels2] (0,0) -- (1,1) node[xi] {} ; 
	\draw[kernels2] (0,0) node[not] {}  -- (-1,1) node[not] {} -- (0,2) node[xi] {};
}

\DeclareSymbol{thin}{1.4}{\draw[pagebackground] (-0.3,0) -- (0.3,0); \draw  (0,0) -- (0,2);}
\DeclareSymbol{thin2}{1.4}{\draw[pagebackground] (-0.3,0) -- (0.3,0); \draw[tinydots]  (0,0) -- (0,2);}

\DeclareSymbol{thick}{1.4}{\draw[pagebackground] (-0.3,0) -- (0.3,0); \draw[kernels2]  (0,0) -- (0,2);}

\DeclareSymbol{thick2}{1.4}{\draw[pagebackground] (-0.3,0) -- (0.3,0); \draw[kernels2,tinydots]  (0,0) -- (0,2);}

\DeclareSymbol{Xi4ind}{2}{\draw (0,0) node[xi,label={[label distance=-0.2em]right: \scriptsize  $ i $}]  { } -- (-1,1) node[xi,label={[label distance=-0.2em]left: \scriptsize  $ j $}] {} -- (0,2) node[xi,label={[label distance=-0.2em]right: \scriptsize  $ k $}] {} -- (-1,3) node[xi,label={[label distance=-0.2em]left: \scriptsize  $ \ell $}] {};}

\DeclareSymbol{Xi4c1}{2}{\draw (0,0) node[xic] {} -- (-1,1) node[xi] {} -- (0,2) node[xic] {} -- (-1,3) node[xi] {};} 
\DeclareSymbol{IXi2ex}{0}{\draw (0,-0.25) node[xie] {} -- (-1,1) node[xi] {} ; \draw (0,-0.25)-- (1,1) node[xi] {};}
\DeclareSymbol{IXi2ex1}{0}{\draw (0,-0.25) node[xie] {} -- (-1,1) node[xi] {} -- (0,2) node[xi] {};}

\DeclareSymbol{Xi4b1}{0}{\draw(0,1.5) node[xic] {} -- (0,0); \draw (-1,1) node[xic] {} -- (0,0) node[xi] {} -- (1,1) node[xi] {};}

\DeclareSymbol{Xi4ec1}{0}{\draw (0,2) node[xi] {} -- (-1,1) node[xic] {} -- (0,0) node[xic] {} -- (1,1) node[xi] {};}
\DeclareSymbol{Xi4ec2}{0}{\draw (0,2) node[xic] {} -- (-1,1) node[xi] {} -- (0,0) node[xic] {} -- (1,1) node[xi] {};}
\DeclareSymbol{Xi4ec3}{0}{\draw (0,2) node[xic] {} -- (-1,1) node[xic] {} -- (0,0) node[xi] {} -- (1,1) node[xi] {};}

\DeclareSymbol{I1Xi4ac1}{2}{\draw[kernels2] (0,0) node[not] {} -- (-1,1) ; \draw[kernels2] (0,0) node[not] {} -- (1,1) node[xic] {} ;
	\draw (-1,1) node[xi] {} -- (0,2) node[xic] {} -- (-1,3) node[xi] {};}

\DeclareSymbol{I1Xi4ac2}{2}{\draw[kernels2] (0,0) node[not] {} -- (-1,1) ; \draw[kernels2] (0,0) node[not] {} -- (1,1) node[xic] {} ;
	\draw (-1,1) node[xi] {} -- (0,2) node[xi] {} -- (-1,3) node[xic] {};}

\DeclareSymbol{I1Xi4bp}{2}{\draw (0,0) node[not] {} -- (-1,1) node[xi] {} -- (0,2) ; \draw[kernels2] (0,2) node[not] {} -- (-1,3) node[xi] {};\draw[kernels2] (0,2)  -- (1,3) node[xi] {};
}

\DeclareSymbol{I1Xi4bc1}{2}{\draw (0,0) node[xic] {} -- (-1,1) node[xi] {} -- (0,2) ; \draw[kernels2] (0,2) node[not] {} -- (-1,3) node[xi] {};\draw[kernels2] (0,2)  -- (1,3) node[xic] {};
}

\DeclareSymbol{I1Xi4bc2}{2}{\draw (0,0) node[xic] {} -- (-1,1) node[xi] {} -- (0,2) ; \draw[kernels2] (0,2) node[not] {} -- (-1,3) node[xic] {};\draw[kernels2] (0,2)  -- (1,3) node[xi] {};
}

\DeclareSymbol{I1Xi4cp}{2}{\draw (0,0) node[not] {} -- (-1,1) node[not] {}; \draw[kernels2] (-1,1) -- (0,2) ; 
	\draw[kernels2] (-1,1) -- (-2,2) node[xi] {} ;
	\draw (0,2) node[xi] {} -- (-1,3) node[xi] {};}

\DeclareSymbol{I1Xi4cc1}{2}{\draw (0,0) node[xic] {} -- (-1,1) node[not] {}; \draw[kernels2] (-1,1) -- (0,2) ; 
	\draw[kernels2] (-1,1) -- (-2,2) node[xi] {} ;
	\draw (0,2) node[xic] {} -- (-1,3) node[xi] {};}

\DeclareSymbol{I1Xi4cc2}{2}{\draw (0,0) node[xic] {} -- (-1,1) node[not] {}; \draw[kernels2] (-1,1) -- (0,2) ; 
	\draw[kernels2] (-1,1) -- (-2,2) node[xi] {} ;
	\draw (0,2) node[xi] {} -- (-1,3) node[xic] {};}

\DeclareSymbol{I1Xi4abc1}{2}{\draw[kernels2] (0,0) node[not] {} -- (-1,1) ; \draw[kernels2] (0,0) node[not] {} -- (1,1) node[xic] {};\draw (-1,1) node[xi] {} -- (0,2) ; \draw[kernels2] (0,2) node[not] {} -- (-1,3) node[xic] {};\draw[kernels2] (0,2)  -- (1,3) node[xi] {}; }

\DeclareSymbol{I1Xi4abc2}{2}{\draw[kernels2] (0,0) node[not] {} -- (-1,1) ; \draw[kernels2] (0,0) node[not] {} -- (1,1) node[xic] {};\draw (-1,1) node[xi] {} -- (0,2) ; \draw[kernels2] (0,2) node[not] {} -- (-1,3) node[xi] {};\draw[kernels2] (0,2)  -- (1,3) node[xic] {}; }

\DeclareSymbol{R1}{0}{\draw (-1,1) node[xi] {} -- (0,0) node[not] {};
	\draw[kernels2] (0,1.5) node[xic] {} -- (0,0) -- (1,1) node[xic] {};}
\DeclareSymbol{R2}{0}{\draw (-1,1) node[xic] {} -- (0,0) node[not] {};
	\draw[kernels2] (0,1.5)  {} -- (0,0) -- (1,1)  {};
	\draw (0,1.5) node[xi] {};
	\draw (1,1) node[xic] {};
}
\DeclareSymbol{R3}{1}{\draw[kernels2] (-1,1.5)  {} -- (0,0) node[not] {} -- (1,1.5);
	\draw (-1,1.5) node[xi] {};
	\draw[kernels2] (0,3) {} -- (1,1.5) -- (2,3)  {};
	\draw  (0,3) node[xic] {} ;
	\draw (2,3) node[xic] {};}
\DeclareSymbol{R4}{1}{\draw[kernels2] (-1,1.5) node[xic] {} -- (0,0) node[not] {} -- (1,1.5);
	\draw[kernels2] (0,3) {} -- (1,1.5) -- (2,3) node[xic] {};
	\draw (0,3) node[xi] {};}

\DeclareSymbol{I1Xi4bcp}{2}{\draw (0,0) node[not] {} -- (-1,1) node[not] {}; \draw[kernels2] (-1,1) -- (0,2) ; 
	\draw[kernels2] (-1,1) -- (-2,2) node[xi] {} ; \draw[kernels2] (0,2) node[not] {} -- (-1,3) node[xi] {};\draw[kernels2] (0,2)  -- (1,3) node[xi] {};
}

\DeclareSymbol{I1Xi4bcc1}{2}{\draw (0,0) node[xic] {} -- (-1,1) node[not] {}; \draw[kernels2] (-1,1) -- (0,2) ; 
	\draw[kernels2] (-1,1) -- (-2,2) node[xi] {} ; \draw[kernels2] (0,2) node[not] {} -- (-1,3) node[xi] {};\draw[kernels2] (0,2)  -- (1,3) node[xic] {};
}

\DeclareSymbol{I1Xi4bcc2}{2}{\draw (0,0) node[xic] {} -- (-1,1) node[not] {}; \draw[kernels2] (-1,1) -- (0,2) ; 
	\draw[kernels2] (-1,1) -- (-2,2) node[xi] {} ; \draw[kernels2] (0,2) node[not] {} -- (-1,3) node[xic] {};\draw[kernels2] (0,2)  -- (1,3) node[xi] {};
} 

\DeclareSymbol{2I1Xi4bc1}{2}{\draw[kernels2] (0,0) node[not] {} -- (-1,1) ;
	\draw[kernels2] (0,0) -- (1,1);
	\draw (-1,1) node[xic] {} -- (-1,2.5) node[xi] {};
	\draw (1,1)  node[xic] {} -- (1,2.5) node[xi] {};
}

\DeclareSymbol{2I1Xi4bc2}{2}{\draw[kernels2] (0,0) node[not] {} -- (-1,1) ;
	\draw[kernels2] (0,0) -- (1,1);
	\draw (-1,1) node[xi] {} -- (-1,2.5) node[xic] {};
	\draw (1,1)  node[xic] {} -- (1,2.5) node[xi] {};
}

\DeclareSymbol{diff2I1Xi4bc2}{2}{\draw (0,0) node[diff] {} -- (-1,1) ;
	\draw (0,0) -- (1,1);
	\draw (-1,1) node[xi] {} -- (-1,2.5) node[xic] {};
	\draw (1,1)  node[xic] {} -- (1,2.5) node[xi] {};
}

\DeclareSymbol{2I1Xi4bc3}{2}{\draw[kernels2] (0,0) node[not] {} -- (-1,1) ;
	\draw[kernels2] (0,0) -- (1,1);
	\draw (-1,1) node[xic] {} -- (-1,2.5) node[xic] {};
	\draw (1,1)  node[xi] {} -- (1,2.5) node[xi] {};
}

\DeclareSymbol{Xi41}{0}{\draw (0,1) -- (0.8,2.2) node[xic] {};\draw (0,-0.25) node[xi] {} -- (0,1) node[xi] {} -- (-0.8,2.2) node[xic] {};} 

\DeclareSymbol{Xi42}{0}{\draw (0,1) -- (0.8,2.2) node[xi] {};\draw (0,-0.25) node[xic] {} -- (0,1) node[xi] {} -- (-0.8,2.2) node[xic] {};}

\DeclareSymbol{Xi4ca1}{0}{\draw (0,1) -- (-1,2.2) node[xic] {};\draw (0,-0.25) node[xi] {} -- (0,1) ; \draw[kernels2] (0,1) node[not] {} -- (1,2.2) node[xic] {};
	\draw[kernels2] (0,1) {} -- (0,2.7) node[xi] {};
}

\DeclareSymbol{Xi4ca2}{0}{\draw (0,1) -- (-1,2.2) node[xi] {};\draw (0,-0.25) node[xi] {} -- (0,1) ; \draw[kernels2] (0,1) node[not] {} -- (1,2.2) node[xic] {};
	\draw[kernels2] (0,1) {} -- (0,2.7) node[xic] {};
}

\DeclareSymbol{Xi4cap}{0}{\draw (0,1) -- (-1,2.2) node[xi] {};\draw (0,-0.25) node[not] {} -- (0,1) ; \draw[kernels2] (0,1) node[not] {} -- (1,2.2) node[xi] {};
	\draw[kernels2] (0,1) {} -- (0,2.7) node[xi] {};
}

\DeclareSymbol{Xi3a}{0}{
	\draw (-1,1)  node[xi] {} -- (0,0); 
	\draw (0,0) node[xi] {}  -- (1,1) node[xi] {};
}

\DeclareSymbol{Xi4ebc1}{0}{
	\draw[kernels2] (0,2) node[xi] {} -- (-1,1) ; \draw[kernels2] (-2,2)  node[xic] {} -- (-1,1) ; \draw (-1,1)  node[not] {} -- (0,0); 
	\draw (0,0) node[xic] {}  -- (1,1) node[xi] {};
}

\DeclareSymbol{Xi4ebc2}{0}{
	\draw[kernels2] (0,2) node[xi] {} -- (-1,1) ; \draw[kernels2] (-2,2)  node[xi] {} -- (-1,1) ; \draw (-1,1)  node[not] {} -- (0,0); 
	\draw (0,0) node[xic] {}  -- (1,1) node[xic] {};
}

\DeclareSymbol{Xi2cbispex}{0}{\draw[kernels2] (0,1) -- (0.8,2.2) node[xi] {};\draw (0,-0.25) node[xie] {} -- (0,1); \draw[kernels2] (0,1) node[not] {} -- (-0.8,2.2) node[xi] {};}

\DeclareSymbol{Xi2cbis1p}{0}{\draw (0,1) -- (-0.8,2.2) node[xi] {};\draw (0,-0.25) node[not] {} -- (0,1) node[xi] {}; }

\DeclareSymbol{Xi2Xp}{-2}{\draw (0,-0.25) node[not] {} -- (-1,1) node[xix] {};} 

\DeclareSymbol{I1XiIXib}{0}{\draw  (0,-0.25) node[xi] {} -- (0,1) node[not] {};
	\draw[kernels2] (0,1) -- (0,2.25) ; \draw (0,2.25) node[xi]{}; }

\DeclareSymbol{IXi2b}{0}{\draw  (0,-0.25) node[xi] {} -- (0,1) node[not] {};
	\draw (0,1) -- (0,2.25) ; \draw (0,2.25) node[xi]{}; }

\DeclareSymbol{IXi2bex}{0}{\draw  (0,-0.25) node[xi] {} -- (0,1) node[xie] {};
	\draw (0,1) -- (0,2.25) ; \draw (0,2.25) node[xi]{}; }

\def\1{\mathbf{\symbol{1}}}

\def\one{\mathbf{1}}
\def\eps{\varepsilon}

\DeclareSymbol{diff}{0}{
	\draw (0,0.5) node[diff] {};
}

\DeclareSymbol{diff1}{0}{
	\draw (0,0.5) node[diff1] {};
}

\DeclareSymbol{diff2}{0}{
	\draw (0,0.5) node[diff2] {};
}

\DeclareSymbol{geo}{0}{
	\draw (0,0) node[diff] {};
	\draw (0.3,0) node[diff] {};
}

\DeclareSymbol{generic}{0}{
	\draw (0,0.6) node[xi] {};
}

\DeclareSymbol{genericD}{0}{
	\draw (0,0.6) node[xide] {};
}

\DeclareSymbol{genericDj}{0}{
	\draw (0,0.6) node[xide1] {};
}

\DeclareSymbol{genericx}{0}{
	\draw (0,0.6) node[xix] {};
}

\DeclareSymbol{g}{0}{
	\draw (0,0.6) node[g] {};
}

\DeclareSymbol{Ito}{0}{
	\draw (0,0.6) node[xies] {};
}

\DeclareSymbol{Itob}{0}{
	\draw (0,0.6) node[xiesf] {};
}

\DeclareSymbol{greycirc}{0}{
	\draw (0,0.3) node[xi] {};
}

\DeclareSymbol{not}{0}{
	\draw (0,0.6) node[not] {};
	\draw[tinydots] (0,0.6) circle (0.8);
}

\DeclareSymbol{genericb}{0}{
	\draw (0,0.6) node[xic] {};
}

\DeclareSymbol{bluecirc}{0}{
	\draw (0,0.3) node[xic] {};
}

\DeclareSymbol{genericxix}{0}{
	\draw (0,0.6) node[xix] {};
}

\DeclareSymbol{genericX}{0}{
	\draw (0,0.6) node[X] {};
}

\DeclareSymbol{diffIto}{1}{
	\draw  (0,2.5) -- (0,0) ;
	\draw (0,-0.1) node[diff] {};
	\draw (0,2.5) node[xies] {};
}
\DeclareSymbol{Itodiff}{2}{
	\draw(0,2.9) -- (0,-0.2);
	\draw (0,2.9) node[diff] {};
	\draw (0,-0.1) node[xies] {};
}

\DeclareSymbol{diffgeneric}{1}{
	\draw  (0,2.5) -- (0,0) ;
	\draw (0,-0.1) node[diff] {};
	\draw (0,2.5) node[xi] {};
}

\DeclareSymbol{genericdiff}{2}{
	\draw(0,2.9) -- (0,-0.2);
	\draw (0,2.9) node[diff] {};
	\draw (0,-0.1) node[xi] {};
}

\DeclareSymbol{diffdot}{2}{
	\draw  (0,3) -- (0,-0.1) ;
	\draw (0,3) node[not] {};
	\draw (0,-0.1) node[diff] {};
}

\DeclareSymbol{diffdotmini}{0}{
	\draw  (0,0) -- (0,1.2) ;
	\draw (0,1.2) node[not] {};
	\draw (0,0) node[diffmini] {};
}

\DeclareSymbol{dotdiff}{2}{
	\draw[kernelsmod]  (0,3) -- (0,-0.1) ;
	\draw (0,3) node[diff] {};
	\draw (0,-0.1) node[not] {};
}

\DeclareSymbol{dotdiff1}{2}{
	\draw[kernelsmod]  (0,3) -- (0,-0.1) ;
	\draw (0,3) node[diff1] {};
	\draw (0,-0.1) node[not] {};
}

\DeclareSymbol{dotdiff1mini}{0}{
	\draw[kernelsmod]  (0,1.2) -- (0,0) ;
	\draw (0,1.2) node[diffmini] {};
	\draw (0,0) node[not] {};
}

\DeclareSymbol{dotdiff2}{2}{
	\draw (0,3) -- (0,-0.1) ;
	\draw (0,3) node[diff] {};
	\draw (0,-0.1) node[not] {};
}

\DeclareSymbol{dotdiff2mini}{0}{
	\draw (0,1.2) -- (0,0) ;
	\draw (0,1.2) node[diffmini] {};
	\draw (0,0) node[not] {};
}

\DeclareSymbol{dotdiffstraight}{0}{
	\draw  (0,3) -- (0,-0.1) ;
	\draw (0,3) node[diff] {};
	\draw (0,-0.1) node[not] {};
}

\DeclareSymbol{arbre1}{0}{
	\draw  (0,0) -- (1.5,1.5) ;
	\draw (1.5,1.5) node[not] {};
	\draw (0,0) node[not] {};
}

\DeclareSymbol{arbre2}{0}{
	\draw  (0,0) -- (1.5,1.5) ;
	\draw[kernelsmod] (0,0) -- (-1.5,1.5);
	\draw (1.5,1.5) node[not] {};
	\draw (0,0) node[not] {};
	\draw (-1.5,1.5) node[xi] {};
}

\DeclareSymbol{arbre3}{0}{
	\draw  (0,0) -- (1.5,1.5) ;
	\draw[kernelsmod] (1.5,1.5) -- (0,3);
	\draw (0,0) node[not] {};
	\draw (1.5,1.5) node[not] {};
	\draw (0,3) node[xi] {};
}

\DeclareSymbol{treeeval}{0}{
	\draw (0,0) -- (1,1);
	\draw (0,0) node[xi] {};
	\draw (1.25,1.25) node[xi] {};
	\draw (-0.6,0.6) node[]{\tiny{$i$}};
	\draw (0.65,1.85) node[]{\tiny{$j$}};
}

\DeclareSymbol{testeval}{0}{
	\draw (0,0) -- (1,1);
	\draw (0,0) -- (-1,1);
	\draw (0,0) node[xi] {};
	\draw (1.25,1.25) node[xi] {};
	\draw (-1.25,1.25) node[xi] {};
	\draw (-0.6,-0.6) node[]{\tiny{$i$}};
	\draw (0.65,1.85) node[]{\tiny{$j$}};
	\draw (-1.95,1.85) node[]{\tiny{$k$}};
}

\DeclareSymbol{treeeval2}{0}{
	\draw[kernelsmod] (-0.25,-1) -- (1,0.5) ;
	\draw[kernelsmod] (1,0.5) -- (-0.25,2);
	\draw (1,0.5) node[diff2] {};
	\draw (-0.25,-1) node[not] {};
	\draw (-0.25,2) node[xi] {};
	\draw (-0.6,1.2) node[]{\tiny{1}};
}

\DeclareSymbol{arbreact}{1}{
	\draw (0,0) node[not] {};
	\draw[kernelsmod] (0,0) -- (1,1);
	\draw[kernelsmod] (0,0) -- (-1,1);
	\draw (-1,1) node[xic] {};
	\draw  (0,2) -- (1,1) ;
	\draw (0,2) node[xic] {};
	\draw (1,1) node[xi] {};
}

\DeclareSymbol{arbreact1}{0}{
	\draw (0,-1.5) -- (0,0);
	\draw[kernelsmod] (0,0) -- (1,1);
	\draw[kernelsmod] (0,0) -- (-1,1);
	\draw  (0,2) -- (1,1) ;
	\draw (0,-1.5) node[diff] {};
	\draw (0,0) node[not] {};
	\draw (-1,1) node[xic] {};
	\draw (0,2) node[xic] {};
	\draw (1,1) node[xi] {};
}

\DeclareSymbol{arbreact2}{0}{
	\draw (0,-0.75) -- (-1,0.5); 
	\draw (0,-0.75) -- (1,0.5);
	\draw (0,1.5) -- (1,0.5);
	\draw (0,1.5) node[xic] {};
	\draw (1,0.5) node[xi] {};
	\draw (-1,0.5) node[xic] {};
	\draw (0,-0.75) node[diff] {};
}

\DeclareSymbol{arbreact3}{0}{
	\draw[kernelsmod] (0,-0.75) -- (-1,0.5); 
	\draw[kernelsmod] (0,-0.75) -- (1,0.5);
	\draw (0,1.75) -- (1,0.5);
	\draw (2,1.75) -- (1,0.5);
	\draw (0,1.75) node[xic] {};
	\draw (1,0.5) node[diff] {};
	\draw (-1,0.5) node[xic] {};
	\draw (2,1.75) node[xi] {};
	\draw (0,-0.75) node[not] {};
}

\DeclareSymbol{pre_im_I1Xitwo}{0}{
	\draw[kernels2] (0,-0.3) node[not] {} -- (-0.6,0.7) ;
	\draw[kernels2] (0,-0.3) -- (0.6,0.7);
	\draw (0,0.9) node[g] {};
}

\DeclareSymbol{pre_im_cI1Xi4ab}{2}{
	\draw[kernels2] (0,-1) node[not] {} -- (-0.6,0) ;
	\draw[kernels2] (0,-1) -- (0.6,0);
	\draw (0,0.2) node[g] {};
	\draw (0,0.6) -- (0,1.5);
	\draw[kernels2] (0,1.5) node[not] {} -- (-0.6,2.5) ;
	\draw[kernels2] (0,1.5) -- (0.6,2.5);
	\draw (0,2.7) node[g] {};
}

\DeclareSymbol{pre_im_I1Xi4acc2}{0}{
	\draw[kernels2] (-1,-0.5) node[not] {} -- (-1.6,0.5) ;
	\draw[kernels2] (-1,-0.5) -- (-0.4,0.5);
	\draw[kernels2] (-1,-0.5) -- (0.2,-1.5) node[not] {} ;
	\draw (-1,1.1) -- (-1,2);
	\draw[kernels2] (0.2,-1.5) -- (0.2,2);
	\draw (-1,0.7) node[g] {};
	\draw (-0.3,2.2) node[g] {};
}

\DeclareSymbol{pre_im_I1Xi4abcc2}{2}{
	\draw[kernels2] (0,-1) node[not] {} -- (-1,0) node[not] {};
	\draw[kernels2] (-1,1.2) node[not] {} -- (-1,0);
	\draw[kernels2] (-1,1.2) -- (-1.5,2.5);
	\draw[kernels2] (-1,1.2) -- (-0.5,2.5);
	\draw[kernels2] (-1,0) -- (0.7,2.5);
	\draw[kernels2] (0,-1) -- (1.5,2.5);
	\draw (-1,2.7) node[g] {};
	\draw (1,2.7) node[g] {};
}

\DeclareSymbol{pre_im_2I1Xi4c1}{2}{
	\draw[kernels2] (0,-0.5) node[not] {} -- (-1,0.5) node[not] {};
	\draw[kernels2] (0,-0.5) -- (1,0.5) node[not] {};
	\draw[kernels2] (-1,0.5) node[not] {}-- (-1.7,2);
	\draw[kernels2]  (-1,2) -- (1,0.5);
	\draw[kernels2] (-1,0.5) -- (1,2);
	\draw[kernels2] (1,0.5) -- (1.7,2);
	\draw (-1.2,2.2) node[g] {};
	\draw (1.2,2.2) node[g] {};
}

\DeclareSymbol{pre_im_Xi4eabisc2}{2}{
	\draw[kernels2] (1.2,-0.5) node[not] {} -- (-0.7,0.8) ;
	\draw[kernels2] (1.2,-0.5) -- (0.4,0.8);
	\draw (0,1.4)  -- (0,2.2);
	\draw (1.2,2.2) -- (1.2,-0.6);
	\draw (0,1) node[g] {};
	\draw (0.6,2.4) node[g] {};
}

\DeclareSymbol{pre_im_Xi4eabisc22}{2}{
	\draw (1.2,-0.5) node[not] {} -- (-0.7,0.8) ;
	\draw[kernels2] (1.2,-0.5) -- (0.4,0.8);
	\draw (0,1.4)  -- (0,2.2);
	\draw[kernels2] (1.2,2.2) -- (1.2,-0.6);
	\draw (0,1) node[g] {};
	\draw (0.6,2.4) node[g] {};
}

\DeclareSymbol{pre_im_Xi4eabisc222}{2}{
	\draw[kernels2] (0.4,-0.5) node[not] {} -- (-0.6,1) ;
	\draw[kernels2] (1.2,0) -- (0.3,1);
	\draw (0,1.1)  -- (0,2.5);
	\draw[kernels2] (1.2,2.5) -- (1.2,0) node[not] {} -- (0.4,-0.6);
	\draw (0,1.2) node[g] {};
	\draw (0.6,2.5) node[g] {};
}

\DeclareSymbol{pre_im_Xi4eabc2}{2}{
	\draw (0,-0.5) node[not] {} -- (-1,0.5) node[not] {};
	\draw[kernels2] (-1,0.5) -- (-1.5,2);
	\draw[kernels2] (0,-0.5)  -- (0.7,2);
	\draw[kernels2] (-1,0.5) -- (-0.5,2);
	\draw[kernels2] (0,-0.5) -- (1.5,2);
	\draw (-1,2.2) node[g] {};
	\draw (1,2.2) node[g] {};
}

\DeclareSymbol{pre_im_Xi4eabbisc2}{2}{
	\draw[kernels2] (0,-0.5) node[not] {} -- (-1,0.5) node[not] {};
	\draw[kernels2] (-1,0.5) -- (-1.5,2);
	\draw[kernels2] (0,-0.5)  -- (0.7,2);
	\draw[kernels2] (-1,0.5) -- (-0.5,2);
	\draw (0,-0.5) -- (1.5,2);
	\draw (-1,2.2) node[g] {};
	\draw (1,2.2) node[g] {};
}

\DeclareSymbol{pre_im_I1Xi4abcc1}{2}{
	\draw[kernels2] (0,-1) node[not] {} -- (-1,0) node[not] {};
	\draw[kernels2] (0,1.1) node[not] {} -- (-1,0);
	\draw[kernels2] (-1,0) -- (-1.5,2.5);
	\draw[kernels2] (0,1.1) node[not] {} -- (-0.5,2.5);
	\draw[kernels2] (0,1.1) -- (0.5,2.5);
	\draw[kernels2] (0,-1) -- (1.5,2.5);
	\draw (-1,2.7) node[g] {};
	\draw (1,2.7) node[g] {};
}

\DeclareSymbol{pre_im_Xi4eabc1}{2}{
	\draw (0,-0.5) node[not] {} -- (-1,0.5) node[not] {};
	\draw[kernels2] (-1,0.5) -- (-1.5,2);
	\draw[kernels2] (0,-0.5)  -- (-0.5,2);
	\draw[kernels2] (-1,0.5) -- (0.8,2);
	\draw[kernels2] (0,-0.5) -- (1.5,2);
	\draw (-1,2.2) node[g] {};
	\draw (1,2.2) node[g] {};
}

\DeclareSymbol{pre_im_Xi4ba1b}{2}{
	\draw[kernels2] (0,0) node[not] {}  -- (1.8,1.5);
	\draw[kernels2] (0,0) -- (0.8,1.5);
	\draw (0,-0.1) -- (-1.8,1.5);
	\draw (0,-0.1) -- (-0.8,1.5);
	\draw (-1,1.7) node[g] {};
	\draw (1,1.7) node[g] {};
}

\DeclareSymbol{pre_im_Xi4ba2}{2}{
	\draw (0,-0.1) node[not] {}  -- (1.8,1.5);
	\draw[kernels2] (0,0) -- (0.8,1.5);
	\draw (0,-0.1) -- (-1.8,1.5);
	\draw[kernels2] (0,0) -- (-0.8,1.5);
	\draw (-1,1.7) node[g] {};
	\draw (1,1.7) node[g] {};
}

\DeclareSymbol{pre_im_Xi4cabc2}{2}{
	\draw[kernels2] (0,-0.5) node[not] {} -- (-1,0.5) node[not] {};
	\draw[kernels2] (-1,0.5) -- (-1.5,2);
	\draw (-1,0.5)  -- (0.7,2);
	\draw[kernels2] (-1,0.5) -- (-0.5,2);
	\draw[kernels2] (0,-0.5) -- (1.5,2);
	\draw (-1,2.2) node[g] {};
	\draw (1,2.2) node[g] {};
}

\DeclareSymbol{pre_im_Xi4cabc1}{2}{
	\draw[kernels2] (0,-0.5) node[not] {} -- (-1,0.5) node[not] {};
	\draw (-1,0.5) -- (-1.5,2);
	\draw[kernels2] (-1,0.5)  -- (0.7,2);
	\draw[kernels2] (-1,0.5) -- (-0.5,2);
	\draw[kernels2] (0,-0.5) -- (1.5,2);
	\draw (-1,2.2) node[g] {};
	\draw (1,2.2) node[g] {};
}

\DeclareSymbol{pre_im_Xi4eabbisc1}{2}{
	\draw[kernels2] (0,-0.5) node[not] {} -- (-1,0.5) node[not] {};
	\draw[kernels2] (-1,0.5) -- (-1.5,2);
	\draw[kernels2] (0,-0.5)  -- (-0.5,2);
	\draw[kernels2] (-1,0.5) -- (0.8,2);
	\draw (0,-0.5) -- (1.5,2);
	\draw (-1,2.2) node[g] {};
	\draw (1,2.2) node[g] {};
}

\DeclareSymbol{pre_im_1}{0}{
	\draw[kernels2] (0,-0.5) node[not] {} -- (-0.6,0.5) ;
	\draw[kernels2] (0,-0.5) -- (0.6,0.5);
	\draw (0,1.1)  -- (-0.55,2);
	\draw (0,1.1)  -- (0.55,2);
	\draw (0,0.7) node[g] {};
	\draw (0,2.2) node[g] {};
}

\DeclareSymbol{disconnect}{0}{
	\draw[kernels2] (0,-0.5) node[not] {} -- (-0.6,0.5) ;
	\draw[kernels2] (0,-0.5) -- (0.6,0.5);
	\draw (-0.55,1.1)  -- (-0.55,2.3);
	\draw (0.55,2.3) -- (0.55,1.5) -- (1.2,1.5) -- (1.2,3.5) -- (0.55,3.5) -- (0.55,2.7);
	\draw (0,0.7) node[g] {};
	\draw (0,2.5) node[g] {};
}

\DeclareSymbol{pre_im_2}{2}{\draw[kernels2] (0,0) node[not] {} -- (-1,1) node[not] {};
	\draw[kernels2] (0,0) -- (1,1) node[not] {};
	\draw[kernels2] (-1,1) -- (-1.5,2.5);
	\draw[kernels2] (-1,1) -- (-0.5,2.5);
	\draw[kernels2] (1,1) -- (0.5,2.5);
	\draw[kernels2] (1,1) -- (1.5,2.5);
	\draw (-1,2.7) node[g] {};
	\draw (1,2.7) node[g] {};
}

\DeclareSymbol{CX_rec}{0}{
	\draw [black] (-0.3,1) to (-0.3,-0.3);
	\draw [black] (0.3,1) to (0.3,-0.3);
	\draw [black] (-0.3,1) to (-0.3,2.3);
	\draw [black] (0.3,1) to (0.3,2.3);
	\draw (0,1) node[rec] {};
}

\DeclareSymbol{CX_cerc}{0}{
	\draw [black] (0,1) to (0,-0.3);
	\draw (0,1) node[cerc] {};
}

\DeclareMathAlphabet{\mathpzc}{OT1}{pzc}{m}{it}

\let\eps\varepsilon

\def\eqref#1{(\ref{#1})}

\makeatletter 
\newcommand*{\bigcdot}{}
\DeclareRobustCommand*{\bigcdot}{%
	\mathbin{\mathpalette\bigcdot@{}}%
}
\newcommand*{\bigcdot@scalefactor}{.5}
\newcommand*{\bigcdot@widthfactor}{1.15}
\newcommand*{\bigcdot@}[2]{%
	\sbox0{$#1\vcenter{}$}
	\sbox2{$#1\cdot\m@th$}%
	\hbox to \bigcdot@widthfactor\wd2{%
		\hfil
		\raise\ht0\hbox{%
			\scalebox{\bigcdot@scalefactor}{%
				\lower\ht0\hbox{$#1\bullet\m@th$}%
			}%
		}%
		\hfil
	}%
}
\makeatother

\tcbset
{colframe=boxcolor,colback=symbols!7!pagebackground,coltext=pageforeground,
	fonttitle=\bfseries,nobeforeafter,center title,size=fbox,boxsep=1.5pt,
	top=0mm,bottom=0mm,boxsep=0mm,tcbox raise base}

\def\two{{\<generic>\kern0.05em\<genericb>}}
\def\twoI{{\<Ito>\kern0.05em\<Itob>}}

\def\mail#1{\burlalt{#1}{mailto:#1}}

\usepackage{thmtools} 

\declaretheorem[style=definition]{example}

\makeatletter
\newcommand*{\defeq}{\mathrel{\rlap{%
			\raisebox{0.3ex}{$\m@th\cdot$}}%
		\raisebox{-0.3ex}{$\m@th\cdot$}}%
	=}
\makeatother

\renewcommand{\lg}{\langle}
\newcommand{\rg}{\rangle}
\newcommand{\bfN}{\mathbf{N}}
\newcommand{\bfR}{\mathbf{R}}
\newcommand{\scrL}{\mathscr{L}}

\newcommand{\scrT}{\mathscr{T}}
\newcommand{\norme}[1]{\left\lVert #1 \right\lVert}
\newcommand{\Linv}{\scrL^{-1}}
\newcommand{\bbrack}[1]{ \llbracket #1 \rrbracket }
\newcommand{\ddd}{ \text{d} }
\newcommand{\bfC}{{\bf C}}

\newcommand{\pix}{\pi_{T_X}^\perp}

\newcommand{\bfT}{\mathbf{T}}
\newcommand{\bbE}{\mathbb{E}}
\newcommand{\rmF}{\mathrm{F}}
\begin{document}
	
	\title{A general paracontrolled ansatz for singular SPDEs}
	\author{Yvain Bruned, Nicolas Moench}
	\institute{ 
		Universite de Lorraine, CNRS, IECL, F-54000 Nancy, France
		\\
		Email:\ \begin{minipage}[t]{\linewidth}
			\mail{yvain.bruned@univ-lorraine.fr}
			\\  \mail{nicolas.moench@univ-lorraine.fr}
	\end{minipage}}
	
	\maketitle
	\begin{abstract}
		We provide a general ansatz for the paracontrolled approach introduced by Gubinelli, Imkeller, Perkowski for treating singular SPDEs. The ansatz proposed covers a large class of equations.
		It is described via decorated trees that encode its coefficients and its stochastic data. The main novelty is the recursive definition of the paracontrolled stochastic iterated integrals that involve a well-chosen combinatorial set of words. The paralinearisation is performed via a lift of iterated paraproducts to modelled distributions. The main idea is to integrate localised iterated paraproducts in order to get a global object with a good regularity.      We also provide a fixed point argument and the renormalised equation with this ansatz when one uses the preparation map formalism that encompasses the BPHZ renormalisation. The last main result is the convergence of the renormalised  stochastic data via the BPHZ renormalisation  using an adaptation of the spectral gap approach. One obtains in the end a general solution theory for parabolic singular SPDEs via the paracontrolled approach.
			{\scriptsize  
		} 
	\end{abstract}
	
	\tableofcontents
		
	\section{Introduction}
	
	The local well-posedness of singular stochastic partial differential equations (SPDEs) has been covered in great generality thanks to the theory of Regularity Structures invented by Martin Hairer \cite{Hai14}. It has been completed via several additional works \cite{BHZ,CH16,BCCH} where a systematic renormalised local ansatz of the solutions of these singular dynamics is provided (see \cite{FrizHai,BH26} for introductions to Regularity Structures). 
	The ansatz extends B-series arising in numerical analysis for describing Runge-Kutta schemes. The main idea is to proceed with an expansion of the solution with iterated integrals. 
	In the present paper, we consider equations of the form
	\begin{equation}\label{eq_gPAM}
		\scrL u = \sum_{\ell=1}^{\ell_0} F_\ell(v)\xi_\ell, \quad (t,x) \in \R_+ \times \bfT^{d},
	\end{equation}
	where $\scrL=\partial_t-L$  is a differential operator. One has $ v= (\partial^k u)_{k \in A_- \sqcup A_+} $ with $A_-$ and $A_+$ being two finite sets of $ \N^{d+1} $. For $k \in A_-$, $\partial^k u$ is a distribution and for $ k \in A_+ $, $\partial^k u$ is a function. We assume that  $F_\ell$ is a smooth nonlinearity with respect to $(\partial^k u)_{k \in A_+}$ and polynomial in the $(\partial^k u)_{k \in A_-}$. The terms $\xi_1,\dots,\xi_{\ell_0}$ are space-time  noises with  $\xi_\ell \in  C^{\alpha_\ell}$ and $\alpha_\ell \in \R$. One can then rewrite \eqref{eq_gPAM} in a mild form
	\begin{equation}
		\label{mild_form}
		u = K * \left(  \sum_{\ell=1}^{\ell_0} F_\ell(v)\xi_\ell \right),
	\end{equation}
	where $ K$ is the kernel associated with $ \scrL^{-1} $ and we assume that convolution with $K$ yields a gain of $\beta$ in the Hölder regularity measured with the scaling $ \s = (\beta,1,\dots,1) $. In practice, the rough noises $\xi_\ell$ are replaced by regularised versions  $\xi_{\ell}^\eps$ depending on a small parameter $\eps >0$. The regularised noises are given by
	\begin{equation*}
		\xi_{\ell}^\eps= \rho_{\eps} * \xi_{\ell}, \quad \rho_{\eps}(t,x) = \eps^{- \beta - d} \rho( \eps^{-\beta} t , \eps^{-1} x )
	\end{equation*}
	where $\rho$ is a mollifier, a smooth and compactly supported function. 
Now iterating \eqref{mild_form} with these smooth noises, one expects the following ansatz for the solution $u^{\eps}$
	\begin{equation*}
		u^{\eps} = \text{Sum of stochastic iterated integrals} + \text{Remainder}.
	\end{equation*}
	where $u^{\varepsilon}$ is the solution of \eqref{mild_form} when the rough noises are replaced by the smooth noises $ \xi_{\ell}^{\varepsilon} $.
	The strategy of Regularity Structures is to provide a local ansatz which requires localising the iterated integrals around a base point producing new monomials. The most recent approach developed by Pawel Duch in \cite{Duc21} introduces an extra cut-off on the differential operator by replacing the kernel $K$ by a family of kernels $ (K_{\mu})_{\mu \geq 0} $. This approach is inspired by the Polchinski flow \cite{P84}. A discrete version based on the renormalisation group had been proposed before in \cite{K16}. In both approaches (Regularity Structures and flow) one can easily compute the ansatz of $F_{\ell}(v^{\eps})$
since small parameters are available to perform the linearisation. One can find general ansatz based on decorated trees in \cite{BCCH} for Regularity Structures and in \cite{BM26} for the flow approach.
	
	At the same time as Regularity Structures, Gubinelli, Imkeller and Perkowski introduced in \cite{GIP} a paracontrolled approach where the main idea is to decompose a distributional product using paraproducts that allows to identify the part in the product which is ill-defined. It uses ideas from paradifferential calculus introduced by Bony in  \cite{Bon81} and controlled rough paths in \cite{Gubinelli2004}.  The main contribution is to provide a global ansatz without a small parameter. This method had been used with great success both for solution theories and discretisation on models with short ansatz like the $\Phi^4_3$ model (see \cite{Phi43Paracontrol,MW17,MWX17,ZZ18,GH19,JP23,BDFT23}), the KPZ equation (see \cite{KPZReloaded}), 2D stochastic Yang-Mills \cite{BC23}, Mean field equations \cite{Meanfield}. 
	Other notable successes of this approach include a variational method for studying the $ \Phi_3^4 $ measure in \cite{BG20}, as well as dispersive singular SPDEs (see \cite{GKO24,BDNY24}). It is also connected to the ansatz used for the random averaging operators in \cite{DNY24} and random tensors in \cite{DNY22}.

	Going to  higher order  expansions requires the linearisation of $F_\ell(v^{\eps})$. A toolkit was provided in \cite{BailleulBernicotHighOrder} following the line of works \cite{HeatKernelSPDE,SpacetimeParaproducts}. This framework was subsequently applied to quasilinear singular SPDEs in \cite{ParacontrolforQuasilinear}. However, no general ansatz with a scope comparable to the one in \cite{BCCH} for Regularity Structures had been provided yet in the paracontrolled approach. 
	Computations were performed by hand for writing the paracontrolled ansatz, the renormalised equation and proving the convergence of the stochastic data. The later is very complex due to the paraproducts inside the stochastic iterated integrals, and it was not clear how to renormalise them. 
	 The aim of this paper is to fill this gap in the literature by providing such an ansatz. The general paracontrolled ansatz takes the form
	\begin{equation}\label{eq_ansatz0}
	u= \sum_{|\CI(\tau)| < \gamma } \P(u_\tau, \, Z_{\mcI(\tau)} ) + u^\#,
\end{equation}
	where $\P$ denotes the paraproduct, the reference distributions $Z_{\mathcal{I}(\tau)}$ are indexed by decorated trees $\tau \in \CT$ . Here,  $ \mathcal{I}(\tau) $ is obtained by grafting $\tau$ to a new root via a new edge.
	One has $ Z_{\CI(\tau)} \in C^{|\CI(\tau)|}  $ and  $u^\#$ is a smooth remainder belonging to $C^{\gamma}$. The degree $|\tau|$ (see \eqref{degree_tree}) is computed from the regularity of the noises and the gain of regularity provided by the convolution with the kernel $K$ which is $\beta$. Consequently, one has $ |\CI(\tau) | = \beta + |\tau| $.
	The coefficients $u_{\tau}$ are given by $u_\tau = \frac{\Upsilon_F[\tau]}{S(\tau)}(v)$ where $ \Upsilon_F[\tau] $  are the elementary differentials and $S(\tau)$ are the symmetry factors associated with the tree $\tau$.   
	We assume that the set of decorated trees $ \tau $ such that $ |\CI(\tau)| < \gamma $ (We denote this set by $\CT_{< \gamma}$) with $u_{\tau} \neq 0$ is finite, which is a consequence of the equation being subcritical.   It is guaranteed via conditions involving the $F_{\ell}$, $\alpha_\ell$ and $\beta$ that we describe in Subsection \ref{subsec_subcriticality}. We also assume Assumption \ref{constant_coefficient} that guarantees the use of the generalised Da Prato-Debussche trick introduced in \cite[Section 5.5]{BCCH}, an extension of the Da Prato-Debussche trick \cite{DPD02,DPD03}.
The stochastic data $Z_\tau$ are defined inductively from the noises $\xi_{\ell}$ and the various operations given below:
	\begin{itemize}
		\item Pointwise multiplication.
		\item Space-time convolution with the kernel $K^{[n]}(x)=K(x)(-x)^n$.
		\item Iterated paraproducts of the form
		\begin{equation*}
		\partial^k \P_\ell\big(Z_{\tau_1},\dots,Z_{\tau_n}\big)
		\end{equation*}
		where the subscript  $\ell \in \N^{d+1}$ is associated with some monomial weights inside the definition of the paraproduct. In particular, one has $ \P = \P_0 $. The $\tau_i$ are combinatorial objects strictly smaller than $\tau$, which provides the inductive structure. 
	\end{itemize}
	
	 This iterative construction of the stochastic data does not make explicit use of the resonant products that usually appear in paracontrolled calculus. However, when $\tau$ is a product of trees, the associated quantity $Z_\tau$ can be interpreted as a generalised resonant product,  obtained by subtracting 
	  its irregular components, given by paraproducts involving lower-order pieces $Z_\mu$ to the corresponding product.

	We construct the $Z_{\tau}$ in the smooth setting replacing $ \xi_{\ell} $ by $ \xi_{\ell}^{\eps} $.
	Given a preparation map $ R : \CT \rightarrow \CT $ introduced in \cite{BR18}, one can renormalise the $Z_{\tau}$ by considering 
	\begin{equation*}
		\widehat{Z}_{\tau} = \widehat{Z}^{\circ}_{R \tau}
	\end{equation*}
	where $\widehat{Z}^{\circ}$ is multiplicative (in the sense given by Lemma \ref{lem_PsiTau_multiplicatif}) and contains further renormalisation. This is similar to the renormalisation of models in Regularity Structures via preparation maps. We denote by $R^{*}$ the adjoint of $R$ and we set
	\begin{equation*}
		R F_{\ell} = \Upsilon_F[R^* \Xi_{\ell}]
	\end{equation*}
	where $ \Xi_{\ell} \in \CT $ is the symbol associated to the noise $\xi_{\ell}$.  This renormalised non-linearity was first defined in \cite{BCCH} and in \cite{RenomrmalisedSpde}, one uses preparation maps.
		We are now able to state the first main theorem of the paper
	
	\begin{theorem} \label{main_theorem}
		The solution $ u^{\eps}$ of the regularised equation
		\begin{equation}
			\label{regularied}
				\scrL u^{\eps} = \sum_{\ell=1}^{\ell_0} R_\eps F_\ell(v^{\eps})\xi_\ell^{\eps}, \quad u(0) = u_0
		\end{equation}
		with $ u_0 \in C^{\gamma_0}(\bfT^d)$ and $\gamma_0$ defined in \eqref{gamma_0_b},  
		is described by the paracontrolled expansion 
		\begin{equation}
		\label{renormalised_ansatz}	u^\eps= \sum_{|\CI(\tau)| < \gamma_0} \P (u_\tau^\eps, \widehat{Z}^\eps_{\mcI(\tau)}) + u^\#.
		\end{equation}
			For any initial condition $u_0 \in  C^{\gamma_0}(\bfT^d)$ and stochastic data $\widehat{\mcZ}^\eps_{<\beta}=(\widehat{Z}^\eps_\tau)_{\tau\in\mcT_{<\beta}}$, there exists a small time $T>0$ such that  \eqref{regularied} has a unique solution. Furthermore the solution $(u,u^\#)$ depends continuously on $u_0$ and $\widehat{\mcZ}^\eps_{< \beta}$. 
		\end{theorem}
	
	Let us stress the main ideas behind the proof of Theorem \ref{main_theorem}.
	It can be decomposed into several steps
	\begin{itemize}
		\item Paralinearisation: Write $ F_{\ell}(v) $ as a paracontrolled distribution from the paracontrolled ansatz of $u$ (see Theorem \ref{prop_paralinearization}).
		\item Define the product $ F_{\ell}(v) \xi_{\ell} $ (see Proposition \ref{prop_nonlinearitxiell_generalized}) with a corrector lemma (see Proposition \ref{lem_SumNegativeCorrector}).
		\item Write $ (\partial_t - \mathcal{L} )^{-1} (F_{\ell}(v) \xi_{\ell}) $ as  a
		paracontrolled distribution in order to close a fixed point argument (see Proposition \ref{prop_integration} and Section \ref{section_fixedpoint}). This is where one uses Schauder estimates.
	\end{itemize}
	To perform these various steps, one travels from local estimates to global ones and vice versa. Local estimates are easier to manipulate for products and general
	non-linearities thanks to Taylor formula.
	They are also useful for obtaining Schauder estimates.
The approach is summarised in the diagram below
	\begin{equation*}
		\xymatrix{
			\text{\tiny{Paracontrolled Ansatz}} \ar[r] \ar[d]_{\text{\tiny{Localisation}}}  & \text{\tiny{Paracontrolled Ansatz}} \\
			\text{\tiny{Local Ansatz}} \ar[r]_{\text{\tiny{Analytical operations}}} & \text{\tiny{Local Ansatz}}
			\ar[u]^{\text{\tiny{Globalisation}}} 
		}
	\end{equation*}
	In the diagram above, one defines  the analytical operations on the Paracontrolled Ansatz by going down to the Local Ansatz.
	Let us first present the globalisation argument that originates from  \cite{BailleulHoshinoRS1}.
		If one has
	\begin{equation*}
		|f_{x,y}| \lesssim |y - x|^{\alpha},
	\end{equation*}
	then
	\begin{equation*}	
		\P f\defeq \sum_{i\geq -1} \int K_{< i-1}(\cdot-x) K_i(\cdot- y) f_{x,y} \ddd x  \ddd y \in C^{\alpha}.
	\end{equation*}
	where the $K_{<i-1}$ and $K_i$ come from the Littlewood-Paley decomposition.
	This is the core of the proof of Propositions \ref{prop_RegCrochets}  and  \ref{prop_RegCrochets_I}.	From these blocks, one can obtain the crucial estimates for the stochastic data $Y_{\tau}$ and $\widehat{Y}_{\tau}$ (see Theorems \ref{thm_regularityZtau} and \ref{prop_regularity_YI_tau}).
	The stochastic data $ Y_{\tau} $ is used to construct the stochastic data $Z_{\tau}$ that appear in the paracontrolled ansatz of the right hand side of the equation. The stochastic data $ \widehat{Y}_{\tau} $ is used to construct the stochastic data $Z_{\CI(\tau)}$ that appear in the paracontrolled ansatz of the solution of the equation. The main difference between the two is an integration step.

	For the localisation argument, we used recentered iterated paraproducts introduced in \cite{LocalExpansionsParacSystems}. The idea is to lift $u$ into a modelled distribution that we denote by $ \bm{u} $ given by
	\begin{equation}\bm{ u} = \sum_{|\mcI(\tau)|<\gamma } u_\tau \, \Psi(\mcI(\tau))  + \sum_{ |n|\leq \gamma } \frac{\partial_*^{n} u}{n!} X^{n}
		\end{equation}
		where the $\partial_*^{n} u$ are modified derivatives solutions and correspond to the coefficients in front of the monomials $X^n$. The map $\Psi$ given in Definition \ref{def_Psitau} sends decorated trees to decorated words of the form
		\begin{equation*}
			\langle \tau_1,..., \tau_n \rangle_{\ell}
		\end{equation*}
		where the letters $\tau_i$ are decorated trees and $\ell \in \mathbf{N}^{d+1}$ is an extra decoration. It is similar in spirit to the Hairer-Kelly map given in \cite[Lemma 4.9]{HK15} that allows us to move from non-geometric to geometric rough paths. In our context, one does not have a shuffle morphism.
		One can then define a pre-model $\PPi^{\mathcal{Z}} $ by
		\begin{equation*}
			\PPi^{\mathcal{Z}}	\langle \tau_1,..., \tau_n \rangle_{\ell} = \P_\ell\big(Z_{\tau_1},\dots,Z_{\tau_n}\big).
		\end{equation*}
		where the superscript $ \mathcal{Z} $ stresses the dependence on the $Z_{\tau_i}$ in the definition of $\PPi^{\mathcal{Z}}$.
		We extend this definition to words with an extra decoration encoding derivatives of paraproducts:
		\begin{equation*}
			\PPi^{\mathcal{Z}} 	\langle \tau_1,..., \tau_n \rangle_{\ell}^k = \partial^k \P_\ell\big(Z_{\tau_1},\dots,Z_{\tau_n}\big).
		\end{equation*}
		Then, the model is given by
		\begin{equation*}
			\Pi_x = \left( \PPi \otimes g_x^{-1}  \right) \Delta, \quad g_x  = (\PPi \CA_{*} \cdot)(x)
		\end{equation*}
where $ \Delta $ is a deconcatenation type coproduct,  $\CA_{*}$ is a  kind of twisted antipode on words with the two decorations (superscript and subscript). This construction has been implemented in \cite{LocalExpansionsParacSystems} which extends the works \cite{Hos20, Hos21, Hos24}. The move from decorated trees to words traduces the fact that one needs paraproducts which are non-commutative products. Decorated trees with the merging root product encode pointwise products and convolution with a kernel.

 The interpretation involving the map $\CA_{*}$ is given in Proposition \ref{def_prop_antipode}. One can connect this pre-model to the one for decorated trees which we denote by $\PPi$. Indeed, from Corollary \ref{PPi_Psi} one has
\begin{equation*}
	\PPi \tau = \PPi^{\mathcal{Z}} \Psi(\tau).
\end{equation*} 
If we denote by ${\bm R}$ the reconstruction operator associated with the previous regularity structure on decorated words and by $ f_{\gamma}(\bm u)$ the modelled distribution of order $\gamma$ obtained by composing $f$ with $\bm u$ we have
\begin{equation*}
f(u)= 	{\bm R}{f_{\gamma}(\bm u)}.
\end{equation*}
From Corollary \ref{cor_repmodelledditribParap}, one can rewrite the right hand side of this identity with a paracontrolled ansatz similar to \eqref{renormalised_ansatz}. This rewriting is close to the main theorems given in \cite{BailleulHoshinoRS1,PCandRS2} that associate a paracontrolled distribution to a modelled distribution.  The main difference between their works and Corollary \ref{cor_repmodelledditribParap} is the stochastic data used for the paracontrolled distribution.
Our general paralinearisation formula is given in Theorem \ref{prop_paralinearization}. Then, from there, one can compute the paracontrolled ansatz  for the multiplication with the noises $ \xi_{\ell}$ in Proposition \ref{prop_nonlinearitxiell_generalized} and we finish via an integration step given in Proposition \ref{prop_integration}.

The second main result of this paper is the convergence of the renormalised stochastic data $\widehat{Z}_{\tau}^m$. Here, we use the BPHZ preparation map $(R^m)_{m \in \mathbb{N}}$ defined in \eqref{BPHZ_ch}. The $m$ stresses the fact that we regularise via $\varrho_{1/m}$ where $\varrho $ is a mollifier. 
	As the proof of the general paracontrolled ansatz relies on a well-chosen regularity structure on words, it is natural to use the spectral gap approach first proposed for multi-indices in \cite{LOTP} and then extended to decorated trees in \cite{HS23,Hos25,BH23}. This avoids
 implementing directly the BPHZ algorithm \cite{BP57,KH69,WZ69} as it was done for Regularity Structures in \cite{CH16}.   
 
\begin{theorem}\label{main_convergence}
We suppose that the noises $ \xi_\ell $, for $ \ell \in \{1,...,\ell_0\} $ belong to $ C^{-\frac{|\mathfrak{s}|}{2} - s_\ell} $ with $ s_{\ell} \geq - \frac{|\mathfrak{s}|}{2} $. The noises are also centered, space invariant and satisfy the spectral gap Assumption \ref{assumption_spectral gap}. Then, for any mollifier $\varrho$ and the BPHZ preparation map $R^m$, the stochastic data $ (\widehat{Z}_{\tau}^m)_{\tau} $ converges in $L^q(\Omega;C^{|\tau|})$ for any $ 1 \leq q < + \infty $. The limit is independent of the mollifier $\varrho$.
	\end{theorem}
	
	Let us explain how one proves such a result via the spectral gap method. In this framework, one controls the moments of $ \widehat{Z}_{\tau}^{m} $ via its expectation and the moments of $ \delta \widehat{Z}_{\tau}^{m} $ where $\delta$ is the Malliavin derivative. We use the following  spectral gap inequality 
	$$
	\mathbb{E}\Big[|\widehat{Z}^m_\tau|^q\Big] \lesssim  \Big| \bbE[ \widehat{Z}^m_\tau]\Big|^q  + \mathbb{E}\Big[|\delta \widehat{Z}^m_\tau|^q\Big].
	$$
	 Choosing the BPHZ scheme  allows us to show the convergence of the first expectation via an inductive argument. The second term is more involved. We summarise the steps below.

	 One introduces new symbols for the noises $ \delta \xi_{\ell} $ whose regularity will be first considered in some Besov spaces for taking into account the integrability.
		Via Besov embedding, one is able to see these objects in some Hölder space. We denote the map of this embedding by $j$:
		\begin{equs} 
			j : B^{\alpha}_2(\bfR^{d+1}) \hookrightarrow C^{\alpha-\frac{|\mathfrak{s}|}{2}}(\bfR^{d+1}),
		\end{equs} 
		where $|\s|$ is the size of the equation scaling as one uses an anisotropic Euclidean norm based on this scaling. See \eqref{besov} for a precise definition of the Besov spaces. 
		One generalises this idea to $\widehat{Z}^m_{D_{\Xi} \tau}$ where the derivative $ D_{\Xi} $ changes one noise of type $ \xi_{\ell} $ into $ \delta \xi_{\ell} $. This process is constructed into the following spaces:
		\begin{equs}
			\widehat{Z}^m_{D_{\Xi} \tau}  \in B^{|D_{\Xi} \tau|}_{2}(\bfR^{d+1}) \hookrightarrow C^{|D_{\Xi}\tau|-\frac{|\mathfrak{s}|}{2}}(\bfR^{d+1})
			\end{equs}
			where here, we have made an abuse of notation as normally $D_{\Xi} \tau$ is a linear combination of decorated trees. One has to consider that the identity above is  for each decorated tree appearing in this linear combination. One important remark is that the embedding forgets the extra regularity provided by the Malliavin derivative in the sense that one has
			\begin{equs}
				|D_{\Xi}\tau|-\frac{|\mathfrak{s}|}{2} = |\tau|.
				\end{equs} 
		 In Proposition \ref{prop_modelwords_besov}, one has a regularity-integrability structure on words built out of these stochastic processes. It relies on an extension of \cite[Theorem 1]{LocalExpansionsParacSystems} to Besov space performed in Appendix \ref{Appendix_B}. Then, an inductive bound on the Besov norm of $ \widehat{Z}^m_{\tau}  $ is established in Lemma \ref{cor_reconstruction_convergence} based on Lemma \ref{lem_reconstructionarg_Ztau}. This bound is crucial for constructing inductively $ \delta \widehat{Z}^m_{\tau} $. One important step toward this result, is
		the key algebraic identity in Theorem \ref{prop_alg_identity_convergence}, given for any decorated tree $ \tau $ by
		\begin{equation} \label{main_Y_delta}
		\delta \widehat{Z}_\tau^m = Y^{\widehat{\mcZ}^m
			}_{j_*\Psi(D_\Xi\tau)}
		\end{equation}
where $j_*$ performs algebraically the Besov embedding on the words by introducing new letters and on decorated trees by changing the noise type nodes. The stochastic data $Y^{\widehat{\mathcal{Z}}^m}$ is constructed from the $ \widehat{Z}^m_{\tau} $ viewed as stochastic processes in the suitable Hölder space.  The proof of \eqref{main_Y_delta} relies on an intermediate object denoted by $ Z^{R_{\varepsilon}}_{j^* (D_{\Xi}\tau)} $ which treats the new noises $  \delta \xi_{\ell}$ as if they were the old ones $ \xi_{\ell} $.  One shows in Proposition \ref{lem_deltaZ_Zj*} that
\begin{equation} \label{id_2}
\delta Z_\tau^{R_{\varepsilon}} = 	Z^{R_{\varepsilon}}_{j^* (D_{\Xi}\tau)}.
\end{equation}
Then, one gets in Proposition \ref{lem_Zj*tau=YjstarPsi}  the other identity which allows to conclude for Theorem \ref{prop_alg_identity_convergence}
\begin{equation} \label{id_1}
	Z_{j_*\tau}^{R_{\varepsilon}} =  Y_{j_*\Psi(\tau)}^{\mcZ}.
\end{equation}
The identities \eqref{id_2} and \eqref{id_1} are purely algebraic whereas \eqref{main_Y_delta} provides a way to get the correct analytical bounds. Lemma \ref{lem_estimate_deltaZtau} uses the right hand side of \eqref{main_Y_delta} to get the correct analytical bounds that depends inductively on Lemma \ref{cor_reconstruction_convergence}.

 Let us point out the parts of the paper where we use ideas coming from Regularity Structures:
\begin{itemize}
	\item Theorem \ref{thm_regularityZtau} uses the analytical bounds of the model on iterated paraproducts to get bounds on the stochastic data $Y_{\tau}$. Corollary \ref{cor_repmodelledditribParap} 	gives a paracontrolled representation for any modelled distribution based on the $Y_\tau$.
	\item Theorem \ref{prop_paralinearization} gives a paralinearisation formula by using the model on iterated paraproducts.
	\item Lemma \ref{integration_RS}  prepares the Schauder estimate by giving a local expansion of the right hand side of the equation.
	\item Lemma \ref{cor_reconstruction_convergence} uses the Polynomial Regularity Structure for a difference of two modelled distributions.
	\end{itemize}
We do not use the reconstruction theorem in the proofs of the paper. The reconstruction is already given and we do not need its existence part.

\begin{remark}
	The main algebraic identities of the various spectral gap proofs \cite{LOTP,HS23,BH23} rely on the same algebraic map $ d \Gamma $ first described on decorated trees in \cite{BN23}. This unification has been established in \cite{BM26a}. The map $d \Gamma$ is used for shortening some Taylor expansions in the recentering of the model.  The identity \eqref{main_Y_delta} seems close in spirit to the previous map but it is rather different and not easy to infer from the definition of $d \Gamma$.
	\end{remark}

	\begin{remark} For scalar-valued equations, one can replace decorated trees via multi-indices as it was initiated for Regularity Structures in \cite{OSSW,LOT}. The main inductive arguments of the construction of the $Z_{\tau}$ seem at a first sight robust to this new combinatorial setting.
		\end{remark}
	
	\begin{remark}
		The use of Local expansion/Regularity Structures limits the scope of the ansatz to the singular SPDEs where one observes a clear gain of regularity when convolving with the kernel $K$ and therefore it excludes stochastic dispersive equations. In the latter case, one mainly considers polynomial non-linearities where such a paralinearisation is not needed. One can have good hope that the structure of this ansatz and its renormalisation could be translated into this context.
		\end{remark}
		
		\begin{remark} In Theorem \ref{main_theorem}, one can work with rougher initial data $u_0$ as in \cite{BCCH}. This will require working with weighted spaces. We refrain from providing such optimal exponents as the main point of the paper is to provide a systematic way to write a paracontrolled ansatz that works for a large class of singular SPDEs. In order to stress the main mechanism of this ansatz, we avoid to work in a too general setting.
			\end{remark}
			
			\begin{remark}
				Let us stress that paracontrolled calculus is very helpful for obtaining global existence results for many singular SPDEs such as 2D stochastic Navier-Stokes \cite{HR24}, 2D generalized parabolic Anderson model \cite{SZZ26}. One can hope to extend these results with the help of this new general ansatz.
				\end{remark}

It is worth developing some example in order to illustrate our ansatz and computations of the right hand side of the equation. We give the example of the gPAM equation in dimension $2$
$$
(\partial_t-\Delta)u=f(u)\xi
$$
where $\xi\in C^{\alpha-2}$ and $\alpha\in(2/3,1)$. This was the equation studied in \cite{GIP} where the paracontrolled method was introduced. Our ansatz is very similar, the difference is that we write the computations in terms of the corrector \[\bfC(h_1,h_2)= \P(h_1,h_2)-h_1h_2\] instead of the resonant product \[h_1\odot h_2= h_1h_2-\P(h_1,h_2)-\P(h_2,h_1).\] We claim that the approach based on $\bfC(h_1,h_2)$ facilitates the generalisation of these computations to higher order expansions, as some structure appears in the recursive definition of the higher order commutation relations. Working with $\bfC$ enables also to gather terms with the same analytical properties, for instance we have the same regularity estimate on $\P(h_1,h_2)$ and $h_1\odot h_2$ where $(h_1,h_2)\in C^{\alpha_1}\times C^{\alpha_2}$ with $\alpha_1<0$ and $\alpha_1+\alpha_2>0$. The ansatz reads as 
$$u=\P\big(f(u),Z_{\mcI(\Xi)}\big) + u^\#$$
with $Z_{\mcI(\Xi)}=K*\xi$ and $u^\#\in C^{2\alpha}$. In the sequel, we will use the following tree notation: 
\begin{equs}
	\<IXi> = \CI(\Xi), \quad \<Xii> = \Xi \CI(\Xi),
\end{equs}
where a thin edge corresponds to the abstract integration symbol $ \CI $ and $ \<generic> $ corresponds to the noise symbol $ \Xi $. 
We have a paralinearisation result that enables us to write
$$f(u) = \P\big(ff'(u),Z_{	\,\<IXi>}\big) + f(u)^\#$$
with $f(u)^\#\in C^{2\alpha}$.
The definition of a product map still comes from a corrector lemma. 
We set \begin{align*}
	\bfC(h_1,h_2,h_3) &\defeq \bfC(\P(h_1,h_2),h_3) - h_1\bfC(h_2,h_3)
	\\
	&= \P(h_1,h_2,h_3)-h_1\P(h_2,h_3)-\P(h_1,h_2)h_3+h_1h_2h_3
\end{align*} 
where still $\bfC (h_1,h_2)= \P(h_1,h_2) - h_1h_2$. It is morally the same corrector as in \cite{GIP}, written with $\bfC$ instead of $\odot$. The corresponding corrector lemma states that 
$$\norme{\bfC(h_1,h_2,h_3)}_{\alpha_1+\alpha_2+\alpha_3}\lesssim \prod_{i=1}^{3}\norme{h_i}_{\alpha_i}.$$
whenever $\alpha_1\in(0,1), \, \alpha_2+\alpha_3<0$ and $ \alpha_1+\alpha_2+\alpha_3>0$. We also note the continuity estimate \[\norme{\bfC(h_1,h_2)}_{\alpha_1+\alpha_2}\lesssim \norme{h_1}_{\alpha_1}\norme{h_2}_{\alpha_2}\] for $\alpha_2<0$ and $\alpha_1+\alpha_2>0$.
Then one can make the formal computations
\begin{align*}
	f(u)\xi &= \P(f(u),\xi) - \bfC(  f(u),\xi   )
	\\
	&=  \P(f(u),\xi) - \bfC(f(u)^\#,\xi)  -\bfC\big(ff'(u), Z_{	\,\<IXi>},\xi\big) - ff'(u) \bfC(Z_{\,\<IXi>},\xi).
\end{align*}
We see that the only problematic term is $ff'(u)\bfC(Z_{	\, \<IXi>},\xi)$, we can therefore define a product map provided we are given $\bfC(Z_{\, \<IXi>},\xi)\in C^{2\alpha-2}$. The extended noise required for the gPAM is then $(\xi,Z_{\<Xii>})\in C^{\alpha-2}\times C^{2\alpha-2}$ where $Z_{\<Xii>}$ plays the role of $-\bfC(Z_{\, \<IXi>},\xi)$. One has
\begin{equs}
	Z_{\<Xii>} & = -\bfC(Z_{\, \<IXi>},\xi) =
	-\bfC(Z_{\, \<IXi>},Z_{\<generic>})
	\\
	& = \xi  K * \xi - \P( Z_{\, \<IXi>}, Z_{\<generic>} )
	= \xi  K * \xi - \P(K*\xi, \xi).
\end{equs}
The paraproduct term can be computed using the Hairer-Kelly map $ \Psi $.
Indeed, one has
\begin{equs} \label{Psi_app}
	\Psi( \<Xii> )= \<Xii> + \<IXi>  \otimes \<generic>.
\end{equs}
Here, $ \<IXi>  \otimes \<generic> $ represents a word with two letters. In the sequel, we will also use the notation $ \lg  \<IXi>  \, \<generic> \rg$.
The first word in \eqref{Psi_app} with one letter is sent to $ \xi K * \xi $ and the second to $ \P( Z_{\, \<IXi>} ,Z_{\<generic>}) $. One can notice that the second letter $\<IXi>$ is a decorated tree with positive degree. This comes from the definition of $\Psi$. The only letter with potentially negative degree is the right most letter. The order of the iterated paraproduct respects the order of the letters.
We also mention that this product map writes as 
$$
f(u)\cdot \xi = \P(f(u),\xi) + \phi^\#
$$
with $\phi^\#\in C^{2\alpha-2}$. Provided we are given some commutation relation \[ K*\P(f(u),\xi) \simeq \P(f(u), K*\xi) \] up to some smoother term, we get
$$
K*(f(u)\cdot \xi) = \P(f(u), Z_{\, \<IXi>}) + \psi^\#.
$$
We are then able to perform a fixed point argument in some adequate functional space, as one recovers the same structure as in the ansatz for $u$.  We provide more examples of computations with more singular noise in Subsection \ref{examples}.

\bigbreak

Let us briefly summarise the content of this paper by presenting the various sections. In Section \ref{Sec::2}, we introduce the space-time paraproducts coming from \cite{SpacetimeParaproducts} but in the flat case. Then, we recall the basics of Regularity Structures following the presentation of \cite{BH26}. We define generic notations for the coaction/coproducts used for the recentering of some pre-model. In the sequel, we introduce two specific regularity structures: one on decorated trees in Subsection \ref{subsection_DecoratedTrees} coming from \cite{Hai14,BHZ} and one on words in Subsection \ref{subsect_theRSIteratedParap} coming from \cite{BM26}.
We also recall the subcriticality conditions coming from \cite{BHZ} after defining elementary differentials in Subsection \ref{subsec_subcriticality}. The main result of the section is Proposition \ref{def_prop_antipode} that provides a recursive map $\CA^*$ for encoding the $\partial_*^k$ operator on paraproducts. 	In Section~\ref{section_fromRStoPS}, we prove that, for any regularity structure $\scrT$ satisfying mild assumptions and for any model  ${ (\PPi,g)}$ over it, one has the paracontrolled representation (see Definition \ref{def_Ytau})
	$$
 Y_{\tau} = \widetilde{ \PPi}(\tau) - \, \sum_{k\geq 0} \, \sum_{\sigma\prec \tau} \, \frac{1}{k!} \, \P_k\Big( \widetilde{ g}\big(\downarrow^k\hspace{-2pt} \tau/\sigma\big) , \enskip Y_\sigma  \Big),
	$$
	where $(\widetilde{\PPi},\widetilde{g})$ is a modification of the model in which polynomials are recentered, yielding periodic distributions (introduced in Subsection \ref{subsect_modifmodel}). The derivation $ \downarrow^k $ is defined in Subsection	\ref{section_derivations}, $ \tau / \sigma $ is associated with the recentering coaction $\Delta$ given in Subsection \ref{subsection_basicsRS} and $\sigma \prec \tau$ is defined in \eqref{thm_regularityZtau}. We also connect $ \widetilde{\PPi} $ to $\widetilde{\PPi}^{\mathcal{Y}}$, the pre-model on words whose letters belong to the basis of $T$ and $T^{+}$ (see \eqref{eq_PiTauPiPsiTau})
	\begin{equation*}
		\widetilde{\PPi}(\tau) = 	\widetilde{\PPi}^\mcY\big(\Psi(\tau)\big)
	\end{equation*}
	where the map $\Psi$ is given in Definition \ref{def_Psitau}.
	 In Theorem \ref{thm_regularityZtau}, we show that the distributions $Y_\tau$ are in $C^{|\tau|}$ with a continuity estimate with respect to $(\PPi,g)$. Its proof relies on algebraic identities that allow one to rewrite the main analytical objects using the model on words and the map $\Psi$ (see Proposition \ref{prop_gDntaustarderivative} and Lemma \ref{lem_sumcrochets}). Corollary \ref{cor_repmodelledditribParap} which is  crucial for the sequel,
	 provides a paracontrolled representation for any modelled distribution based on the $Y_\tau$. A similar construction is provided for the integrated stochastic data denoted by $ \widehat{Y}_{\tau} $ (see Definition \ref{def_YItau}). One works with an extended alphabet to take into account the integration map. With an extended Hairer-Kelly map $ \Psi^I $ (see \eqref{modified_Hairer_Kelly}), one gets similar analytical bounds for $ \widehat{Y}_{\tau} $ in Theorem \ref{prop_regularity_YI_tau}.

	In Section~\ref{section_constructionZtau}, we construct the collection of reference distributions $Z_\tau$ by induction in \eqref{eq_defZtau_Naif} and \eqref{eq_defZtau_T+} using the maps $ \widetilde{\PPi} $, $ \widetilde{g} $ and $\Psi$. We define the notion of admissible stochastic data in Definition \ref{admissible} which is similar to the admissibility notion for pre-models in Regularity Structures. We then state the main theorem of the section, Theorem \ref{prop_extensionthm_Ztau} that extends a family of singular stochastic data given on decorated trees with negative degree to a family of admissible stochastic data on all decorated trees with positive degree. As a consequence, one gets Corollary \ref{PPi_Psi} that connects the admissible pre-model on decorated trees to the admissible pre-model on decorated words via the map $\Psi$.
	 In Propositions \ref{prop_Ztau=YPsitau_integration} and  \ref{prop_Ztau=YPsitau_mult}, we connect $Z_{\tau}$ to $ Y_{\tau}^{\mathcal{Z}} $ via the map $\Psi$. This allows us to apply Theorem \ref{thm_regularityZtau} to prove Theorem \ref{prop_extensionthm_Ztau}. We finish the section by introducing the corresponding multiplicative collection $\mcZ^\circ=(Z^\circ_\tau)_{\tau\in\mcT}$ associated with $\mcZ=(Z_\tau)_{\tau\in\mcT}$. The main result for this new stochastic data is Proposition \ref{lem_SumNegativeCorrector} which gives an explicit expression for the sum of correctors involving $Z$ and $Z^{\circ}$. 
	 
In Section~\ref{section_ParacAnsatz}, we define the paracontrolled ansatz and introduce, in Definition~\ref{def_PAMparacSystem}, the stronger notion of a paracontrolled system associated with \eqref{eq_gPAM}. This notion captures the full paracontrolled structure of a distribution satisfying the ansatz. In Proposition \ref{prop_pamansatz_to_pamsystem}, we show that these two notions are in fact equivalent, using a paralinearisation formula established in Theorem \ref{prop_paralinearization}. The paracontrolled structure provided by the notion of a paracontrolled system will play an important role in computing the right-hand side of \eqref{eq_gPAM} using our ansatz. We show that the resulting expression again admits the same paracontrolled form, which allows us to perform a fixed-point argument to solve the equation. A definition of the singular products is given in Subsection \ref{subsect_productmap}, it relies on the correctors $\bfC$ introduced in \cite{LocalExpansionsParacSystems}. The integration step is carried out in Subsection \ref{subsection_integration}. We finish the section by presenting two known examples the stochastic multiplicative heat equation and $ \Phi^4_3 $ in Subsection \ref{examples} to illustrate our generalised ansatz.

The fixed-point procedure is described in Section~\ref{section_fixedpoint} for smooth initial data. We start with a notion of solution in Definition \ref{def_solution_EqPAM} that involves both $u$ and the smooth remainder $ u^\#$. Then, one constructs a  map $\Gamma : u^\# \mapsto u$ as a perturbation of the identity in Proposition \ref{prop_AppGamma}, using the same trick as in \cite[Lemma 4.12]{AllezChouk}. We conduct the final argument for the fixed-point in Proposition \ref{fixed_point_prop}.

	 In Section~\ref{section_renormalized_equation}, we show how to compute the renormalised equation within this framework.  We first consider a renormalisation map $M$ that is compatible with the coaction $ \Delta$
	(see \eqref{eq_RenormMapCoprodcuct1}) and then we move to preparation maps (see Definition \ref{def_preparation_map}). The main results of the section are Propositions \ref{prop_eqrenormalised_0}
	and \ref{renormalised_preparation_pde}.

	In Section \ref{Sec::8}, we show the convergence of the stochastic data $\widehat{Z}^m_{\tau}$ (Theorem \ref{main_convergence}). We start by implementing the  regularity-integrability structures of \cite{Hos25,BH23} at the level of the words in Proposition \ref{prop_modelwords_besov}. This allows us to revisit the regularity of the main analytical objects of the solution theory in Besov spaces.  Then, one states the main theorem for the convergence, Theorem \ref{prop_alg_identity_convergence}, that provides a key identity between $ \delta Z $ and $Y$. Its proof relies on two algebraic identities Propositions \ref{lem_deltaZ_Zj*} and \ref{lem_Zj*tau=YjstarPsi}.  We finish with the proof of Theorem \ref{main_convergence} which uses the previous identity and a reconstruction argument given in Lemma \ref{cor_reconstruction_convergence}. 
		We finish the paper with several Appendices. Appendix \ref{appendix_A} that explains how the main results of \cite{LocalExpansionsParacSystems} generalise to paraproducts built from more general Littlewood-Paley projectors. 
		Appendix \ref{Appendix_B} proves Proposition \ref{prop_modelwords_besov} that constructs a regularity-integrability structure over iterated paraproducts.
		Appendix \ref{Appendix_C} provides a symbolic index.
		
		 \subsection*{Data Availibility} Data sharing not applicable to this article as no datasets were
		generated or analysed during the current study.
		
		\subsection*{Conflict of interest} The authors declare that they have no conflict of interest.
	
	 \subsection*{Acknowledgements}
	
	{\small
		Y.B. and N.M. gratefully acknowledge funding support from the European Research Council (ERC) through the ERC Starting Grant Low Regularity Dynamics via Decorated Trees (LoRDeT), grant agreement No.\ 101075208. Views and opinions expressed are however those of the author(s) only and do not necessarily reflect those of the European Union or the European Research Council. Neither the European Union nor the granting authority can be held responsible for them. }

	
	

	\section{Notations and prerequisites}
	\label{Sec::2}
	
	\subsection{Analytical setting} \label{subsection_analyticalsetting}
	
	We let $(e_i)_{0\leq i\leq d}$ be the canonical basis of $\bfR^{d+1}$, and we call multi-indices elements $k=(k^0,\dots,k^d)\in \bfN^{d+1}$. We work with the following scaling by setting $$|k|= \sum_{i=0}^{d}\mathfrak{s}_i k^i, \quad \mathfrak{s} = (\beta,1,...,1), \quad |\mathfrak{s}| = \beta + d. $$
	This scaling gives the following norm $|\cdot|$ on $\R^{d+1}$
	\begin{equs}
		|(t,x)| = |t|^{\frac{1}{\beta}} + \sum_{i=0}^d |x_i|.
	\end{equs}
	We also define for $\lambda>0$ the corresponding scaling map $$\mcS^\lambda(t,x_1,\dots,x_d) = (\lambda^\beta t,\lambda x_1,\dots\lambda x_d).  $$ For $n\in\bfN$ and multi-index $k\in\bfN^{d+1}$, we set
	$$
	\mcP_n(k)\defeq \Big\{(k_1,\dots,k_n)\in (\bfN^{d+1})^n,\quad \sum_{j=1}^nk_j = k      \Big\}.
	$$ 
	For any ${\bm k} \in \mcP_n(k)$, we set the multinomial coefficient 
	$ \binom{k}{\bm k } = \frac{k!}{k_1!\cdots k_n!}$.
	Distributions on $\bfT^d$ are regarded as periodic distributions on $\bfR^d$. Define the inverse operator $\Linv$ setting for any periodic distribution $v$ on $\bfR_+\times\bfR^d$
	$$
	(\Linv v)(t,x) = \int_0^t (P_{t-s} v_s)(x) \ddd s,
	$$
	where $(P_t)_{t\geq 0}$ is the semi-group generated by $L$.
		We denote by $K$ the kernel of the inverse operator $\Linv$.  We define for any $\ell\in \bfN^{d+1}$ the heat kernel with polynomial weight by setting
	\begin{equation}
		K^{[\ell]} (x)  = (-x)^\ell K(x).
	\end{equation}
	
	\subsection{Space-time paraproducts}\label{subsect_spacetime_paraproduct}
	We now introduce the space-time paraproduct we are going to use in our analysis. It is very close to the one defined in \cite{SpacetimeParaproducts}, except we are here in a flat setting, so that we define it via Fourier analysis instead of the heat flow. We set once and for all a Littlewood-Paley decomposition $(\tilde{\Delta}_i)_{i\geq -1}$ on $\bfR^d$. It is associated with a dyadic partition of the unity $(\rho, \chi)$ with $\rho$ supported on the annulus $\{ 1/2<|z|<2 \}$, such that for $ i \in \N $
	\begin{equation*}
		\tilde{\Delta}_{\tiny{-1}} = \mathscr{F}^{-1}( \chi \mathscr{F} \cdot ), \quad \tilde{\Delta}_{i} = \mathscr{F}^{-1}( \rho_i \mathscr{F} \cdot ), \quad \rho_i = \rho(2^{-i} \cdot),
	\end{equation*}
	where $ \mathscr{F} $ is the Fourier transform.
	Let $ \tilde{\Delta}_{<j} \defeq \sum_{i\leq j-1}\tilde{\Delta}_i$ and write $\tilde{K}_i$ for the kernel of the $i^{th}$ Littlewood-Paley projector $\tilde{\Delta}_i$. One has for a periodic distribution $ u $
	\begin{equation*}
		(\tilde{\Delta}_i u)(y) = \int_{\R^d} \tilde{K}_i(y-x) u(x) dx.
	\end{equation*}
	Let $B$ be some integer that will be chosen large enough later in Subsection \ref{subsec_subcriticality}, and set some $\phi\in C^\infty(\bfR,\bfR)$ that is supported in $[1/2,2]$, with unit mass and such that
	\begin{equation}
		\int_\bfR \phi(x)x^k \, \ddd x = 0  \quad \text{for } k\in\{1,..,B \}.
	\end{equation}
	We also define for any $i\geq0$ the smooth function 
	$$
	\phi_i \defeq {2^{i\beta}} \phi({2^{i\beta}\, \cdot}).
	$$
	We then set the space-time kernel 
	\begin{equation}
		K_{<i}(t,x) = \tilde{K}_{<i}(x) \, \phi_i(t).
	\end{equation}  
	We also set $K_i = K_{<i+1} - K_{<i}$ and $\Delta_{i}$ the operator associated with this kernel.
We set for any $\alpha\in\bfR$ the associated Hölder-Besov space $C^\alpha$ and write $\norme{\cdot}_\alpha$ for its norm given by
	$$
	\norme{u}_\alpha = \sup_{i\geq-1}2^{i\alpha} \norme{\Delta_i u}_{L^\infty}.
	$$ 
	We also let $ \Delta_{<j} \defeq \sum_{i\leq j-1}\Delta_i$. Define for any multi-indice $\ell \in \bfN^{d+1}$ the Littlewood-Paley projector with polynomial weight $\ell$ by 
	\begin{equation}\label{eqdef_littlewoodprojectorL}
		\Delta_i^\ell h (y)= \int_{\bfR^{d+1}} K_i(y-x)(x-y)^\ell h(x)\, \ddd x.
	\end{equation}
	with the convention $   \Delta_i^{0} = \Delta_i $. 
	We define the polynomial weight paraproduct of two distributions $h_1$ and $h_2$ by 
	$$
	\P_{{ \ell}}(h_1,\, h_2) \defeq \sum_{i\geq -1} \Delta^{\ell}_{<i-1} h_1 \; \Delta_i h_2.
	$$
	For any distributions $h_1,\dots,h_n$ and ${ \bm \ell }=(\ell_1,\dots \ell_{n-1})\in (\bfN^{d+1})^{n-1}$, we set the iterated paraproduct with weight ${\bm \ell}$ 
	$$
	\P_{{\bm \ell}}(h_1,\dots,h_n) =  \P_{\ell_{n-1}}\big( \P_{{\bm \ell}_{<n-1}}(h_1,\dots,h_{n-1}), \, h_n\big),
	$$
	where $ {\bm \ell}_{< n-1} = (\ell_1,\dots \ell_{n-2}) $.
	For $\ell\in\bfN^{d+1}$, we define
	\begin{equation}\label{eqdef_ParapPolWeight2}
		\P_{\ell}\big(h_1,\dots,h_n\big) = \sum_{{\bm \ell}\in \mcP_{n-1}(\ell)} \binom{\ell}{\bm \ell}  \P_{\bm \ell}\big(h_1,\dots,h_n\big).
	\end{equation}
The previous formula corresponds to split $\ell $ into $ n-1 $ polynomial weights that will appear in the iteration of the paraproduct.

	

\subsection{Basics on Regularity Structures  }\label{subsection_basicsRS}
We follow the formalism introduced in \cite{BH26}.	A regularity structure is a pair $\scrT=(T,T^+)$ consisting of a graded vector space $T$ and a graded algebra $T^+$, equipped with linear maps 
	\begin{equation*}
		\Delta : T \rightarrow T \otimes T^+,
		\quad
		\Delta^{\!+}  : T^+ \rightarrow T^+\otimes T^+,
	\end{equation*}
	making $T^+$ a graded Hopf algebra and $T$ a co-module over $T^+$. We suppose that each regularity structure $(T,T^+)$ we use comes with a basis $\mcB$ (resp. $ \mcB^+ $) of $T$ (resp. $T^+$) consisting of homogeneous elements. We denote by $\vert \cdot \vert$ the degree map defined on $T$ and $T^+$ taking values in $\R$. The homogenous subspace of $T$ (resp. $T^+$) of degree $\alpha$ is denoted by $T_{\alpha}$ (resp. $T_{\alpha}^+$). For any $\tau,\sigma\in\mcB$ we define $\tau/\sigma\in T^+$ by requiring
	$$\Delta \tau = \sum_{\sigma \leq \tau} \sigma \otimes  \tau/\sigma$$
	and write $\sigma<\tau$ when $\sigma\leq \tau$ and $\sigma\ne\tau$. We define in the same way $\tau/\sigma\in T^+$ for $\tau,\sigma\in \mcB^+$ and the notation $\sigma<\tau$ using this time the coproduct $\Delta^{\!+}$. The difference between $\Delta$ and $ \Delta^{\!+} $ will be clear from the context. In the same fashion we introduce for $\tau\in \mcB$ and $\sigma\in\mcB^+$ the element $\sigma \backslash \tau\in T$ by the equation
	\begin{equation}\label{eq_sigmabackslashtau}\Delta \tau = \sum_{\sigma} \sigma\backslash \tau \otimes \sigma,
	\end{equation}
	and likewise we define for $\tau,\sigma\in\mcB^+$ an element $\sigma\backslash \tau\in T^+$ by using $\Delta^{\!+}$ instead of $\Delta$. 
	We suppose that $\mathscr{T}$ contains the polynomial regularity structure $T_X$ whose basis is given by 
		 $$ \CB_X = \{  X^k, \, k \in \N^{d+1} \}$$
	where $ X^{k} = \prod_{i=0}^d X_{e_i}^{k_i} $. In the sequel, we will use the short hand notation $ X_{e_i} = X_i $.
	Then, one has
	\begin{equation*}
		\Delta X_i = \Delta^{\!+} X_i = X_i \otimes \one + \one \otimes X_i.
		\end{equation*}
		and extended multiplicatively which gives
		\begin{equation*}
			\Delta X^{k} = \Delta^{\!+} X^{k} = 
			\sum_{\ell \in \N^{d+1}} \binom{k}{\ell}  X^{\ell} \otimes X^{k-\ell}.
		\end{equation*}
	Model over a regularity structure will be written as a couple of maps ${ (\PPi,g)}$, with ${\ \PPi}:T\rightarrow \mcD'(\bfR^{d+1})$ and ${\ g}:T^+\rightarrow L^\infty(\bfR^{d+1})$. \label{model}
	To define the size $(\!| \PPi , g |\!)$ of a model $ (\PPi,g)$ we first let for $\tau\in T_{|\tau|}$ and $\nu\in T_{|\nu|}^+$
	\begin{equs} 
			\Vert\tau\Vert_{ (\PPi,g)} &\defeq  \sup_{x\in{\bfR^{d +1}}}\sup_{\phi \in \mathscr{B}}\sup_{\lambda\in(0,1]} \lambda^{-\vert\tau\vert} \big|\langle { \Pi}_x\tau, \phi_x^\lambda \rangle\big|,   \\
			\Vert\nu\Vert_{ (\PPi,g)} & \defeq  \sup_{x,y\in{\bfR^{d + 1}}}\frac{|{ \gamma}_{yx}(\nu)|}{|y-x|^{|\nu|}},
 \end{equs}
 where $ \mathscr{B} $ is  a suitable space of test functions  $ \varphi $ which are rescaled around a point $z = (t,x)$, $ \varphi_z^{\lambda} = \lambda^{-\beta-d} \varphi(\lambda^{-\beta}t, \lambda^{-1}x) $.
  The maps $  \Pi_x $ and $\gamma_{yx}$ are defined as
 \begin{equation*}
 	\Pi_x = \left( \PPi \otimes g_x \CA_+  \right) \Delta, \quad \gamma_{yx} = \left( g_y  \otimes g_x \CA_+ \right) \Delta^{\!+}
 \end{equation*}
 where $ \CA_+ $ is the antipode for the Hopf algebra $T^+$. Here, we use the notations of \cite{BH26}. In the notations of \cite{Hai14}, one uses a map $f_x$ which is equal to $g_x \CA_+$.  
 Then recursively for $\tau\in \mcB \sqcup \mcB^+ $, one has
	\begin{equation*} \begin{split}
			\Vert\tau\Vert_{ (\PPi,g)}^* \defeq \max \Big( \Vert\tau\Vert_{ (\PPi,g)},\, \max_{\sigma<\tau} \Vert\tau/\sigma\Vert_{ (\PPi,g)} \Vert\sigma\Vert_{ (\PPi,g)}^*\Big).  
	\end{split} \end{equation*}
	We then set
	\[
	(\!| \PPi,g |\!) \defeq \max_{\tau\in\mcB, \mu\in\mcB^+}\Big(\Vert\tau\Vert_{ (\PPi,g)}^*  \,,\, \Vert\mu\Vert_{ (\PPi,g)}^*  \Big).
	\]
	For any regularity structure $\mathscr{T}$ and model $({ \PPi},{ g})$ on it, we set for any $\gamma\in\bfR$, a modelled distribution ${\bm f}\in\mcD^\gamma(T, { g})$ which is a map ${\bm f} : {\bfR^{d+1}} \rightarrow \bigoplus_{r'<\gamma}T_{r'}$ such that
\begin{align*}
	&\max_{\beta<\gamma}\sup_{x\in{\bfR^{d+1}} } \norme{{\bm f}(x)}_{\beta} <+\infty,
	\quad \max_{\beta<\gamma}\sup_{x,y\in{\bfR^{d+1}}} \frac{\big\Vert {\bm f}(y) - {\Gamma}_{yx}({\bm f}(x)) \big\Vert_{\beta}}{|y-x|^{\gamma-\beta}} < +\infty
\end{align*}
where $\Gamma_{yx} = (\id\otimes \gamma_{yx})\Delta$. We also set $\norme{{\bf f}}^*$ to be the maximum of these two supremum.
	We now state the reconstruction theorem.
For any regularity structure $\scrT$, for any model $(\PPi,g)$ over it,
and regularity exponent $\gamma\in\bfR$. There exists a linear continuous operator \label{reconstrcution_op}
$$
{\bm R} : \mcD^\gamma(T,{ g})  \to \mcD'(\bfR^{d+1})
$$
satisfying the property
$$ 
\big|\langle {\bm R} {\bm f} - { \Pi}_x{\bm f}(x) , \varphi_x^\lambda \rangle \big|  \lesssim \norme{\bm f}_{\mcD^\gamma} (| \PPi,g |)\lambda^\gamma
$$ 
uniformly in ${\bm f}\in\mcD^\gamma$, in $\varphi\in \mathscr{B} $ and in $\lambda\in (0,1)$. Such an operator is unique when the exponent $\gamma$
is positive.
	In the sequel, we will work with two regularity structures that we introduce in the next subsections.

	

	\subsection{Decorated trees and their Hopf algebraic structure}	\label{subsection_DecoratedTrees}		
	
	A decorated tree $T_{\frak e}^{\frak n, \frak f}=(\tau,\frak n, \frak e,\frak f)$ consists in a non-planar rooted tree $\tau$ with nodes set $N_\tau$ and edges set $E_\tau$, equipped with maps $\frak{n},\frak{f} : N_\tau\to \bfN^{d+1} \times \{1,...\ell_0\} $ and $\frak e : E_\tau \to \bfN^{d+1}$. The decoration $\frak{n}$ represents monomials at each node, the decoration $\frak{e}$ represents the derivatives of  the integration kernels, and $\frak{f}$ represents the noises. We will often only write $\tau$ for the whole decorated tree $T_{\frak e}^{\frak n, \frak f}$. 
	We also write $\mcB$ for the set of decorated trees, and $\mcB_o$ for the set of decorated trees with no noise at the root which is described by $\mathfrak{f}$ being equal to $1$ on the root. A general decorated tree has a unique decomposition 
	$$
	\tau = X^k\Xi_\ell \prod_{j=1}^m \mcI_{b_j}(\tau_j),
	$$
	where $ k $ corresponds to the monomial decoration at the root $ \rho $ of $\tau$, $\ell$ is the noise decoration at $\rho$, the decorated trees $\tau_j$ are connected to the root $\rho$ via edges decorated by $b_j$. 	If $ \ell= 1 $, we say that  the noise decoration is trivial and the symbol $\Xi_{\ell}$ can be omitted. The tree of the form $\Xi_{1}$ is also identified with the empty tree denoted by $\one$, the unit for the tree product.
	Below, we provide the example of two decorated trees appearing for the stochastic multiplicative equation (see Subsection \ref{examples})
	\begin{equs}
	\Xi = \<generic>, \quad	\CI(\Xi)^2 \Xi  = \<Xiiia>, \quad  \CI( \CI(\CI( \Xi) \Xi) \Xi) \Xi  = \<Xi4>.
	\end{equs}
	Here, a thin edge encodes the symbol $\CI$ an abbreviation of $ \CI_{0}$. In this example, we have only two noises $1$ and $ \xi$ ($ \ell_0 = 2 $).
		If $\tau$ writes as $\tau = X^k\Xi_l \prod_{j=1}^m \mcI_{a_j}(\tau_j)^{b_j}$, where the couples $(\tau_j,a_j)$ are distinct, the symmetry factor of a tree $\tau$ is the integer defined recursively by \label{symmetry_factor}
	$$
	S(\tau)=k! \prod_{j=1}^m S(\tau_j)^{b_j}b_j!.
	$$
	One can compute the symmetry factor of the two trees that we have introduced:
	\begin{equs}
		S(\<Xiiia>) = 2, \quad  S(\<Xi4>) =  1.
	\end{equs}
	The factor $2$ for the symmetry factor of $\<Xiiia>$ comes from the fact that the two branches attached to its root are the same equal to $ \<IXi> $.
We now assign a degree to any decorated tree by setting for any $\tau = X^k\Xi_\ell \prod_{j=1}^m \mcI_{a_j}(\tau_j)^{b_j}$
	\begin{equation} \label{degree_tree}
	|\tau| = |k|+\alpha_\ell + \sum_{j=1}^m b_j(\beta-|a_j| + |\tau_j|).
\end{equation}
	where the $\beta$ corresponds to the gain of regularity from the convolution with the kernel $K$, $\alpha_{\ell}$ is the regularity of the noise $ \xi_{\ell} $ ($ |\Xi_{\ell}| =\alpha_{\ell} $).
	For the stochastic multiplicative heat equation, one has
	\begin{equs}
		\beta = 2, \quad \alpha  = -\frac{3}{2}- \kappa
	\end{equs}
	Therefore, one has
	\begin{equs}
		| \<Xiiia> | & = 2 \beta + 3 |\<generic>| = 2 \times 2 + 3 \times (-\frac{3}{2}- \kappa) = -\frac{1}{2} - 3 \kappa \\  |\<Xi4>| &   = 3 \beta + 4 |\<generic>| = 3 \times 2 + 4 \times (-\frac{3}{2}- \kappa) = - 4 \kappa. 
	\end{equs}
	We denote by  $\mathcal{T}$ the vector space generated by the set of decorated trees and $\mathcal{T}^+$ the vector space
	$$
	\mathcal{T}^+ = \text{Span} \{  X^k\prod_{j=1}^n \mcI^+_{b_j}
	(\tau_j) ,\quad  k\geq0, \, \tau_1,\dots,\tau_n \in\mcT , \, |b_j|<|\tau_j|+\beta    \}.
	$$
\label{I_+}	We have replaced $ \mcI $ by $ \mcI^{+} $ to stress that there is no noise decoration at the root of these decorated trees. The notation $ \prod_{i}^n $ is the tree product. It merges the roots of two  decorated trees and add their root decorations.
	If the decorated trees have a non-trivial noise decorations at their roots, their product is set to be zero. 
		We now give the precise definitions of the maps $ \Delta $ and $ \Delta^{\!+} $.
	First, we set for $\tau\in\mcT$ and $b\in{\bfN}^{d+1}$ 
	
	\begin{equation}\label{eq_CoproductOfMcItau}
	\begin{aligned}	\Delta \mcI_b(\tau) &= ( \mcI_b \otimes \id ) \Delta \tau + \sum_{|n|<|\tau|+\beta-|b|} \frac{X^n}{n !}  \otimes \mcI_{b+n}^{+}(\tau)
		\\
		\Delta^{\!+} \mcI_b^+(\tau) &= ( \mcI_b^+ \otimes \id ) \Delta \tau + \sum_{|n|<|\tau|+\beta-|b|} \frac{X^n}{n !}  \otimes \mcI_{b+n}^{+}(\tau)
		\end{aligned}
	\end{equation}
	and for $\ell\in \{1,...,\ell_0\}$ and $ i \in \{0,...,d\}$
	$$
	\Delta \Xi_\ell = \Xi_\ell\otimes \one, \quad \Delta X_i = \Delta^{\! +}X_i  = X_i \otimes \one + \one \otimes X_i. 
	$$ 	We extend by multiplicativity these maps. 
	One way to compute the coaction $ \Delta $ is to use admissible cuts. One puts on the right hand side of the tensor product the branches that are cut. By admissible, we mean that it is forbidden to cut two edges lying on the same path to the root.
	As an example, one has
	\begin{equs}
		\Delta \, \<Xi4> =  \<Xi4> \otimes \one  + \one \otimes \<Xi4> + 
	\<generic> \otimes	\<Xi4_11>  +  \<genericx> \otimes \<Xi4_12>  +
	\<Xii> \otimes 	\<Xiiibis>   + \<Xiii> \otimes \<IXi>.  
	\end{equs}
	One has the possibility to cut three edges in $ \<Xi4> $. Only the edge connected to the root gives extra polynomial decorations. Here, a brown thick edge encodes $ \CI_{(0,1)} $ and $ \<genericx> $ is for $ X_{(0,1)}\Xi $. In the previous computations, we have made an abuse of notation as we did not distinguish between  $ \CI $ and $ \CI^+ $ in the decorated tree notation. 
One also defines the dual Grossman-Larson product on $T$ setting for $\mu = X^k\prod_{j=1}^n\mcI_{b_j}(\mu_j)$
	\label{star_product}
	$$
	\mu \star \tau = \uparrow^k_{N_\tau} \Big(\prod_{j=1}^n\mcI_{b_j}(\mu_j) \curvearrowright \tau \Big),
	$$
	where $\prod_{j=1}^n \mcI_{b_j}(\mu_j) \curvearrowright \tau$ denotes the sum of the trees obtained by grafting simultaneously each $\mu_j$ to some node of $\tau$ via an edge decorated by $ b_j $. We define the grafting of one tree below: 
	\begin{equation*}
\CI_{a}(\sigma) \curvearrowright \tau	= 	\sigma \curvearrowright_a \tau \defeq \sum_{v\in N_{\tau}}\sum_{m\in\N^{d+1}} {\Labn_v \choose m} \,\sigma  \curvearrowright^v_{a-m}(\uparrow_v^{-m} \tau),
	\end{equation*}
	where $\Labn_v\in\N^{d+1}$ is the polynomial decoration at the node $v$, {\color{black} the binomial coefficient ${\Labn_v \choose m}$ equals $\prod_{0\leq i\leq d}{(\Labn_v)_i \choose m_i}$}, and $\curvearrowright^v_{a-m}$ grafts $ \sigma $ onto $\tau$ at the node $ v $ with an edge decorated by $a -m$. The above sum is finite due to the binomial coefficient $ {\Labn_v \choose m} $ which is equal to zero if $m$ is greater than $ \Labn_v $, by convention. The operator $\uparrow_v^{-m}$ adds $ -m $ to the polynomial decoration at the node $v$.
	One has
	 $$\uparrow^k_{N_\tau} = \sum_{\sum_{v\in N_\tau}k_v=k} {k \choose (k_v)_{v \in N_{\tau}}} \prod_{v\in N_\tau}\uparrow_v^{k_v}.$$ When in particular $\mu\in T^+$, we have the duality formula
	$$
	\lg \mu \star \tau , \nu     \rg = 	\lg \tau \otimes\mu , \Delta\nu     \rg,
	$$
	where the bracket $\lg\, \cdot, \cdot\,\rg$ from the left hand side  is the scalar product on $T$ such that two distinct decorated trees are orthogonal and such that $\lg\tau,\tau\rg =S(\tau)$. This scalar product is extended to $T\otimes T^+$ by setting $\lg a\otimes b,c\otimes d\rg=\lg a,c\rg\lg b, d\rg$.
 We illustrate this duality relation by the following computations:
 \begin{equs}
 	\Delta \<Xiiia> = \<Xiiia> \otimes \one + \one \otimes \<Xiiia> + 2 \, \<Xii> \otimes  \<IXi>    +  \<generic> \otimes
 	\<Xiiiaa>.
 	\quad
 	\<Xiiiaa> \star 	 \<generic> = \<Xiiia>,
 	\\
 	\lg \<Xiiiaa> \star \<generic>,  \<Xiiia>   \rg = S(\<Xiiia>) = 2, \quad 	\lg \<generic> \otimes \<Xiiiaa>  , \Delta\<Xiiia>     \rg =  S(\<generic>  ) S(\<Xiiiaa> ) = 2.
 \end{equs}	
 As $ \sigma \star \tau $ is a sum of decorated trees with the same degree, we will make the following abuse of notation by denoting this degree by $ | \sigma \star \tau| $. 
 
	The naive model $(\PPi, {g})$ is given by
	\begin{equation} \label{def_PPi_trees}
		\begin{aligned}
\PPi \Xi_{\ell} & = \xi_{\ell},	\quad  (\PPi X^k )(x) = x^{k}, \\	\PPi \CI_n(\tau) & = \partial^{n} K * \PPi \tau, \quad  \quad \PPi \tau_1 \tau_2 = \PPi \tau_1 \PPi \tau_2,
\end{aligned}
	\end{equation}
	and 
	\begin{equation} \label{def_g}
		\begin{aligned}
			g_x(\CI^+_k(\tau)) = \sum_{\sigma \leq \tau}  g_x(\tau / \sigma) (\partial^k K * \Pi_x \sigma)(x)
		\end{aligned}
	\end{equation}
	which is given in \cite[Section 3.3]{BH26}. The map $\Pi_x$ is also given recursively as
	\begin{equation}
	\label{def_Pi_trees}
	\begin{aligned}
		(\Pi_x \Xi_{\ell})(y) & = \xi_{\ell}(y),	\quad  (\Pi_x X^k )(y) = (y-x)^{k}, \quad  \Pi_x \tau_1 \tau_2 = \Pi_x \tau_1 \Pi_x \tau_2, \\	(\Pi_x \CI_n(\tau))(y) & = (D^{n} K * \Pi_x \tau)(y) - \sum_{|\CI_{n+k}(\tau)| > 0}   \frac{(y-x)^{k}}{k!}(\partial^{n+k} K * \Pi_x \tau)(x). 
	\end{aligned}
	\end{equation}
	Continuing the examples, one has
	\begin{equs}
		\PPi \, \<Xiiia> & = \xi (K * \xi)^2 , \quad \PPi \, \<Xi4> =  K* (  K * (   (K * \xi) \xi ) \xi ) \xi,
		\\
		\Pi_x \<Xiiia> & = \xi (K * \xi - (K * \xi)(x) )^2 ,
		\\
		\Pi_x \<Xi4> & =  \PPi \, \<Xi4>
		-  (\PPi \<Xi4_11>)(x)  \xi  - (\cdot - x_{(0,1)}) (\PPi \<Xi4_12>)(x)  \xi 	-
	(\PPi \<Xiiibis> )(x)	\PPi \<Xii> 	 \\ &  - (\PPi \<IXi>)(x) \PPi \<Xiii> + 2 (\PPi \<IXi>)(x) (\PPi \<Xiiibis> )(x)  \xi.
	\end{equs}
	

	\subsection{Non-linearities, elementary differentials and subcriticality}
	
	\label{subsec_subcriticality}

	We introduce a set of indeterminates $\mcX=(\mcX_{k})_{k\in\bfN^{d+1}}$ and write $\mcC$ for the set of smooth functions  that depends on a finite number of the indeterminates from $\mcX$. These indeterminates $\mcX$ are meant to represent $u$ and its derivatives.
	For $F\in\mcC$ and $k\in \bfN^{d+1}$, we write $D_kF$ for the partial derivative of $F$ with respect to the variable $\mcX_k$. We also set
	$$
	\partial^k F (\mcX) = \sum_{j\in\bfN^{d+1}}  \mcX_{j+k} \,  (D_{j} F)(\mcX).
	$$
	We consider non-linearities that are polynomials in the $\mcX_k$ for $ k \in A_{-} $ and non-polynomials in $ A_+ $ where $A_-, A_+ \subset \N^{d+1}$. Recalling that the right hand side of \eqref{eq_gPAM}, is a sum of terms of the form $F_\ell(v)\xi_\ell$, we set for any $\ell\in \{1,...,\ell_0 \} $
	$$
	\Upsilon_F[\Xi_\ell]( \mcX ) = F_\ell((\mcX_k)_{k\in A_-\sqcup A_+} ).
	$$
	For a general decorated tree $\tau = \Xi_\ell X^k \prod_{j=1}^n \mcI_{a_j}(\tau_j)$, we set
	\begin{equation}\label{eq_defUpsilon}
	\Upsilon_F[\tau] = \prod_{j=1}^n \Upsilon_F[\tau_j] \,  \partial^k D_{a_1}\dots D_{a_n} \Upsilon_F[\Xi_\ell].
\end{equation}
We can compute the elementary differential for some trees coming from the stochastic multiplicative heat equation. One has
\begin{equs}
	\Upsilon_F[\Xi](u) & = \Upsilon_F[\<generic>](u) = f(u),
	\\
	\Upsilon_F[\<Xii2>](u) &= \Upsilon_F[\<generic>](u)\partial^{(0,1)} D_u \Upsilon_F[\<generic>](u) = \partial_x u f(u) f''(u),
	\\ \Upsilon_F[ \<Xiiia> ](u)&  = (\Upsilon_F[\<generic>](u) )^2 D_u^2 \Upsilon_F[\<generic>](u) = f(u)^2 f''(u)
\end{equs}
where $D_u$ is the derivative in the variable $u$.
In order to be able to solve \eqref{eq_gPAM}, we need to make some subcriticality assumption on the non-linearity of its right hand side. Following \cite{BCCH}, we introduce $\mcL_-=\{\Xi_1,\dots,\Xi_{\ell_0} \}$ the set of noises and $\mcL_+ = \{  \mcX_k \}$ the set of dummy variables representing $u$ and its derivative at first order. We suppose that there exists some application $\reg : \mcL_- \sqcup \mcL_+ \to \bfR$ satisfying the following properties

\begin{itemize}
\item For any $\ell\in\{1,...,\ell_0\}$, we have $\reg(\Xi_\ell)< 0$ and $\reg(\Xi_\ell)<|\Xi_\ell|$.
\item We have $\reg(\mcX_k)>0$ for $k\in A_+$.
\item For any $\ell\in\{1,...\ell_0\}$ and corresponding non linearity $$F_\ell( (\partial^k u)_{k \in A_+} )\prod_{k\in A_-} (\partial^{k}u)^{p_{k,\ell}},$$ we have
\begin{equation} \label{eq_subcriticality}
	\reg(\mcX_0)< \beta  + \reg(\Xi_\ell) + \sum_{k\in A_-} p_{k,\ell} \reg(\mcX_k).
\end{equation}	

\end{itemize}
As the number of noise types $\{1,...,\ell_0\}$ is finite, there exists a constant $\kappa>0$ such that for any $\ell \in\{1,...,\ell_0\}$
\begin{equation}  \label{eq_subcriticality_kappa}
	 \beta  + \reg(\Xi_\ell) + \sum_{k\in A_-} p_{k, \ell} \reg(\mcX_k) -	\reg(\mcX_0) > 2\kappa.
\end{equation}
We let $\mcT^{nc}$ to be the set of trees such that the elementary differential $\Upsilon[\tau]$ is non constant. For any $\tau\in\mcT^{nc}$, we let $r_\tau$ to be the smallest integer such that one can write 
$$
\Upsilon_F[\tau](\mcX)  = \Upsilon_F[\tau]\big( (\mcX_j)_{|j|\leq r_\tau} \big).
$$
We set 
\begin{equs} \label{gamma_0_b}
	\gamma_0 = \beta -\min_{\tau\in\mcT^{nc}}\{  |\mcI(\tau)| - r_{\tau}  \} + \kappa, \quad \gamma'_0 = \gamma_0 +\kappa/2.
\end{equs} 
The space $ \mcT^{nc} $ allows us to determine the optimal length of the paracontrolled ansatz. It appears in Proposition \ref{prop_regularite_utaudash_gammaell} and one can see an example in the $ \Phi^4_3 $ model in Subsection \ref{example_Phi_3_4}.
For $1\leq\ell\leq \ell_0$, we set 
\begin{equation} \label{eq_defgamma_ell}
	\gamma_\ell = \gamma_0 +  \min_{\tau\Xi_\ell\in\mcT^{nc}}\{ |\mcI(\tau\Xi_\ell)|- {r}_{\tau\Xi_\ell} \}.
\end{equation}
It is apparent that 
\begin{equation} \label{inequality_gamma}
	\gamma_\ell \geq \beta +\kappa.  
\end{equation}
This inequality will be useful in Subsection \ref{Subsection_RHS} for defining the singular product $F_\ell(v)\xi_\ell$ through a corrector lemma.


The integer $B$ from the definition of the space-time paraproduct $\P$ in Subsection \ref{subsect_spacetime_paraproduct} is chosen such that 
$$
B >  \max\{ \gamma_\ell, \enskip \ell\in\{1,...,\ell_0\}   \},
$$
as this maximum is the highest order of the local expansions in this work.

The following lemma is a consequence of subcriticality and will be used in the definition of our paracontrolled ansatz.

\begin{lemma}\label{lem_elementary differential_form}
	Under the subcriticality assumption, for any integer $n\geq 1$ and tree $\tau$ such that $|\mcI(\tau)|<n$, the elementary differential attached to $\tau$ has the form
	$$
	\Upsilon_F[\tau](\mcX)  = F_\tau\left((\mcX_k)_{|k|<n}\right).
	$$
	for some function $F_\tau$ that is polynomial on its negative arguments.
\end{lemma}

\begin{proof}
It is already clear from the definition that $\Upsilon_F[\tau]$ is polynomial on its negative arguments, it remains to check that it depends only on the $\mcX_k$ with $|k|$ small enough. We prove it by contradiction by supposing that there exist some tree $\tau_0$ with $|\mcI(\tau_0)|<n$ and non-trivial dependence on $\mcX_p$ for some $|p|\geq n$, that is $D_p\Upsilon_F[\tau_0]\ne 0$.

We let $\tau_1 = \mcI_p(\tau_0) \star \tau_0$, where $\star$ is the Grossman-Larson product introduced in Subsection \ref{subsection_DecoratedTrees}. From the pre-Lie morphism property of the elementary differential from Lemma \ref{lem_ElemDiffMorphism}
$$
\Upsilon_F[\tau_1] =\Upsilon_F[\tau_0] D_p \Upsilon_F[\tau_0]. 
$$
We have $D_p\Upsilon_F[\tau_1]\ne 0$, because both factors $\Upsilon_F[\tau_0]$ and  $D_p\Upsilon_F[\tau_0]$ are polynomials in $\mcX_p$, with $\Upsilon_F[\tau_0]$ being non constant and $D_p\Upsilon_F[\tau_0]$ being non-zero. We also have $|\tau_1|= 2|\tau_0|+\beta-|p|<|\tau_0|$.
Iterating this procedure yields an infinite sequence of trees $(\tau_j)_{j\geq0}$ with strictly decreasing degree, and that are non vanishing as $D_p\Upsilon_F[\tau_j]\ne 0$. This contradicts subcriticality.
\end{proof}



	\subsection{The regularity structure of iterated paraproducts} \label{subsect_theRSIteratedParap}
	
	We describe here a suitable version of the regularity structure introduced in \cite{LocalExpansionsParacSystems} which encodes the local expansion structure of iterated paraproducts.
	Let $\mathcal{A}=\big\{a_1,\dots,a_M\big\}$ be an alphabet, where each letter $a_i$
	is assigned a degree $|a_i|$. For a word $ w = w_1....w_n $ where $w_i \in  \mathcal{A}$,  one sets
	$
		|w| = \sum_{i=1}^n |w_i|$.
	 We assume that the degree of any non-empty word is never an integer. 
	
	A basis $\mcB$ of $T$ consists in a set of words $\mcU$ over $\mathcal{A}$ together with two multi-indices $\ell,p\in\bfN^{d+1}$. The algebra $T^+$ will be generated by words of $\mcU$ now equipped with three multi-indices $k,\ell,p\in\bfN^{d+1}$. We will use the following notation to represent these basis elements
	\begin{equation}
		\mcB = \Big\{ \lg w\rg_\ell X^p, \quad  w \in \mcU, \enskip \ell,p\in\bfN^{d+1}   \Big\}, \quad T =  \lg  \mcB \rg = \mathcal{T}_{\mathcal{A}}
	\end{equation}
	and
	\begin{equation} \begin{aligned}
		\mcB^+ & = \Big\{ \lg w\rg_\ell^k  X^p, \quad  w \in \mcU, \enskip k,\ell,p\in\bfN^{d+1}, \enskip |w|+|\ell|-|k| >0  \Big\}
		\\ T^+ &= \lg \mcB^+ \rg = \mathcal{T}_{\mathcal{A}}^+.
		\end{aligned}
	\end{equation}
	One also sets degrees over the symbols in $\mcB$ and $\mcB^+$ by
	\begin{equation}
		\big| \lg w\rg_\ell X^p \big| = |w|+|\ell|+|p|, \qquad\quad \big| \lg w \rg_\ell^k X^p \big| = |w|+|\ell|-|k|+|p|.
	\end{equation}
	One defines a coproduct and a coaction by setting
	\begin{equation}
		\label{coproduct_words}
		\begin{aligned}
		\Delta \lg w \rg_\ell & = \sum_{w=w_1w_2}\sum_{\ell=\ell_1+\ell_2 } \sum_{\substack{ |r|<|w_1|+|\ell_1| \\ r=r_1+r_2}} \frac{1}{r_1!r_2!} \binom{\ell}{\ell_1}   \lg w_2 \rg_{\ell_2+r_1}X^{r_2} \otimes  \lg w_1 \rg_{\ell_1}^r,
	\\
		\hat{\Delta}^{\!+} \lg w \rg_\ell^k & = \sum_{\substack{ w=w_1w_2 \\ \ell=\ell_1+\ell_2 \\ k=k_1+k_2}} \sum_{ r=r_1+r_2}\frac{1}{r_1!r_2!} \binom{\ell}{\ell_1}\binom{k}{k_1} \, \lg w_2 \rg_{\ell_2+r_1}^{k_2} X^{r_2} \otimes  \lg w_1 \rg_{\ell_1}^{k_1+r},
		\\ 	\Delta^{\!+} & = \left( \mathfrak{p}_+  \otimes \mathfrak{p}_+ \right) \hat{\Delta}^{\!+}, \quad \Delta X_i = \hat{\Delta}^{\!+} X_i = X_i \otimes \one + \one \otimes X_i 
		\end{aligned}
	\end{equation}
	and by extending it to $(T,T^+)$ by multiplicativity.
	Here, $ \mathfrak{p}_+ $ keeps the words with positive degree belonging to $\mathcal{B}_+$. \label{model_words}
	Suppose that we are given for any letter $a\in \mathcal{A}$ a distribution $Z_a\in C^{|a|}$, one defines a model ${ (\PPi^\mathcal{A},g^\mathcal{A})}$ over $\scrT_\mcA = (\mathcal{T}_{\mathcal{A}}, \mathcal{T}_{\mathcal{A}}^+)$ by setting
	\begin{equation}\label{def_model_words}\begin{split}
			{\PPi}^\mathcal{A}\big( \lg a_1\dots a_n\rg_\ell  \big) &= \P_\ell\big(Z_{a_1},\dots,Z_{a_n}\big),
			\\
			{g}^\mathcal{A}\big( \lg a_1\dots a_n\rg ^k_\ell  \big) &= \partial^k_*\P_\ell\big(Z_{a_1},\dots,Z_{a_n}\big),
	\end{split}\end{equation}
	where the derivatives $\partial^k_*$ are defined in \eqref{eqdef_starderivative}. This is a weaker result than the one in \cite{LocalExpansionsParacSystems} as one studies the expansions of $\P_{\bm \ell}(h_1,\dots,h_n)$ for ${\bm \ell}\in(\bfN^{d+1})^{n-1}$, whereas we describe here the local expansion of the paraproduct $\P_\ell$ obtained by regrouping the $\P_{\bm \ell}$ as in \eqref{eqdef_ParapPolWeight2}.
	 The star derivative $\partial^k_*$ of an iterated paraproduct $\P_\ell(h_1,\dots,h_n)$ involving distributions $h_j$ of regularity $\beta_j$ is defined recursively by
	\begin{equation}\label{eqdef_starderivative}	\begin{aligned}
			& \partial^k_* \P_\ell\big(h_1,\dots,h_n\big) = \partial^k \P_\ell\big(h_1,\dots,h_n\big)  
			\\
			&\, -
			\sum_{j=1}^{n-1}\sum_{\substack{\ell=\ell_1+\ell_2\\ k=k_1+k_2}}\sum_{|p|\leq \rho_j(k_1,l_1)} \frac{1}{p!}\binom{k}{k_1}\binom{\ell}{\ell_1} \partial^{k_1+p}_* {\P}_{\ell_1}\big(h_1,\dots,h_j\big) 
		\\ & \times\partial^{k_2}_*\P_{\ell_2+p}\big(h_{j+1},\dots,h_n\big),
	\end{aligned}\end{equation}
	where
	\begin{equation}
		\label{condition_sum}
		\rho_j(k_1,\ell_1) = \min \Big\{ \sum_{r=1}^j \beta_r+|\ell_1|-|k_1|, \quad -\sum_{r=j+1}^n\beta_r-|\ell_2|+|k_2|      \Big\} .
	\end{equation}
	The operator $\partial^k_*\P_\ell$ is dependent on the uplet $(\beta_j)$, we will however omit it in the notations as the $\beta$ corresponding to some distribution will always be its natural regularity in the context where it is used.  
	One can reformulate \eqref{eqdef_starderivative} using the coaction  $  \Delta$ in the next proposition:
	\begin{proposition} \label{def_prop_antipode} One first sets
		\begin{equation}
			\CA_{*} \lg w \rg_\ell^k = \lg w \rg_\ell^k - \mathcal{M} \left( \CA_{*} \mathfrak{p}_{-} \otimes \CA_{*} \mathfrak{p}_+ \right) \tilde{\Delta}^{\! +} \lg w \rg_\ell^k
		\end{equation}
		where $ \tilde{\Delta}^{\! +}$ is the reduced coproduct of $ \hat{\Delta}^{\! +}$\textnormal{:} $$ \tilde{\Delta}^{\! +}  \lg w \rg_\ell^k = \tilde{\Delta}^{\! +}  \lg w \rg_\ell^k - \lg w \rg_\ell^k \otimes \one - \one \otimes \lg w \rg_\ell^k.$$ Then, one has
		\begin{equation}
			g( \lg w \rg_\ell^k )  =\PPi( \CA_{*} \lg w \rg_\ell^k ).
			\end{equation}
		\end{proposition}
		\begin{proof} We proceed by induction on the size of the words. Let $ w = w_1...w_n $ a word on $\CA$. One has from \eqref{def_model_words} and \eqref{eqdef_starderivative}
			\begin{equation*}
				\begin{aligned}
				&	g( \lg w_1....w_n \rg_\ell^k )  = \partial^k_*\P_\ell\big(Z_{w_1},\dots,Z_{w_n}\big)
					\\
					 &  = \partial^k \P_\ell\big(Z_{w_1},\dots,Z_{w_n}\big)  
					\\
					&\, -
					\sum_{j=1}^{n-1}\sum_{\substack{\ell=\ell_1+\ell_2\\ k=k_1+k_2}}\sum_{|p|\leq \gamma_j(k_1,l_1)} \frac{1}{p!}\binom{k}{k_1}\binom{\ell}{\ell_1} g(\lg h_1 \dots h_j \rg_{\ell_1}^{k_1+p})
					 g( \lg h_{j+1} \dots h_n \rg_{\ell_2 + p}^{k_2})
					 \\ & =  \PPi (\lg w \rg_{\ell}^{k})  
					 \\
					 &\, -
					 \sum_{j=1}^{n-1}\sum_{\substack{\ell=\ell_1+\ell_2\\ k=k_1+k_2}}\sum_{|p|\leq \gamma_j(k_1,l_1)} \frac{1}{p!}\binom{k}{k_1}\binom{\ell}{\ell_1} \PPi( \CA_{*}\lg h_1 \dots h_j \rg_{\ell_1}^{k_1+p})
					 \PPi(\CA_{*} \lg h_{j+1} \dots h_n \rg_{\ell_2 + p}^{k_2})
					 \\ & =  \PPi(\lg w \rg_\ell^k) -  \left(  \PPi \CA_{*} \mathfrak{p}_{-} \otimes \PPi \CA_{*} \mathfrak{p}_+ \right) \tilde{\Delta}^{\! +} \lg w \rg_\ell^k
					 \\ & = \PPi( \CA_{*} \lg w \rg_\ell^k)
				\end{aligned}
			\end{equation*}
			where from the second equality to the third we have used the induction hypothesis.
			\end{proof}

	\begin{example}
		If $f,g$ are regarded as functions in $C^{\beta_1}$ with $\beta_1\in(0,1)$, one has
		$$
		\partial^{e_i}_*\P(f,g) = \partial^{e_i} \P(f,g) - f \, \partial^{e_i} g.
		$$
		The second term satisfies the condition \eqref{condition_sum}: $ \beta_1 >  0 $, $ \beta_1-|e_i| = \beta_1-1 < 0 $.
		If $(f,g,h)\in \prod_{i=1}^3 C^{\beta_i} $ with $(\beta_i)_i = (3/2 , -4/3,7/6  ) $
		$$
		\partial_*^{e_i} \P(f,g,h) = \partial^{e_i} \P(f,g,h) - f \, \partial^{e_i}_*\P(g,h)  - \sum_{j=1}^d\partial^{e_j}f \, \partial^{e_i}_*\P_{e_j}(g,h) - \partial^{e_i}\!f \,\P(g,h)
		$$
	\end{example}
	We define  for smooth functions $h_1,\dots,h_n$ the commutator $ \bfC_\ell(h_1,\dots,h_n) $ as 
	\begin{equation}\begin{split}\label{eqdef_corrector}
			\bfC_\ell(h_1,\dots,h_n) = \partial^{0}_* \P_\ell(h_1,\dots, h_n).
	\end{split}\end{equation}

\begin{example}
	If $f\in C^{\beta_1}$ and $g\in C^{\beta_2}$ with $\beta_1>0$ and $\beta_2<0$
	$$
	\bfC (f,g) = \P(f,g) - f g,
	$$
	and if $k\ne 0$, we have 
	$
	\bfC_k(f,g) = \P_k(f,g).
	$
	If $(f,g,h)\in \prod_{j=1}^{3}C^{\beta_j}$ with $\beta_1,\beta_1+\beta_2>0$ and $\beta_3,\beta_3+\beta_2<0$
	$$
	\bfC(f,g,h) = \P(f,g,h) - \P(f,g)h - \sum_{|k|<\min\{\beta_1,-\beta_2-\beta_3\} } \partial^k f \, \P_k(g,h) + fgh.
	$$
\end{example}
	

	

	
	
	

	\section{Stochastic data on words \label{section_fromRStoPS}}
	
	In this section we prove Theorem \ref{thm_regularityZtau} which combined with the paracontrolled representation from Equation \eqref{eq_ParacRepIterated1} will be crucial in Section \ref{section_ParacAnsatz}. It is analogous to \cite[Theorem 1]{BailleulHoshinoRS1} except that this time the reference function $Y_\tau$ will indeed be in $C^{|\tau|}$ and the regularity of these functions will not be integers, so that we will be able to exploit results from \cite{LocalExpansionsParacSystems}.
	We suppose in this section that the  regularity structure $\scrT=(T,T^+)$  satisfies the following assumption that was already present in \cite{PCandRS2}. 
	\begin{assumption}\label{assumption_notint}
		For any integer $n\geq 1$ and $\tau_1,\dots,\tau_n \in \mcB\backslash \mcB_X$, the sum $\sum_{j=1}^n |\tau_j|$ is not an integer. We make the same assumption for $ \mcB $ replaced by $\mcB^+$.
	\end{assumption}
	This assumption implies the Assumption (B) from \cite{PCandRS2}, that restricts the way polynomials interact with coproducts and ensure good algebraic properties for the derivations $D^n$ and $\downarrow^n$ described in the next subsection. We will also use this assumption in order to apply the results from \cite{LocalExpansionsParacSystems} as the closely related Assumption (A) is needed there.

	

	\subsection{Two derivations on $\scrT$ \ }
	\label{section_derivations}
	Assumption \ref{assumption_notint} implies the following structural property on $\scrT$ that was already present in \cite{PCandRS2}.
	
	\begin{lemma}\label{lem_polynomialsintheRS}
		Under Assumption \ref{assumption_notint}, we have the following
		\begin{enumerate}
			\item For any $\sigma,\tau\in\mcB$, one has either $\tau/\sigma\in \text{Span}(\mcB_X)$ or $\tau/\sigma\in \text{Span}(\mcB \backslash \mcB_X)$.
			\item For $\tau\in\mcB\backslash \mcB_X$ and $\sigma\in\mcB_X$, we have  $\tau/\sigma\in \text{Span}(\mcB \backslash \mcB_X)$.
			\item For $\tau\in\mcB\backslash \mcB_X$ and $\sigma<\tau$ such that $\tau/\sigma\in \mcB_X$, we have $\sigma\in \text{Span}(\mcB \backslash \mcB_X)$.
		\end{enumerate}
	\end{lemma}
		Define for any $n\in\bfN^{d+1}$ the linear operator $D^n$ by setting for any $\tau\in\mcB$
	\begin{equation}  \label{linear_operator}
		D^n\tau = n!  \, \tau / X^n, \quad 
		\downarrow^n\hspace{-3pt}\tau =  n!\, X^n \backslash \tau.
	\end{equation}
	The notation $\sigma\backslash \tau$ is defined in \eqref{eq_sigmabackslashtau}. On the polynomials, the operators $D^n$ and $\downarrow^n$ act like classical differentiation
		\begin{equs} \label{derivation_X}
		D^n X^p = \mathbf{1}_{n\leq p}\, \frac{p!}{(p-n)!} \, X^{p-n}  \quad \text{and} \quad \downarrow^n X^p = \mathbf{1}_{n\leq p}\, \frac{p!}{(p-n)!} \, X^{p-n}.
		\end{equs}
	 In the regularity structure of decorated trees, for any decorated tree $\tau$,
		$$
		D^n \mcI(\tau) = \mcI_{n}^+(\tau)  \quad  \text{and} \quad 	D^n \mcI_{k}^+ (\tau) = \mcI_{k+n}^+(\tau).
		$$
	 In the regularity structure of iterated paraproducts, for any nonempty word $w$
		$$
		D^n \lg w\rg_\ell = 
		\lg w \rg_\ell^{n} \quad \text{and} \quad D^n \lg w\rg_\ell^k = \lg w \rg_\ell^{k+n}.
		$$
		If $n\ne 0$
		$$
			\downarrow^n \lg w\rg_\ell = 0 \quad \text{and} \quad \downarrow^n \lg w\rg_\ell^k =0.
		$$
	These two operations will play the role of derivations on  $\scrT$. Lemma \ref{lem_polynomialsintheRS} implies the following algebraic identities.
	\begin{lemma}\label{lem_DnSnpropenvrac}
		We have the following identities (under Assumption 1).
		\begin{enumerate}
			\item For any $n_1,n_2\in\bfN^{d+1}$
			\begin{equation}\label{eq_commutDnDownaroown}
			D^{n_1} D^{n_2} = D^{n_1+n_2},  \quad 
			\downarrow^{n_1}\hspace{-1pt} \downarrow^{n_2}\hspace{-0pt} = \downarrow^{n_1+n_2}, \quad 
			\downarrow^{n_1}D^{n_2} = D^{n_2}\downarrow^{n_1}.
		\end{equation}
			\item  For any $\tau\in T \sqcup T^+$,  $\sigma\in\mcB\sqcup \mcB^+$ and $n\in\bfN^{d+1}$
			\begin{equation} \label{eq_commutSnDiv}
				\downarrow^n \! (\tau/\sigma) = (\downarrow^n \! \tau)/\sigma.
			\end{equation}
			\item  We have the Leibniz rules
			\begin{equation} \label{eq_leibnizSn}
			\begin{aligned}	D^n(\tau_1\tau_2) & = \sum_{n=n_1+n_2} \binom{n}{n_1}  D^{n_1}\tau_1 \, D^{n_2}\tau_2,
			\\
				\downarrow^n\hspace{-2pt}(\tau_1\tau_2) & = \sum_{n=n_1+n_2} \binom{n}{n_1} \downarrow^{n_1}\hspace{-2pt}\tau_1 \, \downarrow^{n_2}\hspace{-2pt}\tau_2.
				\end{aligned}
			\end{equation}
			\item For any $n\in\bfN^{d+1}$ and $\tau\in T$, we have the commutation relations
			\begin{equation} \label{eq_commutDerivationDelta}
				\Delta D^n\tau = (D^n\otimes\id)\Delta \tau, \quad \text{and}\quad
				\Delta \downarrow^n \! \tau = (\id \, \otimes \downarrow^n)\Delta \tau.
			\end{equation}
			The same holds true for $\tau \in \CT^+$ and $\Delta$ replaced by $ \Delta^{\!+} $.
		\end{enumerate}
	\end{lemma}
	For any $\sigma, \tau\in\mcB^{+}$, we write $\sigma\leqq\tau$ if $\sigma<\tau$ and $\tau/\sigma$ is a polynomial, we also write
	$$
	\sigma\prec \tau \textrm{ if } \sigma<\tau \textrm{ but not } \sigma\leqq\tau.
	$$ We use the same notation when $ \sigma \in \mcB $ and $\tau \in \mcB^+$.
	The following Lemma will be useful in the computations for reindexing sums.
	\begin{lemma}\label{lem_SnDnreindex}
		For any $\tau$ in $T$ and $n\in\bfN^{d+1}$, the following identities holds in $T\otimes T^{+}$
		\begin{equation} \label{eq_Dntausigma}
			\sum_{\sigma\leq\tau} \sigma \otimes D^n(\tau/\sigma)= \sum_{\nu\leq\tau} \downarrow^n\hspace{-2pt} \nu \otimes \tau/\nu,
		\end{equation}
		and
		\begin{equation} \label{eq_Dntausigma2}
			\sum_{\sigma\prec\tau} \sigma \otimes D^n(\tau/\sigma) = \sum_{\nu\prec\tau} \downarrow^n\hspace{-2pt} \nu \otimes \tau/\nu.
		\end{equation}
		The same is true for $\tau \in T^+$ and using $ \Delta^{\!+} $.
	\end{lemma}
	\begin{proof}
		We prove the result for $\tau\in T$, the identity with $\tau\in T^+$ follows from the same computations. We let $\pi_{X^n} : T^+\rightarrow T^+$ to be the orthogonal projection onto the subspace generated by $ X^n$. On the one hand 
		\begin{align*}
			\big(\text{id} \otimes \pi_{X^n} \otimes \text{id} \big)\big(\id \otimes \Delta^{\!+} \big)\Delta\tau 
			&=
			\sum_{\sigma\leq\tau} \sigma \otimes \big( ( \pi_{X^n} \otimes \id) \Delta^{\!+} (\tau/\sigma) \big)
			\\
			&=\sum_{\sigma\leq\tau }\frac{1}{n!} \, \sigma \otimes X^n \otimes D^n(\tau/\sigma),
		\end{align*}
		and on the other hand
		\begin{align*}
			(\text{id} \otimes \pi_{X^n} \otimes \text{id}  )(\Delta \otimes\id)\Delta \tau 
			&=
			\sum_{\sigma\leq\tau}   \big(   (\id \otimes \pi_{X^n} ) \Delta \sigma        \big) \otimes \tau/\sigma
			\\
			&=\sum_{\sigma\leq\tau }\frac{1}{n!}  \downarrow^n\hspace{-3pt}\sigma \otimes X^n \otimes \tau/\sigma.
		\end{align*}
		The result follows from the co-associativity property $(\Delta \otimes \text{id} )\Delta = (\text{id}\otimes\Delta^{\!+})\Delta$. 
	We denote by $\pi^{\perp}_{T_X}$ the orthogonal projection parallel to the vector space of polynomials. One obtains the second identity \eqref{eq_Dntausigma2} by applying $\id \otimes \pi^{\perp}_{T_X} $ to \eqref{eq_Dntausigma}, and by using the identity $D^n\pi^{\perp}_{T_X} = \pi^{\perp}_{T_X} D^n$, which is a consequence of Lemma \ref{lem_polynomialsintheRS}.	
	\end{proof}

	

	\subsection{ A modified model} \label{subsect_modifmodel}

	Integrating local remainders ${ \Pi}_x(\tau)$ and ${\gamma}_{xy}(\tau/\sigma)$ against suitable Littlewood-Paley kernels gives functions that have some regularity and that will enable to write paracontrolled representations. This idea was first used in \cite{BailleulHoshinoRS1} for representing models over regularity structures as paracontrolled distributions. A first example of it will be the generalised Littlewood-Paley block $[\tau]_i$, first defined in \cite{LocalExpansionsParacSystems} by setting for any $\tau\in T$
	\begin{equation}\label{eq_defbracket_1} \begin{aligned} 
			&[\tau]_i \defeq \Delta_i \big({ \PPi}(\tau)\big)
		- \sum_{q\geq 1} (-1)^{q-1} \sum_{\tau_q<\cdots<\tau_1<\tau} \\ &\hspace{3cm}  \Delta_{<i-1}\big({ g}\big(\tau/\tau_1\big)\big)\cdots \Delta_{<i-1}\big({ g}\big(\tau_{q-1}/\tau_q\big)\big)  \Delta_i\big({ \PPi}\big(\tau_q\big)\big).
	\end{aligned}\end{equation}
where the $\tau / \tau_1,..., \tau_{q-1}/\tau_{q} \in T^+$ and are computed by iterations of $ \Delta^{\!+} $. 
One has a similar formula for $\tau/\sigma\in T^+$
	\begin{equation}\label{eq_defbracket_2}\begin{split}
			&[\tau/\sigma]_i \defeq \Delta_i \big({ g}(\tau/\sigma)\big)
		- \sum_{q \geq 1} (-1)^{q-1}  \hspace{-5pt}\sum_{\sigma<\tau_q<\cdots<\tau_1<\tau} \\ &  \hspace{3cm}  \Delta_{<i-1}\big({ g}\big(\tau/\tau_1\big)\big)\cdots \Delta_{<i-1}\big({ g}\big(\tau_{q-1}/\tau_q\big)\big)  \Delta_i\big({ g }\big(\tau_q/\sigma\big)\big).
	\end{split}\end{equation}
	The brackets $[\tau]_i$ and $[\tau/\sigma]_i$ could have also been defined by the following recursive relations
	\begin{equation} \begin{aligned}
		[\tau]_i & = \Delta_i\big({ \PPi}(\tau)\big) - \sum_{\sigma<\tau} \Delta_{<i-1} \big({ g}(\tau/\sigma)\big) \, [\sigma]_i,
	\\
		[\tau/\sigma]_i & = \Delta_i\big({ g}(\tau/\sigma)\big) - \sum_{\sigma<\nu<\tau} \Delta_{<i-1} \big( { g}(\tau/\nu) \big) \, [\nu/\sigma]_i.
		\end{aligned}
	\end{equation}
From \cite{LocalExpansionsParacSystems} one has the following continuity estimate
	\begin{proposition} \label{prop_RegCrochets}
		For any $\tau\in T_{|\tau|}$ and $\tau/\sigma\in T_{|\tau/ \sigma|}^+$,  we have
		$$ 
		\sup_{i\geq -1} 2^{i|\tau|} \Vert[\tau]_i\Vert_{L^\infty} \lesssim \Vert \tau \Vert^*_{ ( \PPi , g ) }, \quad 
		\sup_{i\geq -1} 2^{i|\tau/\sigma|} \Vert[\tau/\sigma]_i\Vert_{L^\infty} \lesssim \Vert \tau/\sigma \Vert^*_{ ( \PPi , g)}.
		$$
	\end{proposition}
	
	The proof follows from the ideas from  \cite[Theorem 1]{BailleulHoshinoRS1}. We describe these steps and refer to \cite[Proposition 18]{LocalExpansionsParacSystems} for a complete proof. We set  
		\begin{align*}
		{\P}(\tau)_i(z) &= \int_{(\bfR^{d+1})^2}  K_{<i-1}(z-x )K_i(z-y) \, \Pi_x\tau(y) \, \ddd x\ddd y,
		\\
		{\bf Q}(\tau)_i(z) &= \int_{(\bfR^{d+1})^2}  K_{<i-1}(z-x)K_{<i-1}(z-y  ) \, \gamma_{yx}(\tau) \, \ddd x\ddd y.
	\end{align*}
From the scaling property of the kernels $(K_i)_i$ (see \cite[Proposition 8]{BailleulHoshinoRS1} ), we have 
	$$ 
\sup_{i\geq -1} 2^{i|\tau|} |\P(\tau)_i| \lesssim \Vert\tau\Vert_{ (\PPi,g)}, 
\quad 
\sup_{i\geq -1} 2^{i|\tau/\sigma|} |{\bf Q}(\tau/\sigma)_i| \lesssim \Vert\tau/\sigma\Vert_{ (\PPi,g)}.
$$
In order to obtain the proposition, it suffices to verify the following algebraic identity 
$$
[\tau]_i = \P(\tau)_i + \sum_{\sigma<\tau} {\bf Q}(\tau/\sigma)_i[\sigma]_i.
$$

	The $[\tau]_i$ and $[\tau/\sigma]_i$  can be interpreted as generalised Littlewood-Paley blocks of some functions in $C^{|\tau|}$ or $C^{|\tau/\sigma|}$. We refer to \cite{LocalExpansionsParacSystems} for a proof of this result. This notion of bracket $[\tau]_i$ is not plainly satisfactory for our purpose because on the one hand the natural paracontrolled representation formula involving these brackets is written in terms of some simplified notion of iterated paraproduct that was introduced in \cite{LocalExpansionsParacSystems}, and on the other hand some degrees of some of these brackets are integers, which prevents us from using results from \cite{LocalExpansionsParacSystems} as it contradicts Assumption (A) from there. Also it is not clear that $[\tau]_i$ defines a periodic function when the model is itself periodic.
	These issues are solved by introducing some modified 'model' $\widetilde{ \PPi}$ and $\widetilde{ g}$, that were already defined in \cite{LocalExpansionsParacSystems}, which will prevent us from extracting polynomials when writing coproducts by hiding them in the $\widetilde{\PPi}$ and $\widetilde{ g}$.
We define
	\begin{equation}\label{eqdef_pitilde} \begin{split}
			&\widetilde{ \PPi}(\tau)_i =  \Delta_i\big( { \PPi}(\tau)\big) 
		- \sum_{q\geq 1}(-1)^{q-1} \hspace{-0.2cm} \sum_{\sigma_q\leqq \cdots \leqq\sigma_1 \leqq \tau} \\ & \hspace{3cm}  \Delta_{<i-1}\big({ g}\big(\tau/\sigma_1\big)\big)\cdots \Delta_{<i-1}\big({ g}\big(\sigma_{q-1}/\sigma_q\big)\big) \, \Delta_i\big({ \PPi}(\sigma_q)\big), 
	\end{split}\end{equation}
	and
	\begin{equation}\label{eqdef_gtilde}\begin{split}
			&\widetilde{ g}(\tau/\sigma)_i =  \Delta_i\big( { g}(\tau/\sigma)\big) 
			- \sum_{q\geq 1}(-1)^{q-1} \hspace{-0.2cm} \sum_{\sigma_q\leqq \cdots \leqq\sigma_1 \leqq \tau} \\ & \hspace{2cm}  \Delta_{<i-1}\big({ g}\big(\tau/\sigma_1\big)\big)\cdots \Delta_{<i-1}\big({ g}\big(\sigma_{q-1}/\sigma_q\big)\big) \, \Delta_i\big({ g}(\sigma_q/\sigma)\big). 
		\end{split}
	\end{equation}
	We also set 
	$$
	\widetilde{ \PPi}(\tau) = \sum_{i\geq -1} \widetilde{ \PPi}(\tau)_i
	\quad \text{and} \quad 
	\widetilde{ g}(\tau) = \sum_{i\geq -1} \widetilde{ g}(\tau)_i.
	$$
	The following Lemma gives another way to define $\widetilde{ \PPi}$ and $\widetilde{ g}$. It will be useful in the computations involving the previous maps.
	
	\begin{lemma} \label{easy_tilde_Pi}
		For any $\tau\in T$ we have 
		\begin{equation} \label{eq_pitildetopi}
			\widetilde{ \PPi}(\tau)(y) = \sum_{j\geq 0}\frac{(-y)^{j}}{j!}\, { \PPi}\big(\downarrow^j\hspace{-3pt}\tau\big)(y), \quad 
			{ \PPi}(\tau)(y) = \sum_{j\geq 0}\frac{y^j}{j!} \, \widetilde{ \PPi}\big(\downarrow^j\hspace{-2pt}\tau\big)(y).
		\end{equation}
		And likewise for $\tau\in T^+$ we have 
		\begin{equation}\label{eq_defgtilde2}
			\widetilde{ g}_x(\tau) = \sum_{j\geq 0}\frac{(-x)^{j}}{j!}\, { g}_x\big(\downarrow^j\hspace{-2pt}\tau\big),
\quad
			{ g}_x(\tau) = \sum_{j\geq 0}\frac{x^j}{j!} \, \widetilde{ g}_x\big(\downarrow^j\hspace{-2pt}\tau\big).
		\end{equation}
	\end{lemma}
	
	\begin{proof}
		We prove the two relations involving $\widetilde{\PPi}$, the other two are proven along the same lines.
		As 
	$	\big\{ \sigma, \, \sigma \leqq \tau \big\} = \big\{\downarrow^j\hspace{-2pt}\tau,\, j\geq 1 \big\}$,
		we can rewrite \eqref{eqdef_pitilde} as   
		$$
		\widetilde{\PPi}(\tau)_i(y) = \Delta_i\big( { \PPi}(\tau)\big)(y) - \sum_{j\geq 1} \frac{y^j}{j!}  \widetilde{ \PPi}\big( \downarrow^j\hspace{-2pt}\tau\big)_i(y),
		$$
		that is 
		$$
		\Delta_i\big( {\PPi}(\tau)\big)(y) = \sum_{j\geq 0} \frac{y^j}{j!}   \widetilde{ \PPi}\big(\downarrow^j\hspace{-2pt}\tau\big)_i(y),
		$$
		hence \eqref{eq_pitildetopi}. Then, from Pascal inversion
		\begin{equation}\label{eq_pitopitildei}
			\widetilde{ \PPi}(\tau)_i(y) = \sum_{j\geq 0} \frac{(-y)^j}{j!}   \Delta_i \big({ \PPi}\big(\downarrow^j\hspace{-2pt}\tau\big)\big)(y).
		\end{equation}
	\end{proof}
	
	\begin{example}
		For any $k\in\bfN^{d+1}\backslash \{ 0 \} $ and any symbol $\tau\in\mcB$, we have
		$$
		\widetilde{\PPi}(X^k\tau) = 0.
		$$
		For any $\tau\in\mcB$ such that $\downarrow^{e_i}\tau = 0$, for any $0\leq i \leq d$, we have
		$$
		\widetilde{\PPi}(\tau) = { \PPi}(\tau).
		$$
		In the regularity structure of decorated trees, we have 
		\begin{equs}
			\PPi \big(\mcI(X^\ell\Xi)\big) &= K* X^{\ell}\xi,
			\quad 
			\PPi \big(\mcI(\mcI(X_{e_i}\Xi)\Xi)\big) = K*\big( \xi\, K* X_{e_i} \xi\big) 
			\end{equs}
				where $ X^\ell \xi $ is the function $ x \mapsto x^{\ell} \xi(x) $.
			By using \eqref{derivation_X} and Lemma \ref{easy_tilde_Pi}, we have
			\begin{equs}
			\widetilde{\PPi}\big(\mcI(X^\ell\Xi)\big) &= \sum_{j\geq 0}\frac{(-y)^{j}}{j!}\, { \PPi}\big(\mcI( \downarrow^j X^\ell\Xi)\big)(y)
			\\ &= \sum_{j\geq 0}\frac{(-y)^{j}}{j!}\, { \PPi}\big(\mcI( \mathbf{1}_{j\leq \ell}\, \frac{\ell!}{(\ell-j)!} \, X^{\ell-j}\Xi)\big)(y)
			\\ & = \sum_{j\leq \ell }\frac{\ell !}{j!(\ell-j)!} (-y)^j \int K(y-x)  x^{\ell-j}\xi(x) dx 
			\\& =  K^{[\ell]}*\xi,
			\end{equs}
			where we recall that $ K^{[\ell]}(x) = (-x)^{\ell} K(x) $.
			Via a similar computation, one gets
			\begin{equs}
			\widetilde{\PPi}\big(\mcI(\mcI(X_{e_i}\Xi)\Xi)\big) & = K * \big( \xi \, K^{[e_i]}*\xi\big) +  K^{[e_i]}*\big( \xi \, K*\xi \big).
		\end{equs}
	
	\end{example}
	We now collect some properties of $\widetilde{\PPi}$ and $\widetilde{ g}$ that will be useful later on.
	\begin{lemma} \label{lem_gtildetogitilde}
		For any $\tau\in T$ and for any $\tau/\sigma\in T^+$
		\begin{equation}\label{eq_gitildefromdeltai}
			\widetilde{ \PPi}(\tau)_i = \sum_{m\geq 0} \frac{1}{m!} \, \Delta_i^m\big(\widetilde{ \PPi}(\downarrow^m\hspace{-2pt}\tau)\big),\quad 
			\widetilde{ g}(\tau/\sigma)_i = \sum_{m\geq 0} \frac{1}{m!} \, \Delta_i^m\big(\widetilde{ g}(\downarrow^m\hspace{-2pt}\tau/\sigma )\big).
		\end{equation}
	\end{lemma}
	
	\begin{proof}
		We prove \eqref{eq_gitildefromdeltai}, the other identity is proven along the same lines. From \eqref{eq_pitildetopi},
		\begin{align*}
			\sum_{m\geq 0} \frac{1}{m!} \, \Delta_i^m\big(\widetilde{ \PPi}(\downarrow^m\hspace{-2pt}\tau)\big)(x) 
			&=
			\sum_{m,j\geq 0}\int_{\bfR^{d+1}} K_i(x-y) \frac{(y-x)^m}{m!} \frac{(-y)^j}{j!} { \PPi}\big(\downarrow^{m+j}\hspace{-2pt}\tau\big)(y) \, \ddd y
			\\   
			&=
			\sum_{k\geq 0} \int_{\bfR^{d+1}}  K_i(x-y) \frac{(-x)^k}{k!} {\PPi}\big( \downarrow^k\hspace{-2pt}\tau\big)(y) \, \ddd y
			\\
			&= \sum_{k\geq 0} \frac{(-x)^k}{k!} \Delta_i\big({\PPi}(\downarrow^k\hspace{-2pt}\tau)\big)(x)
			=
			\widetilde{\PPi}(\tau)_i(x).
		\end{align*}
	\end{proof}
		\begin{lemma}\label{lem_multiplicativitygtilde}
		For any $\tau_1,\tau_2\in T^+$, we have the multiplicativity property
		$$
		\widetilde{ g}(\tau_1\tau_2) = \widetilde{ g}(\tau_1)\, \widetilde{ g}(\tau_2).
		$$
	\end{lemma}
	\begin{proof}
		From \eqref{eq_defgtilde2}, \eqref{eq_leibnizSn} and the multiplicativity of ${ g}$,
		\begin{align*}
			\widetilde{ g}_x(\tau_1\tau_2) &= \sum_{n\geq 0} \frac{(-x)^n}{n!} { g}_x\Big( \sum_{n=n_1+n_2}\binom{n}{n_1} \downarrow^{n_1}\hspace{-3pt}\tau_1 \,  \downarrow^{n_2}\hspace{-3pt}\tau_2\Big)
			\\
			&= \sum_{n_1,n_2\geq 0} \frac{(-x)^{n_1+n_2}}{n_1! \, n_2!} { g}_x\big(\downarrow^{n_1}\hspace{-3pt}\tau_1\big) \, { g}_x\big(\downarrow^{n_2}\hspace{-3pt}\tau_2\big)
			=\widetilde{ g}_x(\tau_1)\, \widetilde{ g}_x(\tau_2).
		\end{align*} 
	\end{proof}
	One can rewrite local remainders using ${ (\widetilde{\PPi},\widetilde{g})}$
	\begin{equation}
		\begin{aligned}
		{ \Pi}_{x}(\tau)(y) &= \sum_{j\geq 0}\frac{(-x)^j}{j!} \, { \PPi}(\downarrow^j\hspace{-2pt} \tau)(y) -\sum_{\sigma \prec \tau} \widetilde{ g}_x(\tau/\sigma) { \Pi}_{x}(\sigma)(y),
\\
		{ \gamma}_{yx}(\tau) &= \sum_{j\geq 0}\frac{(-x)^j}{j!} \, { g}_y(\downarrow^j\hspace{-2pt} \tau) -\sum_{\sigma \prec \tau} \widetilde{ g}_x(\tau/\sigma) { \gamma}_{yx}(\sigma).
		\end{aligned}
	\end{equation}
	We have the following alternative definition of the brackets $[\tau]_i$, in terms of the modified model $({\widetilde{\PPi}}, \widetilde{g})$.
	\begin{proposition}\label{lem_sumgtilde}
		For any $\tau\in\mcB$ and $i\geq -1$ we have the relation
		\[
		[\tau]_i =  \widetilde{ \PPi}(\tau)_i - \sum_{q\geq 1} (-1)^{q-1} \sum_{\tau_q\prec\cdots\prec\tau_1\prec\tau} \widetilde{ g}\big(\tau/\tau_1\big)_{<i-1} \cdots \,  \widetilde{ g}\big(\tau_{q-1}/\tau_{q}\big)_{<i-1} \,  \widetilde{ \PPi} \big(\tau_q\big)_i.
		\]
		Likewise for $\tau/\sigma\in \mcB^+$ we have
		\begin{align*}
			[\tau/\sigma]_i &=  \widetilde{ g}(\tau/\sigma)_i 
			\\
			&\quad - \sum_{q\geq 1} (-1)^{q-1}  \hspace{-0.5cm}\sum_{\sigma\prec\sigma_q\prec\cdots\prec\sigma_1\prec\tau} \widetilde{ g}\big(\tau/\sigma_1\big)_{<i-1}\cdots \, \widetilde{ g}\big(\sigma_{q-1}/\sigma_{q}\big)_{<i-1} \,  \widetilde{ g}\big(\sigma_q/\sigma\big)_i.
		\end{align*}
	\end{proposition}
	
	\begin{proof} 
		Plugging the definition of $\widetilde{ \PPi}$ and $\widetilde{ g}$ from \eqref{eqdef_pitilde} and \eqref{eqdef_gtilde} into the right hand side of the identity to prove, developing the products, one recovers the definition of $[\tau]_i$ and $[\tau/\sigma]_i$ as any descending sequence $\tau_q<\cdots<\tau_1<\tau$ takes the form
		$$
		\cdots\leqq\tau_{3,0}\prec \tau_{2,e_2} \leqq \cdots  \leqq  \tau_{2,0} \prec \tau_{1,e_1} \leqq \cdots \leqq \tau_{1,1} \leqq \tau.
		$$
	\end{proof}
	One can rewrite Proposition \ref{lem_sumgtilde} as the following alternative recursive definition for the $[\tau]_i,[\tau/\sigma]_i$
	\begin{equation}\label{eq_defreccrochettilde}
		\begin{aligned}
		[\tau]_i & = \widetilde{\PPi}(\tau)_i - \sum_{\sigma\prec\tau} \widetilde{ g}(\tau/\sigma)_{<i-1} \, [\sigma]_i, \\
		[\tau/\sigma]_i & = \widetilde{ g}(\tau/\sigma)_i - \sum_{\sigma\prec\nu\prec\tau}  \widetilde{ g}(\tau/\nu)_{<i-1} \, [\nu/\sigma]_i.
		\end{aligned}
	\end{equation}

	

	\subsection{A generalised Hairer-Kelly map}\label{subsect_HairerKelly}
	Given a regularity structure $\scrT$ satisfying Assumption \ref{assumption_notint}, one defines a new regularity structure $\scrT_\P$ by setting it to be the regularity structure of iterated paraproducts $\mcT_\mcA$, where $\mcA$ is the alphabet consisting of the symbols from $\mcB\sqcup \mcB^+ $, and where each letter has the same degree as the symbol it represents.

	\begin{definition}\label{def_Psitau}
		One defines a linear map from $\scrT$ to $\scrT_\P$ by setting for any non polynomial $\tau\in\mcB \sqcup \mcB^+$
		\begin{equation} \label{forest_formula}
			\Psi(\tau) = \sum_{\substack{k\geq 0 \\  k=k_1+k_2}} \sum_{\substack{q\geq 0 \\ \tau_q\prec\dots\prec\tau_1\prec\tau}} \, \frac{1}{k_1!\, k_2!} \, \big\lg \downarrow^k\hspace{-2pt}\tau/\tau_1,\, \tau_1/\tau_2, \dots, \tau_q \big\rg_{k_1} X^{k_2},
		\end{equation}
		and $\Psi(X^n) = X^n$ for any $n\in\bfN^{d+1}$.
	\end{definition}
		One can write a recursive definition of the map $ \Psi $ given by
	\begin{equation} \label{recursive_psi}
	\Psi(\tau) = \sum_{\substack{k\geq 0 \\  k=k_1+k_2}} \, \frac{1}{k_1!\, k_2!} \lg \downarrow^k \tau  \rg_{k_1}  X^{k_2} + \mcM_c (\pi_{T_X}^\perp \otimes \Psi \pi_{T_X}^\perp ) \Delta \tau,
\end{equation}
	where $\pi_{T_X}^\perp$ is the  orthogonal projection parallel to the polynomial vector space $T_X$, and where the concatenation map is given by
	$$
	\mcM_c( a\otimes  \lg w\rg_\ell X^k   ) =  \lg wa \rg_{\ell} X^{k}.
	$$
	 The formula \eqref{forest_formula} can be interpreted as a forest formula obtained when one iterates $ \Delta $ and $\Delta^{\!+}$. As we send decorated trees to words, one can draw a parallel with the Hairer-Kelly map used in \cite{HK15} for moving from non-geometric to geometric rough paths. The recursive formula \eqref{recursive_psi} is close in spirit to \cite{BCEF20}.
	In our case, we do not use any shuffle product and we do not claim any Hopf algebra morphism as in the Hairer-Kelly map. 
	\begin{example} Below, we provide some examples of computations. We suppose that $|\Xi| = - \frac{3}{2}-\kappa$ which is the regularity of space-time white noise in dimension $d=1$. Then, one has $ | \CI(\Xi) |  = \frac{1}{2}-\kappa$ and  
		\begin{align*}
		\Psi\left( \mcI(\Xi\mcI(\Xi)) \right) &=  \lg \mcI(\Xi\mcI(\Xi)))\rg + \lg \mcI^+(\Xi) , \, \mcI(\Xi) \rg
	\\
		\Psi\left(X_{(0,1)}\mcI(\Xi) \right) &= \lg X_{(0,1)}\mcI(\Xi)\rg + \lg \mcI(\Xi)\rg X_{(0,1)}.
	\end{align*}
	One can rewrite the previous computations using the tree notations:
	\begin{equs}
		\Psi\left( \<Xiiibis> \right) &=  \<Xiiibis> + \<IXi> \otimes \<IXi>
		\\
		\Psi\left(X_{(0,1)} \<IXi> \right) &=   X_{(0,1)} \<IXi>  + \lg \<IXi>\rg X_{(0,1)}.
	\end{equs}
	where we have made several abuse of notations by removing the brackets $ \lg \cdot \rg $. We have also used the tensor product for separating the letter of the word.  
	The first example follows from the computations below
	\begin{equation*}\begin{aligned}
 \Delta \mcI(\Xi\mcI(\Xi))   &=  \mcI(\Xi\mcI(\Xi)) \otimes \one + \one \otimes \mcI^+(\Xi\mcI(\Xi)) 
 +  \CI(\Xi) \otimes \CI^+(\Xi) 
			\\
 \left( \pi^{\perp}_{T_X} \otimes \pi^{\perp}_{T_X} \right)		\Delta \mcI(\Xi\mcI(\Xi))  & =  \CI(\Xi) \otimes \CI^+(\Xi).
	\end{aligned}
	\end{equation*}
	One gets more examples of computations in Subsection \ref{examples}.
	\end{example}

	It will be useful in later computations to have in hand the value of the coproduct $\Delta\Psi(\tau)$, it is computed in the next lemma. We first introduce for $n\geq0$
	$$
	\Psi_n(\tau) = \sum_{\substack{n=n_1+n_2}} \sum_{\substack{q\geq 0 \\ \tau_q\prec\dots\prec\tau_1\prec\tau}} \, \frac{1}{n_1!\, n_2!} \, \big\lg \tau/\tau_1,\, \tau_1/\tau_2, \dots, \tau_q \big\rg_{n_1} X^{n_2}.
	$$
	The maps $\Psi_n$ and $\Psi$ are related by the following identity
	\begin{equation}\label{eq_PsinAndPsi}
		\Psi(\tau) = \sum_{n\in \N^{d+1}} \Psi_n(\downarrow^n\tau).
	\end{equation}
	
	\begin{lemma}\label{lem_DeltaPsitau}
		For any non polynomial $\tau\in \mcB \sqcup \mcB^+$,
			\begin{align*}
			\Delta \Psi(\tau) = & \sum_{\substack{\sigma\prec\tau\\n\in \N^{d+1}}} \Psi_n(\sigma)\otimes  D^n\left(\Psi(\tau/\sigma)\right)
			+\sum_{n \in \N^{d+1}} \frac{X^n}{n!} \otimes D^n \left( \Psi(\tau)\right) \\ & +\sum_{n\in\bfN^{d+1}} \Psi(\downarrow^n\!\tau)\otimes \frac{X^n}{n!}.
		\end{align*}
		
	
\end{lemma}
	
	\begin{proof}
		From the definition of the coproduct from Subsection \ref{subsect_theRSIteratedParap}, the coproduct $\Delta\Psi(\tau)$ is equal to
		\begin{equation}\label{eq_coproduitPsiTau_proof}
			\begin{aligned}
			&\sum_{\substack{\sigma\prec\tau \\ \sigma\prec\dots\prec\sigma_1\prec\tau \\ \nu_s\prec\dots\prec\nu_1\prec\sigma }}\sum_{\substack{(k_i)\in\mcP_4(k) \\ |p|<|\downarrow^k \tau/\sigma|+|k_2| \\ p=p_1+p_2 }} \hspace{-5pt} \frac{1}{k_1!k_2!k_3!k_4!} \frac{1}{p_1!p_2!} 
			\Big\lg \sigma/\mu_1,\dots,\mu_s \Big\rg_{k_3+p_1} X^{k_4+p_2}
			\\
			& \otimes \Big\lg   \downarrow^k\tau/\sigma_1,\dots,\sigma_r/\sigma     \Big\rg^p_{k_2}X^{k_1} + \,\sum_{n\geq 0} \frac{X^n}{n!} \otimes D^n \left( \Psi(\tau)\right) + \Psi(\downarrow^n\tau)\otimes \frac{X^n}{n!} .
		\end{aligned}
	\end{equation}
		Re-indexing the first sum gives  
		\begin{align*}
			&\sum_{\substack{\sigma\prec\tau \\ \sigma\prec\dots\prec\sigma_1\prec\tau \\ \nu_s\prec\dots\prec\nu_1\prec\sigma }}\sum_{ \substack{k=k_1+k_2+k_{34} \\ |q|<|\tau/\sigma|-|k_1| \\ q=q_1+q_2}} 	\sum_{k_{34}=k_3+k_4}  \frac{1}{k_1!k_2!}
			 \frac{1}{k_3!k_4!(q_1-k_3)!(q_2-k_4)!} \\ & \Big\lg \sigma/\mu_1,\dots,\mu_s\Big\rg_{q_1} X^{q_2}
			 \otimes \Big\lg  \downarrow^k\tau/\sigma_1,\dots,\sigma_r/\sigma  \Big\rg^{q-k_{34}}_{k_2} X^{k_1}.
		\end{align*}
		From the Chu-Vandermonde identity,
		$$
		\sum_{k_{34}=k_3+k_4} \frac{1}{k_3!(q_1-k_3)!}\frac{1}{k_4!(q_2-k_4)!} = \frac{1}{(q-k_{34})!k_{34}!}  \binom{q}{q_1}.
		$$
		Then the first sum in \eqref{eq_coproduitPsiTau_proof} is equal to 
		\begin{align*}
			&\sum_{\sigma\prec\tau } \sum_{\substack{k,q\geq0 \\ k=k_1+k_2+k_3 \\ q=q_1+q_2}} \sum_{\substack{ \sigma\prec\sigma_q\prec\dots\prec\tau \\ \nu_s\prec\dots\prec\nu_1\prec\sigma   }} 	\frac{1}{q_1!q_2!}\Big\lg \sigma/\nu_1,\dots,\nu_s \Big\rg_{q_1}X^{q_2}
			\\
			& \otimes \frac{q!}{k_1!k_2!k_{34}!(q-k_{34})!}  \Big\lg \downarrow^k \tau/\sigma_1,\dots, \sigma_q/\sigma \Big\rg^{q-k_{34}}_{k_2} X^{k_1},
		\end{align*}
		which is indeed equal to $\sum_{\sigma\prec\tau}\sum_{q \in \N^{d+1}} \Psi_q(\sigma) \otimes D^q\Psi(\tau/\sigma)$.
	\end{proof}

	\begin{lemma}\label{lem_commutdownarrowPsi}
		For any $\tau\in\mcB$ and $j\in\bfN^{d+1}$ we have the identity
		$$
		\downarrow^j\Psi(\tau) = \Psi(\downarrow^j\tau).
		$$
	\end{lemma}
	\begin{proof}
		For any non-empty word $w$ and any $\ell\in\bfN^{d+1}$, we have $\downarrow^k\lg w\rg_\ell =0$ if $k\ne 0$. Then from the Leibniz rule of Lemma \ref{lem_DnSnpropenvrac}
		\begin{align*}
			\downarrow^j\Psi(\tau) &= \sum_{q\geq0}\sum_{\tau_q\prec\dots\prec\tau}\sum_{k_1,k_2\geq0} \frac{1}{k_1!}\Big\lg \downarrow^{k_1+k_2}\tau/\tau_1,\dots , \tau_q \Big\rg_{k_1} \downarrow^j \frac{X^{k_2}}{k_2!}
			\\
			&=\sum_{q\geq 0}\sum_{\tau_q\prec\dots\prec\tau}\sum_{k_1,k_2\geq0} \mathbf{1}_{k_2\geq j} \, \frac{1}{k_1!} \frac{X^{k_2-j}}{(k_2-j)!} \Big\lg \downarrow^{k_1+k_2-j}\downarrow^{j}\tau/\tau_1,\dots ,\tau_q \Big\rg_{k_1}  
			\\
			&= \Psi(\downarrow^j\tau).
		\end{align*}
	\end{proof}
	
	

	\subsection{ The distributions $Y_\tau$  \  }
	
	We define here the set of distributions $(Y_\tau)_\tau$ indexed by the symbols of the regularity structure we are using. These will be the building blocks for the commutation relations we will use to build the reference distributions $Z_\tau$ and for computing the right hand side of the equation. In the sequel, we will use the following notations $ [\cdot]_i^{\mathcal{Y}}, \widetilde{\PPi}^{\mathcal{Y}}, \widetilde{g}^{\mathcal{Y}} $ to stress that decorated tree letters $ \tau $ are sent to $Y_{\tau}$.
	
	\begin{definition}\label{def_Ytau}
		Define recursively for any $\tau\in \mcB \sqcup \mcB^+$ 
		\begin{equation}\label{eqdef_Ztau1}
			Y_{\tau} = \widetilde{ \PPi}(\tau) - \, \sum_{k\geq 0} \, \sum_{\sigma\prec \tau} \, \frac{1}{k!} \, \P_k\Big( \widetilde{ g}\big(\downarrow^k\hspace{-2pt} \tau/\sigma\big) , \enskip Y_\sigma  \Big).
		\end{equation}
	\end{definition}
	In practice, we look at specific elements of $\mcB^+$ of the form $ \tau / \sigma $ with $ \tau, \sigma \in \mcB $ and we consider 
	\begin{equation} \label{eqdef_Ztau2}
		Y_{\tau/\sigma} = \widetilde{ g}(\tau/\sigma) - \sum_{k\geq 0} \, \sum_{\sigma\prec \nu\prec \tau} \, \frac{1}{k!} \, \P_k\Big( \widetilde{ g}\big( \downarrow^k\hspace{-2pt}\tau/\nu\big) , \enskip Y_{\nu/\sigma}  \Big).
	\end{equation}
	Re-writing \eqref{eqdef_Ztau1} and \eqref{eqdef_Ztau2} as 
	\begin{equation}\label{eq_parac_representation}\begin{split}
			\widetilde{ \PPi}(\tau) &=Y_{\tau}  + \, \sum_{k\geq 0}\, \sum_{\sigma\prec \tau} \, \frac{1}{k!} \, \P_k\Big( \widetilde{ g}\big(\downarrow^k\hspace{-2pt} \tau/\sigma\big) , \enskip Y_\sigma  \Big),
			\\
			\widetilde{ g}(\tau/\sigma) &=Y_{\tau/\sigma}  + \,  \sum_{k\geq 0} \, \sum_{\sigma\prec \nu \prec \tau} \frac{1}{k!} \, \P_k\Big( \widetilde{ g}\big(\downarrow^k\hspace{-2pt} \tau/\nu\big) , \enskip Y_{\nu / \sigma}  \Big),
	\end{split}\end{equation}
	and iterating these two identities in the first argument of the paraproduct, together with \eqref{eq_commutSnDiv}, yield the following representation of the model ${ (\widetilde{\PPi}, \widetilde{g})}$ as a sum of iterated paraproducts.
	\begin{equation}\label{eq_ParacRepIterated1}
		\widetilde{\PPi}(\tau) = \sum_{k\geq 0 } \, \sum_{q\geq 0} \, \sum_{\tau_q\prec\dots\prec\tau} \, \frac{1}{k!} \, \P_k\Big(Y_{\downarrow^k\hspace{0pt}\tau/\tau_1},\dots, Y_{\tau_q} \Big),
	\end{equation}
	
	\begin{equation}\label{eq_ParacRepIterated2}
		\widetilde{ g}(\tau/\sigma) = \sum_{k\geq 0 } \, \sum_{q\geq 0} \, \sum_{\sigma\prec\sigma_q\prec\dots\prec \tau} \, \frac{1}{k!} \, \P_k\Big(Y_{\downarrow^k\hspace{0pt}\tau/\sigma_1},\dots, Y_{\sigma_q/\sigma} \Big).
	\end{equation}
	These representations can be rewritten using the map $\Psi$ 
	\begin{equation} \label{eq_PiTauPiPsiTau}
		\widetilde{\PPi}(\tau) = 	\widetilde{\PPi}^\mcY\big(\Psi(\tau)\big)
		,\qquad
		\widetilde{ g}(\tau/\sigma) = \widetilde{\PPi}^\mcY\big(\Psi(\tau/\sigma)\big).
	\end{equation}
	The following Lemma gives a stronger version of these last equations in terms of Littlewood-Paley blocks. 
	\begin{lemma} \label{lem_PiiPiPsii_00}
		We have for any $\tau,\sigma\in\mcB$ and $i\geq-1$
		$$
			\widetilde{\PPi}(\tau)_i = 	\widetilde{\PPi}^\mcY\big(\Psi(\tau)\big)_i, 
			\qquad
			 \widetilde{g}(\tau/\sigma)_i = 	\widetilde{\PPi}^\mcY\big(\Psi(\tau/\sigma)\big)_i.
		$$
	\end{lemma}
	\begin{proof}
	We prove the first identity. From Lemma \ref{lem_gtildetogitilde} and \eqref{eq_PiTauPiPsiTau}
	\begin{align*}
		\widetilde{\PPi}(\tau)_i &= \sum_{m\geq 0}\frac{1}{m!}\Delta_i^m\left(\widetilde{\PPi}(\downarrow^m \!\tau)\right) 
		= \sum_{m\geq 0}\frac{1}{m!}\Delta_i^m\left(\widetilde{\PPi}^\mcY(\Psi(\downarrow^m\!\tau))\right). \end{align*}
	From Lemmas \ref{lem_commutdownarrowPsi} and \ref{lem_gtildetogitilde}, this is equal to
	$$
	\sum_{m\geq 0}\frac{1}{m!}\Delta_i^m\left(\widetilde{\PPi}^\mcY(\downarrow^m\Psi(\tau))\right)  = 	\widetilde{\PPi}^\mcY\left(\Psi(\tau)\right)_i.
	$$
	\end{proof}                           	
	\begin{example} We continue the example on the stochastic multiplicative heat equation. We illustrate Definition \ref{def_Ytau} with the computations below
			\begin{equation*}
			\begin{aligned}
				Y_{\<IXi>} &= \tilde{\PPi} \, \<IXi> = K * \xi,
				\\
				Y_{\<Xiiibis>} & = \tilde{\PPi} \, \<Xiiibis> - \P\Big(\tilde{g}(\<IXi>), Y_{\<IXi>} \Big)   =  K* (\xi \, K*\xi) - \P\Big(K*\xi, K*\xi \Big)
			\end{aligned}
		\end{equation*}
		where we have used the fact that $\tilde{g}( \<IXi>)) = g(\CI^{+}(\<IXi>)) =  K * \xi  $ (see \cite[Proposition 15]{BH26}).
		For the regularity structure on the paraproducts, we provide the example of 	$Y_{\lg  a_1,  a_2 \rg X_{(0,1)}}$ below.
		We first compute
		\begin{equation*}
			\begin{aligned}
		\left( \pi^{\perp}_{T_X} \otimes \pi^{\perp}_{T_X} \right)	\Delta \lg  a_1,a_2 \rg X_{(0,1)} &=   \lg a_2 \rg X_{(0,1)} \otimes  \lg a_1 \rg^0 + \lg a_2 \rg  \otimes  \lg a_1 \rg^{0} X_{(0,1)}.
			\end{aligned}
		\end{equation*}
		We have $\widetilde\PPi( 	\lg  a_1,  a_2 \rg X_{(0,1)}) = 0$ as the monomial $X_{(0,1)}$ is in factor. We also have $Y_{\lg a_2\rg} = Z_{a_2}$ and $Y_{\lg a_2 \rg X_{(0,1)}} = 0$. The function $\widetilde{g}(\downarrow^k( \lg a_1 \rg^{0} X_{(0,1)}  ))$  is non zero only when $k=(0,1)$, where it is equal to $Z_{a_1}$.
		Then
		$$
			Y^\mcZ_{\lg  a_1,  a_2 \rg X_{(0,1)}} = 
		- \P_{(0,1)}( Z_{a_1} , Z_{a_2}).
		$$
	\end{example}

	The following theorem is the main result of this section, asserting that the $Y_\tau$ have the regularity we would expect belonging to $ C^{|\tau|} $. 
\begin{theorem} \label{thm_regularityZtau}
		For any symbol $\tau\in \mcB \sqcup \mcB^+$, the distribution $Y_\tau$ is in $C^{|\tau|}$, and furthermore we have the continuity estimate
		\begin{equation}\label{eq_continuity_Y}
		\norme{Y_\tau}_{|\tau|} \lesssim  \Vert\tau\Vert_{ (\PPi,g)}^*.
	\end{equation}
	\end{theorem}
	Theorem \ref{thm_regularityZtau} relies on  Proposition \ref{prop_gDntaustarderivative} and Lemma \ref{lem_sumcrochets}. 
	\begin{proposition}\label{prop_gDntaustarderivative}
		For any $\tau\in \mcB \sqcup \mcB^+$, we have the identity
		\begin{equation}\label{eq_starderivativeZtau}
			\widetilde{ g}\big( D^n\tau\big)_i = \widetilde{ g}^\mcY\left(D^n\Psi(\tau)\right)_i. 
		\end{equation}
	\end{proposition}
	The proof of Proposition \ref{prop_gDntaustarderivative} is postponed to the next subsection, it will follow as a corollary of the stronger Proposition \ref{prop_starderivY} below. 
	
	\begin{lemma}\label{lem_sumcrochets}
		For any symbol $\tau\in\mcB$ and $i\geq-1$, we have the relation
		\begin{equation}\begin{split}\label{eq_propdecompencrochet}
				[\tau]_i = \left[ \Psi(\tau) \right]_i^\mathcal{Y}.
		\end{split}\end{equation}    
	\end{lemma}
	
	\begin{proof}
		We are going to prove the claim by induction.
		We compute the bracket $[\Psi(\tau)]_i^\mcY$. From its definition in \eqref{eq_defreccrochettilde} and from Lemma \ref{lem_DeltaPsitau}, we have
		\begin{align*}
			[\Psi(\tau)]_i^\mcY = \widetilde{\PPi}^\mcY\left(\Psi(\tau)\right)_i - \sum_{\sigma\prec\tau}\sum_{q\geq 0}\, \widetilde{ g }^\mcY\left(D^q\Psi(\tau/\sigma) \right)_{<i-1} \, \left[ \Psi_q(\sigma)  \right]_i^\mcY.
		\end{align*}
		From Lemma \ref{lem_PiiPiPsii_00} we have $\widetilde{\PPi}^\mcY\left(\Psi(\tau)\right)_i=\widetilde{\PPi}\left(\tau\right)_i$, and from Proposition \ref{prop_gDntaustarderivative} we have $\widetilde{ g }^\mcY\left(D^q\Psi(\tau/\sigma)\right)_{<i-1}=\widetilde{ g }\left(D^q(\tau/\sigma)\right)_{<i-1}$. Then 
		\begin{align*}
			[\Psi(\tau)]_i^\mcY = \widetilde{\PPi}\left(\tau\right)_i - \sum_{\sigma\prec\tau}\sum_{q\geq 0}\, \widetilde{ g }\left(D^q(\tau/\sigma) \right)_{<i-1} \, \left[ \Psi_q(\sigma)  \right]_i^\mcY.
		\end{align*}
		From Lemma \ref{lem_SnDnreindex}, this is equal to 
		$$
		\widetilde{\PPi}\left(\tau\right)_i - \sum_{\mu\prec\tau}\sum_{q\geq 0}\, \widetilde{g}\left(\tau/\mu \right)_{<i-1} \, \left[ \Psi_q(\downarrow^q\mu)  \right]_i^\mcY.
		$$
		We have $\sum_{q\geq 0} \Psi_q(\downarrow^q\mu)=\Psi(\mu)$, then from the induction hypothesis on the element $\mu$, we obtain
			$
			\widetilde{ \PPi}(\tau)_i - \sum_{\mu\prec\tau}  \widetilde{ g}\big(\tau/\mu \big)_{<i-1} \, \big[ \mu \big]_i,
		$
		which is indeed equal to $[\tau]_i$ from \eqref{eq_defreccrochettilde}.
	\end{proof}
	
	We are now able to prove Theorem \ref{thm_regularityZtau}.
	
	\begin{proof}[of Theorem \ref{thm_regularityZtau}]
		We prove it by induction on the size of  $\tau$. As $$[\lg\tau\rg]^\mcY_i = \Delta_i( Y_\tau), \quad  [\lg\downarrow^k \! \tau\rg X^k]^\mcY_i=0 \, \, \text{for } k\ne 0.$$ Lemma \ref{lem_sumcrochets} rewrites as 
		$$
		\Delta_i( Y_\tau) = [\tau]_i - \sum_{\substack{k\geq 0\\k=k_1+k_2}}\, \sum_{\substack{q\geq 1 \\ \tau_q\prec\dots\prec\tau_1\prec \tau}} \frac{1}{k_1!k_2!}\Big[ \big\lg {\downarrow^k\hspace{-2pt}\tau/\tau_1}, \dots , {\tau_q}   \big\rg_{k_1} X^{k_2} \Big]_i^\mcY.
		$$
		By induction hypothesis, when $\mu$ is a letter of a word $\lg {\downarrow^k\hspace{-2pt}\tau/\sigma_1}, \dots , {\sigma_q}   \big\rg$ as above, then the distribution $Y_\mu$ is in $C^{|\mu|}$ with bound $\norme{Y_\mu}_{C^{|\mu|}}\lesssim \norme{\mu}^*_{(\PPi,g)}$. 
		Then,  \cite[Theorem 1]{LocalExpansionsParacSystems} ensures that this collection defines a model over $\scrT$, and Proposition \ref{prop_RegCrochets} implies that
		all these terms are indeed dominated by $2^{-i|\tau|} \Vert\tau\Vert_{ (\PPi,g)}^*$.
	\end{proof}

	Theorem \ref{thm_regularityZtau} admits the following corollary, giving a paracontrolled representation for any modelled distribution based on the $Y_\tau$.

	\begin{corollary}\label{cor_repmodelledditribParap}
		Let ${\bm f}(x) = \sum_{\tau\in \mcB} f_\tau(x) \tau$ some modelled distribution in $\mcD^\gamma(T,{ g})$. There exists a function ${\bm f}^\# \in C^\gamma$, and for any $\tau\in\mcB$ there exists a distribution $f_\tau^\# \in C^{\gamma-|\tau|}$ such that
		\begin{equation}\label{eq_repreconstructionfparap}
			{\bm R}{\bm f} = \sum_{\tau \in \mcB, \, |\tau|<\gamma} \P({f_\tau},\,  Y_{\tau})  + {\bm f}^\#,
		\end{equation}
		and
		\begin{equation}\label{eq_repftauparap}
			f_\tau = \sum_{\nu > \tau} \, \,  \P({f_\nu},\,  Y_{\nu /\tau})  + f_\tau^\#.
		\end{equation}
		Furthermore we have the continuity estimate
		\begin{equation}
			\lVert {{\bm f}^\#} \rVert_\gamma + \sup_{\tau\in\mcB }\lVert {f_\tau^\#}\lVert_{\gamma-|\tau|} \lesssim \norme{\bm f}^*(\!| \PPi,g |\!) .
		\end{equation}
	\end{corollary}
	
	\begin{proof}
		We apply Theorem \ref{thm_regularityZtau} in some extension $\scrT_{\bm f}=(T_{\bm f}, T_{\bm f}^+)$ of the regularity structure $\scrT$, that encodes the local developments of ${\bm f}$ and its coefficients $(f_\tau)_\tau$.
		We set $T_{\bm f}=T\oplus \bfR \rmF$ and $T_{\bm f}^+=T^+\oplus T_{<\gamma}$, with natural basis
		$$
		\mcB_{\bm f}^+ = \mcB^+ \sqcup \big\{ \rmF_\tau, \, \tau\in\mcB_{<\gamma}   \big\}, \quad \text{and}\quad \mcB_{\bm f } = \mcB \sqcup \{ \rmF \}.
		$$
		The symbol $\rmF$ represents the reconstruction of the modelled distribution  ${\bm f}$, and $\rmF_\tau$ represents its coefficient $f_\tau$. We set $|\rmF|=\gamma$ and $|\rmF_\tau|=\gamma-|\tau|$, and extends the coproducts from $\scrT$ to $\scrT_{\bm f}$ by setting
		\begin{align*}
		\Delta^{\!+} \rmF_\tau =  \rmF_\tau \otimes \mathbf{1} \, + \,   \sum_{\mu\geq\tau} \mu/\tau\otimes \rmF_\mu, \quad
		\Delta \rmF = \rmF\otimes \mathbf{ 1} + \sum_{|\tau|<\gamma} \tau \otimes \rmF_\tau
		\end{align*}
		And one extends the model ${ (\PPi,g)}$ to $\scrT_{\bm f}$  by setting $\PPi(\rmF)={\bm R}{\bm f}$ and ${ g}_x(\rmF_\tau) = f_\tau(x) $. One obtains the corollary by noting that $\downarrow^j\hspace{-2pt}\rmF_\tau=0$ for any $\tau\in\mcB$ and $j\geq 1$.	\end{proof}
	One can iterate  \eqref{eq_repreconstructionfparap} and \eqref{eq_repftauparap}, and obtain the following representation
	\begin{equation}
		{\bm R}{\bm f} = \sum_{q\geq0}\sum_{\nu_q\succ\dots\succ\nu_1} \P\big( f_{\nu_q}^\#, \, Y_{\nu_q/\nu_{q-1}},\dots, Y_{\nu_1}\big),
	\end{equation}
	\begin{equation}
		f_\tau = \sum_{q\geq 0}\sum_{\nu_q \succ \dots\succ\nu_1\succ\tau} \P\big( f_{\nu_q}^\#, \, Y_{\nu_q/\nu_{q-1}},\dots, Y_{\nu_1/\tau}\big),
	\end{equation}
 where the distributions in the paraproducts have enough regularity.
	For $|k|<\gamma-|\tau|$, we set 
	\begin{equation}
		\partial^k_* \, f_\tau  = \sum_{\nu_q \succ \dots\succ\nu_1\succ\tau} \partial^k_*\P\big( f_{\nu_q}^\#, \, Y_{\nu_q/\nu_{q-1}},\dots, Y_{\nu_1/\tau}\big).
	\end{equation}
	
	\begin{proposition}\label{prop_starderiv_coefficientmodelleddist}
		For any modelled distribution ${\bm f}\in\mcD^\gamma(T,{ g})$, for any $\tau\in\mcB$ and $|k|<\gamma-|\tau|$, we have the identity
		$$
		\partial^k_* \, f_\tau = f_{(\downarrow^k)^*\tau},
		$$
		where $(\downarrow^k)^*$ is the adjoint of $\downarrow^k\hspace{-2pt}$ on $T$.
	\end{proposition}
	\begin{proof}
		We apply Proposition \ref{prop_gDntaustarderivative} on the extended regularity structure described in the proof of Corollary \ref{cor_repmodelledditribParap}, with the symbol $F_\tau$. It suffices to show the identity
		$$
		D^k F_\tau = F_{(\downarrow^k\hspace{0pt})^*\tau}.
		$$
		We still let $\pi_{X^k}$ be the orthogonal projection on the line generated by the monomial $X^k$. From the formula $\Delta^+ F_\tau = \sum_{\sigma\in\mcB}  \sigma/\tau\otimes F_\sigma$, we have
		\begin{align*}
			D^k F_\tau &= \sum_{\sigma\in\mcB}  F_\sigma \, \lg \, \sigma/\tau , \, X^k  \rg 
			=\sum_{\sigma\in\mcB} \frac{1}{k!} F_\sigma  \, \lg \, \downarrow^k\sigma, \, \tau  \rg  
			\\
			&= \sum_{\sigma\in\mcB}  F_\sigma \, \lg  \,  \sigma , \, (\downarrow^k)^* \tau   \rg 
			= F_{(\downarrow^k)^* \tau }.
		\end{align*}
	\end{proof}

	

	\subsection{Paracontrolled representation for the derivatives \ }
	This subsection is dedicated to the proof of Proposition \ref{prop_gDntaustarderivative}.
	\begin{lemma}\label{lem_derivModifModel}
		Given a smooth model ${ (\PPi,g)}$, one has
		\begin{enumerate}
			\item
			For any $\tau\in\mcB$ and $n\geq0$,
			\begin{equation}\label{eq_partialnkpitilde}
				\sum_{j\geq 0} \frac{(-x)^j}{j!}  \partial^n  \big({\PPi}\big(\downarrow^j\hspace{-2pt}\tau\big)\big)(x) =  \sum_{k\leq n}\binom{n}{k} \partial^{n-k} \big( \widetilde{ \PPi}\big(\downarrow^k\hspace{-2pt}\tau\big)\big)(x).
			\end{equation}
			\item \label{item_DnPixTau_general} For any $\tau\in\mcB$ and $|n|>|\tau|$, 
			\begin{equation} \label{eq_DnPixTau_general}
				\begin{aligned}
				\partial^n_y\big(  { \Pi}_x\tau(y)   \big)_{|y=x} 
				& = 
				\sum_{k\leq n}\binom{n}{k} \partial^{n-k} \big( \widetilde{ \PPi}\big(\downarrow^k\hspace{-2pt}\tau\big)\big)(x)
				\, \\ & - 
				\sum_{\substack{\sigma\prec\tau, \,\, \sigma\notin T_X \\  |\sigma|-|n|<0}} \widetilde{ g}_x\big(\tau/\sigma) \, \partial^n_{y} \big({\Pi}_x\sigma(y)\big)_{|y=x}.
				\end{aligned}
			\end{equation}
			
			\item \label{item_gDnTau_general} For any $\tau\in\mcB$ and $|n|<|\tau|$, 
			\begin{equation}\label{eq_gDnTau_general}
				\widetilde{ g}_x\big( D^n\tau \big)	= 
				\sum_{k\leq n} \binom{n}{k}\partial^{n-k} \big( \widetilde{ \PPi}\big(\downarrow^k\hspace{-2pt}\tau\big)\big)(x)
				- 
				\sum_{\substack{\sigma\prec\tau, \,\, \sigma\notin T_X \\  |\sigma|-|n|<0}} \widetilde{ g}_x\big(\tau/\sigma) \, \partial^n_{y} \big({\Pi}_x\sigma(y)\big)_{|y=x}.
			\end{equation}
		\end{enumerate}
	\end{lemma}
	
	\begin{proof}
		We first prove point \textit{1.} From  \eqref{eq_pitildetopi}, we have
		\begin{align*}
			\sum_{j\geq 0} \frac{(-x)^j}{j!}  \partial^n  \big({\PPi}\big(\downarrow^j\hspace{-2pt}\tau\big)\big)(x) 
			&=
			\sum_{j\geq 0} \frac{(-x)^j}{j!}  \partial^n  \Big(  \sum_{m\geq 0} \frac{x^m}{m!}\widetilde{\PPi}\big(\downarrow^{m+j}\hspace{-3pt}\tau\big)\Big)
			\\
			&=
			\sum_{j,m\geq 0}\sum_{n=n_1+n_2} \binom{n}{n_1} \frac{(-1)^j x^{m+j-n_1}}{j! (m-n_1)!} \partial^{n_2} \widetilde{\PPi}\big(  \downarrow^{m+j}\hspace{-3pt}\tau\big)
			\\
			&=\sum_{\substack{p\geq 0}} \enskip \sum_{j\leq p } \frac{(-1)^j x^{p}}{j!(p-j)!}  \sum_{\substack{ p+n =r_1+r_2}}  \binom{n}{r_1}  \partial^{r_1} \widetilde{\PPi} \big(  \downarrow^{r_2}\hspace{-3pt}\tau  \big)
			\\
			&= \sum_{k\leq n} \binom{n}{k}\partial^{n-k} \widetilde{ \PPi}\big(\downarrow^k\hspace{-2pt}\tau\big).
		\end{align*}
		We now turn to points \textit{2} and \textit{3}. We have
		\begin{align*}
			{ \Pi}_x(\tau)(y) &= ({\PPi}\tau)(y) - \sum_{\sigma<\tau}{ g}_x(\tau/\sigma) \,  {\Pi}_x\sigma(y)
			\\
			&=
			({\PPi}\tau)(y) - \sum_{j\geq 1}\frac{x^j}{j!}{ \Pi}_x(\downarrow^j\hspace{-2pt}\tau)(y)  - \sum_{\sigma\prec\tau} { g}_x(\tau/\sigma) \, {\Pi}_x\sigma(y),
		\end{align*}
		that is 
		$$
		\sum_{j\geq 0}\frac{x^j}{j!}{ \Pi}_x(\downarrow^j\hspace{-2pt}\tau)(y)  
		=
		({\PPi}\tau)(y)  - \sum_{\sigma\prec\tau} { g}_x(\tau/\sigma) \, {\Pi}_x\sigma(y),
		$$
		and by Pascal inversion
		\begin{equation}\label{eq_pitildetorestepi}\begin{split}
				{ \Pi}_x\tau(y) &=  \sum_{j\geq 0}\frac{(-x)^j}{j!} \Big( {\PPi}(\downarrow^j\tau)(y) - \sum_{\sigma\prec \downarrow^j \tau} { g}_x(\downarrow^j\tau/\sigma) \, {\Pi}_x\sigma(y)\Big)
				\\
				&=  \sum_{j\geq 0}\frac{(-x)^j}{j!}  {\PPi}(\downarrow^j\hspace{-2pt}\tau)(y) \,  - \,  \sum_{\sigma\prec\tau} \widetilde{ g}_x(\tau/\sigma) \, {\Pi}_x\sigma(y).
		\end{split}\end{equation}
		Using Equation \eqref{eq_partialnkpitilde}, we obtain
		\begin{align*}
			\partial^n_y\big(  { \Pi}_x\tau(y)   \big)_{|y=x} 
			&= 
			\sum_{j\geq 0}\frac{(-x)^j}{j!}  \partial^n{\PPi}(\downarrow^j\hspace{-2pt}\tau)(x) \,  - \,  \sum_{\sigma\prec\tau} \widetilde{ g}_x(\tau/\sigma) \, \partial^n_y\big({\Pi}_x\sigma(y)\big)_{|y=x}
			\\
			&=\sum_{k\leq n} \partial^{n-k} \big( \widetilde{ \PPi}\big(\downarrow^k\hspace{-2pt}\tau\big)\big)(x) \,  - \,  \sum_{\sigma\prec\tau} \widetilde{ g}_x(\tau/\sigma) \, \big({\Pi}_x\sigma(y)\big)_{|y=x},
		\end{align*}
		hence point \textit{2}. We suppose now that $|n|<|\tau|$, one can rewrite \eqref{eq_pitildetorestepi} as
		\begin{align*}
			({\Pi}_x\tau)(y) =\sum_{j\geq 0}\frac{(-x)^j}{j!}  {\PPi}(\downarrow^j\hspace{-2pt}\tau)(y)&-\sum_{\sigma\prec\tau, \,\, \sigma\notin T_X}  \widetilde{ g}_x(\tau/\sigma) \, {\Pi}_x\sigma 
			\\
			&- \sum_{|k|<|\tau|} \widetilde{ g}_x(D^k\tau) \,  \frac{(y-x)^k}{k!}.
		\end{align*}
		As $|n|<|\tau|$, one has $\partial_y^n\big({ \Pi}_x\tau\big)_{|y=x} = 0$, and then from \eqref{eq_partialnkpitilde},
		\begin{align*}
			\widetilde{ g }_x(D^n\tau) &= 
			\partial^n_{y} \Big(\sum_{j\geq 0}\frac{(-x)^j}{j!}  {\PPi}(\downarrow^j\hspace{-2pt}\tau)(y)-\sum_{\sigma\prec\tau, \,\, \sigma\notin T_X}  \widetilde{ g}_x(\tau/\sigma) \,  ({\Pi}_x\sigma)(y) \Big)_{|y=x}
			\\
			&=
			\sum_{k\leq n} \binom{n}{k}\partial^{n-k} \big( \widetilde{ \PPi}\big(\downarrow^k\hspace{-2pt}\tau\big)\big)(x) \, - \, \sum_{\sigma\prec\tau, \,\, \sigma\notin T_X}  \widetilde{ g}_x(\tau/\sigma) \, \partial^n_y\big({\Pi}_x\sigma\big)_{|y=x}.
		\end{align*}	
		For $\sigma\prec\tau$ such that $|n|<|\sigma|$, one has $\partial_y^n\big({ \Pi}_x\sigma\big)_{|y=x} = 0$, then one can restrict the second sum to the $\sigma\prec\tau$ with $\sigma\notin T_X$ and $|\sigma|<|n|$, which gives the result.
	\end{proof}

	\begin{proposition}\label{prop_starderivY}
		Let $\tau\in \mcB\sqcup\mcB^+$ and $n\in\bfN^{d+1}$,
		\begin{itemize}
			\item[-] If $|\tau|-|n|>0$, then 
			\begin{equation}
				\widetilde{ g}^\mcY\big( D^n \Psi(\tau) \big)
				= \widetilde{ g}\big( D^n \tau \big).
			\end{equation}
			\item[-] If $|\tau|-|n|<0$, then 
			\begin{equation}
				\partial^n_{y} \Big({\Pi}^\mcY_x\left(\Psi(\tau)\right)(y)\Big)_{|y=x}=
				\partial^n_y \big({ \Pi}_x \tau (y) \big)_{|y=x}.
			\end{equation}
		\end{itemize}
	\end{proposition}
	\begin{proof}
		We prove both identities in the same induction. We let $\tau\in\mcB\sqcup\mcB^+$ and $|n|<|\tau|$ and we apply point \ref{item_gDnTau_general} of Lemma \ref{lem_derivModifModel} to the element $\Psi(\tau)$. Using Lemma \ref{lem_DeltaPsitau}, we have
		\begin{align*}
		\widetilde{ g}^\mcY_x\big( D^n \Psi(\tau) \big) =& \sum_{k\leq n} \binom{n}{k} \partial^{n-k} \Big(\widetilde{\PPi}^\mcY\left( \downarrow^k\Psi(\tau)\right)\Big)(x) 
		\\&-
			\sum_{\substack{\sigma\prec\tau \\ |q|<|\tau/\sigma| \\ |\sigma|+|q|-|n|<0}} \widetilde{g}^\mcY_x\left(D^q\Psi(\tau/\sigma)\right) \,\, \partial^n_{y} \left({\Pi}^\mcY_x\Psi_q(\sigma)(y)\right)_{|y=x}.
		\end{align*}
		Using Lemma \ref{lem_commutdownarrowPsi} we have $\widetilde{\PPi}^\mcY\left( \downarrow^k\Psi(\tau)\right)=\widetilde{\PPi}\left( \downarrow^k\tau\right)$, and the induction hypothesis ensures that 
		$\widetilde{g}^\mcY_x\left(D^q\Psi(\tau/\sigma)\right) = \widetilde{g}^\mcY_x\left(\Psi(D^q(\tau/\sigma))\right)$. Then $	\widetilde{ g}^\mcY_x\big( D^n \Psi(\tau) \big)$ is equal to 
		\begin{align*}
			& \sum_{k\leq n} \binom{n}{k} \partial^{n-k} \Big(\widetilde{\PPi}\left( \downarrow^k\tau\right)\Big) (x)
			-
				\sum_{\substack{\sigma\prec\tau \\ |q|<|\tau/\sigma| \\ |\sigma|+|q|-|n|<0}} \widetilde{g}^\mcY_x\left(\Psi(D^q(\tau/\sigma))\right) \,\, \partial^n_{y} \left({\Pi}^\mcY_x\Psi_q(\sigma)(y)\right)_{|y=x}.
		\end{align*}
		From the definition of $\mcY$, we have $\widetilde{g}^\mcY_x\left(\Psi(D^q(\tau/\sigma))\right)=\widetilde{g}_x\left(D^q(\tau/\sigma)\right)$. From Lemma \ref{lem_SnDnreindex} and induction hypothesis $	\widetilde{ g}^\mcY_x\big( D^n \Psi(\tau) \big)$ is then equal to 
		\begin{align*}
			& \sum_{k\leq n} \binom{n}{k} \partial^{n-k} \Big(\widetilde{\PPi}\left( \downarrow^k\tau\right)\Big) (x)
			-
			\sum_{\substack{\nu\prec\tau  \\ |\nu|-|n|<0 \\ q\geq 0}} \widetilde{g}^\mcY_x\left(\Psi(\tau/\nu)\right) \, \partial^n_{y} \left({\Pi}^\mcY_x\Psi_q(\downarrow^q\nu)(y)\right)_{|y=x}
			\\
			&=
			\sum_{k\leq n} \binom{n}{k} \partial^{n-k} \Big(\widetilde{\PPi}\left( \downarrow^k\tau\right)\Big) (x)
			-
			\sum_{\substack{\nu\prec\tau  \\ |\nu|-|n|<0 }} \widetilde{g}^\mcY_x\left(\Psi(\tau/\nu)\right) \,\, \partial^n_{y} \left(({\Pi}_x\nu)(y)\right)_{|y=x}.
		\end{align*}
		which is indeed equal to $\widetilde{g}(D^n\tau)$ from point \ref{item_gDnTau_general} of Lemma  \ref{lem_derivModifModel}.
		The second identity of the Lemma is proven along the same lines, by using point \ref{item_DnPixTau_general} of Lemma \ref{lem_derivModifModel}
		
	\end{proof}
	We are now ready to prove Proposition \ref{prop_gDntaustarderivative} from  Proposition \ref{prop_starderivY}.
	
	\begin{proof}[of Proposition \ref{prop_gDntaustarderivative}]
		
		 From Lemma \ref{lem_gtildetogitilde} and \eqref{eq_commutDnDownaroown}
		\begin{align*}
			\widetilde{g}(D^n\tau)_i &= \sum_{m\geq 0}\frac{1}{m!}\Delta_i^m\left(\widetilde{g}(\downarrow^m D^n\tau)\right) 
			= \sum_{m\geq 0}\frac{1}{m!}\Delta_i^m\left(\widetilde{g}(D^n\downarrow^m\tau)\right) 
		\end{align*}
		From Proposition \ref{prop_starderivY} and Lemma \ref{lem_commutdownarrowPsi} and \eqref{eq_commutDnDownaroown}, this is equal to
		$$
		\sum_{m\geq 0}\frac{1}{m!}\Delta_i^m\left(\widetilde{g}\left(D^n\Psi(\downarrow^m\tau)\right)\right)  
		= 	\sum_{m\geq 0}\frac{1}{m!}\Delta_i^m\left(\widetilde{g}\left(\downarrow^m D^n\Psi(\tau)\right)\right)  
		$$
		which is indeed equal to $\widetilde{g}(D^n\Psi(\tau))_i$ from Lemma \ref{lem_gtildetogitilde}.

	\end{proof}

\subsection{Integrated stochastic data}\label{subsect_integration_Y}

In this subsection, we introduce a new stochastic data $ \widehat{Y}_{\tau}$ that will play a role in our Schauder estimates. It is a modification of 
 $Y_{\tau}$, where an additional integration step against  is performed. We provide analytical estimates for this new object in Theorem \ref{prop_regularity_YI_tau} below. Proving this estimate relies on the same step as for Theorem \ref{thm_regularityZtau}, that is combining integration against Littlewood-Paley blocks, except that an additional integration step is performed here.

\begin{definition}\label{def_YItau}
Define recursively for any $\tau\in \mcB \sqcup \mcB^+$ the distribution
\begin{equation}\label{eq_def_YItau}
\widehat{Y}_\tau = \sum_{n\geq 0} \frac{1}{n!} K^{[n]}*\widetilde{\PPi}(\downarrow^n \! \tau) - \sum_{\substack{\sigma\prec \tau \\ k\geq 0}} \frac{1}{k!}\P_k\big(  \widetilde{g}(\downarrow^k\!\tau/\sigma),\widehat Y_\sigma\big),
\end{equation}
with $K^{[n]}(x) = (-x)^nK(x)$.
\end{definition}

 From the paracontrolled representation in \eqref{eq_ParacRepIterated2}, Equation \eqref{eq_def_YItau} can be rewritten as
$$
\sum_{n\geq 0} \frac{1}{n!} K^{[n]}*\widetilde{\PPi}(\downarrow^n \tau) = \sum_{\substack{q\geq 0\\ \tau_q\prec\dots\prec\tau }}\sum_{k\geq 0} \frac{1}{k!}\P_k\Big(  Y_{\downarrow^k\tau/\tau_1},\dots, Y_{\tau_{q-1}/\tau_q},\widehat Y_{\tau_q}\Big).
$$
One observes that $\widehat{Y}_{\tau}$ is at the right most location of the iterated paraproducts. This corresponds to the location of the root after applying the Hairer-Kelly map and this is where the integration is performed.
The goal of this subsection is to prove the following theorem, reminiscent of Theorem \ref{thm_regularityZtau}.

\begin{theorem}\label{prop_regularity_YI_tau}
	For any symbol $\tau\in \mcB \sqcup \mcB^+$, the distribution $\widehat Y_\tau$ is in $C^{|\tau|+\beta}$, and furthermore we have the continuity estimate
	\begin{equation}\label{eq_continuity_Y}
		\lVert{\widehat Y_\tau}\lVert_{|\tau|+\beta} \lesssim  \Vert\tau\Vert_{ (\PPi,g)}^*.
	\end{equation}
\end{theorem}
The result is proved along the same lines as Theorem \ref{thm_regularityZtau}. We will present how to adapt the different lemmas in order to build the proof.
Given any symbol $\tau\in\mcB\sqcup\mcB^+$, one introduces a new letter $\tau^I$, and one extends the collection $\mcY$ by setting $$Y_{\tau^I} = \widehat{Y}_\tau.$$
We define the modified Hairer-Kelly map
\begin{equation} \label{modified_Hairer_Kelly}
	\Psi^I(\tau) = \sum_{\substack{k\geq 0 \\  k=k_1+k_2}} \sum_{\substack{q\geq 0 \\ \tau_q\prec\dots\prec\tau_1\prec\tau}} \, \frac{1}{k_1!\, k_2!} \, \big\lg \downarrow^k\hspace{-2pt}\tau/\tau_1,\, \tau_1/\tau_2, \dots, \tau_q^I \big\rg_{k_1} X^{k_2}.
\end{equation}
So that one can rewrite \eqref{eq_def_YItau} as
\begin{equation}\label{eq_PiTauPiPsiTau_I}
	\sum_{n\geq 0} \frac{1}{n!} K^{[n]}*\widetilde{\PPi}(\downarrow^n \tau) = \widetilde{\PPi}(\Psi^I(\tau)).
\end{equation}
We also set 
\begin{equation}
	\Psi_n^I(\tau) = \sum_{\substack{n=n_1+n_2}} \sum_{\substack{q\geq 0 \\ \tau_q\prec\dots\prec\tau_1\prec\tau}} \, \frac{1}{n_1!\, n_2!} \, \big\lg\tau/\tau_1,\, \tau_1/\tau_2, \dots, \tau_q^I \big\rg_{n_1} X^{n_2}.
\end{equation}
The same computations as in Lemma \ref{lem_DeltaPsitau} give the identity
\begin{equation}
	\begin{aligned}
	\Delta \Psi^I(\tau) & = \sum_{\substack{\sigma\prec\tau\\n\in \N^{d+1}}} \Psi^I_n(\sigma)\otimes  D^n\left(\Psi(\tau/\sigma)\right)
	+\sum_{n \in \N^{d+1}} \frac{X^n}{n!} \otimes D^n \left( \Psi^I(\tau)\right) \\ & +\sum_{n\in\bfN^{d+1}} \Psi^I(\downarrow^n\!\tau)\otimes \frac{X^n}{n!}.
	\end{aligned}
\end{equation}
We also note that for any $n\in\bfN^{d+1}$, we have the commutation \begin{equation*}
	\Psi^I(\downarrow^n\!\tau) = \downarrow^n\!\Psi^I(\tau).
	\end{equation*}
	We define the kernel $\widehat K_i$ by
\begin{equation*}
\widehat K_i = K*K_i.
\end{equation*}
We denote by $\widehat\Delta_i$ the associated operator. For $n\in\bfN^{d+1}$, we also let $\widehat\Delta_i^n$ to be the operator associated to the kernel $\widehat{K}_i(x)(-x)^n$. The following Lemma gives the analytical properties of these kernels, that will be useful in Proposition \ref{prop_RegCrochets_I} below, giving the analytical block for the proof of Theorem \ref{prop_regularity_YI_tau}.

\begin{lemma}\label{lem_scalingprty_I} The family of kernels $(\widehat K_i)_i$ verifies the following properties
	\begin{enumerate}
		\item  For any $|k|\leq B$ and $i\geq0$, we have
		 $\int \widehat K_i(x)x^k\ddd x = 0.$
		\item There exists an integrable function $\Phi$ such that we have the scaling property
		\begin{equation}\label{eq_scalingpropertiy_KHat}
			|\widehat K_i(x)| \lesssim  2^{id} \Phi ( \mathcal{S}^{2^i}(x)).
		\end{equation}
	\end{enumerate}
\end{lemma}

\begin{proof}
	For the first point, the cancellation property of $K_i$ gives for any $|k|\leq B$
	\begin{align*}
		\int \widehat{K_i}(x)x^k\ddd x &= \int \!\!\!\! \int x^k K_i(x-z) K(z) \ddd x \ddd z =  \int \!\!\!\! \int (y+z)^k K_i(y) K(z) \ddd y \ddd z
		\\
		&= \int \Big( \int (y+z)^kK_i(y)\ddd y \Big) K(z)\ddd z = 0.
	\end{align*}
	For proving the second point, we first note that $K_i$ and $K$ verify the scaling properties $K_i(x)=2^{i(d+\beta)}K_0(\mathcal{S}^{2^i}(x))$ and $|K(\mathcal{S}^{2^{-i}}(x))|\lesssim 2^{id} \Phi_0(x)$, for some integrable $\Phi_0$. Then
	\begin{align*}
		|\widehat K_i(z)| &= \Big|\int_{\bfR^{d+1}}2^{i(d+\beta)} K_0(\mathcal{S}^{2^i}( z-x)) K(x) \ddd x \Big|
	\\ &	= \Big|\int_{\bfR^{d+1}} K_0(\mcS^{2^i}(z)-y) K(\mcS^{2^{-i}}(y)) \ddd y \Big|
		\\
		&\lesssim 2^{id} \int |K_0(\mcS^{2^i}(z)-y)\Phi_0(y)| \ddd y  =  2^{id} \Phi(\mathcal{S}^{2^i}(z)).
	\end{align*}
	with $\Phi=|K_0|*\Phi_0$.
\end{proof}

We define the integrated version of the brackets defined in \eqref{eq_defbracket_1}\eqref{eq_defbracket_2}. In each term the operator $\Delta_i$ is replaced by $\widehat\Delta_i$ in the last factor. 
	\begin{equation}
		\label{integrated_bracket}
		\begin{aligned}
		&\widehat{[\tau]}_i \defeq \widehat\Delta_i \big({ \PPi}(\tau)\big)
		- \sum_{q\geq 1} (-1)^{q-1} \sum_{\tau_q<\cdots<\tau_1<\tau} \\ & \hspace{3cm}  \Delta_{<i-1}\big({ g}\big(\tau/\tau_1\big)\big)\cdots \Delta_{<i-1}\big({ g}\big(\tau_{q-1}/\tau_q\big)\big)  \widehat \Delta_i\big({ \PPi}\big(\tau_q\big)\big).
\end{aligned}\end{equation}
We also have the recursive formulas
\begin{equation}\label{eq_defreccrochettilde_I} \begin{aligned}
		\widehat{[\tau]}_i & = \widehat \Delta_i\big({ \PPi}(\tau)\big) - \sum_{\sigma<\tau} \Delta_{<i-1} \big({ g}(\tau/\sigma)\big) \, \widehat{[\sigma]}_i
		\\
		\widehat{[\tau/\sigma]}_i & = \widehat \Delta_i\big({ g}(\tau/\sigma)\big) - \sum_{\sigma<\nu<\tau} \Delta_{<i-1} \big( { g}(\tau/\nu) \big) \, \widehat{[\nu/\sigma]}_i.
	\end{aligned}
\end{equation}

The following result is reminiscent of Proposition \ref{prop_RegCrochets} and follows from the same arguments

\begin{proposition} \label{prop_RegCrochets_I}
	For any $\tau\in T_{|\tau|}$ and $\tau/\sigma\in T_{|\tau/ \sigma|}^+$,  we have
	$$ 
	\sup_{i\geq -1} 2^{i(|\tau|+\beta)} \widehat{[\tau]}_i \lesssim \Vert \tau\Vert_{(\!| \PPi , g |\!)}^*, \quad 
	\sup_{i\geq -1} 2^{i(|\tau/\sigma|+\beta)} \widehat{[\tau/\sigma]}_i \lesssim \Vert \tau/\sigma\Vert^*_{ ( \PPi , g )}.
	$$
\end{proposition}
\begin{proof}
	The proof is done by induction and follows the same line as the proof of Proposition \ref{prop_RegCrochets}. We write 
	$$
	\widehat{[\tau]}_i = \widehat{\P}(\tau)_i + \sum_{\sigma<\tau} {\bf Q}(\tau/\sigma)_i\widehat{[\sigma]}_i
	$$
	where ${\bf Q}$ was defined in Subsection \ref{subsect_modifmodel} and where
	\begin{align*}
	\widehat{\P}(\tau)_i(z) &= \int_{(\bfR^{d+1})^2}  K_{<i-1}(z-x )\widehat K_i(z-y) \, \Pi_x\tau(y) \, \ddd x\ddd y.
	\end{align*}
	We already know that $|{\bf Q}(\tau/\sigma)_i|\lesssim \norme{\tau/\sigma}_{( \PPi , g )}2^{-i|\tau/\sigma|}$. In order to close the induction, we have to prove the regularity estimates 
	$$
	|\widehat{\P}(\tau)_i|\lesssim 2^{-i(|\tau|+\beta)} \norme{\tau}_{ (| \PPi , g |)}.
	$$
	This follows from the scaling property given in Lemma \ref{lem_scalingprty_I} and from the relation $\Pi_x = (\Pi_z\otimes\gamma_{zx})\Delta$.
	\begin{align*}
		& |\widehat \P(\tau)_i(z)| = \Big| \iint K_{<i-1}(z-x) \widehat K_i(z-y) \sum_{\sigma<\tau}\gamma_{zx}(\tau/\sigma)\Pi_z\sigma(y) \ddd x\ddd y \Big|
		\\
		&\lesssim   \iint 2^{i(d+\beta)} |K_0(2^i(z-x))|2^{id}\Phi(2^i(z-y)) \sum_{\sigma<\tau}\Vert \tau\Vert_{( \PPi , g )} |z-x|^{|\tau/\sigma|}|z-y|^{|\sigma|}\ddd x\ddd y 
		\\
		&\lesssim 2^{-i\beta} \Vert \tau\Vert_{(| \PPi , g |)} \iint \sum_{\sigma<\tau} |K_0(z_1)|\Phi(z_2) 2^{-i|\tau/\sigma|} |z_1|^{|\tau/\sigma|}2^{-i|\sigma|} |z_2|^{|\sigma|} \ddd z_1 \ddd z_2
		\\
		&\lesssim  2^{-i(|\tau|+\beta)} \Vert \tau\Vert_{( \PPi , g )}. 
	\end{align*}
\end{proof}
We consider a new modified model
\begin{equation}\label{eq_gtildetogitilde_I}
\widetilde{\PPi}(\tau)^I_i = \sum_{n\geq 0} \frac{1}{n!} \widehat\Delta_i^n \big(\widetilde\PPi(\downarrow^n \! \tau) \big), \quad
\widetilde{g}(\tau/\sigma)^I_i = \sum_{n\geq 0} \frac{1}{n!} \widehat\Delta_i^n \big(\widetilde g(\downarrow^n\! \tau/\sigma)\big).
\end{equation}

The following lemma is reminiscent of Proposition \ref{lem_sumgtilde}, its proof follows from the same computations. 

\begin{lemma}\label{lem_sumgtilde_I}
	For any $\tau\in\mcB$ and $i\geq -1$ we have the relation
	\[
\widehat{[\tau]}_i =  \widetilde{ \PPi}(\tau)_i^I - \sum_{q\geq 1} (-1)^{q-1} \sum_{\tau_q\prec\cdots\prec\tau_1\prec\tau} \widetilde{ g}\big(\tau/\tau_1\big)_{<i-1} \cdots \,  \widetilde{ g}\big(\tau_{q-1}/\tau_{q}\big)_{<i-1} \,  \widetilde{ \PPi} \big(\tau_q\big)^I_i.
	\]
	Likewise for $\tau/\sigma\in \mcB^+$ we have
	\begin{align*}
		\widehat{[\tau/\sigma]}_i^I &=  \widetilde{ g}(\tau/\sigma)^I_i 
		\\
		&\quad - \sum_{q\geq 1} (-1)^{q-1}  \hspace{-0.5cm}\sum_{\sigma\prec\sigma_q\prec\cdots\prec\sigma_1\prec\tau} \widetilde{ g}\big(\tau/\sigma_1\big)_{<i-1}\cdots \, \widetilde{ g}\big(\sigma_{q-1}/\sigma_{q}\big)_{<i-1} \,  \widetilde{ g}\big(\sigma_q/\sigma\big)^I_i.
	\end{align*}
\end{lemma}
Lemma \ref{lem_sumgtilde_I} rewrites as 
	\[
\widehat{[\tau]}_i =  \widetilde{ \PPi}(\tau)_i^I - \sum_{\sigma\prec\tau} \widetilde{ g}\big(\tau/\sigma\big)_{<i-1} \,  \widehat{[\sigma]}_i, \quad
\widehat{[\tau/\sigma]}_i =  \widetilde{ \PPi}(\tau)_i^I - \sum_{\sigma\prec\nu\prec\tau} \widetilde{ g}\big(\tau/\nu\big)_{<i-1} \,  \widehat{[\nu/\sigma]}_i.
\]

\begin{lemma} \label{lem_PiiPiPsii_00_I}
	For any $\tau\in\mcB$, we have	$$
		\widetilde{\PPi}(\tau)_i^I = 	\widetilde{\PPi}^\mcY\big(\Psi^I(\tau)\big)_i
	$$
\end{lemma}
\begin{proof}
	From the equality $\widehat K_i = K*K_i$ and binomial identity
	\begin{equation}\label{eq_binomial_convo}\begin{split}
	\widehat\Delta_i^m u(x)&=\iint K_i(x-y) K(y-z) (z-x)^mu(z)\ddd y\ddd z 
	\\ 
	&=\sum_{m=m_1+m_2} \binom{m}{m_1} \int  K_i(x-y)(y-x)^{m_1}\Big(\int K^{[m_2]}(y-z)u(z) \ddd z\Big) \ddd y 
	\\
	 &=\sum_{m=m_1+m_2} \binom{m}{m_1}  \Delta_i^{m_1}\big( K^{[m_2]}* u\big).
 \end{split}\end{equation}
	 From Lemma \ref{lem_gtildetogitilde} and \eqref{eq_PiTauPiPsiTau_I} and \eqref{eq_binomial_convo}
	\begin{align*}
		\widetilde{\PPi}(\tau)^I_i &= \sum_{m\geq 0}\frac{1}{m!}\widehat\Delta_i^m\left(\widetilde{\PPi}(\downarrow^m \!\tau)\right) 
		=\sum_{m_1,m_2\geq 0}\frac{1}{m_1!m_2!}\Delta_i^{m_1}\left(K^{[m_2]}*\widetilde{\PPi}(\downarrow^{m_1+m_2} \!\tau)\right) 
		\\
		&= \sum_{m\geq 0}\frac{1}{m!}\Delta_i^m\left(\widetilde{\PPi}^\mcY(\Psi^I(\downarrow^m\!\tau))\right). \end{align*}
	From Lemma \ref{lem_commutdownarrowPsi} and \eqref{eq_gtildetogitilde_I}, this is equal to
	$$
	\sum_{m\geq 0}\frac{1}{m!}\Delta_i^m\left(\widetilde{\PPi}^\mcY(\downarrow^m\Psi^I(\tau))\right)  = 	\widetilde{\PPi}^\mcY\left(\Psi^I(\tau)\right)_i.
	$$
	
\end{proof}

\begin{lemma}\label{lem_sumcrochets_I}
	We have
	$$
	\widehat{[\tau]}_i = [\Psi^I(\tau)]^\mcY.
	$$
\end{lemma}
\begin{proof}
	The proof follows the same lines as the induction performed in the proof of Lemma \ref{lem_sumcrochets}.
From the definition of the brackets in \eqref{eq_defreccrochettilde} and from Lemma \ref{lem_DeltaPsitau}
	\begin{align*}
		 [\Psi^I(\tau)]^\mcY = \widetilde{\PPi}^\mcY\left(\Psi^I(\tau)\right)_i - \sum_{\sigma\prec\tau}\sum_{q\geq 0}\, \widetilde{ g }^\mcY\left(D^q\Psi(\tau/\sigma) \right)_{<i-1} \, \left[ \Psi_q^I(\sigma)  \right]_i^\mcY.
	\end{align*}
	From Lemma \ref{lem_PiiPiPsii_00_I} we have $\widetilde{\PPi}^\mcY\left(\Psi^I(\tau)\right)_i=\widetilde{\PPi}\left(\tau\right)^I_i$, and from Proposition \ref{prop_gDntaustarderivative} we have $\widetilde{ g }^\mcY\left(D^q\Psi(\tau/\sigma)\right)_{<i-1}=\widetilde{ g }\left(D^q(\tau/\sigma)\right)_{<i-1}$. Then 
	\begin{align*}
		[\Psi(\tau)]_i^\mcY = \widetilde{\PPi}\left(\tau\right)_i^I - \sum_{\sigma\prec\tau}\sum_{q\geq 0}\, \widetilde{ g }\left(D^q(\tau/\sigma) \right)_{<i-1} \, \left[ \Psi_q^I(\sigma)  \right]_i^\mcY.
	\end{align*}
	From Lemma \ref{lem_SnDnreindex}, this is equal to 
	$$
	\widetilde{\PPi}\left(\tau\right)^I_i - \sum_{\mu\prec\tau}\sum_{q\geq 0}\, \widetilde{g}\left(\tau/\mu \right)_{<i-1} \, \left[ \Psi_q^I(\downarrow^q\mu)  \right]_i^\mcY.
	$$
	We have $\sum_{q\geq 0} \Psi^I_q(\downarrow^q\mu)=\Psi^I(\mu)$, then from the induction hypothesis on the element $\mu$, we obtain
	$
	\widetilde{ \PPi}(\tau)_i^I - \sum_{\mu\prec\tau}  \widetilde{ g}\big(\tau/\mu \big)_{<i-1} \, \widehat{[ \mu ]}_i,
	$
	which is indeed equal to $\widehat{[\tau]}_i$ from \eqref{eq_defreccrochettilde_I}.
\end{proof}

	\begin{proof}[of Theorem \ref{prop_regularity_YI_tau}]
	The proof follows the same lines as the one of Theorem \ref{thm_regularityZtau}, it is performed by induction on the size of  $\tau$. Lemma \ref{lem_sumcrochets_I} rewrites as 
	$$
	\Delta_i( \widehat Y_\tau) = \widehat{[\tau]}_i - [\overline{\Psi}^I(\tau)]^\mcY,
	$$
	with $\overline{\Psi}^I(\tau)={\Psi}^I(\tau)-\lg \tau^I\rg$. By induction hypothesis and Theorem \ref{thm_regularityZtau}, for each letter $\mu$, the distribution $Y_\mu$ is in $C^{|\mu|}$ and $\widehat Y_\mu$ in $C^{|\mu|+\beta}$, with bound $\norme{Y_\mu}_{C^{|\mu|}}\lesssim \norme{\mu}^*_{(\PPi,g)}$. 
	Then,  \cite[Theorem 1]{LocalExpansionsParacSystems} ensures that this collection defines a model over $\scrT$, and Proposition \ref{prop_RegCrochets_I} implies that
	all these terms are indeed dominated by $2^{-i(|\tau|+\beta)} \Vert\tau\Vert_{ (\PPi,g)}^*$.
\end{proof}

	
	

 	\section{Stochastic data on decorated trees} \label{section_constructionZtau}
	
	In this section we construct the stochastic fields we will use to build solutions for Equation \eqref{eq_gPAM}. The stochastic data $\mcZ$ we will use in the resolution of \eqref{eq_gPAM}, will take the form of a tuple of distributions $\mcZ=(Z_\tau)_\tau$ indexed by decorated trees of degree at most $\gamma$. The role of these distributions will be first to define the paracontrolled ansatz, and then to compute the right hand side of \eqref{eq_gPAM} for functions satisfying the ansatz.
	The distributions $Z_\tau$ will live in the space $C^{|\tau|}$, and we are going to impose some algebraic relations on the $(Z_\tau)$, this will give the notion of admissible stochastic data. One sets a topology on this set of stochastic data by defining the distance $\norme{  \cdot \, ; \,  \cdot }_{\frak N}$ with
	$$
	\Vert\mcZ^1;\mcZ^2\Vert_{\frak N} = \sum_{\substack{\tau\in\mcT_{<\gamma}}} \lVert Z_\tau^1 - Z_\tau^2\lVert_{C^{|\tau|}}, 
	$$
	and we set $\norme{\mcZ}_{\frak N} = \norme{\mcZ;0}_{\frak N}$. We also set for any $\tau\in {\mcT}$ the quantities 
	\begin{equation}
		\norme{\tau}^*_\mcZ =  \max\Big\{  \norme{Z_\tau}_{C^{|\tau|}}, \, \max\Big\{\big\lVert{Z_{\downarrow^k\tau/\sigma}\big\lVert}_{C^{|\downarrow^k\tau/\sigma|}}\norme{\sigma}_\mcZ^*,\quad \sigma\prec\tau, k\geq0    \Big\}    \Big\}, 
	\end{equation} 
	and
	\begin{equation}
	\norme{\tau}^{*,-}_\mcZ = \max\Big\{\big\lVert{Z_{\downarrow^k\tau/\sigma}\big\lVert}_{C^{|\downarrow^k\tau/\sigma|}}\norme{\sigma}_\mcZ^*,\qquad \sigma\prec\tau, k\geq0    \Big\},
\end{equation}

	Given such a collection $\mcZ=(Z_\tau)$, it will be convenient for computations to enhance $\tau\mapsto Z_\tau$ as a linear map from $T \oplus T^{+}$ to $S'(\bfR^{d+1})$.

	\subsection{The naive stochastic data} \label{subsect_naivemodel}

We first place ourselves in the situation where all the noises $\xi_\ell$ in \eqref{eq_gPAM} are smooth, this will typically be the case when we use mollified versions $\xi^\eps_\ell$ of the noises. We are going to define the naive collection of stochastic fields, which is the analogue of the naive model in regularity structure. We call this collection naive because it will typically diverge when one removes the regularisation, and we will need to renormalise it.
We start from a collection of smooth noises $(\xi_\ell)_{\ell\in\{1,...,\ell_0\}}$. We now define the smooth field $Z_\tau$ by induction on the tree $\tau\in \mcB\sqcup\mcB^+$. We initialise the induction by setting   
$
Z_{\Xi_\ell} = \xi_\ell
$
for any $\ell\in\{1,...,\ell_0\}$ and 
$
Z_{X^k}=0
$
for any $k\in\bfN^{d+1}$. We introduce the reduced Hairer-Kelly map
\begin{align*}
\overline\Psi(\tau)  &= \Psi(\tau) -  \lg  \tau \rg.
\end{align*}
We notice that for any tree $\tau$ the reduced term $\overline \Psi(\tau)$ is a sum of words involving only letters that are trees strictly smaller than $\tau$. This will enable us to formulate the recursive definition of the $Z_\tau$.
For $\tau = X^p\Xi_\ell\prod_{j=1}^n\mcI_{b_j}(\tau_j) \in T$, we set 
\begin{equation}\label{eq_defZtau_Naif}
	Z_\tau = \mathbf{1}_{p=0} \,\, \xi_\ell  \,  \prod_{j=1}^{n}\sum_{l_j\geq 0} \frac{1}{l_j!} (\partial^{b_j}K)^{[l_j]} * \widetilde{\PPi}^\mcZ \left(\Psi(\downarrow^{l_j}\tau_j)\right) - \widetilde{ \PPi}^\mcZ\left(\overline\Psi(\tau)\right),
\end{equation}
where $ (\partial^{b_j}K)^{[l_j]}(x) =(\partial^{b_j}K)(x) \, (-x)^{l_j}$. For $\tau = X^p\prod_{j=1}^n\mcI^+_{b_j}(\tau_j) \in T^+$, we set
\begin{equation}  \label{eq_defZtau_T+}
		Z_\tau = \mathbf{1}_{p=0} \prod_{j=1}^{n} \widetilde{ g}^\mcZ\left(D^{b_j}\Psi(\mcI(\tau_j))\right)  - \widetilde{ \PPi}^\mcZ\left(\overline\Psi(\tau)\right).
\end{equation}
These two expressions can be made more explicit by writing
$$
\widetilde{ \PPi}^\mcZ(\Psi(\tau))= \sum_{k\geq 0} \sum_{\substack{ q\geq 0 \\ \tau_q\prec\dots\prec\tau_1\prec \tau  }} \frac{1}{k!} \P_k \left(Z_{\downarrow^k\tau/\tau_1},\dots, Z_{\tau_q}\right),
$$
and 
$$
\widetilde{ g}^\mcZ\left(D^{n}\Psi(\tau)\right) = \sum_{\substack{k\geq 0 \\ k=k_1+k_2}} \sum_{\substack{q\geq 0 \\ \tau_q\prec\dots\prec\tau_1\prec \tau }} \frac{1}{k_2!} \binom{n}{k_1}  \partial^{n-k_1}_*\P_{k_2}\left(  Z_{\downarrow^k\tau/\tau_1},\dots, Z_{\tau_q}   \right). 
$$
The right hand side of \eqref{eq_defZtau_Naif} and \eqref{eq_defZtau_T+} should be interpreted as some commutation relation involving paraproducts, integrations and products, that we evaluate on the $Z_\mu$ for smaller trees $\mu$. Equations \eqref{eq_defZtau_Naif} and \eqref{eq_defZtau_T+} can be reformulated as 
\begin{equation} \label{eq_defZtau_Naif_Alternative}
	\widetilde{ \PPi}^\mcZ\left(\Psi(\tau)\right) = \mathbf{1}_{p=0} \, \xi_\ell  \,  \prod_{j=1}^{n}\sum_{l_j\geq 0} \frac{1}{l_j!}(\partial^{b_j}K)^{[l_j]} * \widetilde{\PPi}^\mcZ\left(\Psi(\downarrow^{l_j}\tau_j)\right),  
\end{equation}
for $\tau = X^p\Xi_\ell\prod_{j=1}^n\mcI_{b_j}(\tau_j) \in T$. And for $\tau = X^p\prod_{j=1}^n\mcI^+_{b_j}(\tau_j) \in T^+$,
\begin{equation} \label{eq_defZtau_T+_alternative}
	\widetilde{ \PPi}^\mcZ\left(\Psi(\tau)\right) = \mathbf{1}_{p=0} \prod_{j=1}^{n} \widetilde{ g}^\mcZ\left(D^{b_j}\Psi(\mcI(\tau_j))\right).
\end{equation}

\begin{example} We have the following examples of naive stochastic fields one can associate to some smooth noise $\xi$ where we are in the stochastic multiplicative heat equation. We first compute the image by $ \Psi$ and we use it to get the stochastic data. 
	One has 
	\begin{equs}
		\Psi(\<IXi>) &= \<IXi>, \quad \Psi(X_{(0,1)} \<IXi> ) =  X_{(0,1)} \<IXi> + \lg  \<IXi> \rg X_{(0,1)}, \\ \Psi(\<Xii>) &= \<Xii> + \<IXi> \otimes \<generic>, \quad
		\Psi(\<IXiX>)   = \<IXiX> + \lg \<IXi>\rg X_{(0,1)}, \\ \Psi( X_{(0,1)}  \<Xiiiaa>) & = X_{(0,1)}  \<Xiiiaa> + \lg  \<Xiiiaa>\rg X_{(0,1)}  + 2 \, \<IXi> \otimes X_{(0,1)}\<IXi> + 2 \, X_{(0,1)} \<IXi> \otimes \<IXi>
		\\ + & 2 \, \lg \<IXi> \otimes \<IXi> \rg X_{(0,1)} + 2 \, \lg \<IXi> \otimes \<IXi> \rg_{(0,1)}.
	\end{equs}
	Then from the previous computations, one has
	\begin{equs}
		Z_{\<IXi>} &=  \tilde{\PPi} \, \<IXi> = \PPi \, \<IXi> = K * \xi,
		\quad                  
		Z_{X_{(0,1)} \<IXi>} =  0.
		\end{equs} 
		The two terms of $ \Psi(X_{(0,1)} \<IXi>) $ contain a factor $ X_{(0,1)} $ that is sent to zero.
		\begin{equs}
		Z_{\<Xii>} 
		& = \widetilde{\PPi} \, \<Xii>  - \widetilde{\PPi}^{\mathcal{Z}} ( \<IXi>  \otimes \<generic>)
		= \PPi \, \<Xii>  - \P(Z_{\<IXi>},\, Z_{\<generic>}),=   \, \xi K * \xi - \P(K*\xi,\, \xi),
		\\
		Z_{\<IXiX>} & = K^{[(0,1)]} * \xi, \quad 
		Z_{X_{0,1}(\mcI(\Xi))^2} = -2\P_{(0,1)}( 	K * \xi , 	K * \xi  ), 
	\end{equs}
	where for the last term, the only term not sent to zero is  $   2 \, \lg \<IXi> \otimes \<IXi> \rg_{(0,1)} $.
\end{example}

The naive stochastic data $Z_\tau$ has been defined using the modified model $(\widetilde{\PPi},\widetilde{g})$ in order to show that these are indeed periodic distributions. These equations can also be rewritten using the $\PPi$ maps instead of $\widetilde{\PPi}$, yielding a simpler formula, it is performed in the following proposition
\begin{proposition}\label{prop_alternativedef_Ztaunaif}
	Equations \eqref{eq_defZtau_Naif_Alternative} and \eqref{eq_defZtau_T+_alternative} are respectively equivalent to
\begin{equation} \label{eq_defZtau_Naif_Alternative_PPi}
	{ \PPi}^\mcZ\left(\Psi(\tau)\right)(x) = x^p \, \xi_\ell  \,  \prod_{j=1}^{n} \partial^{b_j}K * {\PPi}^\mcZ\left(\Psi(\tau_j)\right),  
\end{equation}
for $\tau = X^p\Xi_\ell\prod_{j=1}^n\mcI_{b_j}(\tau_j) \in T$, and 
\begin{equation} \label{eq_defZtau_T+_alternative_PPi}
	{ \PPi}^\mcZ\left(\Psi(\tau)\right)(x) = x^p \prod_{j=1}^{n} { g}^\mcZ_x\left(D^{b_j}\Psi(\mcI(\tau_j))\right),
\end{equation}
for $\tau = X^p\prod_{j=1}^n\mcI^+_{b_j}(\tau_j) \in T^+$.
\end{proposition}

\begin{proof}
	We suppose that $\mcZ$ is the naive stochastic data. We consider $\tau = X^p\Xi_\ell\prod_{j=1}^n\mcI_{b_j}(\tau_j) \in \mcT$. From \eqref{eq_pitildetopi}, Lemma \ref{lem_commutdownarrowPsi}, and the Leibniz formula for $\downarrow^k$
\begin{align*}
	{ \PPi}^\mcZ\left(\Psi(\tau)\right)(x) &= \sum_{k\geq0 } \frac{x^k}{k!} \widetilde{ \PPi}^\mcZ\Big(\Psi\Big(\downarrow^k \Big(X^p\Xi_\ell\prod_{j=1}^n\mcI_{b_j}(\tau_j)\Big)      \Big)\Big)(x) 
	\\
	&= \!\! \sum_{\substack{k\geq0 \\ (k_j) \in\mcP_{n+1 }(k)}} \!\! x^k \frac{p!}{(p-k_0)!}\prod_{j=0}^{n}\frac{1}{k_j!} \widetilde{ \PPi}^\mcZ\Big(\Psi \Big(X^{p-k_0}\Xi_\ell\prod_{j=1}^n\mcI_{b_j}(\downarrow^{k_j}\tau_j)\Big)   \Big)(x),
\end{align*}
from \eqref{eq_defZtau_Naif_Alternative}, this is equal to 
	$$
	\sum_{\substack{k\geq0 \\ (k_j) \in\mcP_{n+1 }(k)}}  \mathbf{1}_{k_0=p} \, x^{k_0} \xi_\ell \prod_{j=1}^n\sum_{l_j\geq0} \frac{1}{l_j!} \frac{x^{k_j}}{k_j!} (\partial^{b_j}K)^{[l_j]}*\widetilde\PPi^\mcZ\Big(\Psi(\downarrow^{k_j+l_j}\tau_j)\Big)(x).
	$$
	From the definition of $(\partial^{b_j}K)^{[l_j]}$, we have 
	\begin{align*}
	&\sum_{k_j,l_j\geq0} \frac{1}{l_j!} \frac{x^{k_j}}{k_j!} (\partial^{b_j}K)^{[l_j]}*\widetilde\PPi^\mcZ\Big(\Psi(\downarrow^{k_j+l_j}\tau_j)\Big)(x) 
	\\=& 
	\int_{\bfR^{d+1}} \sum_{k_j,l_j\geq0}(\partial^{b_j}K)(x-z) \frac{(z-x)^{l_j}}{l_j!} \frac{x^{k_j}}{k_j!} \widetilde\PPi^\mcZ\left(\Psi(\downarrow^{k_j+l_j}\tau_j)\right)(z) \ddd z
	\\=&
	\sum_{q_j\geq0}(\partial^{b_j}K) * \Big( \frac{z^{q_j}}{q_j!} \widetilde\PPi^\mcZ\left(\Psi(\downarrow^{q_j}\tau_j)\right)(z)\Big) 
		\\=&
(\partial^{b_j}K) * \Big( \PPi^\mcZ\left(\Psi(\tau_j)\right)(z)\Big) ,
	\end{align*}
	from where \eqref{eq_defZtau_Naif_Alternative} follows. The reciprocal implication is obtained by reverting the computations.
Likewise for $\tau = X^p\prod_{j=1}^n\mcI^+_{b_j}(\tau_j) \in T^+$, from \eqref{eq_pitildetopi} and Lemmas \ref{lem_DnSnpropenvrac} and \ref{lem_commutdownarrowPsi}
	\begin{align*}
		{ \PPi}^\mcZ\left(\Psi(\tau)\right)(x) &= \sum_{k\geq0 } \frac{x^k}{k!} \widetilde{ \PPi}^\mcZ\Big(\Psi\Big(\downarrow^k \Big(X^p\prod_{j=1}^n\mcI^+_{b_j}(\tau_j)\Big)      \Big)\Big)(x) 
		\\
		&= \sum_{\substack{k\geq0 \\ (k_j) \in\mcP_{n+1 }(k)}} x^k \frac{p!}{(p-k_0)!}\prod_{j=1}^{n}\frac{1}{k_j!} \widetilde{ \PPi}^\mcZ\Big(\Psi \Big(X^{p-k_0}\prod_{j=1}^n\mcI_{b_j}^+(\downarrow^{k_j}\tau_j)\Big)   \Big)(x),
	\\
	&=\sum_{\substack{k\geq0 \\ (k_j) \in\mcP_{n+1 }(k)}}  \mathbf{1}_{k_0=p} x^{k_0}  \prod_{j=1}^n \frac{x^{k_j}}{k_j!} \widetilde g^\mcZ\Big(D^{b_j}\Psi(\mcI(\downarrow^{k_j}\tau_j))\Big)(x)
	\\
		&= x^{p}  \prod_{j=1}^n  g^\mcZ\Big(D^{b_j}\Psi(\mcI(\tau_j))\Big)(x).
	\end{align*}
Hence \eqref{eq_defZtau_T+_alternative_PPi}, the reciprocal implication is also obtained by reverting the computations.
\end{proof}

	

	\subsection{General admissible stochastic data and their algebraic properties} 
	
	We define the notion of admissible stochastic data, by retaining some of the algebraic properties of the naive collection that we will need in Section \ref{section_ParacAnsatz}.

	\begin{definition} \label{admissible}
		An uplet $(Z_\tau)_{\tau\in\mcT}$ is said to be admissible if it satisfies \eqref{eq_defZtau_T+} and 
			\begin{equation}\label{eq_ZtauAdmissible}
				{\PPi}^\mcZ\left( \Psi(\mcI(\tau))\right) =
				K * { \PPi}^\mcZ\left( \Psi(\tau)\right) .
			\end{equation}
	\end{definition}
	Equation \eqref{eq_ZtauAdmissible} can be rewritten using the modified model  $\widetilde\PPi$ map instead of  ${\PPi}$ in the following way
	\begin{equation}\label{eq_ZtauAdmissible_tilde}
	 \widetilde{ \PPi}^\mcZ\left( \Psi(\mcI(\tau))\right) = \sum_{l\geq0} \frac{1}{l!}
	K^{[l]} * \widetilde{ \PPi}^\mcZ\left( \Psi(\downarrow^l\tau)\right),
\end{equation}
	with $K^{[l]}(x) = (-x)^lK(x)$. That is
	\begin{equation}
		Z_{\mcI(\tau)} = \widehat{Y}^\mcZ_{\Psi(\tau)},
	\end{equation}
	 where the $\widehat{Y}$ distributions are defined in Subsection \ref{subsect_integration_Y}. Proposition \ref{prop_Ztau=YPsitau_integration} states and proves this equivalence. 
	 
	  The following proposition, analogous to \cite[Theorem 21]{BailleulHoshinoRS1}, enables the construction of admissible stochastic data $\mcZ$ when we are given the ill-defined fields $Z_\tau$, that is those for which $\tau$ has negative degree. 
		\begin{theorem} \label{prop_extensionthm_Ztau}
		Suppose we are given a collection
			$$
		\mcZ^- = \big( Z_\tau \big)_{\tau\in\mcT,\, |\tau|<0}\in \prod_{\tau\in\mcT_{<0}}C^{|\tau|}.
		$$
		There exists a unique admissible stochastic data $\mcZ=(Z_\tau)_{\tau\in\mcT}$ extending $\mcZ^-$.
	\end{theorem}
	Theorem \ref{prop_extensionthm_Ztau} is proved at the end of Subsection \ref{subsect_extansionRSwORDS}.
	In the rest of this subsection, we set  an admissible stochastic data $\mcZ$. We prove some properties of such collection $\mcZ$.
	
	\begin{lemma} \label{lem_PsiTau_multiplicatif}
		For any tree  $\tau=X^p \prod_{j=1}^n \mcI^+_{b_j}(\tau_j) \in T^+$,
		\begin{align*}
		\widetilde{\PPi}^\mcZ\left(\Psi(\tau)\right) & =  \mathbf{1}_{p=0} \, 	\prod_{j=1}^{n} \widetilde{\PPi}^\mcZ\left(\Psi(\mcI^+_{b_j}(\tau_j))\right),  
		\quad\text{and}\quad
		{\PPi}^\mcZ\left(\Psi(\tau)\right) =  x^p  	\prod_{j=1}^{n} {\PPi}^\mcZ\left(\Psi(\mcI^+_{b_j}(\tau_j))\right). 
		\end{align*}
	\end{lemma}
	\begin{proof}
		For any $j\in\{1,...,n\}$ and $\mcI_{b_j}^+(\tau_j)\in T^+$, from \eqref{eq_defZtau_T+_alternative}, we have 
		\begin{align*}
			\widetilde{ \PPi}^\mcZ\left(\Psi(\mcI_{b_j}^+(\tau))\right) & = \widetilde{ g}^\mcZ\left(D^{b_j}\Psi(\mcI(\tau_j))\right),
		\end{align*}
		Then, substituting this equality in \eqref{eq_defZtau_T+_alternative} applied to $\tau$ gives the first identity. The second one is obtained by the same computation using \eqref{eq_defZtau_T+_alternative_PPi} instead of \eqref{eq_defZtau_T+_alternative}.
		

	\end{proof}
	
	A direct consequence of Lemma \ref{lem_PsiTau_multiplicatif} is the following multiplicativity property of the admissible collection $\mcZ$.
	
	\begin{proposition} \label{prop_mult_Ztau_+}
	For any admissible stochastic data $\mcZ$, for any trees $\tau_1,\dots\tau_N\in\mcT^+$
	\begin{equation}
		\widetilde{\PPi}^\mcZ\Big(\prod_{j=1}^{N} \Psi(\tau_j) \Big)= \widetilde{\PPi}^\mcZ\Big(\Psi\Big( \prod_{j=1}^{N}\tau_j \Big) \Big). 
	\end{equation}
\end{proposition}

\begin{proof}
We let $\tau_j=X^{p_j}\prod_{i=1}^{n_j}\mcI^+_{b_{i,j}}(\tau_{i,j})$ for all $1\leq j\leq N$.
Then the multiplicativity of $\widetilde{\PPi}$ and Lemma \ref{lem_PsiTau_multiplicatif} give
\begin{align*}
	\widetilde{\PPi}^\mcZ\Big(\prod_{j=1}^{N} \Psi(\tau_j) \Big) 
	&= \prod_{j=1}^N \mathbf{1}_{p_j=0}\prod_{i=1}^{n_j} \widetilde{\PPi}^\mcZ\big(\Psi(\mcI^+_{b_{i,j}}(\tau_{i,j}))\big) 
	\\
	&=\widetilde{\PPi}^\mcZ\Big(\Psi\Big(\prod_{j=1}^N X^{p_j}\prod_{i=1}^{n_j} \mcI^+_{b_{i,j}}(\tau_{i,j})) \Big)\Big)
	=
	 \widetilde{\PPi}^\mcZ\Big(\Psi\Big( \prod_{j=1}^{N}\tau_j \Big) \Big).
\end{align*}

\end{proof}

Lemma \ref{lem_DeltaPsitau} shows that the $\Psi$ is almost a Hopf morphism, only a  commutation relation between $\Psi$ and the derivation $D^n$ is missing. The following Proposition gives this commutation relation at the level of the model $(\PPi^\mcZ,g^\mcZ)$, which will be crucial in the computations involving the data $\mcZ$.

\begin{proposition} \label{Prop_CommutDerivativePsi}
	We suppose that the admissible collection $\mcZ$ is smooth.
	\begin{enumerate}
		\item \label{item_Dn1}  For any $\tau\in\mcT\sqcup\mcT^+$
		\begin{equation} \label{cor_commut_DnPsi}
			\sum_{\substack{\sigma\prec\tau \\ q\geq 0}} \widetilde{g}^\mcZ\big(  D^q\Psi(\tau/\sigma)  \big) \, \Psi_q(\sigma) = \sum_{\mu\prec\tau} \widetilde{g}^\mcZ\big(  \Psi(\tau/\mu)  \big) \, \Psi(\mu).
		\end{equation}
		and
		\begin{equation} \label{cor_commut_DnPsi_withoutilde}
			\sum_{\substack{\sigma\prec\tau \\ q\geq 0}} {g}^\mcZ\big(  D^q\Psi(\tau/\sigma)  \big) \, \Psi_q(\sigma) = \sum_{\mu\prec\tau} {g}^\mcZ\big(  \Psi(\tau/\mu)  \big) \, \Psi(\mu).
		\end{equation}
		
		\item \label{item_Dn2} 
	 For any tree $\tau \in T^+$ and $n\in\bfN^{d+1}$ such that $|n|<|\tau|$,	
		\begin{equation}\label{eq_gDntau}\begin{split}
				\widetilde{ g}^\mcZ\left(D^n \Psi(\tau)  \right) =  \widetilde{ \PPi }^\mcZ\left( \Psi(D^n\tau)  \right), \qquad \text{and} \qquad  { g}^\mcZ\left(D^n \Psi(\tau)  \right) =  { \PPi }^\mcZ\left( \Psi(D^n\tau)  \right)
		\end{split}\end{equation}

\item \label{item_Dn3}   For any $x,y\in \bfR\times\bfT^d$ and $\tau=\mcI_{b}^+(\sigma)\in\mcT^+$
\begin{equation}\label{eq_gammaDn}
\gamma_{yx}^\mcZ\big(D^{b}\Psi(\mcI(\sigma))\big) = \Pi_x^\mcZ\big(\Psi\big(  \mcI_{b}^+(\sigma)\big)\big)(y),
\end{equation}
and for non-trivial $\tau_1,\tau_2\in\mcT^+$
\begin{equation}
	\Pi_x^\mcZ\big(\Psi(\tau_1\tau_2)\big) = \Pi_x^\mcZ\big(\Psi(\tau_1)\big)\Pi_x^\mcZ\big(\Psi(\tau_2)\big)
\end{equation}
\end{enumerate}
\end{proposition}

\begin{proof}
We prove all identities in the same induction on the number of nodes of the tree $\tau$.	
	We first prove the identities in \ref{item_Dn1}. From \eqref{cor_commut_DnPsi} and Lemma \ref{lem_SnDnreindex}
\begin{align*}
	\sum_{\substack{\sigma\prec\tau \\ q\geq 0}} \widetilde{g}^\mcZ\big(  D^q\Psi(\tau/\sigma)  \big) \, \Psi_q(\sigma) &
	=\sum_{\substack{\sigma\prec\tau \\ q\geq 0}} \widetilde{g}^\mcZ\big(  \Psi(D^q(\tau/\sigma))  \big) \, \Psi_q(\sigma)
	\\
	&= \sum_{\substack{\mu\prec\tau \\ q\geq 0}} \widetilde{g}^\mcZ\big(\Psi(\tau/\mu)  \big) \, \Psi_q(\downarrow^q\mu).
\end{align*}
We obtain the result by using identity $\Psi(\sigma)= \sum_{q\geq0}\Psi_q(\downarrow^q\sigma)$.
We now prove the first identity in \eqref{eq_gDntau}, the proof is based on equation \eqref{eq_DnPsitau_bb} below, which enables us to run an induction.
From \eqref{eq_gDnTau_general} and Lemma \ref{lem_DeltaPsitau} and induction hypothesis  
\begin{align} 	
	\widetilde g^\mcZ(D^n\Psi(\tau)) &= \sum_{k\leq n}\binom{n}{k}\partial^{n-k} \widetilde\PPi^\mcZ(\Psi(\downarrow^k\tau)) - \sum_{\sigma\prec\tau} \widetilde g^\mcZ( D^q \Psi(\tau/\sigma)  ) \partial^n_y\Big( \Pi_x^\mcZ(\Psi_q(\sigma)) \Big)_{|y=x} 
\end{align}
then from \eqref{cor_commut_DnPsi}
\begin{equation}\label{eq_DnPsitau_bb}
	\widetilde g^\mcZ(D^n\Psi(\tau))
	= \sum_{k\leq n}\binom{n}{k}\partial^{n-k} \widetilde\PPi^\mcZ(\Psi(\downarrow^k\tau)) 
	-
	\sum_{\substack{\mu\prec \tau  }}  \widetilde{g}^\mcZ ( \Psi(\tau/\mu)  ) \partial^n_y\Big( \Pi_x^\mcZ(\Psi(\mu))  \Big)_{|y=x}.  
\end{equation}
We first take $\tau = \mcI_b^+(\sigma)$. Equation \eqref{eq_DnPsitau_bb} writes as
\begin{align*}
&	\widetilde{g}^\mcZ(D^n\Psi(\tau))\\  &= \sum_{k\leq n}\binom{n}{k}\partial^{n-k} \widetilde\PPi^\mcZ(\Psi(\mcI_b^+(\downarrow^k\sigma))) 
	-
	\sum_{\mu\prec\sigma} \widetilde g^\mcZ( \Psi(\sigma/\mu)  ) \partial^n_y\Big( \Pi_x^\mcZ(\Psi(\mcI_b^+\mu)) \Big)_{|y=x}.  
\end{align*}
On the other hand, from Lemma \ref{lem_DnSnpropenvrac} and \eqref{eq_gDnTau_general} \eqref{cor_commut_DnPsi} 
\begin{align*}
&	\widetilde{g}^\mcZ\big(D^nD^b\Psi(\mcI(\sigma)))
	\\ &= \sum_{k\leq n}\binom{n}{k}\partial^{n-k} \widetilde g^\mcZ(D^b\Psi(\mcI(\downarrow^k\sigma)))
	-
	\sum_{\substack{\mu\prec\sigma \\ q\geq 0}} \widetilde g^\mcZ( D^q\Psi(\mu/\sigma)  ) \partial^n_y\Big( \Pi_x^\mcZ(D^b\Psi_q(\mu)) \Big)_{|y=x}
	\\
	&= \sum_{k\leq n}\binom{n}{k}\partial^{n-k} \widetilde g^\mcZ(D^b\Psi(\mcI(\downarrow^k\sigma)))
	-
	\sum_{\substack{\mu\prec\sigma \\ q\geq 0}} \widetilde g^\mcZ( \Psi(\mu/\sigma)  ) \partial^n_y\Big( \Pi_x^\mcZ(D^b\Psi(\mcI^+(\mu))) \Big)_{|y=x}.
\end{align*}
From \eqref{eq_defZtau_T+_alternative}, one has
$$
\widetilde g^\mcZ(D^b\Psi(\mcI(\downarrow^k\!\sigma))) 
=
\PPi^\mcZ\Psi(\mcI_b^+(\downarrow^k\!\sigma)).
$$
By induction hypothesis
$$
\Pi_x^\mcZ(D^b\Psi(\mcI(\mu))) = \Pi_x^\mcZ(\Psi(\mcI_b^+(\mu))).
$$
Then from \eqref{eq_defZtau_T+_alternative} 
\begin{align*}
	\widetilde{g}^\mcZ(D^n\Psi(\tau)) & = 	\widetilde{g}^\mcZ(D^nD^b\Psi(\mcI(\sigma)))  = \widetilde{g}^\mcZ(D^{n+b}\Psi(\mcI(\sigma)))
\\ &	=\widetilde{\PPi}^\mcZ(\Psi(\mcI_n^+(\sigma)) ) = \widetilde{\PPi}^\mcZ (\Psi(D^n\tau)).
\end{align*}
Hence the result for $\tau=\mcI_b^+(\sigma)$.

We now prove \eqref{eq_gDntau} for a product $\tau= \tau_1\tau_2$ with non trivial $\tau_1,\tau_2\in\mcT^+$. First from \eqref{eq_DnPsitau_bb} 
\begin{align*}
	\widetilde{g}^\mcZ(D^n\Psi(\tau_1\tau_2)) &= \sum_{\substack{k\leq n \\ k=k_1+k_2}}\binom{n}{k}\binom{k}{k_1}\partial^{n-k}  \widetilde\PPi^\mcZ(\Psi(\downarrow^{k_1}\!\tau_1\downarrow^{k_2}\!\tau_2)) 
	\\
	&\qquad- \sum_{\sigma\prec(\tau_1\tau_2)} \widetilde g^\mcZ( \Psi((\tau_1\tau_2)/\sigma)  ) \partial^n_y\Big( \Pi_x^\mcZ(\Psi(\sigma)) \Big)_{|y=x}.  
\end{align*}
On the other hand, \eqref{eq_gDnTau_general} gives 
\begin{align*}
	& \widetilde{g}^\mcZ\big(D^n(\Psi(\tau_1)\Psi(\tau_2))\big) \\ &=  \sum_{\substack{k\leq n \\ k=k_1+k_2}}\binom{n}{k}\binom{k}{k_1}\partial^{n-k}  \widetilde\PPi^\mcZ(\Psi(\downarrow^{k_1}\!\tau_1)\Psi(\downarrow^{k_2}\!\tau_2)) 
	\\
	&\qquad- \sum_{\substack{\sigma_1\sigma_2\prec\tau_1\tau_2 \\ q_1,q_2\geq 0}} \widetilde g^\mcZ( D^{q_1}\Psi(\tau_1/\sigma_1)D^{q_2}\Psi(\tau_2/\sigma_2)  ) \partial^n_y\Big( \Pi_x^\mcZ\big(\Psi_{q_1}(\sigma_1)\Psi_{q_2}(\sigma_2)\big)(y) \Big)_{|y=x}. 
\end{align*}
From \eqref{cor_commut_DnPsi}, we get 
\begin{align*}
	\widetilde{g}^\mcZ\big(D^n(\Psi(\tau_1)\Psi(\tau_2))\big) &=  \sum_{\substack{k\leq n \\ k=k_1+k_2}}\binom{n}{k}\binom{k}{k_1}\partial^{n-k}  \widetilde\PPi^\mcZ(\Psi(\downarrow^{k_1}\!\tau_1)\Psi(\downarrow^{k_2}\!\tau_2)) 
\\
&\qquad- \sum_{\substack{\mu_1\mu_2\prec\tau_1\tau_2 }} \widetilde g^\mcZ( \Psi(\tau_1/\mu_1)\Psi(\tau_2/\mu_2)  ) \partial^n_y\Big( \Pi_x^\mcZ\big(\Psi(\mu_1)\Psi(\mu_2)\big)(y) \Big)_{|y=x}. 
\end{align*}
By the induction hypothesis, we have
\begin{align*}
 \Pi_x^\mcZ\big(\Psi(\mu_1)\Psi(\mu_2)\big)=\Pi_x^\mcZ\big(\Psi(\mu_1\mu_2)\big)
\end{align*}
and from Proposition \ref{prop_mult_Ztau_+}
$$
\widetilde g^\mcZ\big( \Psi(\tau_1/\mu_1)\Psi(\tau_2/\mu_2)  \big)
=
\widetilde g^\mcZ\big( \Psi(\tau_1/\mu_1\,\tau_2/\mu_2)  \big)
$$
and
$$
\widetilde\PPi^\mcZ(\Psi(\downarrow^{k_1}\!\tau_1)\Psi(\downarrow^{k_2}\!\tau_2))
=
\widetilde\PPi^\mcZ(\Psi(\downarrow^{k_1}\!\tau_1 \, \downarrow^{k_2}\!\tau_2)).
$$
Hence 
\begin{align*}
	\widetilde{g}^\mcZ\big(D^n\Psi(\tau)\big) &= \widetilde{g}^\mcZ\big(D^n(\Psi(\tau_1)\Psi(\tau_2))\big)
	=
	\sum_{n=n_1+n_2}\binom{n}{n_1}\widetilde{g}^\mcZ\big(D^{n_1}\Psi(\tau_1)D^{n_2}\Psi(\tau_2))\big)
	\\&=
	\sum_{n=n_1+n_2}\binom{n}{n_1}\widetilde{g}^\mcZ\big(D^{n_1}\Psi(\tau_1)\big) \widetilde{g}^\mcZ\big(D^{n_2}\Psi(\tau_2))\big)
	\\
	&=	\sum_{n=n_1+n_2}\binom{n}{n_1}\widetilde{\PPi}^\mcZ\big(\Psi(D^{n_1}\tau_1)\big) \widetilde{\PPi}^\mcZ\big(\Psi(D^{n_2}\tau_2))\big)
	\\&= \sum_{n=n_1+n_2}\binom{n}{n_1}\widetilde{\PPi}^\mcZ\big(\Psi(D^{n_1}\tau_1\, D^{n_2}\tau_2)\big)  = \PPi^\mcZ(\Psi(D^n\tau)).
\end{align*}
	We now prove the second identity in \eqref{eq_gDntau} from the first one. From Lemmas \ref{lem_DnSnpropenvrac},\ref{lem_DeltaPsitau} and \ref{lem_commutdownarrowPsi}, we have
\begin{align*}
	{g}_x^\mcZ( D^n\Psi(\tau))&= \sum_{j\geq 0} \frac{x^j}{j!} \widetilde{g}_x^\mcZ( \downarrow^jD^n\Psi( \tau)) =  \sum_{j\geq 0} \frac{x^j}{j!} \widetilde{g}_x^\mcZ( D^n\Psi( \downarrow^j\tau))
	\\
	&=\sum_{j\geq 0} \frac{x^j}{j!} \widetilde{g}_x^\mcZ( \Psi( D^n\downarrow^j\tau))  = \sum_{j\geq 0} \frac{x^j}{j!} \widetilde{g}_x^\mcZ( \downarrow^j\Psi( D^n\tau)) 
	\\&= {g}_x^\mcZ( \Psi(D^n\tau)).
\end{align*}
We now prove \eqref{eq_gammaDn}.  From Lemma \ref{lem_DeltaPsitau} and \eqref{cor_commut_DnPsi_withoutilde}
\begin{align*}
		\Pi^{\mcZ}_x( \Psi(\tau) ) &= 	\PPi^{\mcZ}( \Psi(\tau) ) - \sum_{\substack{\sigma\prec\tau\\ q\geq0}} g_x^\mcZ(D^q\Psi(\tau/\sigma))) \, 	\Pi^{\mcZ}_x( \Psi_q(\sigma) ) 
		\\
		&\quad-\sum_{n\geq 0} \frac{x^n}{n!} \Pi_x^{\mcZ}(\Psi(\downarrow^n\tau))(x) -  g_x^{\mcZ}(D^n\Psi(\tau)) \frac{(\cdot-x)^n}{n!}
		\\
		&=\PPi^{\mcZ}( \Psi(\tau) ) - \sum_{\substack{\sigma\prec\tau}} g_x^\mcZ\left(\Psi(\tau/\sigma)\right) \, 	\Pi^{\mcZ}_x( \Psi(\sigma) )(x)		\\
		&\quad-\sum_{n\geq 0} \frac{x^n}{n!} \Pi_x^{\mcZ}(\Psi(\downarrow^n\tau)) - g_x^{\mcZ}(\Psi(D^n\tau)) \frac{(\cdot-x)^n}{n!}
		\\
		&= \PPi^{\mcZ}( \Psi(\tau) ) - \sum_{\substack{\sigma<\tau}} g_x^\mcZ\left(\Psi(\tau/\sigma)\right) \, 	\Pi^{\mcZ}_x( \Psi(\sigma) ).
\end{align*}
We first take $\tau=\mcI_b^+(\sigma)$,  we have also from Lemma \ref{lem_DeltaPsitau} and \eqref{cor_commut_DnPsi_withoutilde}
\begin{align*}
	\gamma_{yx}^{\mcZ}( D^b\Psi(\mcI(\sigma)) ) &= 	g_y^{\mcZ}( D^b\Psi(\mcI(\sigma)) ) - \sum_{\substack{\nu\prec\sigma\\ q\geq0}} g_x^\mcZ(D^q\Psi(\sigma/\nu))) \, 	\Pi^{\mcZ}_x( D^b\Psi_q(\mcI(\nu)) )(y) 
	\\
	&\quad-\sum_{n\geq 0} \frac{x^n}{n!} \Pi_x^{\mcZ}\big(D^b\Psi(\downarrow^n\mcI(\sigma))\big)(y) -  g_x^{\mcZ}(D^{n+b}\Psi(\mcI(\sigma))) \frac{(y-x)^n}{n!}
	\\
	&=\PPi^{\mcZ}( \Psi(\mcI_b^+(\sigma) )(y) - \sum_{\substack{\nu\prec\sigma}} g_x^\mcZ\left(\Psi(\sigma/\nu)\right) \, 	\Pi^{\mcZ}_x( D^b\Psi(\mcI(\nu)) )(y)		\\
		&\quad-\sum_{n\geq 0} \frac{x^n}{n!} \Pi_x^{\mcZ}(\Psi(\downarrow^n\mcI_b^+(\sigma)))(y) -  g_x^{\mcZ}(\Psi(\mcI_{n+b}^+(\sigma))) \frac{(y-x)^n}{n!}
	\\
	&= \PPi^{\mcZ}( \Psi(\tau) )(y) - \sum_{\substack{\nu<\sigma}} g_x^\mcZ\left(\Psi(\sigma/\nu)\right) \, 	\Pi^{\mcZ}_x( D^b\Psi(\mcI(\nu)) )(y)\\
	&\quad\sum_{n\geq 0}- g_x^{\mcZ}(\Psi(\mcI_{n+b}^+(\sigma))) \frac{(y-x)^n}{n!}.
\end{align*}
By the induction hypothesis, we have 
$$
\gamma^{\mcZ}_{yx}( D^b\Psi(\mcI(\nu)) = \Pi^{\mcZ}_x( \Psi(\mcI_b^+(\nu))(y).
$$
Hence the result for $\tau=\mcI_b^+(\sigma)$.

We now take a product $\tau=\tau_1\tau_2$ with non-trivial $\tau_1,\tau_2\in\mcT^+$.
We have 
\begin{align*}
	\Pi^{\mcZ}_x( \Psi(\tau) ) 
	& = \PPi^{\mcZ^\circ}( \Psi(\tau) ) - \sum_{\substack{(\sigma_1\sigma_2)<(\tau_1\tau_2)}} g_x^\mcZ\left(\Psi(\tau_1/\sigma_1\,\tau_2/\sigma_2)\right) \, 	\Pi^{\mcZ}_x( \Psi(\sigma_1\sigma_2) ).
\end{align*}
Proposition \ref{prop_mult_Ztau_+} gives 
$$
g_x^\mcZ\big(\Psi(\tau_1/\sigma_1\,\tau_2/\sigma_2)\big) = g_x^\mcZ \big(\Psi(\tau_1/\sigma_1)\big)  g_x^\mcZ  \big(\Psi(\tau_2/\sigma_2)\big). 
$$
The induction hypothesis gives 
$$
\Pi^{\mcZ}_x( \Psi(\sigma_1\sigma_2) ) = \Pi^{\mcZ}_x\big( \Psi(\sigma_1) \big) \Pi^{\mcZ}_x\big( \Psi(\sigma_2) \big).
$$
One has
\begin{align*}
	&\sum_{\sigma_1\sigma_2<\tau_1\tau_2 } g_x^\mcZ \big(\Psi(\tau_1/\sigma_1)\big)  g_x^\mcZ  \big(\Psi(\tau_2/\sigma_2)\big) 	\Pi^{\mcZ}_x( \Psi(\sigma_1) ) \Pi^{\mcZ}_x\big( \Psi(\sigma_2) \big)
	\\
	&= 	\sum_{\substack{\sigma_1<\tau_1\\ \sigma_2<\tau_2 }}  g_x^\mcZ \big(\Psi(\tau_1/\sigma_1)\big)  g_x^\mcZ  \big(\Psi(\tau_2/\sigma_2)\big) 	\Pi^{\mcZ}_x( \Psi(\sigma_1) ) \Pi^{\mcZ}_x\big( \Psi(\sigma_2) \big)
	\\
	&\quad+ \sum_{\sigma_1<\tau_1}   g_x^\mcZ \big(\Psi(\tau_1/\sigma_1)\big)  \Pi^{\mcZ}_x( \Psi(\sigma_1) ) \Pi^{\mcZ}_x\big( \Psi(\tau_2) \big)
	\\
	&\quad+ \sum_{\sigma_2<\tau_2} 
	g_x^\mcZ  \big(\Psi(\tau_2/\sigma_2)\big) 	
	\Pi^{\mcZ}_x\big( \Psi(\tau_1) \big)
	\Pi^{\mcZ}_x( \Psi(\sigma_2) ) 
	\\
	&= 	-\sum_{\substack{\sigma_1<\tau_1\\ \sigma_2<\tau_2 }} g_x^\mcZ \big(\Psi(\tau_1/\sigma_1)\big)  g_x^\mcZ  \big(\Psi(\tau_2/\sigma_2)\big) 	\Pi^{\mcZ}_x( \Psi(\sigma_1) ) \Pi^{\mcZ}_x\big( \Psi(\sigma_2) \big)
	\\
	&\quad+\sum_{\sigma_1<\tau_1}  g_x^\mcZ \big(\Psi(\tau_1/\sigma_1)\big)  \Pi^{\mcZ}_x( \Psi(\sigma_1) ) \PPi^{\mcZ}\big( \Psi(\tau_2) \big)
	\\
	&\quad+ \sum_{\sigma_2<\tau_2}
	g_x^\mcZ  \big(\Psi(\tau_2/\sigma_2)\big) 	
	\PPi^{\mcZ}\big( \Psi(\tau_1) \big)
	\Pi^{\mcZ}_x( \Psi(\sigma_2) ) 
	.
\end{align*}
Then, from \eqref{eq_mult_Ztau}
\begin{align*}
	\Pi^{\mcZ}_x( \Psi(\tau) ) 
	&= \Big(\PPi^{\mcZ}( \Psi(\tau_1) ) - \sum_{\substack{\sigma_1<\tau_1}} g_x^\mcZ\left(\Psi(\tau_1/\sigma_1)\right) \, 	\Pi^{\mcZ}_x( \Psi(\sigma_1) ) \Big)
	\\
	&\quad\times \Big(\PPi^{\mcZ}( \Psi(\tau_2) ) - \sum_{\substack{\sigma_2<\tau_2}} g_x^\mcZ\left(\Psi(\tau_2/\sigma_2)\right) \, 	\Pi^{\mcZ}_x( \Psi(\sigma_2) ) \Big)
	\\
	&= \Pi^{\mcZ}_x( \Psi(\tau_1) ) \Pi^{\mcZ}_x( \Psi(\tau_2) ).
\end{align*}

\end{proof}

\begin{proposition}\label{prop_derivetoile_Ztau_Neg}
	For any tree $\tau\in \mcT\sqcup \mcT^+$ and $|n|>|\tau|$, we have 
	\begin{align*}
		\sum_{\substack{k\geq 0 \\ k=k_1+k_2}} \sum_{\substack{q\geq 0 \\ \tau_q\prec\dots\prec\tau_1\prec \tau }} \frac{1}{k_2!} \binom{n}{k_1}  \partial^{n-k_1}_*\P_{k_2}\left(  Z_{\downarrow^k\tau/\tau_1},\dots, Z_{\tau_q}   \right) = \partial^n_y\Big(\Pi_x^\mcZ(\Psi(\tau))(y)\Big)_{|y=x}.
	\end{align*}
\end{proposition}
\begin{proof}
We prove the identity by induction. From Lemma \ref{lem_derivModifModel}
\begin{align*}
&	\sum_{\substack{k\geq 0 \\ k=k_1+k_2}} \sum_{\tau_q\prec\dots\prec\tau } \frac{1}{k_2!} \binom{n}{k_1}\partial^{n-k_1}\P_{k_2}\Big(  Z_{\downarrow^k\tau/\tau_1},\dots, Z_{\tau_q}   \Big) 
\\ &=
	 \sum_{k_1\leq n}\binom{n}{k_1}\partial^{n-k_1} \widetilde{ \PPi}^\mcZ\big( \Psi(\downarrow^{k_1}\tau) \big)
	 = \sum_{j\geq0} \frac{(-1)^jx^j}{j!} \partial^n \PPi^\mcZ(\Psi(\downarrow^j\tau)).
\end{align*}
From the definition of $\partial^k_*{\P}_\ell$ given in \ref{eqdef_starderivative}, 
\begin{align*}
	&\sum_{\substack{k\geq 0 \\ k=k_1+k_2}} \sum_{\tau_q\prec\dots\prec\tau } \frac{1}{k_2!} \binom{n}{k_1}\Big\{ \partial^{n-k_1}\P_{k_2}\big(  Z_{\downarrow^k\tau/\tau_1},\dots, Z_{\tau_q}   \big) \\
	&\hspace{7.5cm}-\partial^{n-k_1}_*\P_{k_2}\big(  Z_{\downarrow^k\tau/\tau_1},\dots, Z_{\tau_q}   \big)\Big\} 
	\\
	&= \sum_{\substack{\sigma\prec\tau \\ \sigma\prec\sigma_q\prec\dots\prec\tau \\ \nu_s\prec\dots\prec\sigma}} \sum_{(k_i)\in\mcP_3(k)}\sum_{\substack{n-k_1=p_1+p_2 \\ |r|< \delta }} \frac{n!}{k_1!k_3!k_4!p_1!p_2!r!}\partial^{p_1+r}_* \P_{k_3} \big(Z_{\downarrow^k\tau/\sigma_1},\dots,Z_{\sigma_q/\sigma}\big) 
	\\
	&\hspace{8cm}\times \partial^{p_2}_*\P_{k_4+r}\big(Z_{\sigma/\nu_1},\dots, Z_{\nu_s}   \big),
\end{align*}
with $\delta=\min\{ |\tau/\sigma| - |p_1|-|k_1|-|k_4|, \enskip |p_2|-|\sigma|-|k_4|\}$. Re-indexing the sum of the right hand side gives
\begin{align*}
	&\sum_{\substack{\sigma\prec\tau \\ \sigma\prec\sigma_q\prec\dots\prec\tau \\ \nu_s\prec\dots\prec\sigma}} \sum_{|q|<|\tau/\sigma |} \sum_{\substack{k\geq 0 \\ k=k_3+k_{14}}} 
	\partial^{ q-k_{14}}_* \P_{k_3} \big(Z_{\downarrow^k\tau/\sigma_1},\dots,Z_{\sigma_q/\sigma}\big) 
	\\
	&\hspace{0.5cm}\times \sum_{\substack{q=q_1+q_2 \\ k_{14}=k_1+k_4 }} \frac{n!}{k_1!k_3!k_4!(q_2-k_4)!(q_1-k_1)!(n-q_1)! } \partial^{n-q_1}_*\P_{q_2}\big(Z_{\sigma/\nu_1},\dots, Z_{\nu_s}   \big).
\end{align*}
From the Chu-Vandermonde identity
$$
\sum_{\substack{ k_{14} = k_1+k_4      }} \frac{1}{k_1!(q_1-k_1)} \frac{1}{k_4! (q_2-k_4)!}= \frac{1}{ q_1!q_2! } \binom{q}{k_{14}}.
$$
From the induction hypothesis and Proposition \ref{Prop_CommutDerivativePsi}, this is equal to 
\begin{align*}
	&\sum_{\substack{\sigma\prec\tau \\ |q|<|\tau/\sigma| \\  |\sigma|-|n|+|q|<0 }}\sum_{\nu_s\prec\dots\prec\sigma}\sum_{q=q_1+q_2} \frac{1}{q_2!}\binom{n}{q_1}  \widetilde{ g}\big(D^{q}\Psi(\tau/\sigma)\big) \, \partial^{n-q_1}_*\P_{q_2}\big(Z_{\sigma/\nu_1},\dots, Z_{\nu_s}   \big)
	\\
	&\qquad= \sum_{\substack{\sigma\prec\tau \\  |\sigma|-|n|<0 }}\sum_{\nu_s\prec\dots\prec\sigma}\sum_{q=q_1+q_2} \frac{1}{q_2!}\binom{n}{q_1}  \widetilde{ g}^\mcZ\big(\Psi(\tau/\sigma)\big) \,   \partial^{n-q_1}_*\P_{q_2}\big(Z_{\downarrow^{q}\sigma/\nu_1},\dots, Z_{\nu_s}   \big)
	\\
	&\qquad =
	\sum_{\substack{\sigma\prec\tau, \, \sigma\notin T_X\\ |\sigma|-|n|<0}} \widetilde{ g}_x(\tau/\sigma) \, \partial^n_y\big( { \Pi}_x\sigma(y)\big)_{|y=x}.
\end{align*}
Hence the proposition.
\end{proof}

Proposition \ref{Prop_CommutDerivativePsi} admits the following corollary, which makes the link between admissible stochastic data and admissible models on decorated trees.  
	\begin{corollary} \label{PPi_Psi}
	Setting for any $\tau\in (T\sqcup T^+)\backslash T_X$
	$$
	{ \PPi}(\tau) = { \PPi}^\mcZ (\Psi(\tau)) \quad\text{and}\quad { g }(\tau) = { \PPi }^\mcZ(\Psi(\tau)), 
	$$
	and $\PPi(X^k)(x)=x^k, \, g_x(X^k)=x^k$ for any $k\in\bfN^{d+1}$, defines an admissible pre-model ${ (\PPi,g)}$  on the regularity structure of decorated trees.
\end{corollary}

\begin{proof}
 Proposition \ref{prop_mult_Ztau_+} gives the multiplicativity of $g$, as well as its admissibility with respect to polynomials. And admissibility of the map $\PPi$ is clear from \eqref{eq_ZtauAdmissible}.
 We have to check the analytical bounds on the increments of the maps $(\PPi,g)$. We prove first by induction for any $\tau\in T^+$  that $$g^{-1}(D^q\tau) = (g^\mcZ)^{-1}(D^q\Psi(\tau)).$$   From Lemmas \ref{lem_DeltaPsitau}, \ref{lem_SnDnreindex} and \eqref{cor_commut_DnPsi_withoutilde}, one has   
	\begin{align*}
	(g^\mcZ)^{-1}(D^n\Psi(\tau) ) &=	 -g^\mcZ(D^n\Psi(\tau) ) - \sum_{\substack{\sigma\prec\tau\\ q\geq 0}} g^\mcZ(D^q\Psi(\tau/\sigma)) \, (g^\mcZ)^{-1}(D^n\Psi_q(\sigma))
	\\ & 	\qquad-\sum_{q\geq 0} g^\mcZ(D^q\Psi(\tau)) g\Big(D^n\frac{X^q}{q!}\Big)
		\\
		&= -\PPi^\mcZ(\Psi(D^n\tau) ) - \sum_{\substack{\sigma\prec\tau}} \PPi^\mcZ(\Psi(\tau/\sigma)) \, (g^\mcZ)^{-1}(D^n\Psi(\sigma))
	\\ & 	\qquad-\sum_{q\geq 0} \PPi^\mcZ(\Psi(D^q\tau)) g\Big(D^n\frac{X^q}{q!}\Big)
		\\
		&= -g(D^n\tau ) - \sum_{\substack{\sigma\prec\tau}} g(\tau/\sigma) \, g^{-1}(D^n\sigma)
		-\sum_{q\geq 0} g(D^q\tau) g\Big(D^n\frac{X^q}{q!}\Big)
		\\
		&= 	g^{-1}(D^n\tau).
	\end{align*}
From Lemma \ref{lem_DeltaPsitau} and \ref{lem_SnDnreindex}, the increment $\Pi_x^\mcZ( \Psi(\tau)   )$ is then equal to
	\begin{align*}
		&\Pi_x^\mcZ( \Psi(\tau)   ) =  (\PPi^\mcZ\otimes (g^\mcZ_x)^{-1})\Delta \Psi(\tau)
		\\
		&=\PPi^\mcZ(\Psi(\tau)) + \sum_{\substack{\sigma\prec\tau \\ q\geq0}} (g^\mcZ_x)^{-1}(D^q\Psi(\tau/\sigma))  \,\, \PPi^\mcZ( \Psi_q(\sigma) ) +\sum_{n\geq0} (g_x^\mcZ)^{-1} (D^n\Psi(\tau)) \frac{x^n}{n!}
		\\
		&\quad + \sum_{n\geq 0} \frac{(-x)^n}{n!} \PPi^\mcZ\big(\Psi(\downarrow^n\tau)\big)
		\\
		&= \PPi(\tau) +  \sum_{\substack{\sigma\prec\tau \\ q\geq0}} g_x^{-1}(D^q(\tau/\sigma))  \,\, \PPi^\mcZ( \Psi_q(\sigma) ) + \sum_{n\geq0} g_x^{-1} (D^n\tau) \frac{x^n}{n!} 
		\\
		&\quad+ \sum_{n\geq 0} \frac{(-x)^n}{n!} \PPi^\mcZ\big(\Psi(\downarrow^n\tau)\big)
		\\
		&= \PPi(\tau) +  \sum_{\sigma\prec\tau} g_x^{-1}(\tau/\sigma)  \,\, \PPi^\mcZ( \Psi(\sigma) ) +\sum_{n\geq0} g_x^{-1} (D^n\tau) \frac{x^n}{n!} + \sum_{n\geq 0} \frac{(-x)^n}{n!} \PPi^\mcZ\big(\Psi(\downarrow^n\tau)\big)
		\\
		&=\Pi_x\tau.
	\end{align*}
	For $\tau\in T^+$, it follows from the same computations that $\gamma_{yx}(\tau) = \Pi_x^\mcZ(\Psi(\tau))(y) $. We then see that the increments $\Pi_x\tau$ and $\gamma_{yx}(\tau)$ have the adequate analytical estimates to ensure that $(\PPi,g)$ defines a model. 
\end{proof}

\subsection{Regularity estimates on admissible stochastic data} \label{subsect_extansionRSwORDS}

In this subsection, we give in Propositions \ref{prop_Ztau=YPsitau_integration} and \ref{prop_Ztau=YPsitau_mult} an alternative characterisation of the admissibility of a stochastic data using the $Y$ and $\widehat{Y}$ operators from Section \ref{section_fromRStoPS}. It will useful for proving Theorem \ref{thm_regularityZtau}, as we will gain regularity estimates from the results of Section \ref{section_fromRStoPS}.

\begin{proposition}\label{prop_Ztau=YPsitau_integration}
		For any $\tau\in\mcT$ and admissible stochastic data $\mcZ$,
	$$
	Z_{\mcI(\tau)} = \widehat{Y}_{\Psi(\tau)}^\mcZ.
	$$
\end{proposition}

\begin{proof}
	The proof is performed by induction on $\tau$.
	From the definition of $\widehat{Y}$ given in Definition \ref{def_YItau}, from Lemma \ref{lem_DeltaPsitau}, \ref{lem_commutdownarrowPsi} and  \ref{lem_DnSnpropenvrac}
	\begin{align*}
		\widehat{Y}_{\Psi(\tau)}^\mcZ &= \sum_{n\geq0} \frac{1}{n!}K^{[n]}* \widetilde{\PPi}^\mcZ(\downarrow^n \! \Psi(\tau)) - \sum_{\substack{\sigma\prec\tau \\ q\geq0}} \sum_{k\geq 0} \frac{1}{k!}\P_k \Big( \widetilde{g}^\mcZ\big(\downarrow^k\! D^q\Psi( \tau/\sigma)\big), \, \widehat{Y}_{\Psi_q(\sigma)}  \Big)
		\\
		&= \sum_{n\geq0} \frac{1}{n!}K^{[n]}* \widetilde{\PPi}^\mcZ( \Psi(\downarrow^n \! \tau)) - \sum_{\substack{\sigma\prec\tau \\ q\geq0}} \sum_{k\geq 0} \frac{1}{k!}\P_k \Big( \widetilde{g}^\mcZ\big( D^q\Psi( \downarrow^k\!\tau/\sigma)\big), \, \widehat{Y}_{\Psi_q(\sigma)}  \Big).
		\end{align*}
From \eqref{cor_commut_DnPsi} we obtain 
		\begin{align*}
			\widehat{Y}_{\Psi(\tau)}^\mcZ&= \sum_{n\geq0} \frac{1}{n!}K^{[n]} * \widetilde{\PPi}^\mcZ(\Psi(\downarrow^n \! \tau)) - \sum_{\substack{\sigma\prec\tau}} \sum_{k\geq 0} \frac{1}{k!}\P_k \Big( \widetilde{g}^\mcZ\big(\Psi(\downarrow^k \! \tau/\sigma)\big), \, \widehat{Y}_{\Psi(\sigma)}  \Big).
	\end{align*}
From the induction hypothesis and from the admissibility of $\mcZ$ defined in \eqref{eq_ZtauAdmissible_tilde}, we get indeed 
	\begin{align*}
	\widehat{Y}_{\Psi(\tau)}^\mcZ&= \sum_{n\geq0} \frac{1}{n!}K^{[n]} * \widetilde{\PPi}^\mcZ(\Psi(\downarrow^n \! \tau)) - \widetilde{\PPi}^\mcZ\big(\overline{\Psi}(\mcI(\tau))\big) = Z_{\mcI(\tau)}.
\end{align*}
\end{proof}

	\begin{proposition}\label{prop_Ztau=YPsitau_mult}
		For $\tau = X^p\prod_{j=1}^n\mcI^+_{b_j}(\tau_j) \in T^+$, one has
		$$
		Z_\tau = Y_{X^p\prod_{j=1}^nD^{b_j}\Psi(\mcI(\tau_j))}^\mcZ.
		$$
	\end{proposition}

	\begin{proof}
		We prove the result by induction on the size of the tree $\tau=X^p\prod_{j=1}^n\mcI^+_{b_j}(\tau_j)$.
		From Lemma \ref{lem_DeltaPsitau} and \eqref{eq_commutDerivationDelta} and \eqref{eq_CoproductOfMcItau}
		\begin{align*}
		\Delta D^b\Psi(\mcI(\nu)) = &\sum_{\substack{\sigma\prec\nu \\ q\geq 0}} D^b\Psi_q(\mcI(\sigma))\otimes D^q\Psi(\nu/\sigma) +  \sum_{q\geq0} D^b\frac{X^q}{q!} \otimes  D^q\Psi(\mcI(\nu))
		\\
		&\quad+  \sum_{q\geq0} D^b\Psi(\downarrow^q\mcI(\nu)) \otimes \frac{X^q}{q!}.
	\end{align*}
		Letting $\mcP^{ad}$ be the set of partitions $(A,B,C)$ of $[\![1,n]\!]$ such that  $A\sqcup B$ and $A\sqcup C$ are not empty, we have
		\begin{align*}
	&	(\pix\otimes\pix)\Delta\Big(X^p\prod_{j=1}^nD^{b_j}\Psi(\mcI(\tau_j))\Big) 
	\\ &= \sum_{(A,B,C)\in\mcP^{ad}}\sum_{\substack{p=p_1+p_2 \\ (q_j)_{1\leq j\leq n}  } } \binom{p}{p_1}
	X^{p_1} \prod_{j\in A} D^{b_j}\Psi_{q_j}(\mcI(\sigma_j)) \prod_{j\in B} D^{b_j}\frac{X^{q_j}}{q_j!} \prod_{j\in C} D^{b_j}\Psi(\downarrow^{q_j}\!\mcI(\tau_j))
	\\
	&\qquad \otimes 
	 X^{p_2}\prod_{j\in A}\sum_{\sigma_j\prec\tau_j}D^{q_j}\Psi(\tau_j/\sigma_j) \prod_{j\in B}D^{q_j}\Psi(\mcI(\tau_j)) \prod_{j\in C}\frac{X^{q_j}}{q_j!}. 
		\end{align*}
	We have from multiplicativity of $\widetilde{g}$, from Lemma \ref{lem_DnSnpropenvrac} and Proposition \ref{Prop_CommutDerivativePsi}
		\begin{align*}
			&\widetilde{g}^\mcZ\Big( \downarrow^k\Big( X^{p_2} \prod_{j\in A}\sum_{\sigma_j\prec\tau_j}D^{q_j}\Psi(\tau_j/\sigma_j) \prod_{j\in B}D^{q_j}\Psi(\mcI(\tau_j)) \prod_{j\in C}\frac{X^{q_j}}{q_j!}            \Big)\Big)
			\\
			&= \widetilde{g}^\mcZ\Big( \downarrow^k\Big( X^{p_2} \prod_{j\in A}\sum_{\sigma_j\prec\tau_j}\Psi\big(D^{q_j}(\tau_j/\sigma_j)\big) \prod_{j\in B}\Psi(\mcI^+_{q_j}\tau_j) \prod_{j\in C}\frac{X^{q_j}}{q_j!}            \Big)\Big).
		\end{align*}
		Writing Definition \ref{def_Ytau}, and using Lemma \ref{lem_SnDnreindex} gives 
		\begin{align*}
		&	 Y_{X^p\prod_{j=1}^nD^{b_j}\Psi(\mcI(\tau_j))}^\mcZ - \widetilde{g}^\mcZ\Big( X^p\prod_{j=1}^nD^{b_j}\Psi(\mcI(\tau_j)) \Big) 
			 \\
			 &=- \sum_{(A,B,C)\in\mcP^{ad}}\sum_{\substack{(q_j)_j\\k\geq0}} \frac{1}{k!}\binom{p}{p_1} \\ & \hspace{1cm}  \P_k\Big( \widetilde{g}^\mcZ\Big\{ \downarrow^k \Big( X^{p_2} \prod_{j\in A}\sum_{\sigma_j\prec\tau_j}\Psi( D^{q_j}(\tau_j/\sigma_j) ) \prod_{j\in B}D^{q_j}\Psi(\mcI(\tau_j)) \prod_{j\in C}\frac{X^{q_j}}{q_j!} 
			 \Big)\Big\},
			 \\
			 &\hspace{2cm}Y_{X^{p_1}\prod_{j\in A}D^{b_j}\Psi_{q_j}(\mcI(\sigma_j)) \prod_{j\in B} D^{b_j}\frac{X^{q_j}}{q_j!} \prod_{j\in C} D^{b_j}\Psi(\mcI(\downarrow^{q_j}\tau_j))}^\mcZ \Big)
			\\
			 &=- \sum_{(A,B,C)\in\mcP^{ad}}\sum_{\substack{(q_j)\\k\geq0}} \frac{1}{k!}\binom{p}{p_1} \\ & \hspace{1cm} \P_k\Big( \widetilde{\PPi}^\mcZ\Big\{ \downarrow^k \Big(X^{p_2} \prod_{j\in A}\sum_{\mu_j\prec\tau_j}\Psi(\tau_j/\mu_j) \prod_{j\in B}\Psi(\mcI^+_{q_j}\tau_j) \prod_{j\in C}\frac{X^{q_j}}{q_j!} 
			 \Big)\Big\},
			 \\
			 &\hspace{2cm}Y_{X^{p_1}\prod_{j\in A}D^{b_j}\Psi(\mcI(\mu_j)) \prod_{j\in B} D^{b_j}\frac{X^{q_j}}{q_j!} \prod_{j\in C} D^{b_j}\Psi\big(\mcI(\downarrow^{q_j}\tau_j)\big)}^\mcZ \Big).
		\end{align*}
		From the induction hypothesis 
		\begin{equs}
	&	Y_{X^{p_2}\prod_{j\in A}D^{b_j}\Psi(\mcI(\mu_j)) \prod_{j\in B} D^{b_j}\frac{X^{q_j}}{q_j!} \prod_{j\in C} D^{b_j}\Psi(\downarrow^{q_j}\mcI(\tau_j))}^\mcZ 
	\\ &	=
		 Z_{X^{p_2}\prod_{j\in A}\mcI^+_{b_j}(\mu_j) \prod_{j\in B} D^{b_j}\frac{X^{q_j}}{q_j!} \prod_{j\in C}\mcI_{b_j}^+(\downarrow^{q_j}\tau_j)}.
		\end{equs}
 And Lemma \ref{lem_PsiTau_multiplicatif} gives
\begin{align*}
	&\widetilde{\PPi}^\mcZ\Big( \downarrow^k \Big( \prod_{j\in A}\Psi(\tau_j/\mu_j) \prod_{j\in B}\Psi(\mcI_{q_j}^+\tau_j) \prod_{j\in C}\frac{X^{q_j}}{q_j!} 
	\Big)\Big) 
	\\&= 
	\widetilde{\PPi}^\mcZ\Big( \Psi\Big( \downarrow^k \Big( \prod_{j\in A} \tau_j/\mu_j \prod_{j\in B}\mcI^+_{q_j}\tau_j \prod_{j\in C}\frac{X^{q_j}}{q_j!} 
	\Big)\Big) .
\end{align*}
	Then 
	\begin{align*}
		&Y_{X^p\prod_{j=1}^nD^{b_j}\Psi(\mcI(\tau_j))}^\mcZ - \mathbf{1}_{p=0} \, \prod_{j=1}^n \widetilde{g}^\mcZ\Big( D^{b_j}\Psi(\mcI(\tau_j)) \Big) 
		\\
		&=- \sum_{(A,B,C)\in\mcP^{ad}}\sum_{\substack{(q_j)\\k\geq0}} \frac{1}{k!}\binom{p}{p_1}\P_k\Big(
			\widetilde{\PPi}^\mcZ\Big( \Psi\Big( \downarrow^k \Big( \prod_{j\in A} \tau_j/\mu_j \prod_{j\in B}\mcI^+_{b_j +q_j}\tau_j \prod_{j\in C}\frac{X^{q_j}}{q_j!} 
		\Big)\Big)\Big),  \\
		&\hspace{5cm}
		Z_{X^{p_2}\prod_{j\in A}\mcI^+_{b_j}(\mu_j) \prod_{j\in B} \frac{X^{q_j}}{q_j!} \prod_{j\in C}\mcI_{b_j}^+(\downarrow^{q_j}\tau_j)} \Big) 
		\\
		&=-\sum_{\substack{\sigma\prec\tau \\ k\geq0}}\frac{1}{k!} \P_k\Big( \widetilde\PPi^\mcZ(\Psi(\downarrow^k\tau/\sigma)),Z_{\sigma} \Big) =- \widetilde{\PPi}^\mcZ(\overline{\Psi}(\tau)).
	\end{align*} 
	Then $Y_{X^p\prod_{j=1}^nD^{b_j}\Psi(\mcI(\tau_j))}^\mcZ=Z_{X^p\prod_{j=1}^n\Psi(\mcI^+_{b_j}(\tau_j))}$ from \eqref{eq_defZtau_T+}.
\end{proof}

The following lemma will be useful for proving Theorem \ref{prop_extensionthm_Ztau}, it will give regularity estimates on the $Z_\tau$. 

\begin{lemma}\label{lem_reconstructionarg_Ztau}
Let $\tau\in\mcT$ such that $|\tau|>0$. We have
$$
\norme{ \PPi^\mcZ\big(\Psi(D^0\tau)\big) - g^\mcZ\big(D^0\overline{\Psi}(\tau)\big) }_{C^{|\tau|}} \lesssim \lVert D^0\tau\lVert_\mcZ^*+ \norme{\tau}^{*,-}_\mcZ.
$$
\end{lemma}
\begin{proof}
	From \cite[Theorem 1]{LocalExpansionsParacSystems} the distribution $\PPi^\mcZ\big(\Psi(D^0\tau)\big)$ admits a lift as a modelled distribution  given by
	\begin{align*}
		{\bm h}_\tau^1 = (\id \otimes g^\mcZ )\Delta\big(\Psi(D^0\tau)\big) - \Psi(D^0\tau).
	\end{align*}
	From Lemmas \ref{lem_DeltaPsitau}, \ref{lem_DnSnpropenvrac} and Proposition \ref{Prop_CommutDerivativePsi}
	\begin{align*}
		{\bm h}_\tau^1 &= \sum_{\substack{\sigma\prec\tau\\ q\geq0}} g^\mcZ\big(D^q\Psi(\tau/\sigma)  \big) \Psi_q(D^0\sigma) + \sum_{n\geq 0} \frac{1}{n!}g^\mcZ(D^n\Psi(D^0\tau)  ) \, X^n
		\\
		&= \sum_{\substack{\sigma\prec\tau\\ q\geq0}} \PPi^\mcZ\big(\Psi(D^q(\tau/\sigma))  \big) \Psi_q(D^0\sigma) + \sum_{n\geq 0} \frac{1}{n!}\PPi^\mcZ(\Psi(D^n\tau)  ) \, X^n
		\\
			&= \sum_{\substack{\sigma\prec\tau\\ }} \PPi^\mcZ\big(\Psi(\tau/\sigma)  \big) \Psi(D^0\sigma) + \sum_{n\geq 0} \frac{1}{n!}\PPi^\mcZ(\Psi(D^n\tau)  ) \, X^n.
	\end{align*}
	Likewise, the function $g^\mcZ\big(D^0\overline{\Psi}(\tau)\big)$ admits the lift 
		\begin{align*}
		{\bm h}_\tau^2 &= \sum_{\substack{\sigma\prec\tau\\ q\geq0}} g^\mcZ\big(D^q\Psi(\tau/\sigma)  \big) D^0\Psi_q(\sigma) + \sum_{n\geq 0} \frac{1}{n!}g^\mcZ(D^n\overline\Psi(\tau)  ) \, X^n
		\\
		&= \sum_{\substack{\sigma\prec\tau\\ q\geq0}} \PPi^\mcZ\big(\Psi(D^q\tau/\sigma)  \big) D^0\Psi_q(\sigma) + \sum_{n\geq 0} \frac{1}{n!} g^\mcZ(D^n\overline\Psi(\tau)  ) \, X^n
		\\
		&= \sum_{\substack{\sigma\prec\tau\\ }} \PPi^\mcZ\big(\Psi(\tau/\sigma)  \big) D^0\Psi(\sigma) + \sum_{n\geq 0} \frac{1}{n!} g^\mcZ(D^n\overline\Psi(\tau)  ) \, X^n.
	\end{align*}
	From Proposition \ref{Prop_CommutDerivativePsi} we have $\Pi_x(\Psi(D^0\sigma))(y) = \gamma_{yx}(D^0\Psi(\sigma))$ for any $x,y\in\bfR^{d+1}$.
	Then $\PPi^\mcZ\big(\Psi(D^0\tau)\big) - g^\mcZ\big(D^0\overline{\Psi}(\tau)\big)$ admits the following lift as a modelled distribution
	$$
	{\bm h}_\tau^{12} = \sum_{n\geq 0} \frac{1}{n!}\Big(  \PPi^\mcZ(\Psi(D^n\tau)  )  -g^\mcZ(D^n\overline\Psi(\tau)  )\Big) \, X^n,
	$$
	which takes value in the polynomial sector. It follows  that $\PPi^\mcZ\big(\Psi(D^0\tau)\big) - g^\mcZ\big(D^0\overline{\Psi}(\tau)\big)$ is indeed in $C^{|\tau|}$ and that
	\begin{align*}
\norme{ \PPi^\mcZ\big(\Psi(D^0\tau)\big) - g^\mcZ\big(D^0\overline{\Psi}(\tau)\big) }_{C^{|\tau|}} \lesssim \norme{{\bm h}_\tau^1-{\bm h}_\tau^2}_{\mcD^{|\tau|}}&\lesssim \norme{{\bm h}_\tau^1}_{\mcD^{|\tau|}}+\norme{{\bm h}_\tau^2}_{\mcD^{|\tau|}}
\\
&\lesssim  \lVert D^0\tau \lVert_\mcZ^*+ \norme{\tau}^{*,-}_\mcZ.
	\end{align*}
\end{proof}

We are now also able to prove Theorem \ref{prop_extensionthm_Ztau}
\begin{proof}[of Theorem \ref{prop_extensionthm_Ztau}]
	The proof is inspired from \cite{BailleulHoshinoRS1}. We construct the stochastic data by induction on the degree.
	For $\alpha\in\bfR$, we let $T_{<\alpha}^+$ to be the subalgebra of $T^+$ generated by the symbols $X^{e_i}$ with $i\in\{0,...,d\}$ and $\mcI_b^+(\tau)$ with $|\tau|<\alpha$.
		We write the set $A$ of degrees of non polynomial decorated trees as  $A=(\alpha_1<\alpha_2<\dots)$.
	We suppose that all the $Z_\tau$ with $\tau\in T_{<\alpha_j}\sqcup T^+_{<\alpha_j}$ are already constructed. 
	For any planted tree $\mcI(\tau)$, for $\tau\in T_{<\alpha_j}$, we set 
	$$
	Z_{\mcI(\tau)} = \widehat Y^\mcZ_{\Psi(\tau)},
	$$
	which is indeed in $C^{|\tau|+\beta}$ from Theorem \ref{prop_regularity_YI_tau}.
	We then define $Z_\tau$ for $\tau\in T^+_{\alpha_{j+1}}$ according to \eqref{eq_defZtau_T+}. 
	For $\tau\in\mcT$ that is not planted with $|\tau|=\alpha_{j+1}$, one sets
	$$
	Z_\tau = { \PPi^\mcZ\big(\Psi(D^0\tau)\big) - g^\mcZ\big(D^0\overline{\Psi}(\tau)\big) } .
	$$
	The Lemma \ref{lem_reconstructionarg_Ztau} shows that $Z_\tau$ has the required regularity.
\end{proof}

\subsection{Multiplicative stochastic data}

Given a smooth stochastic data $\mcZ=(Z_\tau)$, one defines a corresponding multiplicative collection $(Z^\circ_\tau)_{\tau\in\mcT}$, by setting for $\tau = X^p\Xi_\ell\prod_{j=1}^n\mcI_{b_j}(\tau_j) $
\begin{equation} \label{eq_defZotau}
	 Z_{\tau}^\circ= \mathbf{1}_{p=0} \, \xi_\ell  \,  \prod_{j=1}^{n}\sum_{l_j\geq 0} \frac{1}{l_j!} (\partial^{b_j}K)^{[l_j]} * \widetilde{\PPi}^{\mcZ}\left(\Psi(\downarrow^{l_j}\tau_j)\right) - \widetilde{ \PPi}^{\mcZ^\circ}\left(\overline\Psi(\tau)\right). 
\end{equation}
which defines indeed an unique collection $(Z^\circ_\tau)_{\tau\in\mcT}$ by induction on the size of the tree as $\overline\Psi(\tau)$ involves only smaller trees. Extending this collection by setting $Z_\tau^\circ=Z_\tau$ for $\tau\in\mcT^+$ gives a stochastic data, that we write $\mcZ^\circ$. Equation \ref{eq_defZotau} can be rewritten as
\begin{equation}\label{eq_mult_PPi_Zcirc}
	 \widetilde{ \PPi}^{\mcZ^\circ}\left(\Psi(\tau)\right)= \mathbf{1}_{p=0} \, \xi_\ell  \,  \prod_{j=1}^{n}\sum_{l_j\geq 0} \frac{1}{l_j!} (\partial^{b_j}K)^{[l_j]} * \widetilde{\PPi}^{\mcZ}\left(\Psi(\downarrow^{l_j}\tau_j)\right). 
\end{equation}
 From the same computations as in the proof of Proposition \ref{prop_alternativedef_Ztaunaif}, the collection $\mcZ^\circ$ is characterized by the following equation
\begin{equation}\label{eq_mult_PPi_Zcirc1}
	{ \PPi}^{\mcZ^\circ}\left(\Psi(\tau)\right) = x^p \, \xi_\ell  \,  \prod_{j=1}^{n} \partial^{b_j}K * {\PPi}^\mcZ\left(\Psi(\tau_j)\right).
\end{equation}
The same computations as in the proof of Lemma \ref{lem_PsiTau_multiplicatif} give also
\begin{equation}\label{eq_mult_PPi_Zcirc0}
	\PPi^{\mcZ^\circ}( \Psi(\tau) )(x) = x^p \xi_\ell \prod_{j=1}^{n} {\PPi}^{\mcZ} \left( \Psi(\mcI_{b_j}(\tau_j))\right),
\end{equation}
We notice that the naive collection satisfies $Z_\tau=Z^\circ_\tau$.

 For any smooth stochastic data $\mcZ$, any tree $\tau\in\mcT$ and $|n|>|\mcI(\tau)|$, we define the distribution
\begin{equation} \label{eq_defZplusNeg_0}
	\overline Z_{\mcI_{n}(\tau)} = \partial^{n}_y\Big( { \Pi}^\mcZ_x \Psi(\mcI (\tau))(y)\Big)_{|y=x}.
\end{equation}
The main result of this subsection is the following proposition.
	\begin{proposition} \label{lem_SumNegativeCorrector}
	Suppose we are given a smooth admissible stochastic data $\mcZ=(Z_\tau)$. Let a tree $\tau= X^p \Xi_\ell \prod_{j=1}^n\mcI_{b_j}(\tau_j) \, $ with $|\tau|<0$.
	\begin{enumerate}
		\item If $p\ne 0$ or $|\mcI_{b_j}(\tau_j)|>0$ for some $j\in\{1,...,n\}$, then
		\begin{equation}
			\sum_{\substack{q\geq 0 \\ \tau_q\prec\dots\prec\tau_1\prec\tau}}\sum_{k\geq 0} \frac{1}{k!}{\bf C}_k(Z_{\downarrow^k\tau/\tau_1},\dots,Z_{\tau_{q-1}/\tau_q}, Z_{\tau_q}^\circ) = 0.
		\end{equation}
		\item Otherwise 
		\begin{equation}
			\sum_{\substack{q\geq 0 \\ \tau_q\prec\dots\prec\tau_1\prec\tau}}\sum_{k\geq 0} \frac{1}{k!}{\bf C}_k(Z_{\downarrow^k\tau/\tau_1},\dots, Z_{\tau_{q-1}/\tau_q}, Z_{\tau_q}^\circ) = \xi_\ell \prod_{j=1}^n  \overline Z_{\mcI_{b_j}(\tau_j)}.
		\end{equation}
	\end{enumerate}
\end{proposition}
The proposition will be proven after the following Lemmas \ref{lem_sum Negative_correctors} and \ref{lem_multiplicativity_modelPsicirc}.
\begin{lemma}\label{lem_sum Negative_correctors}
	For any smooth admissible data $\mcZ$ and any tree $\tau$ with $|\tau|<0$, one has
	\begin{align*}
		\sum_{\substack{q\geq 0 \\ \tau_q\prec\dots\prec\tau_1\prec\tau}}\sum_{k\geq 0} \frac{1}{k!}{\bf C}_k(Z_{\downarrow^k\tau/\tau_1},\dots,Z_{\tau_{q-1}/\tau_q} ,Z_{\tau_q}^\circ)(x) = { \Pi}_x^{\mcZ^\circ}(\Psi(\tau))(x).
	\end{align*}
\end{lemma}
\begin{proof}
	We prove the result by induction on the size of the tree $\tau$. We let $\tau= X^p \Xi_\ell \prod_{j=1}^n\mcI_{b_j}(\sigma_j)$. From Lemma \ref{lem_DeltaPsitau}
	\begin{align*}
		{ \Pi}_x^{\mcZ^\circ}(\Psi(\tau))(x) &= 	{ \PPi}^{\mcZ^\circ}(\Psi(\tau))(x) - \sum_{ \sigma\prec\tau }\sum_{q\geq 0}  {g}_x^\mcZ(D^q\Psi(\tau/\sigma))  { \Pi}_x^{\mcZ^\circ}(\Psi_q(\sigma))(x).
	\end{align*}
From \eqref{cor_commut_DnPsi_withoutilde} we obtain
		\begin{equation}\label{eq_Pixtau(x)_dev}
		{ \Pi}_x^{\mcZ^\circ}(\Psi(\tau))(x) = 	{ \PPi}^{\mcZ^\circ}(\Psi(\tau))(x) - \sum_{ \mu\prec\tau }  {g}_x^\mcZ(\Psi(\tau/\mu)) \,  { \Pi}_x^{\mcZ^\circ}(\Psi(\mu))(x).
	\end{equation}
On the other hand from \eqref{eqdef_corrector} and \eqref{eqdef_starderivative} 
	\begin{align*}
		&\sum_{\substack{q\geq 0 \\ \tau_q\prec\dots\prec\tau_1\prec\tau}}\sum_{k\geq 0} \frac{1}{k!}{\bf C}_k(Z_{\downarrow^k\tau/\tau_1},\dots, Z_{\tau_q}^\circ) 
		= 	\sum_{\substack{q\geq 0 \\ \tau_q\prec\dots\prec\tau_1\prec\tau}}\sum_{k\geq 0} \frac{1}{k!}{\P }_k(Z_{\downarrow^k\tau/\tau_1},\dots, Z_{\tau_q}^\circ) 
		\\
		&\qquad- \sum_{\substack{\sigma\prec\tau \\ k=k_1+k_2}}\sum_{|r|<|\tau/\sigma|-|k_2|} \frac{1}{k!p!} \binom{k}{k_1} \partial^r_*\P_{k_1}(Z_{\downarrow^k\tau/\sigma_1},\dots,Z_{\sigma_q/\sigma}) \bfC_{k_2+r}( Z_{\sigma/\nu_1},\dots, Z_{\nu_r}^\circ),
		\end{align*}
		re-indexing the last sum gives
		\begin{align*}
			 \widetilde{\PPi}^{\mcZ^\circ}(\Psi(\tau)) -  \sum_{\substack{\sigma\prec\tau \\ |n|<|\tau/\sigma|}}\sum_{\substack{k_1\geq 0 \\ n=r+k_2}} \frac{1}{k_1!n!}\binom{n}{k_2}     &\partial^{n-k_2}_*\P_{k_1}(Z_{\downarrow^{k_1+k_2}\tau/\sigma_1},\dots,Z_{\sigma_q/\sigma}) \\ &\times \bfC_{n}( Z_{\sigma/\nu_1},\dots, Z_{\nu_r}^\circ). 
		\end{align*}
		Proposition \ref{Prop_CommutDerivativePsi} rewrites as 
		$$
	\sum_{\substack{k_1\geq 0 \\ n=r+k_2}} \frac{1}{k_1!}\binom{n}{k_2}     \partial^{n-k_2}_*\P_{k_1}(Z_{\downarrow^{k_1+k_2}\tau/\sigma_1},\dots,Z_{\sigma_q/\sigma}) =  \widetilde{\PPi}^\mcZ( \Psi( D^{n}(\tau/\sigma)  )  ),
		$$
		then from Lemma \ref{lem_SnDnreindex}, we obtain 
		\begin{align*}
			\widetilde{\PPi}^{\mcZ^\circ}(\Psi(\tau)) - \sum_{\substack{\sigma\prec\tau \\ |n|<|\tau/\sigma|}}\frac{1}{n!} \widetilde{\PPi}^\mcZ\left( \Psi( \tau/\sigma  )  \right) \, \bfC_{n}( Z_{\downarrow^n\sigma/\nu_1},\dots, Z_{\nu_r}^\circ).
		\end{align*}
	From the induction hypothesis, one recovers indeed $\Pi_x^{\mcZ^\circ}(\Psi(\tau))(x)$ from \eqref{eq_Pixtau(x)_dev}.
\end{proof}

\begin{lemma}\label{lem_multiplicativity_modelPsicirc}
	For any $\tau = X^p\Xi_\ell\prod_{j=1}^n\mcI_{b_j}(\sigma_j)$, we have
	$$
	\Pi^{\mcZ^\circ}_x( \Psi(\tau) )(x) =  \mathbf{1}_{p=0}\xi_\ell \prod_{j=1}^{n} {\Pi}_x^\mcZ \left( \Psi(\mcI_{b_j}(\sigma_j))\right)(x).
	$$
\end{lemma}

\begin{proof}
	We will prove the following identity for any trees $\tau_1,\tau_2\in\mcT$ 
	\begin{equation}
		\Pi^{\mcZ^\circ}_x(\Psi(\tau_1\tau_2)) = \Pi^{\mcZ^\circ}_x(\Psi(\tau_1))\Pi^{\mcZ^\circ}_x(\Psi(\tau_2)).
	\end{equation}
	The lemma will follow as $\Pi^{\mcZ^\circ}_x(\Psi(\Xi_\ell))=\xi_\ell$ and $\Pi^{\mcZ^\circ}_x(\Psi(X^n)) = (\cdot-x)^n$.
		From the same computations as in the proof of Lemma \ref{lem_PsiTau_multiplicatif}, we have for $\tau=X^p\Xi_\ell\prod_{j=1}^n\mcI_{b_j}(\sigma_j)$
	\begin{equation}\label{eq_mult_PPi_Zcirc0_b}
		\PPi^{\mcZ^\circ}( \Psi(\tau) )(x) = x^p \xi_\ell \prod_{j=1}^{n} {\PPi}^{\mcZ} \left( \Psi(\mcI_{b_j}(\sigma_j))\right),
	\end{equation}
	so that from the same computations in the proof of Proposition \ref{prop_mult_Ztau_+}
	\begin{equation}\label{eq_mult_Ztau}
		\PPi^{\mcZ^\circ}(\Psi(\tau_1\tau_2)) = \Pi^{\mcZ^\circ}(\Psi(\tau_1)) \, \Pi^{\mcZ^\circ}(\Psi(\tau_2)).
	\end{equation}
	From Lemma \ref{lem_DeltaPsitau} and \eqref{cor_commut_DnPsi}
	\begin{align*}
		\Pi^{\mcZ^\circ}_x( \Psi(\tau) )(x) &= 	\PPi^{\mcZ^\circ}( \Psi(\tau) ) - \sum_{\substack{\sigma\prec\tau\\ q\geq0}} g_x^\mcZ(D^q\Psi(\tau/\sigma))) \, 	\Pi^{\mcZ^\circ}_x( \Psi_q(\sigma) )(x) 
		\\
		&\quad-\sum_{n\geq 0} \frac{x^n}{n!} \Pi_x^{\mcZ^\circ}(\Psi(\downarrow^n\tau))(x) -  g_x^{\mcZ^\circ}(D^0\Psi(\tau))
		\\
		&=\PPi^{\mcZ^\circ}( \Psi(\tau) )(x) - \sum_{\substack{\sigma\prec\tau}} g_x^\mcZ\left(\Psi(\tau/\sigma)\right) \, 	\Pi^{\mcZ^\circ}_x( \Psi(\sigma) )(x)		\\
		&\quad-\sum_{n\geq 0} \frac{x^n}{n!} \Pi_x^{\mcZ^\circ}(\Psi(\downarrow^n\tau)) - g_x^{\mcZ^\circ}(\Psi(D^0\tau))
		\\
		&= \PPi^{\mcZ^\circ}( \Psi(\tau) )(x) - \sum_{\substack{\sigma<\tau}} g_x^\mcZ\left(\Psi(\tau/\sigma)\right) \, 	\Pi^{\mcZ^\circ}_x( \Psi(\sigma) )(x).
	\end{align*}
Then 
	\begin{align*}
		\Pi^{\mcZ^\circ}_x( \Psi(\tau) ) 
		& = \PPi^{\mcZ^\circ}( \Psi(\tau) ) - \sum_{\substack{(\sigma_1\sigma_2)<(\tau_1\tau_2)}} g_x^\mcZ\left(\Psi(\tau_1/\sigma_1\,\tau_2/\sigma_2)\right) \, 	\Pi^{\mcZ^\circ}_x( \Psi(\sigma_1\sigma_2) ).
	\end{align*}
Proposition \ref{prop_mult_Ztau_+} gives 
$$
g_x^\mcZ\big(\Psi(\tau_1/\sigma_1\,\tau_2/\sigma_2)\big) = g_x^\mcZ \big(\Psi(\tau_1/\sigma_1)\big)  g_x^\mcZ  \big(\Psi(\tau_2/\sigma_2)\big). 
$$
The induction hypothesis gives 
$$
\Pi^{\mcZ^\circ}_x( \Psi(\sigma_1\sigma_2) ) = \Pi^{\mcZ^\circ}_x\big( \Psi(\sigma_1) \big) \Pi^{\mcZ^\circ}_x\big( \Psi(\sigma_2) \big).
$$
One also has
\begin{align*}
	&\sum_{\sigma_1\sigma_2<\tau_1\tau_2 } g_x^\mcZ \big(\Psi(\tau_1/\sigma_1)\big)  g_x^\mcZ  \big(\Psi(\tau_2/\sigma_2)\big) 	\Pi^{\mcZ^\circ}_x( \Psi(\sigma_1) ) \Pi^{\mcZ^\circ}_x\big( \Psi(\sigma_2) \big)
\\
&= 	\sum_{\substack{\sigma_1<\tau_1\\ \sigma_2<\tau_2 }}  g_x^\mcZ \big(\Psi(\tau_1/\sigma_1)\big)  g_x^\mcZ  \big(\Psi(\tau_2/\sigma_2)\big) 	\Pi^{\mcZ^\circ}_x( \Psi(\sigma_1) ) \Pi^{\mcZ^\circ}_x\big( \Psi(\sigma_2) \big)
\\
&\quad+ \sum_{\sigma_1<\tau_1}   g_x^\mcZ \big(\Psi(\tau_1/\sigma_1)\big)  \Pi^{\mcZ^\circ}_x( \Psi(\sigma_1) ) \Pi^{\mcZ^\circ}_x\big( \Psi(\tau_2) \big)
\\
&\quad+ \sum_{\sigma_2<\tau_2} 
 g_x^\mcZ  \big(\Psi(\tau_2/\sigma_2)\big) 	
 \Pi^{\mcZ^\circ}_x\big( \Psi(\tau_1) \big)
 \Pi^{\mcZ^\circ}_x( \Psi(\sigma_2) ) 
\\
&= 	-\sum_{\substack{\sigma_1<\tau_1\\ \sigma_2<\tau_2 }} g_x^\mcZ \big(\Psi(\tau_1/\sigma_1)\big)  g_x^\mcZ  \big(\Psi(\tau_2/\sigma_2)\big) 	\Pi^{\mcZ^\circ}_x( \Psi(\sigma_1) ) \Pi^{\mcZ^\circ}_x\big( \Psi(\sigma_2) \big)
\\
&\quad+\sum_{\sigma_1<\tau_1}  g_x^\mcZ \big(\Psi(\tau_1/\sigma_1)\big)  \Pi^{\mcZ^\circ}_x( \Psi(\sigma_1) ) \PPi^{\mcZ^\circ}\big( \Psi(\tau_2) \big)
\\
&\quad+ \sum_{\sigma_2<\tau_2}
g_x^\mcZ  \big(\Psi(\tau_2/\sigma_2)\big) 	
\PPi^{\mcZ^\circ}\big( \Psi(\tau_1) \big)
\Pi^{\mcZ^\circ}_x( \Psi(\sigma_2) ) 
.
\end{align*}
Then from \eqref{eq_mult_Ztau}
\begin{align*}
	\Pi^{\mcZ^\circ}_x( \Psi(\tau) ) 
	&= \Big(\PPi^{\mcZ^\circ}( \Psi(\tau_1) ) - \sum_{\substack{\sigma_1<\tau_1}} g_x^\mcZ\left(\Psi(\tau_1/\sigma_1)\right) \, 	\Pi^{\mcZ^\circ}_x( \Psi(\sigma_1) ) \Big)
	\\
	&\quad\times \Big(\PPi^{\mcZ^\circ}( \Psi(\tau_2) ) - \sum_{\substack{\sigma_2<\tau_2}} g_x^\mcZ\left(\Psi(\tau_2/\sigma_2)\right) \, 	\Pi^{\mcZ^\circ}_x( \Psi(\sigma_2) ) \Big)
	\\
	&= \Pi^{\mcZ^\circ}_x( \Psi(\tau_1) ) \Pi^{\mcZ^\circ}_x( \Psi(\tau_2) ).
\end{align*}
\end{proof}

We are now ready to prove Proposition \ref{lem_SumNegativeCorrector}.

\begin{proof}[of Proposition \ref{lem_SumNegativeCorrector}]
For any $\tau= X^p \Xi_\ell \prod_{j=1}^N\mcI_{b_j}(\sigma_j) $, from Lemmas \ref{lem_multiplicativity_modelPsicirc} and \ref{lem_sum Negative_correctors}, we have

	\begin{align*}
			\sum_{\substack{q\geq 0 \\ \tau_q\prec\dots\prec\tau_1\prec\tau}}\sum_{k\geq 0} \frac{1}{k!}{\bf C}_k(Z_{\downarrow^k\tau/\tau_1},\dots,Z_{\tau_{q-1}/\tau_q}, Z_{\tau_q}^\circ)  = \mathbf{1}_{p=0}\xi_\ell \prod_{j=1}^n {\Pi}_x^{\mcZ^\circ} \left(\Psi(\mcI_{b_j}(\tau_j))\right)(x).
	\end{align*}
	If we suppose that $|\mcI_{b_j}(\sigma_j)|> 0 $ for some $j\in\{1,...,N\}$, from the definition of a model we would have ${\Pi}_x(\mcI_{b_j}(\sigma_j))(x)=0$, hence the first point of the Lemma.
	In order to prove the second point, we have to show \begin{equation*}
		\Pi_x^\mcZ(\Psi(\mcI_n(\sigma)))(x) = \overline{Z}_{\mcI_n(\sigma)}(x)
		\end{equation*}
		 for any tree $\sigma$ and $|n|>|\mcI(\sigma)|$, we are going to prove it by induction. We have indeed from the definition of $\Pi_x$ and from Lemmas \ref{lem_DeltaPsitau} and \eqref{cor_commut_DnPsi}
	\begin{align*}
		\overline{Z}_{\mcI_n(\sigma)}(x)&= \partial^{n}_y\Big( { \Pi}_x^\mcZ \Psi(\mcI (\sigma))(y)\Big)_{|y=x}
		\\
		&= \partial^n_y\Big( \PPi^\mcZ\left(\Psi(\mcI(\sigma))\right)(y)  - \sum_{\substack{\mu\prec\sigma \\ q\geq  0}} g_x^\mcZ\left(D^q\Psi(\sigma/\mu)\right) \, \Pi^{\mcZ^\circ}_x \Psi_q(\mcI(\mu))    \Big)_{|y=x}
		\\
		&= \partial^n \PPi^\mcZ\left(\Psi(\mcI(\sigma))\right)(x)  - \sum_{\substack{\nu\prec\sigma }} g^\mcZ_x\left(\Psi(\sigma/\nu)\right) \,  \partial^n_y\Big(\Pi^{\mcZ^\circ}_x \Psi(\mcI(\nu))(y)    \Big)_{|y=x}.
	\end{align*}
	From Lemma \ref{lem_SnDnreindex}, the induction hypothesis and \eqref{eq_ZtauAdmissible}, this is equal to
	\begin{align*}
		& \PPi^\mcZ\left(\Psi(\mcI_n(\sigma))\right) - \sum_{\substack{\nu\prec\sigma }} g_x^\mcZ(\Psi(\sigma/\nu))  \Pi^{\mcZ^\circ}_x \Psi(\mcI_n(\nu))(x)    
		\\
		&= \Pi_x^\mcZ\left(\Psi(\mcI_n(\sigma))\right)(x)
	\end{align*}
which concludes the proof.
\end{proof}

	

	
	
	

	\section{The paracontrolled ansatz} \label{section_ParacAnsatz}
	
	In this section, we define the paracontrolled ansatz, and we show how to compute the right hand side of \eqref{eq_gPAM}. We assume that we are given an admissible stochastic data $\mcZ=(Z_\tau)_{\tau\in \mcT}$ as described in Section \ref{section_constructionZtau}.

	\subsection{The paracontrolled ansatz and paralinearisation}
	
	\subsubsection{Definition of the paracontrolled ansatz and paracontrolled systems}
	We suppose first that the collection $\mcZ$ is smooth. For any planted tree $\mcI(\tau)$ with $|\mcI(\tau)|<|n|$ and $n \in \N^{d+1}$, we recall the definition of the distribution $\overline Z$
	\begin{equation} \label{eq_defZplusNeg}
		\overline Z_{\mcI_{n}(\tau)} = \partial^{n}_y\Big( { \Pi}^\mcZ_x \Psi(\mcI (\tau))(y)\Big)_{|y=x}.
			\end{equation}
	It has the following alternative definition from Proposition \ref{prop_derivetoile_Ztau_Neg} 
	$$
	\overline Z_{\mcI_{n}(\tau)} = \sum_{\substack{q\geq 0 \\ \tau_q\prec\dots\prec\tau_1}}\sum_{k=k_1+k_2\geq0} \frac{1}{k_2!} \binom{n}{k_1} \partial^{n-k_1}_*  \P_{k_2}(Z_{\downarrow^k\tau/\tau_1},\dots,Z_{\mcI(\tau_q)}  ).
	$$
	For any decorated tree $\tau\in\mcT$ such that $|\mcI(\tau)|<n$, we know from Lemma \ref{lem_elementary differential_form} that the elementary differential $\Upsilon_F[\tau]$ is a function of the variables $\mcX_k$ with $|k|<n$. So that one can define recursively the modified derivatives
	\begin{equation}\label{eqdef_partialStarU}
		\partial_*^{n} u = \partial^{n} u - \sum_{|\mcI(\tau)|<|n|} \frac{1}{S(\tau)}\Upsilon_F[\tau]( (\partial^j_*u)_{|j|<n} ) \, \overline Z_{\mcI_{n}(\tau)}.
\end{equation}
	For any tree $\tau\in\mcT$ with $|\mcI(\tau)|<n$, we will use the shorthand 
	\begin{equation}\label{eq_defutau}
		u_\tau = \frac{1}{S(\tau)}\Upsilon_F[\tau]( (\partial^j_*u)_{|j|<n} ).
	\end{equation}
	In order to have \eqref{eqdef_partialStarU} well-defined, we make the following assumption. 
	\begin{assumption} \label{constant_coefficient} Let $\tau $ be a decorated tree such that $ \Upsilon_F[\tau] \neq 0 $ and for every strict subtree $ \bar{\tau} $ containing the root of $\tau$ ($ \bar{\tau} \neq \tau $), one has
		\begin{equs}
			|\bar{\tau}| \geq - \beta.
			\end{equs} 
	\end{assumption}
	This assumption is the same as \cite[Assumption 5.10]{BCCH}. This implies (see \cite[Lemma 5.14]{BCCH}) that $ \Upsilon_F[\tau] $ is a constant and does not depend on $u$ for every $\tau$ such that $|\CI(\tau)| < 0$.
	
In \eqref{eqdef_partialStarU},
$ \partial^n_* u $ depends on the previous derivatives $ \partial^j_*u $ for $ |j| < |n| $.  This triangular structure is preserved for $n=0$ as $ \Upsilon_F[u] $ is constant. The definition \eqref{eqdef_partialStarU} is equivalent to the generalised Da Prato-Debussche trick introduced in \cite{BCCH}.
We illustrate this construction with two examples. 
In the $ \Phi^4_3$ model, $d=3$, $\beta = 2$ and $ |\Xi| = -\frac{5}{2} - \kappa $, one has only one planted decorated tree with negative degree  $ \CI(\Xi) $. Then, one has
\begin{equs}
		\partial_*^{0} u =  u -  \overline Z_{\mcI(\Xi)} = u - K * \xi.
\end{equs}
	For $\gamma > 0$, we say that $u$ satisfies the paracontrolled ansatz up to order $\gamma$ with random data $\mcZ$ if one has   
		\begin{equation}\label{eq_pamansatz}
		u=\sum_{|\mcI(\tau)|<\gamma} \P(u_\tau , \, Z_{\mcI(\tau)} ) + u^\#,
	\end{equation}
		for some $u^\#\in C^\gamma$, where $u_\tau$ is the elementary differential defined in \eqref{eq_defutau}.
		The vector space of couples $(u,u^\#)$ such that $u$ satisfies the paracontrolled ansatz with remainder $u^\#$ is endowed with the norm
		$$
		\lVert(u,u^\#)\lVert_{\alpha,\gamma}  = \norme{u}_{C^\alpha} + \lVert u^\# \lVert_{C^\gamma}.
		$$ 
	
	\begin{remark}
		In order to show that the set of functions satisfying the paracontrolled ansatz is not trivial, we can follow the trick that was used in \cite[Lemma 4.12]{AllezChouk}. This construction is made explicit in the proof of Proposition \ref{prop_AppGamma}. 
	\end{remark}
	
	One also defines the stronger notion of paracontrolled system associated with \eqref{eq_gPAM}. This corresponds to the same notion of paracontrolled systems from \cite{BailleulBernicotHighOrder}, where the generalised derivatives are given by elementary differentials.
	
	\begin{definition}\label{def_PAMparacSystem}
		One says that $u$ develops in the  paracontrolled system associated with \eqref{eq_gPAM} to order $\gamma$, if for any $\tau\in\mcT$ with $|\mcI(\tau)|<\gamma$ we have 
		$$
		u_\tau = \sum_{\substack{\sigma\succ\tau \\ |\mcI(\sigma)|<\gamma}} \P(u_\sigma,Z_{\sigma/\tau} ) + u_\tau^\#
		$$
		with $u_\tau^\# \in C^{\gamma-|\mcI(\tau)|}$.
			We write $\mcD^\gamma_\P(\mcZ)$ for the vector space of such paracontrolled systems, we endow it with the norm
		$$
		\norme{u}_{\mcD^\gamma_\P(\mcZ)} = \sum_{\substack{|\mcI(\tau)|<\gamma} } \lVert u_\tau^\#\lVert_{C^{\gamma-|\mcI(\tau)|}}.
		$$
		\end{definition}

		This definition can be reformulated using the notion of a paracontrolled system introduced in \cite[Section 1.3]{LocalExpansionsParacSystems}. More precisely, we say that $u$ can be written at the top of a paracontrolled system $(u_w)_w$ of order $\gamma$ with alphabet $\mcT$ and reference distributions the $(Z_\tau)_{\tau\in\mcT}$, and where the set of admissible words is
		\begin{align*}
			\mcU_{<\gamma} = \big\{ \lg\tau/\tau_1,\dots, \mcI(\tau_q) \rg , \quad |\mcI(\tau)|<\gamma, \, q\geq 0, \, \tau_q\prec\dots\prec\tau_1\prec \tau   \big\}. 
		\end{align*}
		One has then for any word
		$
			u_{w} = \sum_{\tau \in \CT} \P(u_{w \tau} , \, Z_{\tau} ) + u^\#_w
		$
		with $ u^\#_w \in C^{\gamma - |w|} $.
		For $w=\lg \tau_1, \dots, \mcI(\tau_n) \rg $, we have $$u_w= 
		\frac{1}{\prod_{i=1}^n S(\tau_i)} \Upsilon_F[  \tau_1 \star \dots \star \tau_n].
		$$

	When $u$ develops in the paracontrolled system associated to \eqref{eq_gPAM} we can also write it as a sum of iterated paraproducts 
	\begin{equation}\label{eq_gPamSystem_to_IteratedParap}
		u = \sum_{|\mcI(\tau)|<\gamma} \sum_{\substack{q\geq0 \\ \tau_q\prec\dots\prec\tau_1\prec\tau}} \P( u_\tau^\#, Z_{\tau/\tau_1},\dots, Z_{\tau_{q-1}/\tau_q}, Z_{\mcI(\tau_q)}   ) + u^\#. 
	\end{equation}

	From \eqref{eq_CoproductOfMcItau} giving the interaction between $\mcI$ and the coproduct, we have for any tree $\tau$ the identity $$ \sum_{\sigma\prec\mcI(\tau)} \sigma \otimes (\mcI(\tau))/\sigma  = \sum_{\nu\prec\tau}  \mcI(\nu)  \otimes \tau/\nu.  $$ This is due to the fact that $ \sigma \neq X^k $ as we are considering $ \sigma \prec\mcI(\tau) $. Iterating this identity by co-associativity, we obtain for any integer $q\geq1$ the following relation which will be useful in later computations 
	\begin{equation}\label{eq_CommutMotAvecMcI}
\sum_{\sigma_q\prec \dots \prec \sigma_1 \prec \mcI(\tau)}   \sigma_q \otimes \dots \otimes (\mcI(\tau))/\sigma_1  = \sum_{\tau_q\prec \dots \prec\tau_1\prec \tau}  \mcI(\tau_q) \otimes \dots \otimes \tau/\tau_1.
\end{equation}

The following Lemma is crucial for computing the lift as a modelled distribution of any distribution $u$ that develops in the paracontrolled system associated to \eqref{eq_gPAM}. It enables to compute coefficients of the expansion we obtain when applying  \cite[Theorem 1]{LocalExpansionsParacSystems}. A proof can be found for instance in \cite[Proposition 2]{RenomrmalisedSpde}.
	\begin{lemma}\label{lem_ElemDiffMorphism}
		We have the following two identities  
		\begin{enumerate}
			\item For any $\sigma\in \mcT$ and $\tau=X^k\prod_{j=1}^n\mcI_{b_j}(\tau_j)$ we have
			$$
			\Upsilon_F\Big[ \Big(X^k \prod_{j=1}^n\mcI_{b_j}(\tau_j) \Big) \star \sigma  \Big] = \prod_{j=1}^n \Upsilon_F[\tau_j] \,\, \partial^k \prod_{j=1}^n D_{b_j} \Upsilon_F[\sigma] .
			$$
			\item For any $\tau\in\mcT$ and $k\in\bfN^{d+1}$
			\begin{equation}\label{eq_derivUpsilon}
			\Upsilon_F[(\downarrow^k)^*\tau] = \partial^k \Upsilon_F[\tau].
			\end{equation}
		\end{enumerate}
	\end{lemma}
	
	We finish the subsection with another crucial lemma which is a general Fa\`a di Bruno formula coming from \cite[Lemma A.1]{BCCH}. It will apply in addition to Lemma \ref{lem_ElemDiffMorphism} in some of the main proofs of this section (see Proofs of Lemma \ref{lem_deriv_uk_induction} and Theorem \ref{prop_paralinearization}). 
	\begin{lemma} \label{Faa Di Bruno}
		We suppose given a smooth function $h(v_{*})$ where $v_{*} $ corresponds to a subset of the variables  $  \partial_{*}^{k} u $ with $k \in \N^{d+1}$, one has
		$$
		\frac{\partial^{k}h(v_*)}{k!} =  \sum_{a_1,...,a_r} \, \sum_{k = \sum_{i=1}^r \beta_i k_i} \; \prod_{i=1}^r \frac{1}{\beta_i !} \left(\frac{\partial_{*}^{a_i +k_i} u}{k_i!}\right)^{\beta_i} \prod_{i=1}^r (D_{a_i})^{\beta_i} h(v_*).
		$$
		\end{lemma}

	\subsubsection{Commutation with the derivatives}
	
	This subsection is devoted to the next Proposition \ref{lem_starDerivUpsilon} which  will be crucial for performing computations. It gives roughly a commutation relation between the derivative $D^k$ and the map $\Psi$ at the level of the ansatz. 
	One defines
	$$
	\partial^k_{*,\gamma}u_\tau = \partial^k u_\tau^\# + \sum_{\substack{r\geq 0, \, |\mcI(\nu_r)|<\gamma \\ \nu_r\succ\dots\succ\nu_1\succ\tau}}\partial^k_*\P\big(u_{\nu_r}^\#, Z_{\nu_r/\nu_{r-1}},\dots, Z_{\nu_1/\tau}\big).
	$$
	and 
	$$
	\widetilde{\partial}^k_{*,\gamma} u  = \partial^k u^{\#} + \sum_{\substack{r\geq 0, \, |\mcI(\nu_r)|<|k| \\ \nu_r\succ\dots\succ\nu_1\succ\nu}}\partial^k_*\P\big(u_{\nu_r}^\#, Z_{\nu_r/\nu_{r-1}},\dots,Z_{\nu_1/\nu}, Z_{\mcI(\nu)}\big),
	$$
	where in both equations we associate to $u_{\nu_r}^\#$ the degree $\gamma-|\mcI(\nu_r)|$.
	
	\begin{lemma}\label{lem_starderivutau_notdependonGamma}
		The function $\partial^k_{*,\gamma}u_\tau$ does not depend on $\gamma>|k|+|\mcI(\tau)|$. The function $\widetilde\partial^k_{*,\gamma}u$ does not depend either on $\gamma>|k|$.
	\end{lemma}
	\begin{proof}
		We prove the first claim by induction on $|k|+|\mcI(\tau)|$, the second one is proven along the same lines.
		We take $\gamma_1>\gamma_2>|k|+|\mcI(\tau)|$, from the definition of $\partial^k_*$ in \eqref{eqdef_starderivative}, we have for $i\in\{1,2\}$
		\begin{equation}\label{eq_proof_starderivutau_gamma}
			\partial^k_{*,\gamma_i} u_\tau  =  \partial^k u_\tau - \sum_{\substack{\sigma\succ\tau \\ k=k_1+k_2}} \sum_{|r|< \rho_i    } \frac{1}{r!}\binom{k}{k_1} \partial^{k_1+r}_{*,\gamma_i}u_\sigma \sum_{\tau\prec\sigma_s\prec \dots\prec\sigma} \partial^{k_2}_*\P_r(Z_{\sigma/\sigma_1},\dots, Z_{\sigma_s/\tau}      ) 
		\end{equation}
		with 
		$$
		\rho_i \defeq \min\{  \gamma_i - |\mcI(\sigma)| - |k_1| , \quad - |\sigma/\tau|+|k_2|      \}.
		$$
		We have assumed $\gamma_i>|k|+|\mcI(\tau)| $, so that 
		$$
		\gamma_i - |\mcI(\sigma)| - |k_1| > - |\sigma/\tau|+|k_2|.  
		$$
		Then $\rho_i=-|\sigma/\tau|+|k_2|$,  and we see that it does not depend on $i\in\{1,2\}$. Then the sum in the right hand side of \eqref{eq_proof_starderivutau_gamma} is indexed on the same set for $i\in\{1,2\}$, and we can apply the induction hypothesis to get the lemma.
	\end{proof}

	The Lemma \ref{lem_starderivutau_notdependonGamma} enables us to forget the dependence on $\gamma$ in $\partial^k_{*,\gamma}u_\tau$ and in $\widetilde\partial^k_{*,\gamma}u$, and to simply write $\partial^k_*u_\tau$ and $\widetilde\partial^k_{*}u$.

	\begin{proposition}\label{lem_starDerivUpsilon}
		We suppose that $u$ can be developed in a paracontrolled system associated with \eqref{eq_gPAM} at order $\gamma$ with random data $\mcZ=(Z_\tau)_{\tau\in\mcT}$. For any $\tau\in\mcT$ and $k\in\bfN^{d+1}$ such that $|k| < \gamma-|\mcI(\tau)|$, we have
		\begin{equation}\label{eq_starPartialUTau} \begin{split}
			\partial_*^ku_\tau= u_{(\downarrow^k)^* \tau}.
		\end{split}\end{equation}
	For any $k\in\bfN^{d+1}$ with $|k|<\gamma$, 
	\begin{equation}\label{eq_starPartialU} \begin{split}
			\widetilde\partial^k_{*}u = \partial^k_* u.
	\end{split}\end{equation}
	\end{proposition}

The Proposition \ref{lem_starDerivUpsilon} will be proven by induction, we will first need the following Lemmas \ref{lem_deriv_uk_to_u} \ref{lem_halfderivative} and \ref{lem_deriv_uk_induction} giving all the steps for the induction. We always suppose here that $u$ develops in the paracontrolled system associated to \eqref{eq_gPAM}.
	
	It will be convenient to use the extension $\scrT_{\bm f}$ of the decorated trees regularity structure, that was defined in the proof of Corollary \ref{cor_repmodelledditribParap}. We will use words over the symbols of this extended structure. The model $(\PPi^{\mcZ,u},g^{\mcZ,u})$ we are going to work with is such that each tree $\tau$ will still correspond to $Z_\tau$, and $F_{\tau}$ will correspond to $u_\tau^\#$. We have for instance 
$$
\partial^k_*u_\tau = \widetilde{g}^{\mcZ,u}\big(D^k\Psi(F_\tau)\big).
$$
and 
$$
\widetilde{g}^{\mcZ,u}(\Psi(F))=\partial^k u^{\#} + \sum_{\substack{r\geq 0, \, |\mcI(\nu_r)|<k \\ \nu_r\succ\dots\succ\nu_1\succ\nu}}\partial^k_*\P\big(u_{\nu_r}^\#, Z_{\nu_r/\nu_{r-1}},\dots,Z_{\nu_1/\nu}, Z_{\mcI(\nu)}\big).
$$

		\begin{lemma}\label{lem_deriv_uk_to_u}
		We let $b\in\bfN^{d+1}$ with $|b|<\gamma$ and suppose that $\partial^k_*u_\tau = u_{(\downarrow^k)^*\tau}$ for any $k\in\bfN^{d+1}$ and $\tau\in\mcT$ such that $|k|+|\mcI(\tau)|< |b|$, then
			$$
			\widetilde\partial^b_{*}u = \partial^b_* u.
			$$
		\end{lemma}
	\begin{proof}
		We now prove the identity \eqref{eq_starPartialU}. From the definition of $\partial^b_*$ in \eqref{eqdef_starderivative}
	\begin{align*}
		\partial^b u - \sum_{\substack{|\mcI(\sigma)|<\gamma \\ b=b_1+b_2}} \sum_{\substack{ r\geq  0 \\ |b_1|+r <\gamma-|\mcI(\sigma)| \\ |\mcI(\sigma)|-|b_2|+|r|<0 }} \frac{1}{r!}\binom{k}{k_1} \partial^{b_1+r}_*u_\sigma \,  \sum_{\nu_s\prec\dots\prec\sigma}\partial^{b_2}_*\P_r(Z_{\sigma/\nu_1},\dots, Z_{\mcI(\nu_s)}). 
	\end{align*}
	In the sum  above we have $|b_1|+|r|= |b|-|b_2|+|r|<|b|-|\mcI(\sigma)|$, we can therefore use the induction hypothesis to write $\partial^{b_1+r}_*u_\sigma = u_{(\downarrow^{b_1+r})^*\sigma}$. The left hand side of the identity to prove is then equal to  
	\begin{align*}
		\partial^b u - \sum_{\substack{ |\mcI(\sigma)|<|b| }} \sum_{\substack{ r\geq  0 \\ b=b_1+b_2 }} \frac{1}{r!}\binom{b}{b_1} u_\sigma \,  \sum_{\substack{s\geq0 \\ \nu_s\prec\dots\prec\sigma}}\partial^{b_2}_*\P_r(Z_{\downarrow^{b_1+r}\sigma/\nu_1},\dots, Z_{\mcI(\nu_s)}). 
	\end{align*}
	From  Proposition \ref{prop_derivetoile_Ztau_Neg}, we have 
	$$
	\sum_{\substack{b=b_1+b_2 \\ r\geq 0}} \sum_{\substack{s\geq0 \\ \nu_s\prec\dots\prec\sigma}} \frac{1}{r!}\binom{k}{k_1} \partial^{b_2}_*\P_r(Z_{\downarrow^{b_1+r}\sigma/\nu_1},\dots, Z_{\mcI(\nu_s)}) = \overline{Z}_{\mcI_b(\sigma)}.
	$$ 
	One recovers indeed $\partial^b_*u$ from its definition.
	\end{proof}

	\begin{lemma}\label{lem_halfderivative}
	Let $p,a\in\bfN^{d+1}$ such that $|p|+|a|<\gamma$, suppose that $\partial^q_*u_\nu = u_{(\downarrow^q)^*\nu}$ for any $|q|+|\mcI(\nu)|<  |p|+|a|$. We have 
	\begin{equation}\label{eq_halfstarderiv}
	\partial^p\widetilde\partial^a_* u - \sum_{|\mcI_a(\mu)|\in(0,|p|)} u_\mu \overline{Z}_{\mcI_{p+a}\mu} = \widetilde\partial^{p+a}_* u .
\end{equation}
	\end{lemma}
	
	\begin{proof}

			We have  $\widetilde\partial^{p+a}_* u =\widetilde{g}^{\mcZ,u}(D^{p+a}\Psi(F)  )$, then from Lemma \ref{lem_derivModifModel}
		\begin{align*}
			\widetilde\partial^{p+a}_*u &= \widetilde{g}^{\mcZ,u}\big(D^{p}D^a\Psi(F)\big)
			\\
			&=  \partial^{p} \big( \widetilde{g}^{\mcZ,u}\big(D^a\Psi(F)\big)\big) \\
			&\enskip - \sum_{\substack{\gamma-|\mcI(\nu)|-|q|>0 \\ |\mcI(\nu)|+|q|-|a|-|p|<0 }} \widetilde{g}^{\mcZ,u}\big(D^q\Psi(F_\nu)\big) \,\, \partial^{p}_y \Big( \Pi_x^\mcZ D^a\Psi_q(\mcI(\nu))  (y)   \Big)_{|y=x}. 
		\end{align*}
		In the sum above, we have $|q|+|\mcI(\nu)|<|p|+|a|$. We use the induction hypothesis to write $\widetilde{g}^{\mcZ,u}(D^q\Psi(F_\nu))= \widetilde{g}^{\mcZ,u}(\Psi(F_{(\downarrow^q)^*\nu}))$, then using Proposition \ref{Prop_CommutDerivativePsi} and \eqref{eq_PsinAndPsi}
		\begin{align*}
			\partial^{p+a}_*u &= 
			\partial^{p} \partial^a_*u \sum_{\substack{|\mcI_b(\nu)|\in(0,|p|)\\ q\geq  0}} \widetilde{g}^{\mcZ,u}(\Psi(F_\nu)) \,\, \partial^{p}_y \Big( \Pi_x^\mcZ( D^a\Psi_q(\downarrow^q\mcI(\nu)))  (y)   \Big)_{|y=x}
			\\
			 &= 
			\partial^{p} \partial^a_*u \sum_{\substack{|\mcI_a(\nu)|\in(0,|p|) \\ q\geq 0 }} u_\nu \,\, \partial^{p}_y \Big( \Pi_x^\mcZ( D^a\Psi(\mcI(\nu)))  (y)   \Big)_{|y=x}
			\\
			 &= 
			\partial^{p} \partial^a_*u \sum_{\substack{|\mcI_a(\nu)|\in(0,|p|) }} u_\nu \,\, \partial^{p}_y \Big( \Pi_x^\mcZ( \Psi(\mcI_a^+(\nu)))  (y)   \Big)_{|y=x}.
		\end{align*}		
		Using Proposition \ref{prop_derivetoile_Ztau_Neg}, we get indeed
		\begin{align*}
		  \partial^{p} \partial^a_*u - \sum_{|\mcI_a(\nu)|\in(0,|p|)}  u_\nu \overline Z_{\mcI_{p+a}(\nu)}. 
		\end{align*}
		
	\end{proof}
	
	\begin{lemma}\label{lem_deriv_uk_induction}
		We let $\tau\in\mcT$ and $k\in\bfN^{d+1}$ with $|\mcI(\tau)|+|k|<\gamma$.
 		Suppose $\partial^q_*u_\sigma = u_{(\downarrow^q)^*\sigma} $ for any $q\in\bfN^{d+1}$ and $\sigma\in\mcT$ such that $|\mcI(\sigma)|+|q|<|\mcI(\tau)|+|k|$. We suppose also that \ref{eq_halfstarderiv} holds true for $|p|+|a|<|\mcI(\tau)| + |k|$ and $\widetilde\partial^b_*u= \partial^b_*u$ for $|b|<|k|+|\mcI(\tau)|$.
 		Then
 		$$
 		\partial^k_*u_\tau = u_{(\downarrow^k)^*\tau}.
 		$$
	\end{lemma}
	\begin{proof}

	For $k\in\bfN^{d+1}$, from the definition of $\partial^k_*\P$ given in \eqref{eqdef_starderivative}
		\begin{align*}
			\partial_*^{k} u_\tau &= \partial^{k} u_\tau  -\sum_{\substack{\sigma\succ\tau \\ q\geq0 }} \sum_{\substack{ \gamma-|\mcI(\sigma)|-|q| \\ |\sigma/\tau|+|q|-|k|<0           }} \partial^q_*u_\sigma\times \\ & \hspace{3cm}  \sum_{\substack{q=q_1+q_2\\ \tau\prec\sigma_r\prec\dots\prec\tau }} \frac{1}{q_2!}\binom{k}{q_1}\partial^{k-q_1}_*\P_{q_2}(Z_{\sigma/\sigma_1},\dots,Z_{\sigma_r/\tau}).
		\end{align*}
		In the sum above, we have $|\mcI(\sigma)|+|q|<|\mcI(\tau)|+|k|$, then from our assumptions we get 
		\begin{align*}
		&	\partial_*^{k} u_\tau 
		\\ & = \partial^{k} u_\tau  -\sum_{\substack{\sigma\succ\tau \\ q\geq0 }} \sum_{\substack{ \gamma-|\mcI(\sigma)|-|q| \\ |\sigma/\tau|+|q|-|k|<0           }} u_{(\downarrow^q)^*\sigma} \,\,  \sum_{\substack{q=q_1+q_2\\ \tau\prec\sigma_r\dots\prec\tau }} \frac{1}{q_2!}\binom{k}{q_1}\partial^{k-q_1}_*\P_{q_2}(Z_{\sigma/\sigma_1},\dots,Z_{\sigma_r/\tau})
			\\
			&= \partial^{k} u_\tau  -\sum_{\substack{\sigma\succ\tau \\ q\geq0 }} \sum_{\substack{ \gamma-|\mcI(\sigma)| \\ |\sigma/\tau|-|k|<0           }} u_\sigma \, \, \sum_{\substack{q=q_1+q_2\\ \tau\prec\sigma_r\dots\prec\tau }} \frac{1}{q_2!}\binom{k}{q_1}\partial^{k-q_1}_*\P_{q_2}(Z_{\downarrow^q\sigma/\sigma_1},\dots,Z_{\sigma_r/\tau}).
		\end{align*}
	From Proposition \ref{prop_derivetoile_Ztau_Neg}, we have 
	$$\sum_{\substack{q=q_1+q_2\\ \tau\prec\sigma_r\dots\prec\tau }} \frac{1}{q_2!}\binom{k}{q_1}\partial^{k-q_1}_*\P_{q_2}(Z_{\sigma/\sigma_1},\dots,Z_{\sigma_r/\tau}) = \partial^k_y \big(\Pi_x^\mcZ \Psi(\sigma/\tau)(y)\big)_{|y=x}.$$
	Writing $\sigma/\tau = X^p \prod_{j=1}^{N} \mcI^+_{b_j}(\mu_j)$, we have from the Leibniz rule and from \ref{prop_derivetoile_Ztau_Neg}

	\begin{align*}
	\partial^k_y \Big( \Pi^\mcZ_x \Psi(\sigma/\tau) (y)  \Big)_{|y=x} &=
	 \partial^k_y\Big( (y-x)^p \prod_{j=1}^N \Pi^\mcZ_x \Psi(\mcI_{b_j}(\mu_j) ) \Big)_{|y=x}
	\\
	&= \sum_{\substack{k=p+\sum_{j=1}^N k_j \\ |\mcI_{b_j}(\mu_j)|\in(0,|k_j|) }} k! p! \prod_{j=1}^{N}\frac{1}{k_j!}\overline{Z}_{\mcI_{b_j+k_j}(\mu_j)}.
\end{align*}
We obtain 
		\begin{align*}
			\partial_*^{k} u_\tau  &= \partial^{k} u_\tau -\sum_{\substack{\mu=X^p\prod_{j=1}^N \mcI_{b_j}(\mu_j) \\ k=p+\sum_{j=1}^N k_j\\  |\mcI_{b_j}(\mu_j)|\in(0,|k_j|)}}  u_{\mu\star \tau} \, p!k! \, \prod_{j=1}^N \frac{1		}{k_j!} \overline Z_{\mcI_{b_j+k_j}(\mu_j) }.
		\end{align*}		
For $\mu=X^p\prod_{j=1}^n \mcI^+_{b_j}(\mu_j)^{\beta_j}$, Lemma \ref{lem_ElemDiffMorphism} and the Faà di Bruno formula (Lemma \ref{Faa Di Bruno}) gives   
\begin{align*}
& \Upsilon[\mu\star\tau](v_*) = \prod_{j=1}^n  S(\mu_j)^{\beta_j} u_{\mu_j}^{\beta_j} \, \, \partial^p_\mcX\prod_{j=1}^{n}  D_{b_j}^{\beta_j} \Upsilon[\tau](v_*)
\\
&= \sum_{\substack{p=\sum_{i=1}^m \alpha_i p_i  \\ a_i\in\bfN^{d+1}   }} S(\mu\star\tau)\prod_{j=1}^n  \frac{1}{\beta_j!}u_{\mu_j}^{\beta_j} \prod_{i=1}^m \frac{1}{\alpha_i !} \Big(\frac{\partial_*^{p_i+a_i} u}{p_i !}\Big)^{\alpha_i}    \prod_{i=1}^m\prod_{j=1}^{n}D_{a_i}^{\alpha_i} D_{b_j}^{\beta_j} \Upsilon[\tau](v_*).
\end{align*}	
Then 
		\begin{align*}
			\partial^k_*u_\tau&=\partial^{k} u_\tau - \sum_{ \substack{ 0\leq N\leq M  \\ k=\sum_{j=1}^M \beta_j k_j \\     b_j\in\bfN^{d+1}}}   \sum_{\substack{\prod_{j=1}^N (\mcI^+_{b_j}(\mu_j))^{\beta_j} \\|\mcI_{b_j}(\mu_j)|\in(0,|k_j|)  }} \prod_{j=1}^N  u_{\mu_j}^{\beta_j}\overline Z_{\mcI_{b_j+k_j}(\mu_j)}^{\beta_j} \\&\prod_{j=N+1}^M  \frac{1}{\beta_j!}\Big(\frac{1}{k_j!}\partial_*^{k_j+b_j} u\Big)^{\beta_j}  \prod_{j=1}^{M} D_{b_j}^{\beta_j} \Upsilon[\tau](v_*).
		\end{align*}	
	From the Faà di Bruno formula (Lemma \ref{Faa Di Bruno})
		$$
		\partial^{k}u_\tau=\partial^{k}\Upsilon_F[\tau](v_*) = \sum_{\substack{k=\sum_{j=1}^M \beta_jk_j    \\ b_j\in\bfN^{d+1}  }} \prod_{j=1}^{M} \frac{1}{\beta_j!} D_{b_j}^{\beta_j}\Upsilon_F[\tau](v_*) \prod_{j=1}^{M}\frac{1}{\beta_j!}\Big(\frac{1}{k_j!}\partial^{k_j} \partial^{b_j}_*u \Big)^{\beta_j}.
		$$
	Then  
		\begin{align*}
		&	\partial_*^{k} u_\tau  = \sum_{\substack{k=\sum_{j=1}^{M} \beta_jk_j  \\ c_j\in\bfN^{d+1}  }} \prod_{j=1}^M D_{b_j}^{\beta_j}\Upsilon_F[\tau](v_*)  \Big\{ \prod_{j=1}^M \frac{1}{\beta_j!}\Big(\frac{1}{k_j!}\partial^{k_j} \partial^{b_j}_*u \Big)^{\beta_j} 
		\\ &	- \sum_{\substack{\prod_{j=1}^N (\mcI_{b_j}\mu_j)^{\beta_j} \\ |\mcI_{b_j}(\mu_j)|\in(0,|k_j|)}}    \prod_{j=1}^N \frac{1}{\beta_j!}  u_{\mu_j}^{\beta_j} \overline Z_{\mcI_{b_j+k_j}(\mu_j)}^{\beta_j} \prod_{j=N+1}^M \frac{1}{\beta_j!} \Big(\frac{1}{k_j!}\partial_*^{k_j+b_j} u\Big)^{\beta_j}   \Big\}.
		\end{align*}		
		We use \eqref{eq_halfstarderiv} under the binomial identity
		\begin{equs}
	&	\big(\partial^k\partial^b_* u\big)^{\beta}
		-
		\sum_{\beta'<\beta}\binom{\beta}{\beta'} \Big(\sum_{|\mcI_b(\mu)|\in(0,|k|)} u_\mu \overline{Z}_{\mcI_{b+k}\mu}\Big)^{\beta-\beta'} (\partial^{k+b}_*u)^{\beta'}  
	 	= 
		(\partial^{k+b}_* u)^\beta 
		\end{equs}
		to get
			\begin{equs}
		\partial^k_*u_{\tau} & = \sum_{\substack{k=\sum_{j=1}^{M} \beta_jk_j   \\ b_j\in\bfN^{d+1}  }} \prod_{j=1}^{M} D_{b_j}^{\beta_j}\Upsilon_F[\tau](v_*)    \prod_{j=1}^{M}\frac{1}{\beta_j!}\Big(\frac{1}{k_j!} \partial^{k_j+b_j}_*u \Big)^{\beta_j}
		\end{equs}
	which is indeed equal to $u_{(\downarrow^k)^*\tau}$ from Lemma \ref{lem_ElemDiffMorphism} and from the Faà di Bruno formula (Lemma \ref{Faa Di Bruno}).
	\end{proof}

	We are now able to prove the Proposition \ref{lem_starDerivUpsilon}.

	\begin{proof}[of Proposition \ref{lem_starDerivUpsilon}]
	We prove by induction on $\alpha$ the following statement 
	$$
	(E_\alpha) : \left\{ \begin{array}{lll}
	&\forall |k|+|\mcI(\tau)|<\alpha , \qquad  &\partial^k_* u_\tau = u_{(\downarrow^k)^*\tau}
	\\
	&\forall |b|<\alpha, \qquad &\widetilde{\partial}^b_*u = \partial^b_*u
	\\
&\forall |a|+|p|< \alpha, \qquad 	&\partial^p\widetilde\partial^a_* u - \sum_{|\mcI_a(\mu)|\in(0,|p|)} u_\mu \overline{Z}_{\mcI_{p+a}(\mu)} = \widetilde\partial^{p+a}_* u 	
\end{array}\right.
	$$
where $\alpha$ is an element of  $A = (\alpha_0<\alpha_1<\dots)$ the set of degrees of the decorated trees in $\mcT$.
From the conjunction of the Lemmas \ref{lem_deriv_uk_to_u} \ref{lem_halfderivative} and \ref{lem_deriv_uk_induction}, we see that $(E_{\alpha_j})$ implies  $(E_{\alpha_{j+1}})$ for any $j\geq 0$ with $\alpha_{j+1}<\gamma$. we obtain in particular the proposition.
	\end{proof}

	\begin{corollary}\label{cor_welldefinition_starpartialu}
		If $u$ develops in the paracontrolled system associated to \eqref{eq_gPAM} at order $\gamma$, then for any $|k|<\gamma$ the derivative $\partial^k_*u$ is in $L^\infty$, and furthermore $\norme{\partial^k_* u}_{L^\infty}\lesssim \norme{u}_{\mcD^\gamma_\P}$. 
	\end{corollary}
	
	\begin{proof}
		We write $\partial^k_*u$ as in  \eqref{eq_starPartialU}, and we use  \cite[Theorem 1]{LocalExpansionsParacSystems} which asserts in particular that any star derivative $\partial^k_*$ is in $L^\infty$ and depends continuously on its arguments, as $|k|$ is smaller than the sum of degrees associated to each distribution, which is $\gamma$. 
	\end{proof}

	\subsubsection{Paralinearisation}
	
	\begin{lemma}  \label{lem_LiftGpamsystemToRS}
		For any admissible collection of stochastic fields $\mcZ=(Z_\tau)$, if the distribution $u$ develops in a paracontrolled system associated with \eqref{eq_gPAM} at order $\gamma > 0$, then the function $\partial_*^{k}u$ admits a lift as a modelled distribution on the regularity structure $\scrT_\P$ with model induced by $\mcZ=(Z_\tau)$, given by
		$$
		\bm{\partial_{*}^{k} u} = \sum_{|\mcI(\tau)|<\gamma } u_\tau \, \Psi(\mcI_{k}^+ (\tau))  + \sum_{ |n|\leq \gamma - |k|} \frac{1}{n!}\partial_*^{k + n}u \, X^{n}.
		$$
		and we have the continuity estimate $$\lVert\bm{\partial_{*}^{k} u}\lVert_{\mcD^{\gamma-|k|}} \lesssim \lVert {u^\#} \lVert_{C^\gamma}  +  \sum_{|\mcI(\tau)|<\gamma}\lVert u_\tau^\#\lVert_{C^{\gamma-|\mcI(\tau)|}}.$$
	\end{lemma}
	
	\begin{proof}
		One can write
		$$
		u = \sum_{ |\mcI(\tau)|<\gamma}\sum_{q\geq 0}\sum_{\tau_q\prec\dots\prec\tau_1\prec\tau} \P\big( u_\tau^\#, Z_{\tau/\tau_1}, \dots, Z_{\mcI(\tau_q)}    \big) + u^\#.
		$$		
	And for $|k|<\gamma$ we have from \eqref{eq_starPartialU}
		$$
		\partial_*^{k}u = \sum_{|\mcI(\tau)|<\gamma}\sum_{q\geq 0}\sum_{\tau_q\prec\dots\prec\tau_1\prec\tau} \partial_*^{k}\P\big( u_\tau^\#, Z_{\tau/\tau_1}, \dots, Z_{\mcI(\tau_q)}    \big) + \partial^{k}u^\#.
		$$
		Then, from \cite[Theorem 1]{LocalExpansionsParacSystems}, applied to the uplet of distributions $(u_\tau^\#,\dots,Z_{\mcI(\tau_q)})$, the function 
		$\partial_*^{k}\P\big( u_\tau^\#, Z_{\tau/\tau_1}, \dots, Z_{\mcI(\tau_q)}    \big)$ admits a lift as a modelled distribution
		\begin{align*}
		&\sum_{\substack{0\leq j\leq q \\ k=k_1+k_2 \\ r=r_1+r_2 }} \frac{1}{r_1!r_2!}\binom{k}{k_1}\partial^{k_1+r}_*\P\big( u_\tau^\#, Z_{\tau/\tau_1}, \dots, Z_{\tau_{j-1}/\tau_{j}}\big) \, \lg \tau_j/\tau_{j+1},\dots \mcI(\tau_q)   \rg^{k_2}_{r_1}X^{r_2}
		\\
		&+\sum_{|r|<\gamma-|\mcI(\tau_q)|-|k|} \frac{1}{r!} \partial_*^{k+r}\P\big( u_\tau^\#, Z_{\tau/\tau_1}, \dots, Z_{\mcI(\tau_q)}    \big) \,  X^r,
		\end{align*}
		where the first sum over $r$ is taken over $r_1,r_2$ such that $|k_1|+|r|<\gamma-|\mcI(\tau_j)|$ and $|\mcI(\tau_j)|-|k_2|+|r_1|>0$.
		 Summing these lifts, we obtain the following lift of $\partial^{k}_*u$ as a modelled distribution
		\begin{align*}
			{\bm \partial_*^{k}}{\bm u} &= \sum_{|n|<\gamma-|k|}\partial_*^{k+n} u \, \frac{X^n}{n!} 
			\\&+ \sum_{\substack{|\mcI(\sigma)|<\gamma \\k=k_1+k_2 }}\sum_{\substack{q\geq 0 \\ \sigma_q\prec\dots\prec\sigma}}\sum_{\substack{ r=r_1+r_2 \geq0 \\ \gamma-|\sigma|-|k_1|-|r|>0\\|\mcI(\sigma)|+|r|-|k_2|>0 }} \frac{1}{r_1!r_2!}\binom{k}{k_1}\partial^{k_1+r}_* u_\sigma \, \lg \sigma/\sigma_1,\dots,\mcI(\sigma_q)    \rg^{k_2}_{r_1}X^{r_2} .
		\end{align*}
		From Proposition \ref{lem_starDerivUpsilon}, we get 
		\begin{align*}
				{\bm \partial_*^{k}}{\bm u} &= \sum_{|n|<\gamma-|k|}\partial_*^{k+n} u \, \frac{X^n}{n!} 
				+ \sum_{\substack{|\mcI(\sigma)|<\gamma \\ k=k_1+k_2} }\sum_{\substack{q\geq 0 \\ \sigma_q\prec\dots\prec\sigma}} \sum_{\substack{ r=r_1+r_2 \geq0 \\ |\mcI(\sigma)|-|k|-|r_2|>0 }} \\ & \frac{1}{r_1!r_2!}\binom{k}{k_1}   u_\sigma \,\, \lg \downarrow^{k_1+r} \sigma/\sigma_1,\dots,\mcI(\sigma_q)    \rg^{k_2}_{r_1}X^{r_2} .
	\end{align*}
	We verify from the Leibniz formula that for any tree $\sigma$
	\begin{align*}&\sum_{\substack{q\geq 0 \\ \sigma_q\prec\dots\prec\sigma}}\sum_{\substack{k=k_1+k_2\\ r=r_1+r_2\geq 0}}\frac{1}{r_1!r_2!}\binom{k}{k_1} u_\sigma \mathbf{1}_{|\mcI(\sigma)|-|k|-|r_2|>0}\lg \downarrow^{k_1+r} \sigma/\sigma_1,\dots,\mcI(\sigma_q)    \rg^{k_2}_{r_1}X^{r_2} 
		\\
		&= D^k \Psi(\mcI(\sigma)),
		\end{align*}
	so that 
	\begin{align*}
		{\bm \partial_*^{k}}{\bm u} &= \sum_{|n|<\gamma-|k|}\partial_*^{k+n} u \, \frac{X^n}{n!} 
	+ \sum_{|\mcI(\sigma)|<\gamma } u_\sigma \,\, D^k\Psi(\mcI(\sigma)).
	\end{align*}
	We have from Proposition \ref{Prop_CommutDerivativePsi} the equality of increments $$\Pi_x^\mcZ\big( D^k\Psi(\mcI(\sigma))\big)=\Pi^\mcZ_x\left(\Psi(\mcI_k^+(\sigma))\right),$$ then the lift above for $\partial^k_*u$ is indeed equivalent to $$	\bm{\partial_{*}^{k} u} = \sum_{|\mcI(\tau)|<\gamma } u_\tau \, \Psi(\mcI_{k}^+ (\tau))  + \sum_{ |n|\leq \gamma - |k|} \frac{1}{n!}\partial_*^{k + n} u \, X^{n}.$$
	\end{proof}
	
	It is clear that if some function $u$ can be developed in the paracontrolled system associated with \eqref{eq_gPAM}, then $u$ satisfies the ansatz up to the same order, the following Proposition \ref{prop_pamansatz_to_pamsystem} shows that the two notions are actually equivalent. We first need the following paralinearisation result.

 	\begin{theorem}\label{prop_paralinearization}
		Suppose that $\mcZ=(Z_\tau)$ is an admissible collection of stochastic fields and that $u$ can be developed in the paracontrolled system associated with \eqref{eq_gPAM} to order $\gamma$,.
For any smooth function $h$ depending on $v_* = (\partial^k_{*} u)_{k \in A}$ where $A$ is a finite subset of $\N^{d+1}$ such that $r_h<\gamma$ with $r_h =\sup \{ |k|,k\in A\} $, one has 
			$$
			h(v_*) = \sum_{n\geq 0}\sum_{\substack{\tau=  X^r \prod_{i=1}^{n}\mcI^+_{a_i}(\tau_i)^{\beta_i}\in\mcT^+ \\ |\tau|<\gamma-r_h}} \frac{1}{ r!}\P\Big( \partial^r \prod_{i=1}^n D_{a_i}^{\beta_i} h(v_*) \prod_{j=1}^{n} \frac{u_{\tau_j}^{\beta_j}}{\beta_j!} \, ,Z_\tau\Big)  + h(v)^\#,
			$$
			where the $ \mathcal{I}_{a_i}(\tau_i) $ are pairwise distinct. One has $h(v)^\#\in C^{\gamma-r_h}$, and that $\bm{u}\in\mcD^\gamma_\P(\mcZ) \mapsto h(v_*)^\#\in C^{\gamma-r_h}$ is local Lipschitz.
	\end{theorem}
	
	\begin{proof}
	Lemma \ref{lem_LiftGpamsystemToRS} ensures that $\partial^k_{*}u$ admits the lift
	$$	\bm{\partial_{*}^{k} u} = \sum_{|\mcI(\tau)|<\gamma } u_\tau \, \Psi(\mcI_{k}^+ (\tau))  + \sum_{ |n|\leq \gamma - |k|} \frac{1}{n!}\partial_*^{k + n} u X^{n}$$
	as a modelled distribution in $\mcD^{\gamma-|k|}(\scrT_\P)$. Then, from  \cite[Theorem 4.16]{Hai14}, one gets
	$$
	{\bm h}\bm{(v_{*} )} = \mathcal{P}_{\gamma} \sum_{n \geq 0 } \sum_{\substack{ a_1,...,a_n \\ \beta_1,...,\beta_n}} \prod_{i=1}^n \frac{1}{\beta_i !} ( 	\bm{\partial_{*}^{a_i} u} -  \partial_{*}^{a_i} u \one)^{\beta_i} D_{a_i}^{\beta_i} h(v_*),
	$$
	where $\mathcal{P}_{\gamma}$ keeps only the elements of degree $\gamma$. After expanding the various power of $\beta_i$, one gets
	\begin{equation*}
		\begin{aligned}
			&	{\bm h}\bm{(v_{*} )} = \mathcal{P}_{\gamma} 
			\sum_{{\color{black} X^k \prod_{i=1}^m\CI_{a_i}(\tau_i)^{\beta_i}}}
			\sum_{a_{m+1},...,a_{m+n}} \, \sum_{k = \sum_{i=m+1}^{m+n} \beta_i k_i} X^k
			\\ & \prod_{i=1}^{m} \frac{1}{\beta_i !} \left( u_{\tau_i} \Psi(\CI_{a_i}^+(\tau_i)) \right)^{\beta_i}     \prod_{i=m+1}^{n} \frac{1}{\beta_i !} \left(  \frac{ \partial_{*}^{a_i +k_i} u}{k_i!} \right)^{\beta_i} \Big\{ \prod_{i=1}^{m+n} (D_{a_{i}})^{\beta_{i}} h(v_*) \Big\}.
		\end{aligned}
	\end{equation*}
	where the $ \CI_{a_i}(\tau_i) $ are pairwise disjoint.
	We use the Fa\`a di Bruno formula (see Lemma \ref{Faa Di Bruno}) 
	$$
	\frac{\partial^{k}h(v_*)}{k!} =  \sum_{a_1,...,a_r} \, \sum_{k = \sum_{i=1}^r \beta_i k_i} \; \prod_{i=1}^r \frac{1}{\beta_i !} \left(\frac{\partial_{*}^{a_i +k_i} u}{k_i!}\right)^{\beta_i} \prod_{i=1}^r (D_{a_i})^{\beta_i} h(v_*).
	$$
	to get
	\begin{equation*}
		\begin{aligned}
			&	{\bm h}\bm{(v_{*})} = \mathcal{P}_{\gamma} 
			\sum_{{\color{black}  X^k \prod_{i=1}^n\CI_{a_i}^+(\tau_i)^{\beta_i}}}
			\frac{X^k}{k!}
			\prod_{i=1}^{n} \frac{1}{\beta_i !} \left( u_{\tau_i} \Psi(\CI_{a_i}^+(\tau_i)) \right)^{\beta_i}     \Big\{ \partial^k \prod_{i=1}^{n} (D_{a_{i}})^{\beta_{i}} h(v_*) \Big\}.
		\end{aligned}
	\end{equation*}
	And Corollary \ref{cor_repmodelledditribParap} gives the paracontrolled representation
	\begin{align*}
		h(v_*) & = \sum_{|X^k \prod_{i=1}^n\CI_{a_i}^+(\tau_i)^{\beta_i} | \leq \gamma}  \P \Big( \prod_{i=1}^n u_{\tau_i}^{\beta_i} \partial^k \prod_{i=1}^{n} (D_{a_{i}})^{\beta_{i}} h(v_*) , \,\,  Y^\mcZ_{X^k \prod_{i=1}^n\Psi(\CI_{a_i}^+(\tau_i))^{\beta_i} } \Big) \\ & + (h(v_*))^\#
	\end{align*}
	with $h(v_*)^\#\in C^\gamma$. 
	We conclude from  Proposition \ref{prop_Ztau=YPsitau_mult} that gives
	$$
	Z_{X^k \prod_{i=1}^n (\CI_{a_i}^+(\tau_i))^{\beta_i} }  = Y^\mcZ_{X^k \prod_{i=1}^n\Psi(\CI_{a_i}^+(\tau_i))^{\beta_i} }.
	$$
\end{proof}

\begin{example}
	We consider the following paracontrolled ansatz for the stochastic multiplicative heat equation at order $ \gamma =  1 + \kappa $: 
	\begin{equs}
		u = \P(f(u), Z_{\<IXi>} ) + \P(f(u)f'(u), Z_{\<Xiiibis>}) + u^{\#}
	\end{equs}
	One has
	\begin{equs}
		\Psi(\<IXi>) & = \<IXi>, \quad	\Psi(\<Xiiibis>)  = \<Xiiibis> + \<IXi> \otimes \<IXi>,
		\\
		Z_{\<IXi>} &= K * \xi, \quad
		Z_{\<Xiiibis>}  = K * ((K*\xi) \xi) - \P( K*\xi, K*\xi),
	\end{equs}
	Then the lifts are given by
	\begin{equs}
		\bm{\partial_{*}^{0} u} & = \sum_{|\mcI(\tau)|< \gamma } u_\tau \, \Psi(\mcI^+ (\tau))  + \sum_{ |n|\leq \gamma } \frac{1}{n!}\partial_*^{ n} u X^{n}
		\\
		& = f(u) \,  \<IXi> + f(u) f'(u) (\<Xiiibis> + \<IXi> \otimes \<IXi>  ) +  u \,  \one + \partial_{*}^{(0,1)} u   X^{(0,1)}
		\\
		\bm{\partial_{*}^{(0,1)} u} & =   \partial_{*}^{(0,1)} u \,  \one
	\end{equs}
	where
	\begin{equs}
		\partial_{*}^{0} u & = u, \quad \partial_{*}^{(0,1)} u = \partial_x u -   f(u) \overline{Z}_{\<IiiXi>} - f(u) f'(u) \overline{Z}_{\<iXiiibis>}
		\\  \overline{Z}_{\<IiiXi>} & = \partial_x K * \xi, \quad \overline{Z}_{\<iXiiibis>} =  \partial_x K * ( (K*\xi) \xi)  -K * \xi \,\, \partial_x K*\xi.
	\end{equs}
	The paralinearisation at order $ \gamma $ is given by
	\begin{equs}
		h(u) & = \P(f(u) h'(u), Z_{\<IXi>} ) + \P(f(u)f'(u) h'(u), Z_{\<Xiiibis>}) \\ & \quad +
		\frac{1}{2}\P(f(u)^{2} h''(u), Z_{\<Xiiiaa>}) +  h(u)^{\#}.
	\end{equs}
\end{example}

	\begin{corollary}\label{cor_paralinearistaion2}
		For any tree $\tau\in\mcT$, we have the identity
		$$
			\frac{\Upsilon_F[\tau](v_*)}{S(\tau)} =  \sum_{\substack{ |\mu|<\gamma- r_\tau}} \frac{1}{S(\mu) S(\tau)}\P\Big(  \Upsilon_F[\mu \star \tau](v_*)  \, ,Z_\mu\Big)  + u_\tau^\#
		$$
		where  $r_\tau=\lfloor |\mcI(\tau)| \rfloor$, and $u_\tau^\# \in C^{\gamma-r_\tau}$.
	\end{corollary}
	
	\begin{proof}
		We are going to use Theorem \ref{prop_paralinearization} in order to paralinearise the elementary differentials $ \Upsilon_F[\tau](v_*)$.
		Lemma \ref{lem_ElemDiffMorphism} gives
		\begin{equation}\label{eq_morphismUpsilon_ProofCoro}
			\Upsilon_F\Big[X^k\prod_{i=1}^{n}\mcI^+_{a_i} (\tau_i)^{\beta_i} \star \tau  \Big] = \prod_{i=1}^{n} \Upsilon_F[\tau_i]^{\beta_i} \partial^k \prod_{i=1}^n D_{a_i}^{\beta_i} \Upsilon_F[\tau].
		\end{equation}
		From Lemma \ref{lem_elementary differential_form}, $\Upsilon_F[\tau](v_*)$ depends only on the $\partial^j_*u$ for $|j|\leq r_\tau$ . Then from Theorems \ref{prop_paralinearization} and \eqref{eq_morphismUpsilon_ProofCoro} we have
		\begin{equation*}
			\begin{aligned}
			& 	\frac{\Upsilon_F[\tau](v_*)}{S(\tau)} 
			\\	& = \sum_{n\geq 0}\sum_{\substack{\mu =  X^k \prod_{i=1}^{n}\mcI^+_{a_i}(\tau_i)^{\beta_i}\in\mcT^+ \\ |\mu|<\gamma-r_\tau}} \frac{1}{  k! S(\tau)}\P\Big( \partial^k \prod_{i=1}^n D_{a_i}^{\beta_i} \Upsilon_F[\tau](v_*) \prod_{i=1}^{n} \frac{u_{\tau_i}^{\beta_i}}{\beta_i!} \, ,Z_\mu\Big)  + u_\tau^\#
				\\ & = \sum_{n\geq 0}\sum_{\substack{\mu=  X^k \prod_{i=1}^{n}\mcI^+_{a_i}(\tau_i)^{\beta_i}\in\mcT^+ \\ |\mu|<\gamma-r_\tau}} \frac{1}{ k!S(\tau) } \prod_{i=1}^n \frac{1}{\beta_i ! S(\tau_i)}\P\Big(  \Upsilon_F[\mu \star \tau](v_*)  \, ,Z_\mu\Big)  + u_\tau^\#
				\\ & = \sum_{\substack{ |\mu|<\gamma-r_\tau}} \frac{1}{S(\mu) S(\tau)}\P\Big(  \Upsilon_F[\mu \star \tau](v_*)  \, ,Z_\mu\Big)  + u_\tau^\#
			\end{aligned}
		\end{equation*}
		with $u_\tau^\#\in C^{\gamma-r_\tau}$.
	\end{proof}

	\begin{proposition}\label{prop_pamansatz_to_pamsystem}
		If a function $u$ verifies the paracontrolled ansatz of \eqref{eq_gPAM} up to order $\gamma$ given below
		\begin{equation*}
			\begin{aligned}
			u=\sum_{ |\mcI(\tau)|<\gamma} \P ( u_\tau , \, Z_{\mcI(\tau)} ) + u^\#,
	\end{aligned}
			\end{equation*}
				with $u^\#\in C^\gamma$.
		 Then $u$ can be developed in the  paracontrolled system associated with \eqref{eq_gPAM} up to order $\gamma$, as defined in Definition \ref{def_PAMparacSystem}. 
	\end{proposition}
	\begin{proof}
		We let $A=(\alpha_1<\alpha_2<\dots)$ be the set of degrees of the elements $\mcI(\tau)$ for $\tau\in\mcT$, and we prove the proposition for $\gamma\in (\alpha_j,\alpha_{j+1}]$ by induction on $j$. As the proposition is apparent for $j=1$, we suppose that $u$ satisfies the paracontrolled ansatz of \eqref{eq_gPAM} up to order $\gamma\in(\alpha_j,\alpha_{j+1}]$ with $j\geq 2$. 
		
	In particular $u$ satisfies the paracontrolled ansatz up to order $\alpha_j$, and from the induction hypothesis it develops in the paracontrolled system at order $\alpha_j$.
	 Then for any tree $\tau$ with $|\mcI(\tau)|<\gamma$, we have from Corollary \ref{cor_paralinearistaion2}
		\begin{equation*}\label{eq_devatorderAlphaj-rtau}
			\begin{aligned}
		\frac{\Upsilon_F[\tau](v_*)}{S(\tau)} 
		 & = \sum_{\substack{ |\mu|<\alpha_j-r_\tau}} \frac{1}{S(\mu) S(\tau)}\P\Big(  \Upsilon_F[\mu \star \tau](v_*)  \, ,Z_\mu\Big)  + u_\tau^\#,
		\\& = \sum_{\substack{\sigma \succ\tau \\ |\sigma|<\alpha_j-r_\tau+|\tau|  }} \frac{1}{S(\sigma)}\P\Big(\Upsilon_F[\sigma](v_*), \, Z_{\sigma/\tau}\Big)  + u_\tau^\# 
		\end{aligned}
		\end{equation*}
		where we have used
	\begin{equation*}
		\langle \frac{\mu}{S(\mu)} \star  \frac{\tau}{S(\tau)} , \sigma  \rangle  = \langle \frac{\mu}{S(\mu)} \otimes  \frac{\tau}{S(\tau)} , \Delta \sigma  \rangle.
	\end{equation*}
		One has $u_\tau^\#\in C^{\alpha_j-r_\tau}$. We would like to truncate the expansion \eqref{eq_devatorderAlphaj-rtau} at order $\gamma-|\mcI(\tau)|$
		$$
		\frac{\Upsilon_F[\tau](v_*)}{S(\tau)} 
		= \sum_{\substack{\sigma \succ\tau \\ |\mcI(\sigma)|<\gamma  }} \frac{1}{S(\sigma)}\P\Big(\Upsilon_F[\sigma](v_*), \, Z_{\sigma/\tau}\Big)  + u_\tau^\# 
		$$
		for another $u_\tau^\#$ that belongs to $C^{\gamma-|\mcI(\tau)|}$. To do this, we have to check the inequality \begin{equation}\label{eq_ineqalpha_jgammartau}\gamma-|\mcI(\tau)|\leq \alpha_j-r_\tau.
		\end{equation} 
		Setting $\nu_j\in\mcT$ such that $|\mcI(\nu_j)|=\alpha_j$, we have $\alpha_j+|\mcI(\tau)|-r_\tau=|\mcI(  \mcI_{r_\tau}(\tau) \star \nu_j  )|$ and $\alpha_j+|\mcI(\tau)|-r_\tau>\alpha_j$. 
		From the definition of the collection $(\alpha_j)$, we get $\alpha_j+|\mcI(\tau)|-r_\tau\geq\alpha_{j+1}\geq \gamma$, hence \eqref{eq_ineqalpha_jgammartau}.
One can iterate the decomposition on  $ \frac{1}{S(\sigma) }\Upsilon_F[\sigma ]  $ to get the full paracontrolled system.
	\end{proof}

	Proposition \ref{prop_pamansatz_to_pamsystem} enables us to define the paracontrolled ansatz when the stochastic data is not smooth. The main difficulty is that the coefficients $u_\tau$ depend on the derivatives $\partial_*^k u$ with $|k|<|\mcI(\tau)|$, which are not a priori well defined in the non-smooth setting for $\mcZ$. The key idea is to exploit the fact that $u$ can be developed in the paracontrolled system from Definition \ref{def_PAMparacSystem} and to define $\partial_*^k u$ instead through \eqref{eq_starPartialU}. More precisely, we define the paracontrolled ansatz at order $\gamma\in [j,j+1)$ by induction on $j$. We say that $u$ satisfies the paracontrolled ansatz at order $\gamma\in[j,j+1)$ if it already satisfies the ansatz at any order $r$ with $r<\gamma$, in that case $u$ develops in the paracontrolled system at order $r$ and Corollary \ref{cor_welldefinition_starpartialu} ensures that $\partial^j_*u$ is well defined as an element of $L^\infty$. We then finally impose \eqref{eq_pamansatz} at order $\gamma$.

	\subsection{The right hand side of the equation} \label{Subsection_RHS}

	One defines here a map $u\mapsto \Linv (\sum_{\ell=1}^{\ell_0} F_\ell(v) \cdot \xi_\ell  ) $ defined on functions $ v_* = (\partial_{*}^{k} u)_{k \in A}$ satisfying the  paracontrolled ansatz of \eqref{eq_gPAM}. We suppose here that $u$ satisfies the paracontrolled ansatz at order $\gamma = \gamma_0$, where $\gamma_0$ is the exponent defined in Subsection \ref{subsec_subcriticality}.

	\subsubsection{The paracontrolled structure of $\Upsilon_F[\Xi_\ell](v_*)$.}	 \label{subsect_paracstructure_UpsilonXi}

	In Definition \ref{def_PAMparacSystem} of the paracontrolled systems of order $\gamma$, the coefficient $u_{\Xi_\ell}= \Upsilon_F[\Xi_\ell](v_*)$ admits an expansion at order $\gamma-|\mcI(\Xi_\ell)|$. We will need an expansion at a higher order to define the singular product $F_\ell(v)\xi_\ell$, which can be rewritten as $\Upsilon_F[\Xi_\ell](v)\xi_\ell$.
	
	We have set in Subsection \ref{subsec_subcriticality} the quantities $r_{\tau\Xi_\ell}$ and  $\gamma_\ell$. We are going to expand $u_{\Xi_\ell}$ at order $\gamma_\ell$, for $\tau\in\mcT_o$ let 
	\begin{equation}\label{eq_udash_gammaell}
		{u}_{\tau\Xi_\ell}^{\#,\gamma_\ell} = {u}_{\tau\Xi_\ell} - \sum_{\substack{\sigma\succ\tau \\ |\mcI(\sigma\Xi_\ell)|<\gamma_\ell}} \P(   u_{\sigma\Xi_\ell} , Z_{\sigma/\tau}  ).
	\end{equation}

	The following identity is the same as the one from Proposition \ref{lem_starDerivUpsilon}, except that the paracontrolled expansion is performed at the higher order $\gamma_\ell$.
	
	\begin{lemma}\label{lem_starDerivUpsilon_hat}
	We suppose that $u$ can be developed in a  paracontrolled system associated with \eqref{eq_gPAM} at order $\gamma$ with admissible stochastic data $\mcZ=(Z_\tau)_{\tau\in\mcT}$. For any $\ell\in\{1,...,\ell_0\}$ and  $\tau\in\mcT_o$, for any $k\in\bfN^{d+1}$ such that $|k|< \gamma_\ell-|\mcI(\tau\Xi_\ell)|$, we have
	\begin{equation}\label{eq_starPartialUTau_hat} \begin{split}
			&\partial^k u_{\tau\Xi_\ell}^{\#,\gamma_\ell} + \sum_{\substack{r\geq 0, \, |\mcI(\nu_r\Xi_\ell)|<\gamma_\ell \\ \nu_r\succ\dots\succ\nu_1\succ\tau}}\partial^k_*\P\big( u_{\nu_r\Xi_\ell}^{\#,\gamma_\ell}, Z_{\nu_r/\nu_{r-1}},\dots, Z_{\nu_1/\tau}\big) 
			= u_{(\downarrow^k)^* \tau\Xi_\ell}.
	\end{split}\end{equation}
	\end{lemma}
	
	\begin{proof}
		From Proposition \ref{lem_starDerivUpsilon}, it is sufficient to prove
		\begin{align*}
		&\partial^k u_{\tau\Xi_\ell}^{\#,\gamma_\ell} + \sum_{\substack{r\geq 0, \, |\mcI(\nu_r\Xi_\ell)|<\gamma_\ell \\ \nu_r\succ\dots\succ\nu_1\succ\tau}}\partial^k_*\P\big( u_{\nu_r\Xi_\ell}^{\#,\gamma_\ell}, Z_{\nu_r/\nu_{r-1}},\dots, Z_{\nu_1/\tau}\big) 
		\\&=
		\partial^k u_{\tau\Xi_\ell}^{\#} + \sum_{\substack{r\geq 0, \, |\mcI(\nu_r\Xi_\ell)|<\gamma \\ \nu_r\succ\dots\succ\nu_1\succ\tau}}\partial^k_*\P\big( u_{\nu_r\Xi_\ell}^{\#}, Z_{\nu_r/\nu_{r-1}},\dots, Z_{\nu_1/\tau}\big).
		\end{align*}
		It is done by induction, using the recursive definition \eqref{eqdef_starderivative}. We have to check that we have the same set of cuts in both cases, and for the terminal cases we note that
		\begin{align*}
			& u_{\tau\Xi_\ell}^{\#,\gamma_\ell} + \sum_{\substack{r\geq 0, \, |\mcI(\nu_r\Xi_\ell)|<\gamma_\ell \\ \nu_r\succ\dots\succ\nu_1\succ\tau}}\P\big( u_{\nu_r\Xi_\ell}^{\#,\gamma_\ell}, Z_{\nu_r/\nu_{r-1}},\dots, Z_{\nu_1/\tau}\big) 
			\\&=
		 u_{\tau\Xi_\ell}^{\#} + \sum_{\substack{r\geq 0, \, |\mcI(\nu_r\Xi_\ell)|<\gamma \\ \nu_r\succ\dots\succ\nu_1\succ\tau}}\P\big( u_{\nu_r\Xi_\ell}^{\#}, Z_{\nu_r/\nu_{r-1}},\dots, Z_{\nu_1/\tau}\big),
		\end{align*} 
		as both sides are equal to $u_{\tau\Xi_\ell}$.
	\end{proof}

	\begin{proposition}\label{prop_regularite_utaudash_gammaell}
		For any admissible stochastic data $\mcZ$ and $u$ satisfying the paracontrolled ansatz, we have ${u}_{\tau\Xi_\ell}^{\#,\gamma_\ell}\in C^{  \gamma_\ell-|\mcI(\tau\Xi_\ell)|   }.   $
	\end{proposition}

	\begin{proof}
		We take $\tau\Xi_\ell\in\mcT^{nc}$. From Lemma \ref{lem_elementary differential_form}, the elementary differential $\Upsilon[\tau\Xi_\ell](v_*)$ depends on the $\partial^j_*u$ with $|j|\leq r_{\tau\Xi_\ell}$. Theorem \ref{prop_paralinearization} gives the paracontrolled representation
		\begin{align*}
		u_{\tau\Xi_\ell}& = \sum_{n\geq 0}\sum_{\substack{\sigma=  X^r \prod_{i=1}^{n}\mcI^+_{a_i}(\tau_i)^{\beta_i}\in\mcT^+ \\ |\sigma|<\gamma-r_{\tau\Xi_\ell}}} \frac{1}{ r!}\P\Big( \partial^r \prod_{i=1}^n D_{a_i}^{\beta_i} \Upsilon_F[\tau\Xi_\ell](v_*) \prod_{j=1}^{n} \frac{u_{\tau_j}^{\beta_j}}{\beta_j!} \, ,Z_\sigma\Big) \\ &  + u_{\tau\Xi_\ell}^{\#,\gamma-r_{\tau\Xi_\ell}},
		\end{align*}
		with $u_{\tau\Xi_\ell}^{\#,\gamma-r_{\tau\Xi_\ell}}\in C^{\gamma-r_{\tau\Xi_\ell}}$.
		From \eqref{eq_defgamma_ell} we have $$\gamma-r_{\tau\Xi_\ell}\geq \gamma_\ell-|\mcI(\tau\Xi_\ell)|,$$ so that one can truncate this expansion at order $\gamma_\ell-|\mcI(\tau\Xi_\ell)|$. The coherence property from Proposition \ref{prop_pamansatz_to_pamsystem} ensures that the expansion we obtain corresponds indeed to the one in \eqref{eq_udash_gammaell}.
		
		If $\tau\Xi_\ell\notin \mcT^{nc}$ the proposition is apparent as Lemma \ref{lem_ElemDiffMorphism} implies that $u_{\sigma*\tau\Xi_\ell}$ is zero for $\sigma\ne {\bf 1}$, and as any constant function is in $C^\infty$.
	\end{proof}
	

	\subsubsection{The non-linearity $F_\ell(v) \, \xi_\ell$} \label{subsect_productmap}
	
	We define a map $(u,u^\#) \mapsto F_\ell(v)\cdot\xi_\ell$ from the set of distributions satisfying the paracontrolled ansatz of \eqref{eq_gPAM} to the set of distributions by setting
	\begin{equation}\label{eqdef_Produit_Fluxil1}
		F_\ell(v)\cdot\xi_\ell = \sum_{\substack{\tau\in\mcT_o \\ |\mcI(\tau\Xi_\ell) |<\gamma_\ell  }  } \P(u_{\tau\Xi_\ell}, Z_{\tau\Xi_\ell} ) + \Phi_\ell^\#,
	\end{equation}
	with
	\begin{equation}\label{eqdef_Produit_Fluxil2}
		\Phi_\ell^\# = - \sum_{\substack{\tau\in\mcT_o \\ |\mcI(\tau\Xi_\ell)| <\gamma_\ell  }  } \sum_{ \substack{q\geq 0 \\ \tau_q\prec\dots\prec \tau}} {\bf C} \Big(u^{\#,\gamma_\ell}_{\tau\Xi_\ell},Z_{\tau/\tau_1}, \dots, Z_{\tau_q\Xi_\ell}\Big) .
	\end{equation}
	This formula will be justified in Proposition \ref{prop_nonlinearitxiell_generalized} below, which implies that it coincides with the classical products when the stochastic data is the naive one.
	When the stochastic data $\mcZ$ is smooth, we are able to give the following definition of some multiplicative version of the product map.
		\begin{equation}\label{eqdef_Produit_Fluxil1_mult}
		(F_\ell(v)\cdot\xi_\ell)^\circ = \sum_{\substack{\tau\in\mcT_o \\ |\mcI(\tau\Xi_\ell) |<\gamma_\ell  }  } \P(u_{\tau\Xi_\ell}, Z_{\tau\Xi_\ell}^\circ ) + \Phi_\ell^{\circ,\#},
	\end{equation}
	with
	\begin{equation}\label{eqdef_Produit_Fluxil2_mult}
		\Phi_\ell^{\circ,\#} = - \sum_{\substack{\tau\in\mcT_o \\ |\mcI(\tau\Xi_\ell)| <\gamma_\ell  }  } \sum_{ \substack{q\geq 0 \\ \tau_q\prec\dots\prec \tau}} {\bf C} \Big( u^{\#,\gamma_\ell}_{\tau\Xi_\ell},Z_{\tau/\tau_1}, \dots,Z_{\tau_{q-1}/\tau_q} ,Z_{\tau_q\Xi_\ell}^\circ\Big) .
	\end{equation}
Remember that when $\mcZ$ is the naive stochastic data we have $Z_\tau=Z^\circ_\tau$, in particular $(F_\ell(v)\cdot\xi_\ell)^\circ = F_\ell(v)\cdot\xi_\ell$.
		More generally we set for any $\tau\in\mcT_o$ the product
		\begin{equation}\label{eqdef_Produit_ElementaryDiff1}
			\left(\Upsilon_F[\tau\Xi_\ell](v) \cdot \xi_\ell\right)^{\circ} = \sum_{ \substack{\sigma\in\mcT_o \\ |\mcI(\sigma \star \tau\Xi_\ell)|<\gamma_\ell}} \P(u_{\sigma \star \tau\Xi_\ell}, Z^{\circ}_{\sigma\Xi_\ell} ) + \Phi_{\tau,\ell}^{\circ,\#},
		\end{equation}
		with
		\begin{equation}\label{eqdef_Produit_ElementaryDiff2}
			\Phi_{\tau,\ell}^{\circ,\#} = - \sum_{\substack{\sigma\in\mcT_o \\ |\mcI(\sigma \star \tau\Xi_\ell)|<\gamma_\ell    }} \sum_{\substack{q\geq0 \\ \sigma_q\prec\dots\prec\sigma_1\prec\sigma }} {\bf C} \Big(u^{\#,\gamma_\ell}_{\sigma \star \tau\Xi_\ell},Z_{\sigma/\sigma_1}, \dots,Z_{\sigma_{q-1}/\sigma_q}, Z^{\circ}_{\sigma_q\Xi_\ell}\Big) .
		\end{equation}
		
	The formulae above do not explicitly involve resonant products, as one might expect. These resonant products are however hidden in the correctors $\bfC$ appearing above. For instance, one has $$-\bfC(u_\Xi^\#,Z_\Xi) = u_\Xi^\# Z_\Xi - \P(u^\#_\Xi,Z_\Xi) =  u_\Xi^\# \odot Z_\Xi +  \P(Z_\Xi,u^\#).$$
		The next proposition justifies the product formulae above.
			\begin{proposition}\label{prop_nonlinearitxiell_generalized}
		For any smooth admissible stochastic data $\mcZ=(Z_\tau)$  and any $u$ satisfying the paracontrolled ansatz, the distribution $(\Upsilon_F[\tau\Xi_\ell](v)\cdot\xi_\ell)^\circ$ given in \eqref{eqdef_Produit_ElementaryDiff1} and \eqref{eqdef_Produit_ElementaryDiff2} coincides with the classical pointwise product $$(\Upsilon_F[\tau\Xi_\ell](v)\cdot\xi_\ell)^\circ = \Upsilon_F[\tau\Xi_\ell](v)\xi_\ell.$$
	\end{proposition}
	In particular if we take $\tau=\mathbf{1}$, we get indeed
	$
	\left(F_\ell(v)\cdot\xi_\ell\right)^\circ = F_\ell(v)\xi_\ell.
	$
	The more general Proposition \ref{prop_nonlinearitxiell_generalized} will be useful for the computation of the renormalised equation in Section \ref{section_renormalized_equation}.
		\begin{proof}
		We compute the remainder term $\Phi_{\tau,\ell}^{\circ,\#}$ as defined in \eqref{eqdef_Produit_ElementaryDiff2}, using the definition of the corrector ${\bf C}$ given in Subsection \ref{subsect_theRSIteratedParap} via \eqref{eqdef_corrector} and  \eqref{eqdef_starderivative}
		\begin{align*}
			\Phi_{\tau,\ell}^{\circ,\#}  &=
		-\sum_{\substack{\sigma\in\mcT_o \\ |\mcI(\sigma \star \tau\Xi_\ell)|<\gamma_\ell    }} \sum_{\substack{q\geq0 \\ \sigma_q\prec\dots\prec\sigma_1\prec\sigma }} {\bf C} \Big(u^{\#,\gamma_\ell}_{\sigma \star \tau\Xi_\ell},Z_{\sigma/\sigma_1}, \dots, Z^{\circ}_{\sigma_q\Xi_\ell}\Big) 
			\\
			& = -\sum_{\substack{\sigma\in\mcT_o \\ |\mcI(\sigma \star \tau\Xi_\ell)|<\gamma_\ell}} \P(u_{\sigma \star \tau\Xi_\ell}, Z_{\sigma\Xi_\ell}^\circ)  
			\\
			&+  \sum_{\substack{\sigma\in\mcT_o \\ |\mcI(\sigma \star \tau\Xi_\ell)|<\gamma_\ell}} \sum_{\substack{\mu\prec\sigma \\ |k|<\gamma_\ell-|\mcI(\mu \star \tau\Xi_\ell)|}}\sum_{\substack{ \mu\prec\sigma_r\prec\dots\prec\sigma \\ \nu_s\prec\dots\prec\nu_1\prec\mu  }} \frac{1}{k!} \partial^k_*\P(u_{\sigma \star \tau\Xi_\ell}^{\#,\gamma_\ell},\dots, Z_{\sigma_r/\mu}) \\ &\hspace{7cm}\times \bfC_k(Z_{\mu/\nu_1},\dots,Z^{\circ}_{\nu_s\Xi_\ell}).
		\end{align*}
		From Lemma \ref{lem_starDerivUpsilon_hat}, for any $\sigma\in\mcT_o$, one has 
		$$
		\partial^k u_{\mu*\tau\Xi_\ell}^{\#,\gamma_\ell} + \sum_{\substack{r\geq0 \\ \mu\prec \sigma_r\prec\dots\prec\sigma}} \partial^k_*\P( u_{\sigma \star \tau\Xi_\ell}^{\#,\gamma_\ell},\dots, Z_{\sigma_r/\mu})  = u_{(\downarrow^k)^*(\mu*\tau\Xi_\ell) }.
		$$
		And then 
		\begin{align*}
			&\Phi_{\tau,\ell}^{\circ,\#} - \sum_{\substack{\sigma\in\mcT_o \\ |\mcI(\sigma \star \tau\Xi_\ell)|< \gamma_{\ell}  }} \P(u_{\sigma \star \tau\Xi_\ell}, Z^{\circ}_{\sigma\Xi_\ell})  
			\\&=
			 \sum_{ \substack{\mu \in\mcT_o \\ |\mcI(\mu \star \tau\Xi_\ell)|< \gamma_{\ell}   }} \frac {\Upsilon_F[\mu\Xi_\ell](v_*)}{S(\mu \Xi_{\ell})} \sum_{\substack{\nu_s\prec\dots\prec\mu \\ k  \geq 0 }}  \frac{1}{k!}\bfC_k(Z_{\downarrow^k\mu/\nu_1},\dots,Z^{\circ}_{\nu_s\Xi_\ell}    ).
		\end{align*}
		From Proposition \ref{lem_SumNegativeCorrector}, the sum $\sum_k\sum_{\nu_q\prec\dots\prec\mu} \frac{1}{k!}\bfC_k(Z_{\downarrow^k\mu/\nu_1},\dots,Z^{\circ}_{\nu_q\Xi_\ell}    )$ is non-zero only for trees $\mu\in\mcT_o$ with the form 
		$\mu= \prod_{j=1}^n  (\mcI_{b_j} (\mu_j))^{\beta_j}$ with $|\mcI_{b_j} (\mu_j)|<0$ and one has
		\begin{equation*}
			\sum_{\substack{q\geq 0 \\ \nu_q\prec\dots\prec\mu  }} \sum_{k\geq 0} \frac{1}{k!} \bfC_k(Z_{\downarrow^k\mu/\nu_1},\dots,Z^{\circ}_{\nu_q\Xi_\ell}    ) = \xi_\ell \prod_{j=1}^n  \overline Z_{\mcI_{b_j}(\mu_j)}^{\beta_j}.
		\end{equation*}
		From the definition of the elementary differential in \eqref{eq_defUpsilon}, we have
		$$
		\Upsilon_F[\mu \star \tau\Xi_\ell] = \prod_{j=1}^n \Upsilon_F[\mu_j]^{\beta_j} \prod_{j=1}^n D_{b_j}^{\beta_j} \Upsilon_F[\tau\Xi_\ell].  
		$$
		Then,
		\begin{align*}\begin{split}
				&	\Phi_{\tau,\ell}^{\circ,\#} + 
				\sum_{\substack{\sigma\in\mcT_o \\ |\mcI(\sigma \star\tau\Xi_\ell)|< \gamma_{\ell}  }} \P(u_{\sigma \star\tau\Xi_\ell}, Z^{\circ}_{\sigma\Xi_\ell})
					\\
					&= \sum_{ \substack{\mu  = \prod_{j=1}^n \CI_{b_j}(\mu_j)^{\beta_i}   }} \frac {\Upsilon_F[\mu\star \tau\Xi_\ell](v_*)}{S(\mu \Xi_{\ell})} 
				\xi_\ell \prod_{j=1}^n  \overline Z_{\mcI_{b_j}(\mu_j)}^{\beta_j}
				\\
				&  = \sum_{ \substack{\mu  = \prod_{i=1}^n \CI_{b_j}(\mu_j)^{\beta_j}   }}  \prod_{j=1}^n \Upsilon_F[\mu_j](v_*)^{\beta_j} \prod_{j=1}^n D_{b_j}^{\beta_j} \Upsilon_F[\tau\Xi_\ell](v_*) \,  
				\xi_\ell \prod_{j=1}^n \frac{1}{\beta_j ! S(\mu_j)}  \overline Z_{\mcI_{b_j}(\mu_j)}^{\beta_j}
				\\
				&= \Upsilon_F[\tau\Xi_\ell](     (\partial_*^{k} u + \sum_{|\mcI(\tau)|< |k|} u_\tau \overline{Z}_{\mcI_{k}(\tau) } )_{k \in A_-}, (\partial^{n} u)_{n \in A_+}) \xi_\ell
		\end{split}\end{align*}
		where the $\CI_{b_j}(\mu_j)^{\beta_i}$ are pairwise disjoint and  we have used the fact that $\Upsilon_F[\tau\Xi_\ell]$ is polynomial in the $ \partial^k u $ that are distributions.
		From the definition of $\partial_*^{k}u$ given in \eqref{eqdef_partialStarU}, this is indeed equal to $\Upsilon_F[\tau\Xi_\ell](v)\xi_\ell$.
	\end{proof}

	\begin{proposition}
		For any $\ell\in\{1,...,\ell_0\}$, we have $\Phi_\ell^\# \in C^{-\kappa/2}$.
		Furthermore the map $(u,u^\#)\mapsto \Phi_\ell^\#$ is local Lipschitz.
	\end{proposition}
	
	\begin{proof}
		For any tree $\tau$, from Proposition \ref{prop_regularite_utaudash_gammaell} we have $u_{\tau\Xi_\ell}^{\#,\gamma_\ell} \in C^{\gamma_\ell-|\mcI(\tau\Xi_\ell)|}.$ Then the sum of the regularities in the corrector ${\bf C} (u^{\#,\gamma_\ell}_{\tau\Xi_\ell},Z_{\tau/\tau_1}, \dots, Z_{\tau_q\Xi_\ell})$ is equal to 
		$$
		\gamma_\ell-|\mcI(\tau\Xi_\ell)| +|\tau/\tau_1|+\sum_{j=1}^{q-1} |\tau_j/\tau_{j+1}| + |\tau_q\Xi_\ell| = \gamma_\ell-\beta  \geq\kappa>0,
		$$ 
		where we have used inequality \eqref{inequality_gamma}. From \cite[Theorem 1]{LocalExpansionsParacSystems}, the corrector ${\bf C} \big(u^{\#,\gamma_\ell}_{\tau\Xi_\ell},Z_{\tau/\tau_1}, \dots, Z_{\tau_q\Xi_\ell}\big)$ is indeed in $L^\infty \subset C^{-\kappa/2}$.
		
	\end{proof}

	\subsubsection{Integration} \label{subsection_integration}
	
	\begin{lemma} \label{integration_RS}
		For any $u$ satisfying the paracontrolled ansatz of \eqref{eq_gPAM} and any stochastic data $\mcZ=(Z_\tau)$, for any $\ell\in\{1,...,\ell_0\}$ the distribution $F_\ell(v)\cdot \xi_\ell$ admits a lift $\boldsymbol{\Phi}_{\ell}$ in $\mcD^{-\kappa/2 }(\scrT_\P)$ given by 
		$$
		\boldsymbol{\Phi}_{\ell}= \sum_{\substack{\tau\in\mcT_o \\ |\tau\Xi_\ell|<-\kappa/2}} u_{\tau\Xi_\ell} \Psi(\tau\Xi_\ell). 
		$$
	\end{lemma}
	
	\begin{proof}
		The proof is very similar to the proof of Lemma \ref{lem_LiftGpamsystemToRS}. We can write 
		$$
		F_\ell(v)\cdot \xi_\ell =  \sum_{\substack{\tau\in\mcT_o \\  |\mcI(\tau\Xi_\ell)|<\gamma_\ell}} \sum_{\substack{q\geq 0\\ \tau_q\prec\dots\prec\tau}} \P( u_{\tau\Xi_\ell}^{\#,\gamma_\ell}, Z_{\tau/\tau_1},\dots, Z_{\tau_q\Xi_\ell}     )  + \Phi_\ell^{\#}
		$$
		with $u_{\tau\Xi_\ell}^{\#,\gamma_\ell}\in C^{\gamma_\ell-|\mcI(\tau\Xi_\ell)|}$ and $\Phi_\ell^\#\in C^{\gamma_\ell-\beta}$.
		From \cite[Theorem 1]{LocalExpansionsParacSystems} applied to each uplet $(u_{\tau\Xi_\ell}^{\#,\gamma_\ell}, Z_{\tau/\tau_1},\dots, Z_{\tau_q\Xi_\ell} )$, this distribution admits the following lift in $\mcD^{\gamma_\ell-\beta}$
		\begin{align*}
			{\boldsymbol{\widehat\Phi_\ell}} &= \sum_{\substack{\tau\in\mcT_o \\ |\mcI(\tau\Xi_\ell)|<\gamma_\ell }}\sum_{\substack{ \sigma\prec\tau \\ \sigma \prec\sigma_q \prec\dots\prec\mcI(\tau)  \\ \nu_s\prec\dots\prec\nu_1\prec\sigma  }} \sum_{\substack{|k|<\gamma_\ell- |\tau/\sigma| \\ k=k_1+k_2 }} \frac{1}{k_1!\, k_2!}\partial^k_*\P\Big(u_{\tau\Xi_\ell}^{\#,\gamma_\ell}, Z_{\tau/\sigma_1},\dots, Z_{\sigma_q/\sigma}\Big) 
			\\
			&\hspace{7cm} \times \lg {\sigma/\nu_1},\dots , \nu_s\Xi_\ell  \rg_{k_1}X^{k_2} 
			\\
			&+\sum_{|k|<\gamma_\ell-\beta} \Big\{   \sum_{\substack{ \tau\in\mcT_o \\  |\mcI(\tau\Xi_\ell)|<\gamma_\ell}} \sum_{\substack{q\geq 0\\ \tau_q\prec\dots\prec\tau}} \partial^k_*\P\Big( u_{\tau\Xi_\ell}^{\#,\gamma_\ell}, Z_{\tau/\tau_1},\dots, Z_{\tau_q\Xi_\ell}     \Big)  + \partial^k \Phi_\ell^{\#}\Big\} \frac{1}{k!}X^k.
		\end{align*}
		We truncate the lift at order $-\kappa/2$, and apply Lemma \ref{lem_starDerivUpsilon_hat} to obtain the new lift
		\begin{align*}
			\boldsymbol{\Phi_\ell} &= \sum_{\substack{\sigma\in\mcT_o \\ |\sigma\Xi_\ell|< -\kappa/2 }}\sum_{\substack{ \nu_s\prec\dots\prec\nu_1\prec\sigma  }} \sum_{\substack{|k|<-\kappa/2- |\sigma\Xi_\ell| \\ k=k_1+k_2 }} \frac{1}{k_1!\, k_2!} u_{(\downarrow^k)^*(\sigma\Xi_\ell)} \lg {\sigma/\nu_1},\dots , \nu_s\Xi_\ell  \rg_{k_1}X^{k_2}
			\\
			&=  \sum_{\substack{\sigma\in\mcT_o \\ |\sigma\Xi_\ell|< -\kappa/2 }} \sum_{\substack{ \nu_s\prec\dots\prec\nu_1\prec\sigma  }} \sum_{\substack{k\geq 0 \\ k=k_1+k_2 }} \frac{1}{k_1!\, k_2!} u_{\sigma\Xi_\ell} \, \lg (\downarrow^k\sigma)/\nu_1,\dots , \nu_s\Xi_\ell  \rg_{k_1}X^{k_2}
			\\
			&= \sum_{\substack{\sigma\in\mcT_o \\  |\sigma\Xi_\ell|< -\kappa/2     }}  u_{\sigma\Xi_\ell } \Psi( \sigma\Xi_\ell  ).
		\end{align*}	
	\end{proof}

	\begin{proposition}\label{prop_integration}
		We have 
		\begin{equation}\label{eq_prop_integration}
		\scrL^{-1}\Big( \sum_{\ell=1}^{\ell_0} F_\ell(v) \cdot\xi_\ell  \Big) = \sum_{\substack{\tau\in\mcT \\  |\mcI(\tau)|<\gamma'_0 }} \P(u_\tau, Z_{\mcI(\tau)}  ) + \Psi^\#
	\end{equation}
		with $\Psi^\# \in C^{\gamma_0'}$, and furthermore the map $(u,u^\#) \mapsto \Psi^\#$ is local Lipschitz.
	\end{proposition}
	
	\begin{proof}
	For each $\ell\in [\![ 1;\ell_0]\!]$, as in the proof of Corollary \ref{cor_repmodelledditribParap} we introduce an extension of the regularity structure of trees  $\scrT_{\boldsymbol{\Phi}_\ell}=(T_{\boldsymbol{\Phi}_\ell},T^+_{\boldsymbol{\Phi}_\ell})$ by setting  $T_{\boldsymbol{\Phi}_\ell}=T\oplus \bfR \rmF^\ell$ and $T_{\boldsymbol{\Phi}_\ell}^+=T^+\oplus T_{<\gamma}$, with natural basis
	$$
	\mcB_{\boldsymbol{\Phi}_\ell}^+ = \mcB^+ \sqcup \big\{ \rmF^\ell_{\tau\Xi_\ell}, \quad \tau\Xi_\ell \in\mcB_{<\gamma}   \big\}, \quad \text{and}\quad \mcB_{\boldsymbol{\Phi}_\ell} = \mcB \sqcup \{ \rmF^\ell \}.
	$$
	One sets $|\rmF_{\tau\Xi_\ell}^\ell|=-\kappa/2-|{\tau\Xi_\ell}|$, and extends the coproducts from $\scrT$ to $\scrT_{\boldsymbol{\Phi}_\ell}$ by setting
	\begin{align*}
		\Delta^{\!+} \rmF_{\tau\Xi_\ell}^\ell & =  \rmF_{\tau\Xi_\ell}^\ell \otimes \mathbf{1} \, + \,   \sum_{\mu\geq\tau} \Psi(\mu/\tau) \otimes \rmF^\ell_{\mu\Xi_\ell}, \\
		\Delta \rmF^\ell & = \rmF^\ell\otimes \mathbf{ 1} + \sum_{|\tau|<-\kappa/2} \Psi(\tau\Xi_\ell) \otimes \rmF^\ell_{\tau\Xi_\ell}.
	\end{align*}
	According to Lemma \ref{integration_RS}, one defines a model on $\scrT_{\boldsymbol{\Phi}_\ell}$ by setting for $\tau\in\mcT_o$
	$$
	\PPi(\rmF^\ell)=F_\ell(v)\cdot \xi_\ell  , \qquad g(\rmF_{\tau\Xi_\ell}^\ell) = u_{\tau\Xi_\ell}.
	$$
	We let 
	\begin{equation}\label{eq_decomposition_Psi}{\Psi}^\# = \sum_{\ell=1}^{\ell_0}\widehat{Y}_{\Psi(\mathrm{F}^\ell)}.
	\end{equation}
	 Theorem \ref{prop_regularity_YI_tau} and Lemma \ref{integration_RS} ensure that $\widehat{Y}_{\Psi(\mathrm{F}^\ell)}\in C^{\gamma_0'}$, as $\gamma_0'\geq-\kappa/2+\beta$ from \eqref{inequality_gamma}.
	It remains to check that the distribution defined in \eqref{eq_decomposition_Psi} is indeed the paracontrolled remainder in \eqref{eq_prop_integration}.
	From \eqref{def_YItau}, Lemma \ref{lem_DeltaPsitau} and Proposition \ref{prop_Ztau=YPsitau_integration}, we have
	\begin{align*}
	\widehat{Y}_{\Psi(\rmF^\ell)} &= K * \widetilde{\PPi}(\Psi(\mathrm{F}^\ell))- \sum_{\substack{\tau\in\mcT_o\\ q\geq 0 }} \P\big( \widetilde{g}(D^q\Psi( \mathrm{F}^\ell_{{\tau\Xi_\ell}})) , \widehat{Y}_{\Psi_q(\tau)} \big)
	\\
	&=K * \big(  F_\ell(v) \cdot\xi_\ell\big)- \sum_{\substack{\tau\in\mcT_o }} \P\big( \widetilde{\PPi}(\Psi(\mathrm{F}^\ell_{{\tau\Xi_\ell}})) , \widehat{Y}_{\Psi(\tau)} \big)
	\\
	& =  K *  \big(F_\ell(v) \cdot\xi_\ell\big)- \sum_{\substack{\tau\in\mcT_o }} \P\big(u_{\tau\Xi_\ell} , Z_{\mcI(\tau)} \big),
	\end{align*}
	so that we have indeed
\begin{align*}
		\Linv\Big( \sum_{\ell=1}^{\ell_0} F_\ell(v) \cdot\xi_\ell  \Big) = \sum_{\ell=1}^{\ell_0}\sum_{|\sigma\Xi_\ell|<-\kappa/2} \P( u_{\sigma\Xi_\ell}, Z_{\mcI(\sigma\Xi_\ell)}   ) + \Psi^\#.
\end{align*}

	\end{proof}

	\subsection{Two examples}
	\label{examples}
	
	In this subsection we illustrate the paracontrolled ansatz described above in the examples of the stochastic multiplicative equation and $\Phi_3^4$ models. 
	the stochastic multiplicative equation shows the complexity of the ansatz and can be seen as an extension of the gPAM example. The $\Phi_3^4$ model, on the other hand, provides an example in which the derivatives $\partial_*^k u$ appear.
	
	\subsubsection{The example of the stochastic multiplicative equation}
	
	\label{ex_gpam}
	
	We consider the equation 
	$$
	(\partial_t-\Delta)u=f(u)\xi, \quad (t,x) \in \R_+ \times \T, 
	$$
	where $\xi$ is a space-time white noise in $C^{\alpha-2}$ with $\alpha = \frac{1}{2} - \kappa$ with $\kappa < \frac{1}{8}$. As the non-linearity in front of the noise only depends on $u$, one has to consider only the variable $ \mathfrak{X}_0 $ and set
	\begin{equs}
	\beta = 2, \quad r_{\Xi} = 0, \quad	\reg( \Xi ) = -\frac{3}{2} - 2 \kappa, \quad \reg(\mathfrak{X}_0) = \frac{1}{2} - 4 \kappa.
	\end{equs}
	Then, 
	\begin{equs}
		\reg(\Xi) < | \Xi|, \quad  
		\beta  + \reg(\Xi) -	\reg(\mcX_0) \geq 2 \kappa .
	\end{equs}
	and
	\begin{equs}
		 \gamma_0 = \beta -\{|\mcI(\Xi)| - r_{\Xi}  \}+\kappa = 3/2 + 2 \kappa, \quad
	\end{equs}
	which is due to the fact that the decorated tree with the smallest degree in $\mcT^{nc}$  is $ \Xi $.
	 The only trees $\mcI(\tau)$ with degree strictly smaller than $\gamma$ are the  $\mcI(\tau)$ where $\tau$ belongs to 
	\begin{equs}
	\CT_{< -\frac{1}{2} + 2 \kappa } = 	\{  \<generic> \,, \<Xii> \,, \<Xiiia> \,, \<Xiii>  \}.
	\end{equs}
	One has
	\begin{equs}
		\Upsilon_F[\<generic>] &=  f, \quad \Upsilon_F[ \<Xii> ]= f f', \quad \Upsilon_F[ \<Xiiia>  ]= f^2 f'', 
		\quad \Upsilon_F[ \<Xiii>  ]= f (f')^2, 
	\end{equs}
	and
	\begin{equs}
			S(\<generic>) & =   S(\<Xii>)= S(\<Xiii>) = 1, \quad S( \<Xiiia>  )  = 2.
		\end{equs}
The paracontrolled ansatz writes then as
	$$
	u= \sum_{\mcI(\tau) \in \CT_{< \gamma_0}} \P\big(\frac{\Upsilon_F[\tau](u)}{S(\tau)},Z_{\mcI(\tau)}\big) + u^\#
	$$ 
	with $u^\#\in C^{\gamma_0}$. The additional singular stochastic fields we need are given by
	$Z_{\tau}$ with $ \tau $ belonging to 
	\begin{equs}
		\{   \<generic> \,,  \<genericx> \,, \<Xii> \,, \<Xii2> \,,  \<Xii1> \,, \<Xiiia> \,, \<Xiii> \,,  \<Xi4b>\,,  \<Xi4>\,,   \<Xi4ccccc> \}.
	\end{equs}
	  Below, we provide values for some $ \Psi(\tau) $ and $ Z_{\tau} $. In the sequel, we will use the following short hand notation: $ U = K * \xi $. We use a brown thick edge for encoding derivative in space $ \partial_x K $. The node $ \<genericx> $ encodes the pointwise multiplication between the noise $\xi$ and the monomial in space $x$. 
	  For $ \Xi = \<generic> $ with $ |\<generic>| = -\frac{3}{2}-\kappa $, one has
	 \begin{equs}
	 	\Psi(\<generic>) = \<generic>, \quad Z_{\<generic>} = \xi.
	 \end{equs}
	 For $ X_{(0,1)}\Xi = \<genericx> $ with $ |\<genericx>| = -\frac{1}{2}-\kappa $, one has
	 \begin{equs}
	 	\Psi(\<genericx>) = \<genericx> + \lg \<generic> \rg X_{(0,1)}, \quad Z_{\<genericx>} = \widetilde{\PPi}^{\mathcal{Z}} \<genericx>  +  \widetilde{\PPi}^{\mathcal{Z}}( \lg \<generic> \rg X_{(0,1)}) = 0.
	 \end{equs}
	 This is due to
	 \begin{equs}
	 	\widetilde{\PPi}^{\mathcal{Z}} \<genericx> = \widetilde{\PPi} \<genericx> = 0,
	 	\quad \widetilde{\PPi}^{\mathcal{Z}}( \lg \<generic> \rg X_{(0,1)})  = \widetilde{\PPi}(  \<generic>) \widetilde{\PPi}(  X_{(0,1)}) = 0
	 \end{equs}
	 which is a consequence of the fact that $ \widetilde{\PPi} $ sends to zero monomials.	 
	 For $ \Xi \CI(\Xi) = \<Xii> $ with $ |\<Xii>| = -1 - 2 \kappa $, one has
	 \begin{equs}
	 	\Psi(\<Xii>) = \<Xii> + \<IXi>  \otimes \<generic>, \quad  Z_{\<Xii>} = \xi U - \P( U , \xi).
	 \end{equs}
	 For $ X_{(0,1)}\Xi \CI(\Xi) = \<Xii2> $ with $ |\<Xii2>| =  - 2 \kappa $, one has
	 \begin{equs}
	 	\Psi(\<Xii2>) & = \<Xii2> + \<IXi>  \otimes \<genericx> + X_{(0,1)} \<IXi>  \otimes \<generic> +  \lg  \<IXi>  \otimes \<generic> \rg X_{(0,1)} + \lg  \<IXi>  \otimes \<generic> \rg_{(0,1)}
	 	 \\  Z_{\<Xii2>} & = \widetilde{\PPi} \, \<Xii2>  - \P( Z_{ \<IXi>  } , Z_{\<genericx>})- \P( Z_{X_{(0,1)} \<IXi>  } , Z_{\<genericx>}) - \P( Z_{ \<IXi>  } , Z_{\<generic>}) \widetilde{\PPi} X_{(0,1)}
	 	\\ & - \P_{(0,1)}( Z_{\<IXi>} , Z_{\<generic>} ) \\
	 	  Z_{\<Xii2>} & = - \P_{(0,1)}( U, \xi )
	 \end{equs}
	as the other terms are zero for the same reason as for $ Z_{\<genericx>} $. 
	 For $ \Xi \CI(X_{(0,1)}\Xi) = \<Xii1> $ with $ |\<Xii1>| =  - 2 \kappa $, one has 
	\begin{equs}
		\Psi(\<Xii1>) & = \<Xii1> +  \lg \<Xii> \rg X_{(0,1)} + \<IXiX>  \otimes \<generic> +  \lg \<IXi>  \otimes \<generic> \rg_{(0,1)}  +   \lg  \<IXi>  \otimes \<generic> \rg X_{(0,1)} + \<IiXiX>  \otimes \<generic>
		\\  Z_{\<Xii1>} & =  \xi K^{[(0,1)]} * \xi - 
		\P( Z_{\<IXiX>} , \xi) - \P_{(0,1)}(U, \xi)
		- \P(Z_{\<IiXiX>},\xi)
	\end{equs}
	where
	\begin{equs}
		Z_{ \<IXiX> } = K^{[0,1]}* \xi, \quad  Z_{ \<IiXiX> } = \partial_{x} K^{[0,1]}* \xi.
		\end{equs}
	 For $ \Xi \CI(\Xi)^2 =  \<Xiiia> $ with $ |  \<Xiiia>| = - \frac{1}{2} - 3 \kappa $, one has
	 \begin{equs} \Psi(\<Xiiia>)  & = \<Xiiia> + 2 \, \<IXi>  \otimes \<Xii>  +
	 	\<Xiiiaa> \otimes \<generic> + 2 \, \<IXi> \otimes \<IXi> \otimes \<generic>, 
	 	\\
	 	 Z_{\<Xiiia>} & = \xi U^2 - 2 \, \P( U, Z_{\<Xii>}  ) - \, \P( Z_{\<Xiiiaa>}, \xi ) - 2\P(U,U,\xi )
	 	 \\
	 	 &= \xi U^2 - 2 \, \P( U, Z_{\<Xii>}  ) - \, \P( U^2, \xi ) 
	 	\end{equs}	 
	 	where we have used
	 	\begin{equs}
	 		Z_{\<Xiiiaa>} = U^2 - 2 \P(U,U).
	 		\end{equs}
	 For $ \Xi \CI(\Xi \CI(\Xi)) =  \<Xiii> $ with $ |  \<Xiii>| = - \frac{1}{2} - 3 \kappa $, one has
	 \begin{equs}
	 	\Psi(\<Xiii>) & = \<Xiii> 
	 	+ \<Xiiibis> \otimes \<generic>	+ \<IXi>  \otimes \<Xii> +  \<IXi> \otimes \<IXi> \otimes \<generic>,
	 	\quad
	 	\Psi(\<Xiiibis>)  = \<Xiiibis> + \<IXi> \otimes \<IXi>,
	 	\\
	 	Z_{\<Xiiibis>} & = K * (U \xi) - \P( U, U),
	 	\\ 
	 	Z_{\<Xiii>} & = \xi (K *( U \xi ) ) - \P( Z_{\<Xiiibis>} , \xi )  -\P( U , U, \xi)-   \P( U, Z_{\<Xii>}   )
	 	\\
	 	&= \xi (K *( U \xi ) ) - \P( K*(U\xi) , \xi ) -   \P( U, Z_{\<Xii>}   )   . 
	 \end{equs}
	  For $ \CI(\CI(\CI(\Xi) \Xi) \Xi) \Xi =    \<Xi4> $ with $ |   \<Xi4> | =  - 4 \kappa $, one has
	  \begin{equs}
	  	\Psi(   \<Xi4> ) & =   \<Xi4> +
	  \<Xi4_11> \otimes \<generic> + \<Xi4_12> \otimes \<genericx> +
	  	 \<Xiiibis>  \otimes \<Xii> + \<IXi> \otimes \<Xiii> 
	  +  \<Xiiibis> \otimes \<IXi>  \otimes \<generic>
	  	+ \<IXi>  \otimes \<Xiiibis>  \otimes \<generic>
	  \\
	   &	\quad+ 
	   \, \<IXi>  \otimes \<IXi> \otimes \<Xii> 
	   	+ 
	   \, \<IXi>  \otimes \<IXi> \otimes \<IXi> \otimes \<generic>,
	   \\
	   Z_{\<Xi4>} 
	   & =  K * ( K * ( U \xi ) \xi ) \xi
	  -\P( Z_{\<Xi4_11>}, \xi)  - \P( Z_{\<Xiiibis>},  Z_{\<Xii>}) - \P(U, Z_{\<Xiii>})
	  \\
	  & \quad-
	   \P( Z_{\<Xiiibis>}, U, \xi)
	  - \P( U, Z_{ \<Xiiibis>}, \xi) 
	  - 
	  \P( U, U, Z_{\<Xii>}) 
	  - 
	  \P(U, U, U, \xi)
	  \\
	  &=  K * ( K * ( U \xi ) \xi ) \xi
	  -\P( Z_{\<Xi4_11>}, \xi)  - \P( Z_{\<Xiiibis>},  Z_{\<Xii>}) - \P(U, Z_{\<Xiii>})
	  - \P( U, Z_{ \<Xiiibis>}, \xi) 
	   \end{equs}
	   and
	   \begin{equs}
	   \Psi( \<Xi4_11> ) & = \<Xi4_11>  +  \<Xiiibis> \otimes \<IXi>  
	   + \<IXi>  \otimes \<Xiiibis>   + \<IXi>  \otimes \<IXi> \otimes \<IXi>, \quad
	   \Psi( \<Xi4_12> )  = \<Xi4_12>,
	   \\
	   Z_{\<Xi4_11>} & =  K * ( K * ( U \xi )  \xi ) - \P( Z_{\<Xiiibis>}, U ) - \P( U, Z_{\<Xiiibis>} ) - \P(U,U,U),
	   \\ 
	   Z_{\<Xi4_12>} & = \partial_x K * ( K * ( U \xi )  \xi ),
	   \\
	  	\end{equs}
	   For $ \CI(\CI(\Xi)^2 \Xi) \Xi =    \<Xi4c> $ with $ |   \<Xi4c> | =  - 4 \kappa $, one has 
	  	\begin{equs}
	  		\Psi(\<Xi4c>) & = \<Xi4c> + 2 \, \<IXi> \otimes \<Xiii> +  \<Xiiiaa> \otimes \<Xii>
	  		 +  \<Xi4c11> \otimes \<generic> + \<Xi4c12> \otimes \<genericx>  + 2 \, \<IXi> \otimes  \<IXi> \otimes \<Xii>
	  		 \\ & + 2 \, \<IXi>  \otimes \<Xiiibis> \otimes \<generic> +  \<Xiiiaa>  \otimes \<IXi> \otimes \<generic>   + 2 \,  \<IXi>  \otimes \<IXi> \otimes \<IXi> \otimes  \<generic>
	  		\\
	  		Z_{\<Xi4c>} & = (K*( U^2  \xi))\xi - 2\P(U, Z_{\<Xiii>}) -\P(U^2, Z_{\<Xii>})- \P( K*( U^2  \xi) ,\xi )
	  	\end{equs}
	  		For $\CI(\Xi)^3\Xi=\<Xi4b>$ with $ |   \<Xi4b> | =  - 4 \kappa $, one has 
	  	\begin{equs}
	  		\Psi(\<Xi4b>) & = \<Xi4b> + 3 \,  \<IXi> \otimes \<Xiiia> + 3 \, \<Xiiiaa> \otimes   \<Xii> +
	  		\<Xi4b11> \otimes \<generic> + 
	  		6 \, \<IXi> \otimes \<IXi>\otimes   \<Xii>
	  		\\ &\quad +  3 \, \<IXi>\otimes \<Xiiiaa>\otimes \<generic> +  3 \,  \<Xiiiaa>\otimes \<IXi>\otimes \<generic> + 6 \, \<IXi> \otimes \<IXi>\otimes \<IXi>\otimes \<generic>
	  		\\	Z_{\<Xi4b>} & = U^3\xi - 3 \P(U, Z_{\<Xiiia>}) - 3\P(U^2,Z_{\<Xii>}) - \P(U^3,\xi) 
	  		- 6 \P( U, U, Z_{\<Xii>} )\\ & \quad - 3 \P(U, Z_{\<Xiiiaa>}, \xi) - 3 \P(Z_{\<Xiiiaa>}, U, \xi)
	  		- 6 \P(U,U,U,\xi).
	  		\end{equs}
	  	For $\CI(\CI(\Xi)\Xi)\CI(\Xi)\Xi=\<Xi4ccccc>$ with $ |   \<Xi4ccccc> | =  - 4 \kappa $, one has 
	  	\begin{equs}
	  		\Psi(\<Xi4ccccc>) & = \<Xi4ccccc> +  \,  \<IXi> \otimes \<Xiiia> +  \, \<Xiiiaa> \otimes   \<Xii> +  \, \<Xiiibis> \otimes   \<Xii> 
	  		+ \<Xi4ccccb> \otimes \<generic> + 
	  		 2 \, \<IXi> \otimes \<IXi>\otimes   \<Xii>
	  		 \\ 
	  		 & \quad+   \, \<IXi>\otimes \<Xiiiaa>\otimes \<generic> +   \,  \<Xiiiaa>\otimes \<IXi>\otimes \<generic> 
	  		 +   \, \<IXi>\otimes \<Xiiibis>\otimes \<generic> \\&\quad+   \,  \<Xiiibis>\otimes \<IXi>\otimes \<generic>
	  		 + 6 \, \<IXi> \otimes \<IXi>\otimes \<IXi>\otimes \<generic>
	  	\\	
	  	Z_{\<Xi4ccccc>} & = K*(U\xi)U\xi -  \P(U, Z_{\<Xiiia>}) - \P(U^2,Z_{\<Xii>})- \P(Z_{\<Xiiibis>}, Z_{\<Xii>})\\& \quad-  \P(Z_{\<Xi4ccccb>},\xi) 
	  	- 2 \P( U, U, Z_{\<Xii>} )  -  \P(U, Z_{\<Xiiiaa>}, \xi) -  \P(Z_{\<Xiiiaa>}, U, \xi) \\&\quad-   \P(U, Z_{\<Xiiibis>}, \xi) -  \P(Z_{\<Xiiibis>}, U, \xi)
	  	- 6 \P(U,U,U,\xi).
	  	\end{equs}

	\subsubsection{The example of $\Phi_3^4$}
	\label{example_Phi_3_4}
	
	We consider 
	\begin{equation}
	(\partial_t-\Delta) u = -u^3 + \xi, \quad (t,x) \in \R_+ \times \T^3
\end{equation}
	where $\xi\in C^{-5/2-\kappa}$.
	One has to consider only the variable $ \mathfrak{X}_0 $ and set
	\begin{equs}
		\beta = 2, \quad	\reg( \Xi ) = -\frac{5}{2} - 2 \kappa, \quad \reg(\mathfrak{X}_0) = -\frac{1}{2} - 4 \kappa.
	\end{equs}
	Then, 
	\begin{equs}
		\reg(\Xi) < | \Xi|, \quad  
		\beta  + \reg(\Xi) -	 \reg(\mcX_0) \geq 2 \kappa,
	\end{equs}
	and
	\begin{equs}
		\gamma_0 = \beta -\{|\mcI(\mcI(\Xi)^2)| - r_{\mcI(\Xi)^2}  \}+\kappa = 1 + 3 \kappa.
	\end{equs}
	The set of  trees $ \tau $ such that  $\mcI(\tau)$ is of degree smaller than $\gamma_0$ and non vanishing elementary differential is $$ \CT_{< -1 + 3 \kappa} = \{ \<generic>, \<Xiiiaa>, \<Xi4b11>  \}.$$ We note that the tree $\<IXi>$ has negative degree and that $\overline{Z}_{\<IXi>} = Z_{\<IXi>}=K*\xi$ for the naive model, then $\partial^0_*u = u - Z_{\<IXi>}$. 
One has also
\begin{equs}
	\Upsilon_F[ \<generic> ](\partial^0_*u) & = 1, \quad \Upsilon_F[ \<Xiiiaa> ](\partial^0_*u) = -6 \partial^0_*u, \quad  \Upsilon_F[ \<Xi4b11> ](\partial^0_*u) = -6, 
	\\
	S(\<generic>)  & = 1, \quad  S(\<Xiiiaa>)  = 2, \quad  S(\<Xi4b11>) = 6.
\end{equs}	
	The paracontrolled ansatz then takes the form
	$$
	u= \P\big( 1, Z_{\<IXi>}  \big) -  \P\big(1 ,Z_{\mcI(\<Xi4b11>)} \big) - 3\P\big(\partial^0_*u , Z_{\mcI(\<Xiiiaa>)}\big)  + u^\#,
	$$
	with $u^\#\in C^{\gamma_0}$.
We note that for any distribution $Z$ we have $\P(1,Z)=Z $ up to the smooth term $\Delta_{\leq0}Z$. Then the ansatz above  is equivalent to the one we can find in \cite{Phi43Paracontrol} for instance.
	The set of singular stochastic fields is here indexed by the following decorated trees
	$$
	\{ \<IXi>, \, \<Xiiiaa>,\, \<Xi4b11>, \, \<Xi4b111b>, \, \<Xi4b111a>, \, \<Xi4b111>, \, X_{e_i}\<Xiiiaa>   \}
	$$
	where $X_{e_i}\<Xiiiaa>   = X_{e_i}\CI(\Xi)^2$. It is similar to $ \<Xiiiaa>  $ but now the root is decorated by $e_i$. One has
	\begin{equs}
	\Psi(\<IXi>)& = \<IXi>, \quad	\Psi( \<Xiiiaa> ) = \<Xiiiaa>, \quad \Psi(\<Xi4b11>) = \<Xi4b11>, \quad 
	\Psi(X_{e_i}\<Xiiiaa>) = X_{e_i}\<Xiiiaa> + \lg \<Xiiiaa> \rg X_{e_i},
	\\ \Psi( \<Xi4b111b> ) & = \<Xi4b111b> +  \<Xi4b111b1>\otimes \<IXi>, \quad   \Psi(\<Xi4b111a>) = \<Xi4b111a> + \<Xi4b111b2> \otimes \<Xiiiaa> \quad \Psi(  \<Xi4b111>) = \<Xi4b111> + \<Xi4b111b1> \otimes \<Xiiiaa>.
	\end{equs}
	 The naive version of these fields is given by 
	\begin{align*}
		&Z_{\<IXi>} = U, \quad  Z_{\<Xiiiaa>}= U^2, \quad  Z_{\<Xi4b11>} = U^3, \quad Z_{X_{e_i}\<Xiiiaa> }=0, 
		\\
		& Z_{\<Xi4b111b>} = U \, (K*U^3) - \P(K*U^3,\, U) =-\bfC(K*U^3,\, U), 
		\\
		& Z_{\<Xi4b111a>} = U^2 \, (K*U^2)  - \P(K*U^2,\, U^2)=-\bfC(K*U^2,\, U^2),
		\\
		&Z_{\<Xi4b111>} = U^2 \, (K*U^3) - \P(K*U^3,\, U^2)=-\bfC(K*U^3,\, U^2).
			\end{align*}
	The correctors $\bfC$ above can be interpreted as some resonant product in the same way as in the examples of gPAM and the stochastic multiplicative heat equation. In the end we recover the same enhanced noise as in \cite{Phi43Paracontrol}.
	
	
	

	\section{A fixed point argument} \label{section_fixedpoint}
	
	The difficulty in performing a fixed point argument is that we have to work with functions $u,u^\#$ satisfying the constraint $$u=\sum_{|\CI(\tau)| < \gamma} \P(u_\tau,Z_{\mcI(\tau)})+u^\#.$$ We cannot carry out a fixed point argument directly in a space of paracontrolled distributions satisfying this constraint, because the distribution $\scrL^{-1}(\sum_\ell F_\ell(v)\cdot \xi_\ell) $ does not necessarily satisfy the same structural constraint.

	Instead, we perform the fixed point argument only on the remainder $u^\#$. To do so, we have to construct from any $u^\#\in C^\gamma$ with $u^\#_{|t=0}=u_0$ a function $u$ satisfying the paracontrolled ansatz of \eqref{eq_gPAM} with remainder $u^\#$. Computing the right hand side of \eqref{eq_gPAM}
	$$
	\scrL^{-1}\Big(\sum_{\ell=1}^{\ell_0} F_\ell(v)\cdot \xi_\ell \Big)   = \sum_{|\mcI(\tau)|<\gamma} \P(u_\tau,Z_{\mcI(\tau)}) + \Psi^\#, 
	$$
	we see that $u$ is solution to \eqref{eq_gPAM} with initial condition $u_0$ if and only if $$\Psi^\# + P_t u_0=u^\#.$$
	We set $\boldsymbol{\Psi}_{\mcZ}^\#$ to be the map that associates to any function $u$ satisfying the paracontrolled ansatz the remainder $\Psi^\#$ defined from \ref{prop_integration} when the stochastic data is $\mcZ$.

	\begin{definition}\label{def_solution_EqPAM}
		A solution of  \eqref{eq_gPAM} with initial condition $u_0\in C^\gamma(\bfT^d)$ and stochastic data $\mcZ\in{\frak N}$, is a function $u$ satisfying the paracontrolled ansatz with remainder $u^\#$, such that 
		$$
		\boldsymbol{\Psi}_{\mcZ}^\#(u) + P_t u_0= u^\#.
		$$ 
	\end{definition}

	One can indeed construct a function $u$ satisfying the paracontrolled ansatz of \eqref{eq_gPAM} with remainder $u^\#$ provided one slightly modification of the definition of the ansatz by applying a Fourier cut-off to the reference distributions $Z_{\mcI(\tau)}$. This cut-off preserves the singular component of the distributions while making them sufficiently small in the appropriate functional spaces. To achieve this, we employ the same idea as in  \cite[Lemma 4.12]{AllezChouk}, constructing the map $\Gamma : u^\# \mapsto u$ as a perturbation of the identity, see Proposition \ref{prop_AppGamma} below.
	For any distribution $Z$ and any $a\in\bfN$, we set 
	\begin{equation}\label{eq_defZ_FourierCutoff}
		Z^{>a} = \sum_{i>a} \Delta_i Z.
	\end{equation}
	For any regularity exponents $r_1<r_2$, we have the estimate 
	\begin{equation}\label{eq_normeZ_FourierCutoff}
		\norme{Z^{>a}}_{r_1} \leq 2^{-a(r_2-r_1)}\norme{Z}_{r_2}.
	\end{equation}
	In order to apply this trick, we will have to suppose that the stochastic objects $(Z_{\mcI(\tau)})_\tau$ are slightly more regular than the regularity given by their degree $|\mcI(\tau)|$. To this end, we introduce an additional degree $|\cdot|'$, constructed as in Subsection \ref{subsection_DecoratedTrees}, from exponents satisfying $|\Xi_\ell|' > |\Xi_\ell|$, and such that the subcriticality assumption of Subsection \ref{subsec_subcriticality} still holds. We now regard the noises $\xi_\ell$ as elements of $C^{|\Xi_\ell|'}$, and construct the stochastic data with respect to this modified degree. In particular, this implies that $Z_\tau \in C^{|\tau|'}$. We let $\eps_0 = \min\{ |\Xi_\ell|'-|\Xi_\ell|, \, \ell\in\{1,...,\ell_0\}  \}$. Observe that, for every non-polynomial tree $\tau$, one has $|\tau|'\geq|\tau|+\eps_0$. Then \eqref{eq_normeZ_FourierCutoff} rewrites as 
	\begin{equation}\label{eq_normeZtau_smallfactor}
	\norme{Z_\tau^{>a}}_{C^{|\tau|}} \lesssim 2^{-a\eps_0}.
\end{equation}
	
	
	\begin{proposition}\label{prop_AppGamma}
		Let $\gamma >0 $ and a stochastic data $\mcZ$ satisfying \eqref{eq_normeZtau_smallfactor}, for $R>0$ big enough, there exist an integer $a\in\bfN$ and a continuous map $\Gamma : \{ \phi \in C^\gamma, \, \norme{\phi}_{C^\gamma}\leq R \} \to C^{\alpha_0}$, that associates to any $\phi\in C^\gamma$ a function $u=\Gamma \phi$ that satisfies the paracontrolled ansatz of \eqref{eq_gPAM} at order $\gamma$ and such that  
		\begin{equation}
			u=\sum_{|\mcI(\tau)|<\gamma}\P(  u_\tau   , \, Z_{\mcI(\tau)}^{>a})+\phi.
		\end{equation}
	\end{proposition}
	
	\begin{proof}
		We use the same trick as in the proof of \cite[Lemma 4.12]{AllezChouk}, an additional difficulty is the presence of the derivative $\partial_*^ku$ in the elementary differentials for trees $\tau$ with $|\mcI(\tau)|>|k|$.
		We define for any exponent $0\leq\alpha\leq\gamma$ a map  $	\Gamma^\alpha_{a,\mcZ}$  mapping any function $\phi\in C^\alpha$ to a function  $\Gamma^\alpha_{a,\mcZ}\phi$ satisfying the paracontrolled ansatz with truncated stochastic data up to order $\alpha$ with remainder $\phi$, that is
		\begin{equation}\label{eq_defGamma_beta}
			\Gamma^\alpha_{a,\mcZ} \phi = \sum_{|\mcI(\tau)|<\alpha } \frac{1}{S(\tau)}\P\left( \Upsilon_F[\tau]\big((\partial^j_*\Gamma_{a,\mcZ}^k(\phi))_{|j|\leq k} \big) , \,   Z_{\mcI(\tau)}^{>a}  \right) + \phi.
		\end{equation}
		We construct the map $	\Gamma^\alpha_{a,\mcZ}$ with $\alpha\in (k,k+1]$, by a recursion on the integer $k$. This recursion is started by setting $\Gamma^0_{a,\mcZ}= \id$. We now let $\alpha\leq\gamma$ and $k$ to be the unique integer such that $\alpha\in (k,k+1]$ and we suppose that the map $\Gamma^k_{a,\mcZ}$ is given. 
		For any $v\in C^{k}$, the function $\Gamma^k_{a,\mcZ}(v)$ satisfies the paracontrolled ansatz up to order $k$, so that $\partial^j_*\Gamma_{a,\mcZ}^k(v)$ is well defined for any $|j|\leq k$. This observation enables us to define for any $\phi \in C^{\alpha}$ the map
		$$
		\Lambda^{k,\alpha}_{\phi,a,\mcZ} v  = \sum_{k<|\mcI(\tau)|<\alpha } \frac{1}{S(\tau)}  \P \Big( \Upsilon_F[\tau](((\partial^j_*\Gamma_{a,\mcZ}^k(v))_{|j|\leq k} ) , \, Z_{\mcI(\tau)}^{>a} \Big) +  \phi,
		$$
		defined on $C^k$. 
		We are going to prove that the map $\Lambda^{k,\alpha}_{\phi,a,\mcZ}$ is a contraction on some big enough ball of $C^k$ when one chooses $a$ small enough. We write the shorthand $$(\Gamma v)_\tau =\Upsilon_F[\tau](((\partial^j_*\Gamma_{a,\mcZ}^k(v))_{|j|\leq k} ).$$ As $(\Gamma v)_\tau\in L^\infty$ is a local Lipschitz function of $v\in C^k$, and from the estimate $$\norme{\P(f,g)}_{C^k}\lesssim \norme{f}_{L^\infty}\norme{g}_{C^k}$$ and \eqref{eq_normeZ_FourierCutoff}, we have the following estimate for $R_1>0$ and $v\in B_{R_1}^k$ 
		\begin{align*}
			\lVert \Lambda^{k,\alpha}_{\phi,a,\mcZ} v\lVert_{C^{k}} &\lesssim \norme{\phi}_{C^{\alpha_0}} + \sum_{k<|\mcI(\tau)|<\alpha} \norme{(\Gamma v)_\tau}_{L^\infty} \lVert Z_{\mcI(\tau)}^{> a}\lVert_{C^{k}}
			\\
			&\lesssim_{R_1} \norme{\phi}_{C^{\alpha_0}} +  \sum_{k<|\mcI(\tau)|<\alpha }  2^{-a\eps_0}\lVert Z_{\mcI(\tau)}\lVert_{C^{|\mcI(\tau)|}} .
		\end{align*}
		If one chooses $a$ and $R_1$ big enough, the application $\Lambda^{k,\alpha}_{\phi,a}$ maps the ball $B_{R_1}^{\alpha_0}$ into itself. For any $v_1,v_2\in B_{R_1}^{k}$, we have from the local  Lipschitz continuity of $v\mapsto(\Gamma v)_\tau$  
		\begin{align*}
			\norme{ \Lambda_{\phi,a,\mcZ}^{k,\alpha} v_1 - \Lambda_{\phi,a,\mcZ}^{k,\alpha} v_2    }_{C^{k}} &\lesssim \sum_{k<|\mcI(\tau)|<\alpha} \norme{(\Gamma v_1)_\tau-(\Gamma v_2)_\tau}_{L^\infty} \lVert Z_{\mcI(\tau)}^{> a}\lVert_{C^{k}}
			\\
			&\lesssim_{R_1} \sum_{k<|\mcI(\tau)|<\alpha}  2^{-a\eps_0}\norme{Z_{\mcI(\tau)}}_{C^{k}} \norme{v_1-v_2}_{C^{k}} .
		\end{align*}
			Choosing $a$ big enough, the contraction mapping theorem ensures the existence of a unique fixed point in $C^{k}$ for the map $\Lambda_{\phi,a}^{k,\alpha}$. We set $\Gamma^{k,\alpha}_{a,\mcZ} \phi$ to be this fixed point. We let 
			$$
		\Gamma^{\alpha}_{a,\mcZ} \phi = \Gamma^k_{a,\mcZ} \,\Gamma^{k,\alpha}_{a,\mcZ}  \phi.
		$$
		Combining the definition of $\Gamma^k_{a,\mcZ}$ which reads as
		\begin{equation*}
			\Gamma^k_{a,\mcZ} \phi = \sum_{|\mcI(\tau)|<k } \frac{1}{S(\tau)}\P\left( \Upsilon_F[\tau]\big( ( \partial^j_*\Gamma^k_{a,\mcZ} \phi)_{|j|<k} \big) , \,   Z_{\mcI(\tau)}^{>a}  \right) + \phi
		\end{equation*}
		 and the definition of $\Gamma^{k,\alpha}_{a,\mcZ}$ which reads as 
		 $$
		 \Gamma^{k,\alpha}_{a,\mcZ}\phi = \sum_{k<|\mcI(\tau)|<\alpha} \frac{1}{S(\tau)}  \P \Big( \Upsilon_F[\tau]\big(((\partial^j_*\Gamma_{a,\mcZ}^k(\Gamma^{k,\alpha}_{a,\mcZ}\phi))_{|j|\leq k} \big) , \, Z_{\mcI(\tau)}^{>a} \Big) +  \phi ,
		 $$ 
		 we see that $\Gamma^{\alpha}_{a,\mcZ} \phi$ satisfies indeed the paracontrolled ansatz up to order $\alpha$ with truncated stochastic data and remainder $\phi$.
	\end{proof}

		We set the exponent $\gamma_0$ to be the one defined in Subsection \ref{subsec_subcriticality}. The $\Gamma$ map constructed above gives from $\phi\in C^{\gamma_0}$ a function satisfying the paracontrolled ansatz of \eqref{eq_gPAM} without a Fourier cut-off as described in Section \ref{section_ParacAnsatz}, as one can add the terms involving the low frequencies of $Z_{\mcI(\tau)}$ in the remainder. Setting $u=\Gamma^{\gamma_0}_{a,\mcZ}\phi$, this writes as
		\begin{equation}\label{eq_Ansat ithUdashPhi}
			u=\sum_{|\mcI(\tau)|<\gamma_0}\P(u_\tau, Z_{\mcI(\tau)} ) + {\bold u}_{a,\mcZ}^\# (\phi),
		\end{equation}
		where the map ${\bold u}_{a,\mcZ}^\# : C^{\gamma_0} \to C^{\gamma_0}$ is defined by 
		\begin{equation}\label{eq_defUdashPhi}
			{\bold u}_{a,\mcZ}^\# (\phi) \defeq \phi - \sum_{|\mcI(\tau)|<\gamma_0} \P( {\mathbf u}_\tau(\phi) , Z_{\mcI(\tau)}^{\leq a} ),
		\end{equation}
		where we also looked at the elementary differential $u_\tau$ as a function of $\phi$ setting
			$$
		{\mathbf u}_\tau(\phi) = 	{\mathbf u}_{a,\mcZ,\tau}(\phi) = \frac{1}{S(\tau)}\Upsilon_F[\tau]\big( (\partial^j_* \Gamma^{\gamma_0}_{a,\mcZ} \phi)\big).
		$$
		We will also write by abuse of notation
		$
		{\boldsymbol{\Psi}}^\#_{a,\mcZ}(\phi) = 	{\boldsymbol{\Psi}}_{a,\mcZ}^\#( \Gamma^{\gamma_0}_{a,\mcZ} \phi , {\mathbf u }_{a,\mcZ}^\#(\phi)) .
		$																																				
		We fix a time horizon $T$ and for $\gamma>0$ we set $C^\gamma_T$ to be set of functions with $\gamma$ Hölder regularity on $(0,T)\times\bfT^d$. The equation $u^\# = \Psi^\# + P_tu_0$ from Definition \ref{def_solution_EqPAM} rewrites as 
		\begin{equation}\label{eq_pointfixe1}
			\phi = \Psi^\# +  \sum_{|\mcI(\tau)|<\gamma_0} \P(u_\tau,Z_{\mcI(\tau)}^{\leq a})  + P_t u_0.
		\end{equation}
		As we noted above, given an initial condition $u_0$ and a stochastic data  $\mcZ$, all functions in \eqref{eq_pointfixe1} depend only on $\phi\in C^{\gamma_0}$. We can then define a fixed point map $\boldsymbol{\Theta}_{u_0,\mcZ}^T : C^{\gamma_0}_T \to C^{\gamma_0}_T $ by setting 
		\begin{equation}\label{eq_DefAppTheta}
		\boldsymbol{\Theta}_{u_0,\mcZ}^T (\phi) =  \boldsymbol{\Psi}_{a,\mcZ}^\#(\phi) + \sum_{|\mcI(\tau)|<\gamma_0} \P({\mathbf u}_\tau(\phi) ,Z_{\mcI(\tau)}^{\leq a})  + P_t u_0,
	\end{equation}

		\begin{proposition} \label{fixed_point_prop}
			For any initial condition $u_0 \in  C^{\gamma_0}(\bfT^d)$ and admissible stochastic data $\mcZ=(Z_\tau)_{\tau\in\mcT_{<\gamma_0}}$, there exists a small time $T>0$ such that the Equation \eqref{eq_gPAM} has a unique solution in the sense of Definition \ref{def_solution_EqPAM}. Furthermore the solution $(u,u^\#)$ depends continuously on $u_0$ and $\mcZ$.
		\end{proposition}
		
		\begin{proof}
			We apply the contraction mapping theorem in the space $C^\gamma_T$ to the map $\boldsymbol{\Theta}_{u_0,\mcZ}^T$.
		We set a large constant $R>0$, and we suppose we are given a couple of initial conditions $u_0^1,u_0^2\in C^{\gamma_0}(\bfT^d)$ and a couple of stochastic data $\mcZ^1,\mcZ^2\in\frak{N}$ such that $\norme{u_0^i}_{C^{\gamma_0}},\norme{\mcZ^i}_{\frak N}< R$ for $i\in\{1,2\}$.

			We are going to get a bound on $\lVert\boldsymbol{\Theta}_{u_0^1,\mcZ^1}^T(\phi_1) -\boldsymbol{\Theta}_{u_0^2,\mcZ^2}^T(\phi_2) \lVert_{C^{\gamma_0}_T}$ by estimating each term of the right hand side of \eqref{eq_DefAppTheta}. In order to gain a factor $T^\delta$ we will exploit the fact that terms depending on $\phi$ in \eqref{eq_DefAppTheta} are slightly more regular than $\phi$ by using the following fact, already used in \cite{BailleulBernicotHighOrder} for their fixed point argument. For exponents $0<r_1<r_2$ and any $w\in C^{r_2}_T$ such that $w_{|t=0} = 0$, we have
			\begin{equation}\label{eq_facteurT_Cgamma}
				\norme{w}_{C^{r_1}_T} \lesssim T^{\frac{r_2-r_1}{\beta}} \norme{w}_{C^{r_2}_T}.
			\end{equation}
			We first estimate the term in $\Psi^\#$. Combining Propositions \ref{prop_pamansatz_to_pamsystem} and \ref{prop_integration}, for any $\phi_1,\phi_2 \in C^{\gamma_0}_T$ with $\norme{\phi_i}_{C_T^{\gamma_0}} \leq R/2$, we have the estimate
			\begin{align*}
				\lVert \boldsymbol{\Psi}^\#_{\mcZ^1}(\phi_1) - \boldsymbol{\Psi}^\#_{\mcZ^2}(\phi_2) \lVert_{C_T^{\gamma'_0}} \lesssim \norme{\mcZ^1;\mcZ^2}_{\frak{N}} + \norme{\phi_1 - \phi_2}_{C_T^{\gamma_0}}.
			\end{align*}
			Then, from \eqref{eq_facteurT_Cgamma}  
			$$
				\lVert \boldsymbol{\Psi}^\#_{\mcZ^1}(\phi_1) - \boldsymbol{\Psi}^\#_{\mcZ^2}(\phi_2) \lVert_{C_T^{\gamma_0}} 
				\lesssim 
				T^{\frac{\kappa}{2\beta}}\Big(\norme{\mcZ^1;\mcZ^2}_{\frak{N}} + \norme{\phi_1 - \phi_2}_{C_T^{\gamma_0}}\Big).
			$$			
Still using \eqref{eq_facteurT_Cgamma} and the paraproduct estimate
\begin{align*}
	&\Big\lVert \sum_{|\mcI(\tau)|<\gamma_0} \P\big(u_\tau(\phi_1),Z_{\mcI(\tau)}^{1,\leq a}\big)- \sum_{|\mcI(\tau)|<\gamma_0} \P\big(u_\tau(\phi_2),Z_{\mcI(\tau)}^{2,\leq a}\big)\Big\lVert_{C^{\gamma_0}_T}
	\\
	&\lesssim  	T^{\frac{\kappa}{2\beta}}  \sum_{|\mcI(\tau)|<\gamma_0} \norme{u_\tau(\phi_1) - u_\tau(\phi_2)}_{L^\infty} \lVert Z_{\mcI(\tau)}^{1,\leq a}\lVert_{C^{\gamma'_0}_T} + \norme{u_\tau(\phi_2)}_{L^\infty} \lVert Z_{\mcI(\tau)}^{1,\leq a} - Z_{\mcI(\tau)}^{2,\leq a} \lVert_{C^{\gamma'_0}_T}.
\end{align*}
Then from the local Lipschitz continuity of the $\Gamma$ map
\begin{align*}
	&\Big\lVert \sum_{|\mcI(\tau)|<\gamma_0} \P(u_\tau(\phi_1),Z_{\mcI(\tau)}^{1,\leq a})- \sum_{|\mcI(\tau)|<\gamma_0} \P(u_\tau(\phi_2),Z_{\mcI(\tau)}^{2,\leq a})\Big\lVert_{C^{\gamma_0}_T}
	\\
	&\hspace{6cm}\lesssim_R  	T^{\frac{\kappa}{2\beta}}  \left( \norme{\phi_1-\phi_2}_{C_T^{\gamma_0}}  +    \norme{\mcZ^1;\mcZ^2}_{\frak N}      \right).
\end{align*}
We also have $\norme{P_tu_0^1-P_tu_0^2}_{C^{\gamma_0}_T} \lesssim \norme{u_0^1-u_0^2}_{C^{\gamma_0}} $. Then we end up with the bound
			\begin{equation*}\begin{split}
					\norme{\boldsymbol{\Theta}_{u_0,\mcZ^1}^T(\phi_1) -\boldsymbol{\Theta}_{u_0,\mcZ^2}^T(\phi_2)}_{C^{\gamma_0}_T} 
					&\lesssim \norme{u_0^1-u_0^2}_{C^{\gamma_0}(\bfT^d)}  + \norme{\mcZ^1;\mcZ^2}_{\frak N} 
					\\
					&  \quad+ T^{\frac{\kappa}{2\beta}}\norme{\phi_1-\phi_2}_{C^{\gamma_0}_T}  .
			\end{split}\end{equation*}
			If one chooses $T$ sufficiently small, the application $\boldsymbol{\Theta}_{u_0,\mcZ}^T$ maps the ball $B^{\gamma_0}(R)=\{ \norme{\phi}_{C^{\gamma_0}}\leq R \}$ into itself. And choosing eventually $T$ smaller makes it a contraction. Then the map $\boldsymbol{\Theta}_{u_0,\mcZ}^T$ admits a unique fixed point in $C^{\gamma_0}_T$, which depends continuously on $\mcZ\in\frak{N}$  and $u_0\in C^{\gamma_0}(\bfT^d)$. 
			
		\end{proof}

		\begin{remark}
			We have established a local well-posedness result for regular initial conditions. To extend this result to initial conditions in $C^{\alpha_0}$, we need to work in weighted Hölder spaces $C^\alpha_w$, where the weight $w$ exhibits a singularity near the hyperplane $\{t = 0\}$. This would require adapting the results of \cite{LocalExpansionsParacSystems} to this setting.
		\end{remark}



	
	
	

	\section{Renormalised equations} \label{section_renormalized_equation}

	\subsection{Renormalisation with renormalisation maps}

	Taking a mollification of the noises $(\xi^\eps_{\ell})_{\eps}$, one obtains  a regularised enhanced stochastic data $(\mcZ^\eps)_\eps$. There is no hope for $(\mcZ^\eps)$ to converge in $\prod_{\tau\in\mcT} C^{|\tau|}$ as $\eps$ goes to $0$, one can renormalise $\mcZ^\eps$ by defining
	$$
	\widehat Z^\eps_\tau = Z^\eps_{M_\eps\tau}
	$$
	for some adequate renormalisation maps $M_\eps$.

	\begin{definition}
		An admissible renormalisation map is a couple $(M,M^+)$ made of a linear maps $M : T \to T$ and an algebra morphism $M^+ : T^+ \to T^+$ such that 
		\begin{equation}\label{eq_RenormMapCoprodcuct1}
			(M\otimes M^+)\Delta = \Delta M, \qquad \text{and} \qquad (M^+\otimes M^+)\Delta^{\!+} = \Delta^{\!+} M^+,
		\end{equation}
	for any $n\in \N^{d+1}$
		\begin{equation}\label{eq_commut_MandMcI}
			M\mcI_n(\tau) = \mcI_n(M\tau), \qquad \text{and} \qquad 	M^+\mcI^+_n(\tau) = \mcI^+_n (M^+\tau),
		\end{equation}
		and for any $k\in\bfN^{d+1}$
		\begin{equation}\label{eq_commutRenormFleche}
			\downarrow^k M =M\downarrow^k,\qquad \text{and}\qquad \downarrow^k M^{+} =M^{+}\downarrow^k.
		\end{equation}
	\end{definition}
In the sequel, we will denote $M^+$ by $M$. It will be clear from the context when one has to use $M^+$ instead of $M$.
	
	\begin{lemma} \label{lem_CommutModifCoprodRenorm}
		For any $\tau\in\mcB\sqcup \mcB^+$
		$$
		\sum_{\sigma\prec M\tau} \sigma\otimes (M\tau)/\sigma   = M\mu \otimes \sum_{\mu\prec\tau} M(\tau/\mu).
		$$
	\end{lemma}
	
	\begin{proof}
		From  \eqref{eq_RenormMapCoprodcuct1} and \eqref{eq_commutRenormFleche}
		\begin{align*}
		&	\sum_{\sigma\prec M\tau} \sigma \otimes (M\tau)/\sigma  = \Delta(M\tau) - \sum_{j\geq 0}   \downarrow^j(M\tau) \otimes \frac{X^j}{j!}
			\\
			&= (M\otimes M)\Delta\tau - \sum_{j\geq 0}   M(\downarrow^j\tau) \otimes \frac{M(X^j)}{j!} 
			=\sum_{\mu\prec\tau}  M\mu\otimes M(\tau/\mu).
		\end{align*}
	\end{proof}
	Iterating the Lemma \ref{lem_CommutModifCoprodRenorm} gives the relation
	\begin{align*} \label{lem_CommutModifCoprodRenorm_Iterated}
		\sum_{\tau_q\prec\dots\prec\tau_1\prec\tau}  M(\tau/\tau_1) \otimes  \cdots \otimes  M\tau_q
		=
		\sum_{\sigma_q\prec\dots\prec\sigma_1\prec M\tau} (M\tau)/\sigma_1 \otimes  \cdots \otimes  \sigma_q.
	\end{align*}

	\begin{proposition}
		For any admissible renormalisation map $M$ and any admissible stochastic data $\mcZ=(Z_\tau)$, the collection $\mcZ^M = (Z_{M\tau})_\tau$ is still admissible.
	\end{proposition}	
	
	\begin{proof}
		We prove that the condition in \eqref{eq_ZtauAdmissible} holds true. We have from Lemma \ref{lem_CommutModifCoprodRenorm}
		\begin{align*}
			K * \widetilde{ \PPi}^{\mcZ^M}\left( \Psi(\tau)\right) &=  K * \sum_{\substack{q\geq 0 \\ \tau_q\prec\dots\prec\tau_1\prec\tau}}\sum_{k\geq0} \frac{1}{k!}\P_k( Z_{M(\tau/\tau_1)} ,\dots,Z_{M\tau_q}  )
			\\
			&=K * \sum_{\substack{q\geq 0 \\ \sigma_q\prec\dots\prec\sigma_1\prec M\tau}}\sum_{k\geq0} \frac{1}{k!}\P_k( Z_{(M\tau)/\sigma_1} ,\dots,Z_{\sigma_q}  ).
		\end{align*}
		Then, from the admissibility of $\mcZ$ and from \eqref{eq_commut_MandMcI}
		\begin{align*}
			&\sum_{\substack{q\geq 0 \\ \sigma_q\prec\dots\prec\sigma_1\prec M\tau}}\sum_{k\geq0} \frac{1}{k!}\P_k( Z_{(\mcI (M\tau))/\sigma_1} ,\dots,Z_{\sigma_q}  )
			\\ &
= \sum_{\substack{q\geq 0 \\ \sigma_q\prec\dots\prec\sigma_1\prec M\tau}}\sum_{k\geq0} \frac{1}{k!}\P_k( Z_{(M(\mcI(\tau)))/\sigma_1} ,\dots,Z_{\sigma_q}  ).
		\end{align*}
		And from Lemma \eqref{lem_CommutModifCoprodRenorm}, we obtain
		\begin{align*}
			& \sum_{\substack{q\geq 0 \\ \sigma_q\prec\dots\prec\sigma_1\prec M\tau}}\sum_{k\geq0} \frac{1}{k!}\P_k\left( Z_{M((\mcI(\tau))/\sigma_1)} ,\dots,Z_{M\sigma_q}  \right)
			= \widetilde{ \PPi}^{\mcZ^M}\left( \Psi(\mcI(\tau)) \right).  
		\end{align*}
	Equation \eqref{eq_defZtau_T+_alternative} giving the multiplicativity property of $\mcZ_{M\tau}$ for $\tau\in T^+$ is proven from similar computations.
	\end{proof}

	For $\eps>0$, we let $Z^\eps_\tau$ to be the naive model as described in Subsection \ref{subsect_naivemodel}, built from the mollified noises $\xi^\eps_{\ell}$. We suppose given a family of renormalisation maps $(M_\eps)_\eps$. For any renormalisation map $M$, one defines the renormalised non-linearity 
	\begin{equation}
		MF_\ell = \Upsilon_F[M^*\Xi_\ell ].
	\end{equation}
	From \cite[Lemma 3.23]{BCCH}, one has
	\begin{equation}\label{eq_renormalisedElementaryDiff}
		\Upsilon_F[M^*\tau] = \Upsilon_{MF}[\tau].
	\end{equation}

	\begin{proposition}\label{prop_eqrenormalised_0}
		Suppose we are given a renormalised collection of random fields $(\widehat Z_\tau^\eps)_\tau = (Z^\eps_{M_\eps\tau})_\tau$ and that the function $u^\eps$ is a solution of the fixed point described in Definition \ref{def_solution_EqPAM}. Then $u^\eps$ is solution to the renormalised equation
		$$
		\scrL u^\eps =\sum_{\ell=1}^{\ell_0} (M_\eps F_\ell)(v^{\eps})\, \xi_\ell^\eps  .
		$$
	\end{proposition}
	
	\begin{proof}
		We compute the right hand side as defined in Subsection \ref{Subsection_RHS}, it is given by 
		\begin{align*}
			& \sum_{\ell=1}^{\ell_0}\sum_{\substack{\tau\in\mcT_{0} \\ |\mcI(\tau\Xi_\ell)|<\gamma_\ell }}\Big( \P(u_{\tau\Xi_\ell}, \widehat Z_{\tau\Xi_\ell}^\eps ) + \sum_{\substack{q\geq 0 \\ \tau_q\prec\dots\prec\tau}} \bfC\big(u_{\tau\Xi_\ell}^{\#,\gamma_\ell},\widehat Z_{\tau/\tau_1}^\eps,\dots, \widehat Z_{\tau_q\Xi_\ell}^\eps \big)\Big) 
			\\
			=& \sum_{\ell=1}^{\ell_0}\sum_{\substack{\tau\in\mcT_{0} \\ |\mcI(\tau\Xi_\ell)|<\gamma_\ell }}\Big( \P(u_{\tau\Xi_\ell},\, Z_{M_\eps(\tau\Xi_\ell)}^\eps) + \sum_{\substack{q\geq 0 \\ \tau_q\prec\dots\prec\tau}} \bfC (u_{\tau\Xi_\ell}^{\#,\gamma_\ell},Z_{M_\eps(\tau/\tau_1)}^\eps,\dots,Z_{M_\eps(\tau_q\Xi_\ell)}^\eps)\Big).
		\end{align*}
		From Lemma \ref{lem_CommutModifCoprodRenorm}, this is equal to 
		\begin{align*}
			\sum_{\ell=1}^{\ell_0}\sum_{\substack{\tau\in\mcT_{0} \\ |\mcI(\tau\Xi_\ell)|<\gamma_\ell }} & \Big(\frac{1}{S(\tau)} \P(\Upsilon_F(\tau\Xi_\ell), Z_{M_\eps\tau}^\eps) + \sum_{\mu_q\prec\dots\prec M_\eps\tau}\bfC\Big(u_\tau^{\#,\gamma_\ell},Z_{(M_\eps(\tau\Xi_\ell))/\mu_1}^\eps,\dots,Z_{\mu_q\Xi_\ell}^\eps\Big)\Big)
			\\
			=&\sum_{\ell=1}^{\ell_0}\sum_{\substack{\tau\in\mcT_{0} \\ |\mcI(\tau\Xi_\ell)|<\gamma_\ell }} \Big( \frac{1}{S(\tau)} \P(\Upsilon_F(M_{\eps}^* \tau\Xi_\ell), \, Z_{\tau\Xi_\ell}^\eps )
			\\& + \frac{1}{S(\tau)} \sum_{\mu_q\prec\dots\prec \tau}\bfC\Big(\Upsilon_{ F}(M_{\eps}^*\tau\Xi_\ell)^{\#,\gamma_\ell},Z_{\tau/\mu_1}^\eps,\dots,Z_{\mu_q\Xi_\ell}^\eps\Big)\Big).
		\end{align*}
		And then from  \eqref{eq_renormalisedElementaryDiff} and from Proposition \ref{prop_nonlinearitxiell_generalized}, we get 
		\begin{align*}
			&\sum_{\ell=1}^{\ell_0}\sum_{\substack{\tau\in\mcT_{0} \\ |\mcI(\tau\Xi_\ell)|<\gamma_\ell }}  \frac{1}{S(\tau)}\P(  \Upsilon_{M_\eps F}[\tau\Xi_\ell]  ,\, Z_{\tau\Xi_\ell}^\eps) \\ & + \frac{1}{S(\tau)}\sum_{\sigma_q\prec\dots\prec \tau}\bfC\Big(\Upsilon_{M_\eps F}(\tau\Xi_\ell)^{\#,\gamma_\ell},Z_{\tau/\sigma_1}^\eps,\dots,Z_{\sigma_q\Xi_\ell}^\eps\Big)
			= \sum_{\ell=1}^{\ell_0} (M_\eps F_\ell)(v^\eps)\, \xi_\ell^\eps.
		\end{align*}
	\end{proof}

	\subsection{Renormalisation with preparation maps}
	
In this subsection, we also compute the renormalised equation, with the distinction that the renormalised stochastic data is defined through a preparation map instead of a renormalisation map.
	\begin{definition}  \label{def_preparation_map}
		A strong preparation map is a linear map $R:T\to T$ such that for every decorated trees $\sigma, \tau$, one has
		\begin{equation}\label{eq_DeltaRTau}
			R^{*} (\sigma \star \tau) = \sigma \star R^{*} \tau
		\end{equation}
		and such that we have 
		$
		R\tau = \tau+ \sum_{i=1}^N \lambda_i\tau_i,
		$
		for constants $\lambda_i\in\bfR$ and trees $\tau_i$ with $|\tau_i|>|\tau|$ and where $\tau_i$ has fewer nodes than $\tau$.
	\end{definition}

\begin{lemma}
	For any strong preparation map $R$, one has 
	\begin{equation} \label{lem_DeltaModif_RTau}
	\sum_{\mu\prec R\tau} \mu \otimes (R\tau)/\mu = 	\sum_{\sigma\prec \tau} R\sigma \otimes \tau/\sigma.
\end{equation}
\end{lemma}
\begin{proof} One can observes that \eqref{eq_DeltaRTau} implies
	\begin{equation*}
 \left( R \otimes \id \right) \Delta = \Delta R, \quad 		R\downarrow^k = \downarrow^k R.
	\end{equation*}
		Then, one obtains
	\begin{align*}
		\sum_{\sigma\prec R\tau} \sigma \otimes (R\tau)/\sigma  &= \Delta(R\tau) - \sum_{j\geq 0} \downarrow^j(R\tau)    \otimes \frac{X^j}{j!} 
		\\
		&= (R\otimes \id )\Delta\tau - \sum_{j\geq 0} R(\downarrow^j\tau) \otimes    \frac{X^j}{j!}
		=\sum_{\mu\prec\tau}  R\mu\otimes \tau/\mu.
	\end{align*}
\end{proof}
Iterating \eqref{lem_DeltaModif_RTau} gives for any $q\geq2$ and $\tau\in\mcT$
	\begin{equation} \label{lem_DeltaModif_RTau_iterated}
		\begin{aligned}
	\sum_{\mu_q\prec\dots\prec\mu_1\prec R\tau} & \mu_q\otimes \, \mu_{q-1}/\mu_q \, \otimes \dots\otimes \, (R\tau)/\mu_1  	\\ & = \sum_{\tau_q\prec\dots\prec\tau_1\prec \tau} (R\tau_q) \, \otimes \, \tau_{q-1}/\tau_q \, \otimes \dots \otimes \, \tau/\tau_1.
	\end{aligned}
\end{equation}
Given any preparation map $R$ and smooth noises $(\xi_\ell)_{\ell\in\{1,...,\ell_0\}}$, one defines a renormalised stochastic data $\widehat \mcZ = (\widehat Z_\tau)_\tau$. We introduce an auxiliary collection  $\widehat Z^\circ_{\tau}$ indexed by decorated trees from $T$, that will have a multiplicativity property. 
These renormalised data $\widehat\mcZ=(\widehat Z_{\tau})_\tau$ and $\widehat\mcZ^\circ=(\widehat Z^\circ_{\tau})$ are computed in the same induction by setting
\begin{equation}\label{eq_Renorm_ZtauFromZotau}
	\widehat Z_\tau = \widehat Z^\circ_{R\tau}.
\end{equation}
And by requiring 
	\begin{equation} \label{eq_defZotau_renorm}
		\widetilde{ \PPi}^{\widehat\mcZ^\circ}\left(\Psi(\tau)\right) = \mathbf{1}_{p=0} \, \xi_\ell  \,  \prod_{j=1}^{n}\sum_{l_j\geq 0} \frac{1}{l_j !} (\partial^{b_j}K)^{[l_j]} * \widetilde{\PPi}^{\widehat\mcZ}\left(\Psi(\downarrow^{l_j}\tau_j)\right). 
	\end{equation}
for $\tau = X^p\Xi_\ell\prod_{j=1}^n\mcI_{b_j}(\tau_j) \in T$. We define $\widehat Z_\tau$ for trees $\tau\in T^+$ from the previously constructed fields $Z_\sigma$ by \eqref{eq_defZtau_T+}.
From the same computations as in the proof of Proposition \ref{prop_alternativedef_Ztaunaif}, Equation \eqref{eq_defZotau_renorm} can be rewritten as 
	\begin{equation} \label{eq_defZotau_renorm_alt}
	{ \PPi}^{\widehat\mcZ^\circ}\left(\Psi(\tau)\right) = x^{p} \, \xi_\ell  \,  \prod_{j=1}^{n} (\partial^{b_j}K) * {\PPi}^{\widehat\mcZ}\left(\Psi(\tau_j)\right). 
\end{equation}
\begin{proposition}
	For any preparation map $R$ such that $R\mcI(\tau)=\mcI(\tau)$ for all $\tau\in\mcT$, the associated renormalised stochastic data $\widehat\mcZ=\mcZ^R$ is admissible.
	
\end{proposition}	

\begin{proof}
The renormalised collection $\mcZ^R$ satisfies already \eqref{eq_defZtau_T+} by definition. We also have from \eqref{eq_Renorm_ZtauFromZotau}
\begin{align*}
	\PPi^{\mcZ^R} \Psi(\mcI\tau)  &= \sum_{k\geq 0}\sum_{\substack{q\geq 0 \\  \tau_q\prec\dots\prec\tau}} \frac{1}{k!}\P_k(\widehat Z_{\downarrow^k\tau/\tau_1},\dots, \widehat Z^\circ_{R \mcI(\tau_q)}   ) 
	\\&= \sum_{k\geq 0}\sum_{\substack{q\geq 0 \\  \tau_q\prec\dots\prec\tau}} \frac{1}{k!}\P_k(\widehat Z^\circ_{\downarrow^k\tau/\tau_1},\dots, \widehat Z^\circ_{\mcI(\tau_q)}   ).
\end{align*}
Then $\PPi^{\widetilde\mcZ} \Psi(\mcI(\tau))= \PPi^{\widetilde\mcZ^{\circ}} \Psi(\mcI(\tau))$, which is indeed equal to $K*	\PPi^{\widetilde\mcZ} \Psi(\mcI(\tau)) $ from \eqref{eq_defZotau_renorm_alt}. 
\end{proof}

\begin{proposition} \label{renormalised_preparation_pde}
Given a renormalised stochastic data $(\widehat Z_\tau) = (\widehat Z^{\circ}_{R_\eps\tau})$ defined from some preparation map $R_\eps$, any function $u^\eps$ solution of \eqref{eq_gPAM} in the sense of Definition \ref{def_solution_EqPAM}, is a solution to the renormalised equation
	$$
	\scrL u^\eps = \sum_{\ell=1}^{\ell_0} \Upsilon_F[R_\eps^* \Xi_\ell](v^{\eps})\xi_\ell^\eps. 
	$$
\end{proposition}
\begin{proof}
	As it was done in the proof of Proposition \ref{prop_eqrenormalised_0}, we compute the right hand side of \eqref{eq_gPAM} as defined in Subsection \ref{Subsection_RHS}. From \eqref{eq_Renorm_ZtauFromZotau}, it is equal to
\begin{align*}
	 \sum_{\ell=1}^{\ell_0}\sum_{\substack{\tau\in\mcT_{0} \\ |\mcI(\tau\Xi_\ell)|<\gamma_\ell }} \Big\{ \P \Big( u_{\tau\Xi_\ell} ,\, \widehat Z_{R_{\eps}(\tau\Xi_\ell)}^{\circ}\Big) + \sum_{\tau_q\prec\dots\prec\tau} \bfC\Big(u_{\tau\Xi_\ell}^{\#,\gamma_\ell},\widehat Z_{\tau/\tau_1},\dots, \widehat Z_{R_{\eps} (\tau_q\Xi_\ell)}^{\circ}\Big)\Big\}.
\end{align*}
Using Proposition \ref{lem_DeltaModif_RTau}, we get
\begin{equs}
	\sum_{\ell=1}^{\ell_0}\sum_{\substack{\tau\in\mcT_{0} \\ |\mcI(\tau\Xi_\ell)|<\gamma_\ell }} & \Big\{ \P\Big(u_{\tau\Xi_\ell}, \,  \widehat Z_{R_\eps(\tau\Xi_\ell)}^{\circ} \Big)  + \sum_{\substack{q\geq 0 \\ \mu_q\prec\dots\prec R_\eps \tau}}\bfC\Big(u_{\tau\Xi_\ell}^{\#,\gamma_\ell},\widehat Z_{(R_\eps\tau)/\mu_1},\dots,\widehat Z_{ \mu_q\Xi_\ell}^{\circ}\Big)\Big\}
	\\
	=&\sum_{\ell=1}^{\ell_0}\sum_{\substack{\tau\in\mcT_{0} \\ |\mcI(\tau\Xi_\ell)|<\gamma_\ell }}\Big\{ \frac{1}{S(\tau)} \P \Big( \Upsilon_F(R_\eps^*(\tau\Xi_\ell)), \,   \widehat Z_{\tau\Xi_\ell}^{\circ} \Big) \\ & + \sum_{\substack{q\geq 0 \\ \mu_q\prec\dots\prec \tau}}\bfC\Big(\Upsilon_F(R_\eps^*(\tau\Xi_\ell))^{\#,\gamma_\ell},\widehat Z_{\tau/\mu_1},\dots,\widehat Z_{\mu_q\Xi_\ell}^{\circ}\Big)\Big\}
	\\ 	=&\sum_{\ell=1}^{\ell_0}\sum_{\substack{\tau\in\mcT_{0} \\ |\mcI(\tau\Xi_\ell)|<\gamma_\ell }}\Big\{ \frac{1}{S(\tau)} \P \Big( \Upsilon_F(\tau \star R_\eps^*(\Xi_\ell)), \,   \widehat Z_{\tau\Xi_\ell}^{\circ} \Big) \\ & + \sum_{\substack{q\geq 0 \\ \mu_q\prec\dots\prec \tau}}\bfC\Big(\Upsilon_F(\tau \star R_\eps^*(\Xi_\ell))^{\#,\gamma_\ell},\widehat Z_{\tau/\mu_1},\dots,\widehat Z_{\mu_q\Xi_\ell}^{\circ}\Big)\Big\}
	\\ & = \sum_{\ell=1}^{\ell_0}\Upsilon_F[R^*_\eps\Xi_{\ell}](v^{\eps})\xi_\ell^\eps
\end{equs}
where we have first used 
\begin{equation} \label{identity_noise_R}
	R_\eps^*(\tau\Xi_\ell) = 	R_\eps^*(\tau \star\Xi_\ell) = \tau \star R_{\eps}^* \Xi_\ell
	\end{equation}
	which is coming from \eqref{eq_DeltaRTau}. Then for the last line, we have applied Corollary \ref{cor_paralinearistaion2}.
\end{proof}

\begin{remark}
	Let us stress that the main steps of the previous section are  the ones of the renormalised equation proof given \cite{RenomrmalisedSpde}. We list them below:
	\begin{itemize}
		\item Multiplicativity of $ \widehat{Z}^{\circ} $ which corresponds to the multiplicativity of the reconstruction operator in \cite{RenomrmalisedSpde}. 
		\item The identity \eqref{identity_noise_R} is the same in both cases.
		\item A version of Corollary \ref{cor_paralinearistaion2} is also used in \cite{RenomrmalisedSpde}  and it is connected to the coherence property that is the fact that the coefficients $u_{\tau}$ are given by the elementary differentials.
	\end{itemize}
	\end{remark}
	
	\section{Convergence of the stochastic data}
	\label{Sec::8}
	
	For proving the convergence of the stochastic data $(\widehat{Z}_{\tau})_{\tau}$, we use the spectral gap approach introduced with multi-indices in \cite{LOTP}. It has been extended to decorated trees in \cite{HS23,Hos25,BH23}. In this section, we follow mainly the formalism of \cite{BH23}. We make the following assumption based on \cite[Sec. 3.5]{BH23}.

	Let $ \Omega  $ be a separable space and $H$ a separable Hilbert space embedded densely into $\Omega$. A function $ F : \Omega \rightarrow \mathbb{R} $ is $H$-differentiable if there exists a function $\delta F : \Omega \rightarrow H^*$ such that \label{Malliavin_derivative}
	\begin{equs}
		(\delta_{\omega} F)(h) = \frac{d}{dt} F(\omega + th) |_{t=0}.
	\end{equs} 
	Then, we set  $ \Vert \ell \Vert_{H^*} =  \sup_{\Vert h \Vert \leq 1} |  \ell(h)| $ and
	one has the $H$-spectral gap inequality if there exists a constant $C > 0$ such that 
	\begin{equs}
		\label{spectral_gap_2}
		\mathbb{E}[  (F - \mathbb{E}(F))^2 ] \leq C \mathbb{E}( \Vert \delta F \Vert_{H^*}^2 )
		\end{equs}
	for any $H$-differentiable $ F \in L^2(\Omega)$ such that $\delta F \in L^2(\Omega ; H^*)$. One further has
	\begin{equs}
		\label{spectral_gap_q}
	\mathbb{E}[ |F|^q ] \lesssim |\mathbb{E}[F]|^q + \mathbb{E}[ \Vert \delta F\Vert_{H^*}^q ]
	\end{equs}
	for $q= 2^{r}$.
	
	\begin{assumption} \label{assumption_spectral gap}
		Let $ H:= \prod_{\ell=1}^{\ell_0} H^{-s_{\ell}} $ and $ \Omega = \prod_{\ell=1}^{\ell_0} C^{r_\ell} $ with 
		$r_\ell < -\frac{|\mathfrak{s}|}{2} - s_\ell$ and $ \frac{|\mathfrak{s}|}{2} + s_\ell \geq 0 $, we assume that the noises $  \xi_1,...,\xi_{\ell_0} $ satisfy the spectral gap inequality \eqref{spectral_gap_2}. 
		\end{assumption}

\subsection{Analytical prerequisites on Besov spaces}

For $\alpha\in\bfR$ and $p,q\in[1,+\infty]$ we define the Besov space 
\begin{equs} \label{besov}
B^{\alpha}_{p,q}(\bfR^{d+1}) = \Big\{ u\in \mcD'(\bfR^{d+1}), \, \norme{u}_{B^{\alpha}_{p,q}} \defeq \norme{ \Big( 2^{i\alpha} \norme{\Delta_i u}_{L^p}  \Big)  }_{\ell^q} < +\infty       \Big\}.
\end{equs}
The integrability exponent $q$ will almost every time be equal to $\infty$,  we will therefore use the shorthand notation $B^{\alpha}_p =B^\alpha_{p,\infty}$.
We retain the definition of regularity-integrability structure from \cite{Hos25,BH23}.
A regularity-integrability structure is a regularity structure $\scrT=(T,T^+)$ equipped with an additional degree map $i:\mcB\sqcup \mcB^+ \to [1,+\infty]$ that  verifies for any $\sigma,\tau\in\mcB$ the relation  
\begin{equation} \label{integrability}
	\frac{1}{i(\tau)}= 	\frac{1}{i(\tau/\sigma)}+	\frac{1}{i(\sigma)}.
\end{equation} 
We will write $r$ for the degree $|\cdot|$ on the underlying regularity structure. 
The couple $(r,1/i)$ defines a bigrading over $\scrT$.
A model over a regularity-integrability structure is written as a couple $(\PPi,g)$ such that
\begin{align*}
	\Vert\tau\Vert_{ (\PPi,g)}=\sup_{\phi\in\mcB}\sup_{\lambda\in(0,1)}\lambda^{-|\tau|}\norme{\lg\Pi_x\tau, \phi_x^\lambda\rg}_{L_x^{i(\tau)}}<+\infty
	\\
	\Vert\nu\Vert_{ (\PPi,g)}=\sup_{|h|\in(0,1)}\norme{   \frac{\gamma_{x,x+h}(\nu)}{|h|^{r(\nu)}}}_{L_x^{i(\nu)}} <+\infty
\end{align*}
where the subscript $x$ in $L_x^{i(\nu)}$ means that we are looking at the integrability in the $x$ variable. We also set recursively
\begin{equation*} \begin{split}
		\Vert\tau\Vert_{ (\PPi,g)}^* \defeq \max \Big( \Vert\tau\Vert_{ (\PPi,g)},\, \max_{\sigma<\tau} \Vert\tau/\sigma\Vert_{ (\PPi,g)} \Vert\sigma\Vert_{ (\PPi,g)}^*\Big).   
\end{split} \end{equation*}


\subsection{Decorated trees with Malliavin derivatives.} \label{subsect_Malliavin_Trees}

We introduce new symbols denoted by $ \dot{\Xi}_\ell $ and $\dot{\Xi}_\ell^0$ that both encode the Malliavin derivative of $\xi_\ell$ denoted by $\delta \xi_\ell$. We extend the degree map $|\cdot|$ setting
\begin{equation*}
	| \dot{\Xi}_\ell | = |\Xi_\ell| + \frac{|\frak s|}{2}, \quad | \dot{\Xi}^0_\ell | = |\Xi_\ell|.
\end{equation*}
The symbol $\dot\Xi_\ell$ takes into account the gain of regularity provided by the Malliavin derivative whereas the symbol $ \dot\Xi_\ell^{0}$ does not. We will consider decorated trees which have only one symbol $ \dot{\Xi}$ or $\dot{\Xi}^0$.	We  define a derivative $D_{\Xi} $ on decorated trees that makes appear this new symbol. It is given inductively for $\sigma, \tau$ decorated trees as
\label{D_Xi}
\begin{equs}
	D_{\Xi} \Xi_\ell & = \dot{\Xi}_{\ell}, \quad D_{\Xi} X^k = 0, \quad D_{\Xi} \CI_{k}(\tau) = \CI_{k}(D_{\Xi} \tau), \\ D_{\Xi}(\sigma \tau) & = D_{\Xi}(\sigma) \tau + \sigma D_{\Xi}(\tau).
\end{equs}
Below, we provide an example of computation fro the stochastic multipplicative heat equation with the graphical notation introduced before
\begin{equs}
	D_{\Xi} \,  \<generic> =  \<genericD> =  \dot{\Xi}, \quad D_{\Xi}  \, \<Xii> = \<Xiig1> + \<Xiig2>. 
\end{equs}

The set of decorated trees with no Malliavin derivatives is still denoted by $\mcT$. The set of trees containing exactly one noise with a Malliavin derivative $\dot\Xi_\ell$ is denoted $\dot\mcT$ and the set of trees containing exactly one noise with the Malliavin derivative $\dot\Xi_\ell^0$ is denoted $\dot\mcT_0$. We also write $\dot\mcT^+$ and $\dot\mcT_0^+$ for the corresponding set of positive trees.
One also define an integrability exponent $i : \mcT \sqcup \mcT^+ \sqcup \dot{\mcT} \sqcup \dot\mcT^{+} \to [1;+\infty]$ by setting 
$$
i(\tau) = \left\{ \begin{array}{cc} 2 & \text{for  }\tau\in\dot{\mcT}\sqcup \dot\mcT^+ \\ +\infty &\text{for  }\tau\in\mcT\sqcup{\mcT}^{+} 
\end{array}  \right. .
$$
From the Besov injection
\begin{equs} \label{Besov_injection}
	j : B^{\alpha}_2(\bfR^{d+1}) \hookrightarrow C^{\alpha-\frac{|\mathfrak{s}|}{2}}(\bfR^{d+1}),
\end{equs} one defines a map $j_*$ on decorated trees.
The map $j_* : \mcT \sqcup \mcT^+ \sqcup \dot{\mcT} \sqcup \dot\mcT^{+}  \to \mcT \sqcup \mcT^+   \sqcup \dot\mcT_0 \sqcup \dot\mcT^+_0 $ is defined inductively by setting for any $\ell\in\{1,...,\ell_0\}$ and $n\geq0$
$$
j_* (\dot\Xi_\ell) = \dot\Xi_\ell^0 , \qquad j_*(\Xi_\ell) = \Xi_\ell,\qquad j_*(X^n) = X^n
$$
and for any $\tau,\tau_1,\dots,\tau_n\in\mcT$
$$
j_* (\mcI_k (\tau)) = \mcI_k(j_*(\tau)),\qquad j_*\Big(\prod_{j=1}^n \tau_j\Big) =\prod_{j=1}^n  j_*(\tau_j) 
$$
and for $\tau\in\mcT$ and $\tau_1,\dots,\tau_n\in\dot{\mcT}^+$
$$
j_* (\mcI_k^+ (\tau)) = \mathbf{1}_{|j_*(\tau)|+\beta-|k|>0}\, \mcI_k^+(j_*(\tau)),\qquad j_*\Big(\prod_{j=1}^n \tau_j\Big) =\prod_{j=1}^n  j_*(\tau_j).
$$
We illustrate the value of $j^*$ on the same examples as for $ D_{\Xi} $: 
\begin{equs}
	j^{*}( \<generic>) = \<generic>, \quad j^{*}(\<genericD>) = \<genericDj> = \dot{\Xi}^0, \quad j^*(D_{\Xi}  \, \<Xii>) = j^*(\<Xiig1>) + j^*(\<Xiig2>) = \<Xiip1> + \<Xiip2>.
	\end{equs}
We can show that we have the following commutation relation for $\tau\in\dot{\mcT}$
\begin{equation}\label{eq_commut_j*_deriv}
	\downarrow^kj_*=j_*\downarrow^k,\qquad\qquad D^nj_*(\tau)=\mathbf{1}_{|j_*\tau|-|n|>0  } \, j_*(D^n\tau).
\end{equation}

\begin{lemma}
	For any $\tau\in\dot\mcT$ we have
	\begin{equation}\label{eq_commut_jstar_Delta}
		\Delta \, j_*(\tau) = (j_*\otimes j_*)\Delta \tau,
	\end{equation}
	and for any $\tau\in\mcT$
	\begin{equation}\label{eq_commut_jstar_DXi0}
		\Delta \, j_*(D_\Xi\tau) = \big(\id\otimes j_*D_\Xi+j_*D_\Xi\otimes \id\big)\Delta\tau.
	\end{equation}
\end{lemma}

\begin{proof}
	We prove the identities by induction.
	We first verify the identity for $\tau=X^n$, $\Xi_\ell$ or $\dot{\Xi}_\ell$. For $n\geq0$ and $\tau\in\dot\mcT$ we have by induction
	\begin{align*}
		\Delta \big(j_*(\mcI_n(\tau))\big) = \Delta \mcI_n(j_*(\tau)) &= (\mcI_n\otimes \id )\Delta j_*(\tau)   + \sum_{|p|<|\mcI_n(j_*(\tau))| } \frac{X^p}{p!} \otimes \mcI_{n+p}^+(j_*(\tau))
		\\
		&=(\mcI_nj_*\otimes j_*)\Delta\tau + \sum_{|p|<|\mcI_n(\tau)| } j_*(\frac{X^p}{p!}) \otimes j_*(\mcI_{n+p}^+(\tau))
		\\
		&=(j_*\otimes j_*)\Big((\mcI_n\otimes \id)\Delta\tau + \sum_{|p|<|\mcI_n(\tau)| } \frac{X^p}{p!} \otimes \mcI_{n+p}^+(\tau)\Big)
		\\
		&=(j_*\otimes j_*)\Delta\tau
	\end{align*}
	and if $\tau=\tau_1\tau_2$, from the multiplicativity of the coproduct
	\begin{align*}
		\Delta(j_*(\tau_1\tau_2)) &= \Delta\big(j_*(\tau_1)j_*(\tau_2)\big) =
		\Delta\big(j_*(\tau_1)\big)\,\Delta\big(j_*(\tau_2)\big) \\&=
		(j_*\otimes j_*)\Delta(\tau_1) \,	(j_*\otimes j_*)\Delta(\tau_2) 
		\\&= (j_*\otimes j_*)\Delta\tau_1 \Delta\tau_2  =(j_*\otimes j_*)\Delta(\tau_1\tau_2).
	\end{align*} 
	We now prove the second identity by induction.
	For any $n\geq0$ and $\tau\in\mcT$, as $|j_*(D_\Xi\tau)|=|\tau|$, we have
	\begin{align*}
		\Delta\big(j_*(D_\Xi\mcI_n(\tau) )\big) & = \Delta\big(\mcI_n(j_*(D_\Xi\tau)) \big)\\ &= \big(\mcI_n \otimes \id \big)\Delta(j_*(D_\Xi\tau)) +   \sum_{|p|<|\mcI_n(j_*(D_\Xi\tau))| } \frac{X^p}{p!} \otimes \mcI_{n+p}^+(j_*(D_\Xi\tau))
		\\
		&= \big(\mcI_nj_*D_\Xi \otimes \id \big)\Delta\tau  + \big(\mcI_n \otimes j_*D_\Xi \big)\Delta\tau
		\\
		&\quad+ \sum_{|p|<|\mcI_n(\tau)| } \frac{X^p}{p!} \otimes j_*(D_\Xi\mcI_{n+p}^+(\tau))
		\\
		&= \big(j_*D_\Xi\otimes\id + \id\otimes j_*D_\Xi   \big) (\mcI_n\otimes\id)\Delta\tau 
		\\
		&\quad+\Big(j_*D_\Xi\otimes\id + \id\otimes j_*D_\Xi   \Big)\Big( \sum_{|p|<|\mcI_n(\tau)| } \frac{X^p}{p!} \otimes\mcI_{n+p}^+(\tau)      \Big)
		\\
		&= \big(j_*D_\Xi\otimes\id + \id\otimes j_*D_\Xi   \big)\Delta\tau,
	\end{align*}
	where we have used $D_\Xi X^k=0$.
	For $\tau=\tau_1\tau_2\in\mcT$ we have $j_*(D_\Xi\tau)=j_*(D_\Xi \tau_1)\,\tau_2+\tau_1\,j_*(D_\Xi \tau_2) $, from the induction hypothesis
	\begin{align*}
		& \Delta\big(j_* D_\Xi(\tau_1\tau_2)\big) 
		\\
		&=  \big(\id\otimes j_*D_\Xi+j_*D_\Xi\otimes \id\big)\Delta\tau_1\, \Delta \tau_2  + \Delta\tau_1\, \big(\id\otimes j_*D_\Xi+j_*D_\Xi\otimes \id\big)\Delta\tau_2
		\\
		&= \big(\id\otimes j_*D_\Xi+j_*D_\Xi\otimes \id\big)\Delta(\tau_1\tau_2).
	\end{align*}
\end{proof}
We extend naturally the definition of $Z$ on these new decorated trees by interpreting $\dot{\Xi}_\ell$ and $\dot{\Xi}_\ell^0$ as $ \delta \xi_\ell$. The preparation map used for renormalising trees containing Malliavin derivatives is such that
\begin{equation}
	\label{commutation_R}
	R D_{\Xi} = D_{\Xi} R, \quad \text{and}\qquad Rj_*=j_*R.
\end{equation}

\subsection{Regularity-integrability structure on words}

One defines the regularity-integrability structure of words setting for any letter $a\in\mcA$ a degree $r(a)$ and an integrability $i(a)$.
For $w=\lg a_1,\dots,a_n\rg_\ell X^k$, we set 
\begin{equation}\label{eq_degree_besov_words}
r(w) =  \sum_{i=1}^n r(a_i) +  |\ell|+|k|,
\quad \quad
\frac{1}{i(w)} =  \sum_{i=1}^n \frac{1}{i(a_i)}.
\end{equation}
The following proposition is a generalisation of  \cite[Theorem 1]{LocalExpansionsParacSystems} towards Besov spaces. Proof elements are given in Appendix \ref{Appendix_B}.
\begin{proposition}\label{prop_modelwords_besov}
	Let $Z_a\in B^{r(a)}_{i(a)}$ for any $a\in\mcA$ setting 
	\begin{align*}
		\PPi(\lg a_1,\dots,a_n\rg_\ell ) &= \P_\ell(Z_{a_1},\dots,Z_{a_n})
		\\
		g(\lg a_1,\dots,a_n\rg_\ell^k) &= \partial^k_*\P_\ell(Z_{a_1},\dots,Z_{a_n})
	\end{align*}
	defines a model over the regularity-integrability structure of the words.
	\end{proposition}
From the Besov injection \eqref{Besov_injection},
one defines a map $j_*$ on words, by setting
\begin{equation}\label{eq_def_jstar_word}
	j_*\big(\lg a_1,\dots a_n\rg_{\ell}X^k \big) = \big\lg a_1^0,\dots a_n^0\big\rg_{\ell}X^k.
\end{equation}
where one introduces a new alphabet $ (\CA)^0  $ given by \label{inj_ab}
\begin{equs}
\CA^0 = \{  a^0, \, a \in \CA \}.
\end{equs}
Here, $a^0$ are new letters whose degrees are given by 
\begin{equation}\label{eq_def_jstar_letter}
r(a^0)=r(a)-\frac{\mathfrak{|s|}}{i(a)},\qquad i(a^0)=\infty.
\end{equation}
We notice that the map $j_*$ maps an element from a regularity-integrability structure to an element of a regularity structure, forgetting the integrability of the element. We have defined this map both on decorated trees and words. We make an abuse of notation by using the same name for both maps. Below, we illustrate the value of $j^{*}$ on some words whose letters are decorated trees
\begin{equs}
	j^{*}(\<generic>) = (\<generic>)^0, \quad j^{*}(  \<IXi> \otimes \<genericD> ) =  (\<IXi>)^0 \otimes (\<genericD>)^0. 
\end{equs}

We want to construct $Z_{\tau}$ with $ \tau $ containing a noise of type $\dot{\Xi}_{\ell}$. We will show that $ Z_{\tau} \in  B^{|\tau|}_2(\bfR^{d+1}) $, then by the Besov injection one has $ Z_{\tau} \in  C^{|\tau|-\frac{|\mathfrak{s}|}{2}}(\bfR^{d+1}) $. Therefore, one can set
\begin{equs}
	Z_{\tau^0} := Z_{\tau} \in  C^{r( \tau^0)}(\bfR^{d+1}).
\end{equs}

	 Given a regularity integrability structure and a model $(\PPi,g)$ over it, we can associate the distributions $(Y_\tau)$ via  \eqref{eqdef_Ztau1} and  \eqref{eqdef_Ztau2}. The Proposition \ref{prop_regularity_Y_Besov} below generalises Theorem \ref{thm_regularityZtau} and gives the Besov regularity of these distributions $(Y_\tau)$.  
	
	We recall the bracket $ [\tau]_i $ defined via the formula
	$$
	[\tau]_i = (\P \tau)_i + \sum_{\sigma\prec\tau} ({\bf Q}(\tau/\sigma))_i \, [\sigma]_i,
	$$
	with 
	\begin{align*}
		(\P \tau)_i(z) &= \int_{(\bfR^{d+1})^2} K_{<i-1}(z-x)K_i(z-y) \, {\Pi}_x\tau(y) \,  \ddd y \ddd x,
		\\
		(\mathbf{Q} \tau)_i(z) &= \int_{(\bfR^{d+1})^2} K_{<i-1}(z-x)K_{<i-1}(z-y) \, {\ \gamma}_{yx}(\tau) \, \ddd y \ddd x.
	\end{align*}
 It was proven in \cite[Proposition 8]{BailleulHoshinoRS1}, that if $(\PPi, { g})$ is a model, then we have the bounds \begin{equs}
	|(\P \tau)_i|\lesssim \Vert\tau\Vert_{ (\PPi,g)}^* 2^{-i|\tau|}, \quad |(\mathbf{Q} \tau)_i|\lesssim \Vert\tau\Vert_{ (\PPi,g)}^* 2^{-i|\tau|}.
\end{equs}
In the next proposition, we extend these bounds to regularity-integrability structure.

\begin{lemma}\label{lem_crochets_Besov}
	Let $\scrT$ a regularity-integrability structure and a model over it.
	For any $\tau\in\mcB$ and $\nu\in\mcB^+$, we have
	$$
	\sup_{i\geq-1}2^{ir(\tau)}\norme{\P(\tau)_i}_{L^{i(\tau)}} \lesssim  (\!| \PPi , g |\!) , \quad\quad \sup_{i\geq-1} 2^{ir(\nu)}\norme{{\bf Q}(\nu)_i}_{L^{i(\nu)}} \lesssim  (\!| \PPi , g |\!)
	$$
	and 
	$$
		\sup_{i\geq-1}2^{ir(\tau)}\norme{[\tau]_i}_{L^{i(\tau)}} \lesssim  (\!| \PPi , g |\!).
		$$
\end{lemma}
\begin{proof}
	We prove the first claim, the estimate on ${\mathbf{Q}(\nu)}_i$ follows from the same computations. We have 
	\begin{align*}
		 \P(\tau)_i(z) &= \int_{(\bfR^{d+1})^2} K_{<i-1}(z-x)K_i(z-y)\Pi_x\tau(y) \ddd x \ddd y 
		\\& =  \int_{(\bfR^{d+1})^2} K_{<i-1}(z-x)K_i(z-y)\Big(\Pi_z\tau(y)+ \sum_{\sigma<\tau}\gamma_{zx}(\tau/\sigma)\Pi_z\sigma(y)\Big) \ddd x \ddd y
		\\
		&= \int_{(\bfR^{d+1})^2} K_{<i-1}(h_1)K_i(h_2) \\ & \Big(\Pi_z\tau(z+h_2)+ \sum_{\sigma<\tau}\gamma_{z,z+h_1}(\tau/\sigma)\Pi_z\sigma(z+h_2)\Big) \ddd h_1 \ddd h_2.
	\end{align*}
Then, from the triangular and the Hölder inequalities 
\begin{align*}
		& \norme{\P(\tau)_i(z)}_{L_z^{i(\tau)}} \leq \int_{(\bfR^{d+1})^2} |K_{<i-1}(h_1)K_i(h_2)|\Big(\norme{\Pi_z\tau(z+h_2)}_{L_z^{i(\tau)}}
		\\&+ \sum_{\sigma<\tau}\norme{\gamma_{z,z+h_1}(\tau/\sigma)}_{L_z^{i(\tau/\sigma)}}\big\lVert{\Pi_z\sigma(z+h_2)\Big)}\big\lVert_{L_z^{i(\sigma)}} \ddd h_1 \ddd h_2
		\\
		&\leq (\!| \PPi , g |\!)  \int_{(\bfR^{d+1})^2} |K_{<i-1}(h_1)K_i(h_2)|\Big( |h_2|^{r(\tau)} + \sum_{\sigma<\tau}|h_1|^{r(\tau/\sigma)} |h_2|^{r(\sigma)} \Big) \ddd h_1 \ddd h_2.
\end{align*}
From the scaling property of  the $K_i$ we get indeed
$$
\norme{\P(\tau)_i(z)}_{L_z^{i(\tau)}} \leq C(\tau) (\!| \PPi , g |\!) 2^{-ir(\tau)}.
$$
We prove the estimate on $[\tau]_i$ by induction on $\tau$ by writing 
$$
[\tau]_i = \P(\tau)_i + \sum_{\sigma<\tau} {\mathbf{Q}(\tau/\sigma)}_i [\sigma]_i,
$$
and applying Hölder inequality.
\end{proof}

\begin{proposition}\label{prop_regularity_Y_Besov}
	For any regularity-integrability structure $\scrT$ and model $(\PPi,g)$ over it, we have $Y_\tau\in B^{r(\tau)}_{i(\tau)}$ for any $\tau$, and 
	$$
	\norme{Y_\tau}_{B^{r(\tau)}_{i(\tau)}} \lesssim \Vert\tau\Vert_{ (\PPi,g)}^*.
	$$
\end{proposition}

\begin{proof}
	The proof is the same as the one of Theorem \ref{thm_regularityZtau}. It is proven by induction on $\tau$ with the relation $[\tau]_i=[\Psi(\tau)]_i$ from Lemma \ref{lem_sumcrochets}, written under the form
	$$
	\Delta_i Y_\tau = [\tau]_i - [\overline{\Psi}(\tau)]^\mcY_i.
	$$
	From Lemma \ref{lem_crochets_Besov} and Proposition \ref{prop_modelwords_besov}, we have indeed 
	$$
		\sup_{i\geq-1}2^{ir(\tau)}\norme{[\tau]_i}_{L^{i(\tau)}} \lesssim  (\!| \PPi , g |\!), \quad
		\sup_{i\geq-1}2^{ir(\tau)}\norme{[\overline{\Psi}(\tau)]^\mcY_i}_{L^{i(\tau)}} \lesssim  (\!| \PPi , g |\!).
	$$
	Hence $\norme{Y_\tau}_{B^{r(\tau)}_{i(\tau)}}\lesssim  (\!| \PPi , g |\!).  $
\end{proof}

%
%

\subsection{Main algebraic identity}

This subsection is devoted to Theorem \ref{prop_alg_identity_convergence} and its proof.
In general we do not have the identity $\delta \widehat{Z}_\tau=\widehat{Z}_{D_{\Xi}\tau}$, the following proposition gives the correct commutation relation between $\delta$ and the map $\widehat{Z}_{\cdot}$.

\begin{theorem}\label{prop_alg_identity_convergence}
	For any tree $\tau\in\mcT$, we have 
	$$
	\delta \widehat{Z}_\tau = Y^{\widehat{\mcZ}}_{j_*\Psi(D_\Xi\tau)}.
	$$
\end{theorem}
The map $j_*$ on words was defined in \eqref{eq_def_jstar_letter} and \eqref{eq_def_jstar_word} and can be viewed as some type of projection that performs at the level of the algebra the Besov injection. 
For proving Theorem \ref{prop_alg_identity_convergence}, we  will need the following lemma.

\begin{lemma} \label{lem_PiDXiTau}
	We let $R$ to be some preparation map and $\widehat\mcZ$ the associated renormalised stochastic data. For any $\tau\in\dot\mcT$ we have
	$$
	\widetilde{\PPi}^{\widehat{\mcZ}}(\Psi(\tau)) = \widetilde{\PPi}^{\widehat{\mcZ}}(\Psi(j_*\tau)).
	$$
\end{lemma}
\begin{proof}
The identity is proven by induction on the number of nodes of $\tau$. The identity is trivial when $\tau$ has form  $X^n\Xi_\ell$ or $X^n\dot{\Xi}_\ell$.
For $\tau$ with more than one node, we first prove that 
\begin{equation}\label{eq_PiMult_tau=j*tau}
	\widetilde{\PPi}^{\widehat{\mcZ}^\circ}(\Psi(\tau)) = \widetilde{\PPi}^{\widehat{\mcZ}^\circ}(\Psi(j_*\tau)). 
\end{equation}
If $\tau = X^p \Xi_\ell \,\mcI_{b_0} (\tau_0)\,\prod_{j=1}^{n}\sum_{l_j\geq0}\mcI_{b_j}(\tau_j)$ with $\tau_0\in\dot\mcT$, from \eqref{eq_mult_PPi_Zcirc} and the induction hypothesis
 \begin{align*}
 &	\widetilde{\PPi}^{\widehat{\mcZ}}\big(\Psi(\tau)\big)
 	\\ &=  \mathbf{1}_{p=0} \,\xi_\ell \,\sum_{l_0\geq0} \frac{1}{l_0 !}(\partial^{b_0}K)^{[l_0]} * \PPi^{\widehat{\mcZ}}\big(\Psi(\downarrow^{l_0}\!\tau_0)\big) 
 	 \prod_{j=1}^{n}\sum_{l_j\geq0} \frac{1}{l_j !} (\partial^{b_j} K)^{[l_j]} \!*\! \PPi^{\widehat{\mcZ}}(\Psi(\downarrow^{l_j}\!\tau_j)) 
 	\\
 	&= \mathbf{1}_{p=0} \,\xi_\ell \,\sum_{l_0\geq0} \frac{1}{l_0 !} (\partial^{b_0}K)^{[l_0]} * \PPi^{\widehat{\mcZ}}\big(\Psi(j_*\downarrow^{l_0}\!\tau_0)\big) 
 	\prod_{j=1}^{n}\sum_{l_j\geq0}   \frac{1}{l_j !} (\partial^{b_j} K)^{[l_j]} \!*\! \PPi^{\widehat{\mcZ}}(\Psi(\downarrow^{l_j}\!\tau_j)).
 \end{align*}
 On the other hand, we have $j_*\tau = X^p \Xi_\ell \mcI_{b_0} (j_*\tau_0)\prod_{j=1}^{n}\mcI_{b_j}(\tau_j)$ and \eqref{eq_mult_PPi_Zcirc} gives indeed
 \begin{align*}
& \widetilde{\PPi}^{\widehat{\mcZ}^\circ}\big(\Psi(j_*\tau)\big) \\ & = \mathbf{1}_{p=0} \,\xi_\ell \,\sum_{l_0\geq0} \frac{1}{l_0 !} ((\partial^{b_0}K)^{[l_0]} * \widetilde\PPi^{\widehat{\mcZ}}\big(\Psi(\downarrow^{l_0}\!j_*\tau_0)\big) \, \prod_{j=1}^{n}\sum_{l_j\geq0} \frac{1}{l_j !} (\partial^{b_j} K)^{[l_j]} \!*\! \widetilde\PPi^{\widehat{\mcZ}}(\Psi(\downarrow^{l_j}\tau_j)).
 \end{align*}
 If now we let $\tau = X^p \dot{\Xi}_\ell \prod_{j=1}^n \mcI_{b_j}\tau_j$, we see from \eqref{eq_mult_PPi_Zcirc} that both $ {\PPi}^{\widehat{\mcZ}^\circ}\big(\Psi(\tau)\big)$ and ${\PPi}^{\widehat{\mcZ}}\big(\Psi(j_*\tau)\big)$ are equal to
 $
 \mathbf{1}_{p=0}\,\delta\xi_\ell \, \prod_{j=1}^{n} \sum_{l_j\geq0} \frac{1}{l_j !} (\partial^{b_j} K)^{[l_j]} * \PPi^{\widehat{\mcZ}}(\Psi(\downarrow^{l_j}\tau_j)).
 $
 Hence \eqref{eq_PiMult_tau=j*tau}.
 Then using \eqref{eq_PiMult_tau=j*tau} and $Rj_*=j_*R$, we have
 \begin{align*}
 	\widetilde{\PPi}^{\widehat{\mcZ}}\big(\Psi(\tau)\big) &= \widetilde{\PPi}^{\widehat{\mcZ}^\circ}\big(\Psi(R\tau)\big) = {\PPi}^{\widehat{\mcZ}^\circ}\big(\Psi(j_*R\tau)\big)  
 	=\widetilde{\PPi}^{\widehat{\mcZ}}\big(\Psi(j_*\tau)\big).
 \end{align*}
 
\end{proof}

\begin{proposition} \label{lem_deltaZ_Zj*}
	For any tree $\tau\in\mcT$, we have
	$
	\delta Z_\tau = Z_{j_*(D_\Xi\tau)}.
	$
\end{proposition}
\begin{proof}
We first prove the following identities by induction on $\tau$
\begin{equation}\label{eq_commut_Pi_delta}
\widetilde{\PPi}^{\widehat{\mcZ}^\circ}\big(\Psi(D_\Xi\tau)\big) =\delta \widetilde{\PPi}^{\widehat{\mcZ}^\circ}\big( \Psi(\tau)  \big), \quad \widetilde{\PPi}^{\widehat{\mcZ}}\big(\Psi(D_\Xi\tau)\big) =\delta \widetilde{\PPi}^{\widehat{\mcZ}}\big( \Psi(\tau)  \big).
\end{equation}
Both identities are seen to be true if $\tau$ has the form $X^n\Xi_\ell$. If we take $\tau=X^p\Xi_\ell\prod_{j=1}^n\mcI_{b_j}(\tau_j)$, writing 
$$
D_\Xi\tau = X^p \dot{\Xi}_\ell\prod_{j=1}^n\mcI_{b_j}(\tau_j) + X^p\Xi_\ell \sum_{j_0=1}^n  \mcI_{b_{j_0}}(D_\Xi\tau_{j_0}) \prod_{j\ne j_0}\mcI_{b_j}(\tau_j),
$$
we get from \eqref{eq_mult_PPi_Zcirc}
\begin{align*}
&	\widetilde{\PPi}^{\widehat{\mcZ}^\circ}(\Psi(D_\Xi\tau)) =
	 \mathbf{1}_{p=0} \delta\xi_\ell\prod_{j=1}^n\sum_{l_j\geq0} \frac{1}{l_j !} (\partial^{b_j}K)^{[l_j]}*\widetilde{\PPi}^{\widehat{\mcZ}}(\Psi(\downarrow^{l_j}\tau_j)) 
	\\
	&\quad+ \mathbf{1}_{p=0}\xi_\ell \sum_{j_0=1}^n \sum_{l_{j_0}\geq 0} \frac{1}{l_{j_0} !}  \Big((\partial^{b_{j_0}}K)^{[l_{j_0}]}* \widetilde{\PPi}^{\widehat{\mcZ}}\big(\Psi(D_\Xi\downarrow^{l_{j_0}}\tau_{j_0})\big)\Big) \\ & \hspace{2cm} \prod_{j\ne j_0}\sum_{l_j\geq0} \frac{1}{l_j !} (\partial^{b_j}K)^{[l_j]}* \widetilde{\PPi}^{\widehat{\mcZ}}(\Psi(\downarrow^{l_j}\tau_j)).
\end{align*}
Then from the induction hypothesis and the Leibniz rule
\begin{align*}
	& \widetilde{\PPi}^{\widehat{\mcZ}^\circ}(\Psi(D_\Xi\tau)) =
	 \mathbf{1}_{p=0} \delta\xi_\ell\prod_{j=1}^n\sum_{l_j\geq0} \frac{1}{l_j!} (\partial^{b_j}K)^{[l_j]}*\widetilde{\PPi}^{\widehat{\mcZ}}(\Psi(\downarrow^{l_j}\tau_j)) 
	\\
	&\quad+  \mathbf{1}_{p=0}\xi_\ell \sum_{j_0=1}^n\sum_{l_{j_0}\geq 0 }   \frac{1}{l_{j_0}!} \delta\Big((\partial^{b_{j_0}}K)^{[l_{j_0}]}* \widetilde{\PPi}^{\widehat{\mcZ}}\big(\Psi({\downarrow^{l_{j_0}}}\tau_{j_0})\big)\Big) \\
  & \vspace{4cm} \prod_{j\ne j_0}\sum_{l_j\geq0}  \frac{1}{l_j!} (\partial^{b_j}K)^{[l_j]}* \widetilde{\PPi}^{\widehat{\mcZ}}(\Psi(\downarrow^{l_j}\tau_j))
	\\
	&= 	\delta \Big( \mathbf{1}_{p=0} \xi_\ell\prod_{j=1}^n\sum_{l_j\geq0}  \frac{1}{l_j!} (\partial^{b_j}K)^{[l_j]}*\widetilde{\PPi}^{\widehat{\mcZ}}(\Psi(\downarrow^{l_j}\tau_j))\Big).
\end{align*}
Then \eqref{eq_mult_PPi_Zcirc} gives indeed $\widetilde{\PPi}^{\widehat{\mcZ}^\circ}(\Psi(D_\Xi\tau)) = \delta \widetilde{\PPi}^{\widehat{\mcZ}^\circ}( \Psi(\tau)  )$.
Using the commutation relation $RD_\Xi=D_\Xi R$, we get
\begin{align*}
	\widetilde{\PPi}^{\widehat{\mcZ}}(\Psi(D_\Xi\tau)) =
	\widetilde{\PPi}^{\widehat{\mcZ}^\circ}(\Psi(RD_\Xi\tau)) =\delta \widetilde{\PPi}^{\widehat{\mcZ}^\circ}( \Psi(R\tau)  )  = \delta \widetilde{\PPi}^{\widehat{\mcZ}}( \Psi(\tau)  ).
\end{align*}
Hence \eqref{eq_commut_Pi_delta}. From Lemma \ref{lem_PiDXiTau} we have
\begin{equation}\label{eq_delta_Pi_jstarTau}
\widetilde{\PPi}^{\widehat{\mcZ}}(\Psi(j_*D_\Xi\tau)) =\delta \widetilde{\PPi}^{\widehat{\mcZ}}( \Psi(\tau)  ).
\end{equation}
We are now able to prove the proposition by induction, from the definition of the $\Psi$ map and from \eqref{eq_commut_jstar_DXi0}
\begin{align*}
	\widehat{Z}_{j_*D_\Xi\tau} &= \widetilde{\PPi}^{\widehat{\mcZ}}\big(\Psi(j_*D_\Xi\tau)\big)-\sum_{\substack{\mu\prec j_*D_\Xi\tau\\k\geq0}} \frac{1}{k!}\P_k\Big( \widetilde{\PPi}^{\widehat{\mcZ}}\big(\Psi(\downarrow^k (j_*D_\Xi\tau)/\mu)\big) , \widehat{Z}_\mu \Big)
	\\
	&= \widetilde{\PPi}^{\widehat{\mcZ}}\big(\Psi(j_*D_\Xi\tau)\big)-\sum_{\substack{\sigma\prec\tau\\ k \geq0}} \frac{1}{k!}\P_k\Big( \widetilde{\PPi}^{\widehat{\mcZ}}\big(\Psi(\downarrow^k j_*D_\Xi(\tau/\sigma))\big) , \widehat{Z}_\sigma \Big)
	\\
	&\hspace{3cm}-\sum_{\substack{\sigma\prec\tau\\ k \geq0}} \frac{1}{k!}\P_k\Big( \widetilde{\PPi}^{\widehat{\mcZ}}\big(\Psi(\downarrow^k \tau/\sigma)\big) , \widehat{Z}_{j_*D_\Xi\sigma}\Big).
\end{align*}
Then from the induction hypothesis, from \eqref{eq_delta_Pi_jstarTau} and from the Leibniz rule
\begin{align*}
	\widehat{Z}_{j_*D_\Xi\tau} &= \delta\widetilde{\PPi}^{\widehat{\mcZ}}\big(\Psi(\tau)\big)-\sum_{\substack{\sigma\prec\tau\\ k \geq0}} \frac{1}{k!}\P_k\Big( \delta\widetilde{\PPi}^{\widehat{\mcZ}}\big(\Psi(\downarrow^k(\tau/\sigma))\big) , \widehat{Z}_\sigma \Big)
	\\
	&\hspace{2.5cm}-\sum_{\substack{\sigma\prec\tau\\ k \geq0}} \frac{1}{k!}\P_k\Big( \widetilde{\PPi}^{\widehat{\mcZ}}\big(\Psi(\downarrow^k \tau/\sigma)\big) , \delta \widehat{Z}_{\sigma}\Big)
	\\
	&= \delta \Big( \widetilde{\PPi}^{\widehat{\mcZ}}\big(\Psi(\tau)\big)   -\sum_{\substack{\sigma\prec\tau\\ k \geq0}} \frac{1}{k!}\P_k\Big( \widetilde{\PPi}^{\widehat{\mcZ}}\big(\Psi(\downarrow^k \tau/\sigma)\big) , \widehat{Z}_{\sigma}\Big)\Big)
	\\
	&=\delta \widehat{Z}_\tau.
\end{align*}
\end{proof}

\begin{lemma}\label{lem_g_jstarPsi}
	For any tree $\tau\in \dot\mcT\sqcup \dot\mcT^{+}$ and $n\geq0$
	\begin{enumerate}
	\item IF $|n|<|j_*\tau|$, we have 
	\begin{equation}
	\widetilde{g}_x^{\widehat{\mcZ}}\big(D^nj_* \Psi(\tau) \big) =  \widetilde{g}_x^{\widehat{\mcZ}}\big(D^n \Psi(j_*\tau) \big).
\end{equation}
	\item IF $|n|>|j_*\tau|$, we have 
	\begin{equation}
		\partial^n_y \Big( \Pi_x^{\widehat{\mcZ}}\big( j_*\Psi(\tau)\big)(y)    \Big)_{|y=x} =  \partial^n_y \Big(  \Pi_x^{\widehat{\mcZ}} \big( \Psi(j_*\tau)   \big)(y)    \Big)_{|y=x}.
	\end{equation}
	\end{enumerate}
\end{lemma}

\begin{proof}
We prove both identities in the same induction on $\tau$.
We first consider $\tau\in\dot\mcT$, from \eqref{eq_gDnTau_general}  we have 
\begin{align*}
\widetilde{ g}^{\widehat{\mcZ}}\Big(D^nj_*\Psi(\tau)      \Big) &= \partial^n \Big(\widetilde{\PPi}^{\widehat{\mcZ}}\big(j_*\Psi(\tau)\big)\Big)
\\
&\quad- \sum_{\substack{\sigma\prec \tau}} \sum_{\substack{q\geq0 \\ |j_*\sigma|+|q|-|n|<0}}   \widetilde{g}^{\widehat{\mcZ}}\big(  D^q j_*\Psi(\tau/\sigma) \big)\partial^n_y\Big( \Pi^{\widehat{\mcZ}}_x \big(j_*\Psi_q(\sigma)\big)(y)\Big)_{|y=x}.  
\end{align*}
From Lemma \ref{lem_PiDXiTau} we have 
$$
 \widetilde{\PPi}^{\widehat{\mcZ}}\big(j_*\Psi(\tau)\big) = \widetilde{\PPi}^{\widehat{\mcZ}}\big(\Psi(\tau)\big)= \widetilde{\PPi}^{\widehat{\mcZ}}\big(\Psi(j_*\tau)\big).
$$
The induction hypothesis, Proposition \ref{Prop_CommutDerivativePsi} and Lemma \ref{lem_PiDXiTau} give
$$
\widetilde{g}^{\widehat{\mcZ}}\big(  D^q j_*\Psi(\tau/\sigma) \big) = \widetilde{g}^{\widehat{\mcZ}}\big(  D^q\Psi(j_*(\tau/\sigma)) \big) =  \widetilde{\PPi}^{\widehat{\mcZ}}\big(  \Psi(D^qj_*(\tau/\sigma)) \big) = \widetilde{\PPi}^{\widehat{\mcZ}}\big(  \Psi(j_*D^q(\tau/\sigma)) \big) .
$$
Then Lemma \ref{lem_SnDnreindex} and the induction hypothesis give
\begin{align*}
\widetilde{ g}_x^{\widehat{\mcZ}}\Big(D^nj_*\Psi(\tau)      \Big) &= \partial^n \big(\widetilde{\PPi}^{\widehat{\mcZ}}\big(\Psi(j_*\tau)\big)\big)
\\
&\quad-
\sum_{\substack{\mu\prec \tau}} \sum_{\substack{q\geq0 \\ |j_*\mu|-|n|<0}}   \widetilde{\PPi}^{\widehat{\mcZ}}\big(   \Psi(j_*(\tau/\mu)) \big)\partial^n_y\Big( \Pi^{\widehat{\mcZ}}_x \big(j_*\Psi_q(\downarrow^q \mu)\big)(y)\Big)_{|y=x} 
\\
 &= \partial^n \big(\widetilde{\PPi}^{\widehat{\mcZ}}\big(\Psi(j_*\tau)\big)\big)
 \\
 &\quad-
 \sum_{\substack{\mu\prec \tau \\ |j_*\mu|-|n|<0}}    \widetilde{\PPi}^{\widehat{\mcZ}}\big(  \Psi(j_*(\tau/\mu)) \big)\partial^n_y\Big( \Pi^{\widehat{\mcZ}}_x \big(\Psi( j_*\mu)\big)(y)\Big)_{|y=x} .
\end{align*}
On the other hand from \eqref{eq_gDnTau_general}  we have 
\begin{align*}
	\widetilde{ g}^{\widehat{\mcZ}}\Big(D^n\Psi(j_*\tau)      \Big) &= \partial^n \Big(\widetilde{\PPi}^{\widehat{\mcZ}}\big(\Psi(j_*\tau)\big)\Big)
	\\
	&\quad- \sum_{\substack{\sigma\prec j_*\tau}} \sum_{\substack{q\geq0 \\ |j_*\sigma|+|q|-|n|<0}}   \widetilde{g}^{\widehat{\mcZ}}\big( j_*\Psi(\tau/\sigma) \big)\partial^n_y\Big( \Pi^{\widehat{\mcZ}}_x \big(j_*\Psi(\sigma)\big)(y)\Big)_{|y=x}.
\end{align*}
We now take $\tau\in\dot\mcT^+$. There exists a unique tree $\tau_0\in\dot\mcT$ such that $\tau=D^0\tau_0$. We first prove 
\begin{equation}\label{eq_prooflem8.5} 	\widetilde{g}_x^{\widehat{\mcZ}}\big(D^nj_* \Psi(\tau) \big)=	\widetilde{g}_x^{\widehat{\mcZ}}\big(D^nj_* D^0\Psi(\tau_0) \big). 
\end{equation} 
From \eqref{eq_gDnTau_general} we have
\begin{align*}
	\widetilde{g}_x^{\widehat{\mcZ}}\Big( D^n j_* \Psi(\tau)     \Big) =  \partial^n\Big(\widetilde{\PPi}^{\widehat{\mcZ}}(\Psi(\tau))  \Big)  - \sum_{\sigma\prec \tau} \widetilde{g}^{\widehat{\mcZ}}_x\big(D^q\Psi(\tau/\sigma)\big) \partial^n_y \Big(\Pi_x^{\widehat{\mcZ}}\Psi_q(\sigma)(y)\Big)_{|y=x}. 
\end{align*}
On the other hand \eqref{eq_gDnTau_general} gives also
\begin{align*}
	\widetilde{g}^{\widehat{\mcZ}}\Big( D^n j_* D^0\Psi(\tau_0)     \Big) 
	=  \partial^n\Big(\widetilde{\PPi}^{\widehat{\mcZ}}(D^0\Psi(\tau))  \Big)  - \sum_{\sigma\prec\tau} \widetilde{g}^{\widehat{\mcZ}}_x\big(D^q\Psi(\tau/\sigma)\big) \partial^n_y \Big(\Pi_x^{\widehat{\mcZ}}\Psi_q(\sigma)(y)\Big)_{|y=x}
\end{align*}
where we have used Proposition \ref{Prop_CommutDerivativePsi}. Hence \eqref{eq_prooflem8.5}. Then 
\begin{align*}
	\widetilde{g}^{\widehat{\mcZ}}\big(D^nj_* \Psi(\tau) \big)&=	\widetilde{g}^{\widehat{\mcZ}}\big(D^nj_* D^0\Psi(\tau_0) \big) 
	= 
	\widetilde{g}^{\widehat{\mcZ}}\big(D^nj_* \Psi(\tau_0) \big) 
	\\
	&= \widetilde{g}^{\widehat{\mcZ}}\big(D^n \Psi(j_*\tau_0) \big) 
	=
	 \widetilde{g}^{\widehat{\mcZ}}\big( D^n\Psi( j_*\tau) \big) .
\end{align*}
\end{proof}

\begin{proposition}\label{lem_Zj*tau=YjstarPsi}
For any $\tau\in\mcT\sqcup\dot\mcT$ we have 
$
	\widehat{Z}_{j_*\tau} =  Y_{j_*\Psi(\tau)}^{\widehat{\mcZ}}.
$
\end{proposition}
\begin{proof}
We prove the identity by induction on $\tau\in\dot\mcT$. We have
$$
\Delta (j_*\Psi(\tau)) = \sum_{\substack{\sigma\prec\tau\\ q\geq 0}}  j_*\Psi_q(\sigma) \otimes  D^qj_*\Psi(\tau/\sigma)  + \sum_{n\geq 0}\frac{X^n}{n!}\otimes D^nj_*\Psi(\tau)   +  j_*\Psi(\downarrow^n\!\tau) \otimes \frac{X^n}{n!}. 
$$
From the definition of the $Y$ distributions, we have
\begin{align*}
	Y^{\widehat{\mcZ}}_{j_*\Psi(\tau)} = \widetilde{\PPi}^{\widehat{\mcZ}}\big(\Psi(\tau)\big) &- \sum_{\substack{\sigma\prec\tau\\ k\geq 0 }}  \frac{1}{k!}\P_k\Big(  \widetilde{g}^{\widehat{\mcZ}}\big(\downarrow^k\!D^qj_* \Psi(\tau/\sigma )\big) , Y_{j_*\Psi_q(\sigma)}^{\widehat{\mcZ}}   \Big)
	\\
	=\widetilde{\PPi}^{\widehat{\mcZ}}\big(\Psi(\tau)\big) &- \sum_{\substack{\sigma\prec\tau\\ k\geq 0 }}  \frac{1}{k!}\P_k\Big(  \widetilde{g}^{\widehat{\mcZ}}\big(D^qj_* \Psi(\downarrow^k\tau/\sigma )\big) , Y_{j_*\Psi_q(\sigma)}^{\widehat{\mcZ}}   \Big).
\end{align*}
From Lemma \ref{lem_PiDXiTau}
$
\widetilde{\PPi}^{\widehat{\mcZ}}\big(\Psi(\tau)\big) = \widetilde{\PPi}^{\widehat{\mcZ}}\big(\Psi(j_*\tau)\big).
$
From  Proposition \ref{Prop_CommutDerivativePsi} 
\begin{align*}
 \P_k\Big(  \widetilde{g}^{\widehat{\mcZ}}\big(  D^q j_*\Psi(\downarrow^k\!\tau/\sigma) \big) , Y^{\widehat{\mcZ}}_{j_*\Psi_q(\sigma)}   \Big)
	&
	=   \P_k\Big(  \widetilde{g}^{\widehat{\mcZ}}\big(  D^q \Psi(j_*\downarrow^k\!\tau/\sigma) \big) , Y^{\widehat{\mcZ}}_{j_*\Psi_q(\sigma)}   \Big)
	\\
	&=    \P_k\Big(  \widetilde{g}^{\widehat{\mcZ}}\big(   \Psi(j_*\downarrow^k\!D^q(\tau/\sigma)) \big) , Y^{\widehat{\mcZ}}_{j_*\Psi_q(\sigma)}   \Big).
\end{align*}
From  Lemma \ref{lem_SnDnreindex} and the induction hypothesis and \eqref{eq_PsinAndPsi}
\begin{align*}
	\sum_{\substack{\sigma\prec\tau \\ q\geq 0}}\P_k\Big(  \widetilde{g}^{\widehat{\mcZ}}\big(   \Psi(j_*\downarrow^k\!D^q(\tau/\sigma)) \big) , Y^{\widehat{\mcZ}}_{j_*\Psi_q(\sigma)}   \Big)
	&=
		\sum_{\substack{\sigma\prec\tau \\ q\geq 0}}\P_k\Big(  \widetilde{g}^{\widehat{\mcZ}}\big(   \Psi(j_*\downarrow^k\!\tau/\sigma) \big) , Y^{\widehat{\mcZ}}_{j_*\Psi_q(\downarrow^q\sigma)}   \Big)
		\\
		&=	\sum_{\substack{\sigma\prec\tau }}\P_k\Big(  \widetilde{g}^{\widehat{\mcZ}}\big(   \Psi(j_*\downarrow^k\!\tau/\sigma) \big) , Z_{j_*\sigma}  \Big).
\end{align*}
Then 
\begin{align*}
	Y^{{\widehat{\mcZ}}}_{j_*\Psi(D_\Xi\tau)} = \widetilde{\PPi}^{{\widehat{\mcZ}}}\big(\Psi(j_*\tau)\big) &- \sum_{\substack{\sigma\prec\tau \\ k \geq 0 }}\frac{1}{k!}\P_k\Big(  \widetilde{g}^{\widehat{\mcZ}}\big(   \Psi(j_*\downarrow^k\!\tau/\sigma) \big) , Z_{j_*\sigma}  \Big)
	\\
	&= \widetilde{\PPi}^{{\widehat{\mcZ}}}\big(\Psi(j_*\tau)\big) - \widetilde{\PPi}^{{\widehat{\mcZ}}}\big(\overline\Psi(j_*\tau)\big) = Z_{j_*\tau}.
\end{align*}
Hence the proposition.
\end{proof}

\begin{proof}[of Theorem \ref{prop_alg_identity_convergence}]
From Propositions \ref{lem_deltaZ_Zj*} and \ref{lem_Zj*tau=YjstarPsi} we have
\begin{equs} \label{triple_ident}
	\delta \widehat{Z}_\tau = \widehat{Z}_{j_*D_\Xi\tau} = Y^{\widehat{\mcZ}}_{j_*\Psi(D_\Xi\tau)}.
\end{equs}

\end{proof}

\begin{example}
We finish the section by providing an example that illustrates the identity \eqref{triple_ident}. We consider $ \tau = \<Xii> $ in the set-up of Subsection \ref{ex_gpam}, and we consider the naive stochastic data $Z_{\tau}$. One has
\begin{equs}
	\delta Z_{\<Xii>} & = \delta \left( \xi K*\xi - \P(K* \xi ,\xi) \right)
	\\ &  = \delta  \xi \, K * \xi +  \xi \, K * \delta \xi - \P(K* \xi , \delta \xi) - \P(K* \delta \xi ,\xi).
\end{equs}
One has
\begin{equs}
	Z_{j^*( D_{\Xi} \,  \<Xii> )} & = Z_{\<Xiip1>}
	+  Z_{\<Xiip2>}, \quad  \Psi(\<Xiip1>) =  \<Xiip1> + \<IXi> \otimes \<genericDj>, \quad
\Psi(\<Xiip2>) = \<Xiip2>+ \<IXip> \otimes \<generic> + \<IXippp> \otimes \<genericx>.
\end{equs}
Therefore, 
\begin{equs}
		Z_{j^*( D_{\Xi} \,  \<Xii> )} & = \widetilde{\PPi} \<Xiip1>  - \P( Z_{\<IXi>}  , Z_{\<genericDj>}) + \widetilde{\PPi} \<Xiip2>  - \P( Z_{\<IXip>}  , Z_{\<generic>})
			\\ & = \delta \xi K * \xi +  \xi K * \delta \xi - \P(K* \xi , \delta \xi) - \P(K* \delta \xi ,\xi).
\end{equs}
For the third expression, one has
\begin{equs}
	Y^{\mcZ}_{j_*\Psi(D_\Xi \, \<Xii> )} = Y^{\mcZ}_{ (\<Xiig1>)^0 } +
Y^{\mcZ}_{ 	(\<IXi>)^0 \otimes (\<genericD>)^0 } 	+ 
	 Y^{{\mcZ}}_{ (\<Xiig2>)^0 } +  Y^{{\mcZ}}_{ (\<IXig>)^0 \otimes (\<generic>)^0 }
\end{equs}
By definition of $ Y^{\widehat{\mcZ}} $, it is zero one words with more than one letter and zero polynomial decorations. Therefore
\begin{equs}
	Y^{\mcZ}_{ 	(\<IXi>)^0 \otimes (\<genericD>)^0 }  = Y^{{\mcZ}}_{ (\<IXig>)^0 \otimes (\<generic>)^0 } =0. 
\end{equs}
And 
$$
Y^{{\mcZ}}_{ (\<Xiig2>)^0 } =  Z_{\<Xiig2>},
\qquad 
Y^{{\mcZ}}_{ (\<Xiig1>)^0 }  = Z_{\<Xiig1>}
$$
Then 
\begin{equs}
		Y^{\mcZ}_{j_*\Psi(D_\Xi \, \<Xii> )} = Z_{ \<Xiig1> } +
	Z_{ \<Xiig2> } + 
	=Z_{D_\Xi \<Xii> }.
\end{equs}

\end{example}

\subsection{Inductive proof of convergence}

\label{BPHZ_preparation}
According to the assumptions of Theorem \ref{main_convergence}, we take for each $\ell\in\{1,...,\ell_0\}$ a sequence $(\xi_\ell^m)_{m\in\bfN}$ of mollified noise $\xi_\ell^m = \rho_{1/m} * \xi_\ell$ converging towards $\xi_\ell\in C^{|\Xi_\ell|}$. One wants to take a sequence of preparation map $(R^m)_{m \in \mathbb{N}}$ such that the associated renormalised stochastic data $\widehat{\mcZ}^m= (\widehat{Z}_\tau^m)_{m \in \mathbb{N}}$ verifies that $\bbE[\widehat{Z}_\tau^m]$ is bounded.  Here, $ \widehat{Z}_{\tau}^m  $ is built using the preparation map $R^m$. We also use the notation $ \widehat{\mathcal{Z}}^m $ when one uses the stochastic data $ \widehat{Z}_{\tau}^m $ for interpreting the letters $ \tau $. One can use the so-called BPHZ preparation maps $R^m$ which satisfy
\begin{equs}
	\label{BPHZ_ch}
	\mathbb{E}((\widehat{\PPi}^m \tau)(z)) = \mathbb{E}((\widehat{\PPi}^m \tau)(0)) = 0
\end{equs}
where $ \widehat{\PPi}^m $ is given by
	\begin{equation} \label{def_PPi_trees_R}
	\begin{aligned}
		\widehat{\PPi}^{\times,m} \Xi_{\ell} & = \xi_{\ell}^m,	\quad  (\widehat{\PPi}^{\times,m} X^k )(x) = x^{k}, \quad \widehat{\PPi}^m \tau = \widehat{\PPi}^{\times,m} R^m \tau\\	\widehat{\PPi}^{\times,m} \CI_n(\tau) & = D^{n} K * \widehat{\PPi}^{m} \tau, \quad  \quad \widehat{\PPi}^{\times,m} \tau_1 \tau_2 = \widehat{\PPi}^{\times,m} \tau_1 \widehat{\PPi}^{\times,m} \tau_2,
	\end{aligned}
\end{equation}
The fact that one can remove the dependence in $z$ comes from the invariance by translation in law of the noises $\xi_{\ell}$.
This definition was first given in \cite[Theorem 6.18]{BHZ}.
 This preparation map satisfies the key algebraic identities \eqref{commutation_R} as it was noticed in \cite{BN23}. We denote by $(\widehat{\PPi}^m, \widehat{g}^m)$ the BPHZ model. \label{BPHZ_model}

The proof of convergence of $(\widehat{\mcZ}^m)_{m \in \mathbb{N}}$ is performed by induction similar to the one in the proof of Theorem \ref{prop_extensionthm_Ztau}, except that we use a spectral gap argument for building the $\widehat{Z}^m_\tau$ associated to the trees $\tau$ with $|\tau|<0$. For trees  $\tau\in\dot\mcT$, we will also prove the convergence of the $\widehat{Z}^m_\tau$ in $B_2^{|\tau|}$.

For $\alpha\in\bfR$ we still let $T_{<\alpha}$ to be the sub-vector space spanned by the trees $\tau\in\mcT$ such that $|\tau|<\alpha$, we also let $\dot T_{<\alpha}$ to the image of $T_{<\alpha}$ under $D_\Xi$. We let $T_{<\alpha}^+$ the subalgebra of $T^+$ generated by the symbols $X_i$ with $i\in\{1,...,d+1\}$ and $\mcI_b^+(\tau)$ with $|\tau|<\alpha$.

We write the set $A$ of degrees of non polynomial decorated trees as  $A=(\alpha_1<\alpha_2<\dots)$. We let $\mcT_{\leq j}$ to be the set of trees in $\mcT$ forming a basis 
of $T_{<\alpha_{j+1}}$ and $\mcT_{\leq j}^+$ is the set of trees in $\mcT^+$ forming a basis 
of $T^+_{<\alpha_{j+1}}$. Likewise  $\dot{\mcT}_{\leq j}$ and ${\dot\mcT}^+_{\leq j}$ are the natural basis of $\dot T_{<\alpha_{j+1}}$ and $\dot T_{<\alpha_{j+1}}^+$. We also set $\widehat{\mcT}_{\leq j} = {\mcT}_{\leq j} \sqcup \dot{\mcT}_{\leq j} \sqcup {\mcT}_{\leq j}^+ \sqcup \dot{\mcT}_{\leq j}^+$.
The inductive scheme is the following
$$
\textbf{cv}(\mcT_{\leq 0})\rightarrow \dots \rightarrow
\textbf{cv}(\dot\mcT_{\leq j}) \rightarrow \textbf{cv}(\mcT_{\leq j}) \rightarrow 
\textbf{cv}(\mcT_{\leq j}^+, \dot\mcT_{\leq j}^+) \rightarrow 
\textbf{cv}(\dot\mcT_{\leq j+1}) \rightarrow \dots
$$
where $\mathbf{cv}(\mcT^{(1)},\dots,\mcT^{(n)})$  means the convergence for the subset of trees $\mcT^{(1)}\sqcup\dots\sqcup\mcT^{(n)}$.

We use the spectral gap in the step $\textbf{cv}(\dot\mcT_{\leq j}) \rightarrow \textbf{cv}(\mcT_{\leq j})$. The other steps rely on Corollary \ref{cor_repmodelledditribParap}, which amounts to a reconstruction argument.
Given a stochastic data $\mcZ$ we let for $\tau\in \widehat{\mcT}$ the quantity 
\begin{equation*}
	\norme{\tau}^*_\mcZ =  \max\Big\{  \norme{Z_\tau}_{B_{i(\tau)}^{|\tau|}}, \, \max\Big\{\big\lVert{Z_{\downarrow^k\tau/\sigma}\big\lVert}_{B^{|\downarrow^k\tau/\sigma|}_{i(\downarrow^k\tau/\sigma)}}\norme{\sigma}_\mcZ^*,\quad \sigma\prec\tau, k\geq0    \Big\}    \Big\}. 
\end{equation*} 
We also set 
\begin{equation*}
	\norme{\tau}^*_{j(\mcZ)} =  \max\Big\{  \norme{Z_\tau}_{C^{|j_*\tau|}}, \, \max\Big\{\big\lVert{Z_{\downarrow^k\tau/\sigma}\big\lVert}_{C^{|j_*\downarrow^k\tau/\sigma|}}\norme{\sigma}_{j(\mcZ)}^*,\quad \sigma\prec\tau,k\geq0    \Big\}    \Big\}. 
\end{equation*} 
From the Besov embedding inequality \eqref{Besov_injection}, for any $\tau\in\dot\mcT\sqcup\dot\mcT^+$ we have \begin{equation}\label{eq_besovIneq_normeTauStar}
	\norme{\tau}^*_{j(\mcZ)}\lesssim \norme{\tau}^*_{\mcZ}.
	\end{equation}
We also define a distance between two stochastic data $\mcZ,\mcZ'$ setting 
\begin{align*}
\norme{\tau}_{\mcZ:\mcZ'}^* & = \max\Big\{ \norme{Z_\tau-Z_\tau'}_{B_{i(\tau)}^{|\tau|}}, \\& \max\Big\{ \big\lVert{Z_{\downarrow^k\tau/\sigma} - Z'_{\downarrow^k\tau/\sigma}\big\lVert}_{B^{|\downarrow^k\tau/\sigma|}_{i(\downarrow^k\tau/\sigma)}} \norme{\sigma}_{\mcZ}^*, \quad \sigma\prec\tau,k\geq0   \Big\}, 
\\& \max\Big\{ \big\lVert { Z'_{\downarrow^k\tau/\sigma}\big\lVert}_{B^{|\downarrow^k\tau/\sigma|}_{i(\downarrow^k\tau/\sigma)}} \norme{\sigma}_{\mcZ:\mcZ'}^*, \quad \sigma\prec\tau,k\geq0   \Big\}\Big\}.
\end{align*}
Likewise one defines for $\tau\in\dot\mcT\sqcup\mcT^+$ the distance  $\norme{\tau}_{j(\mcZ):j(\mcZ')}$.

\paragraph{ Reconstruction argument}

We define 
$$
\norme{\tau}^{*,-}_\mcZ = \max\Big\{\big\lVert{Z_{\downarrow^k\tau/\sigma}\big\lVert}_{B^{|\downarrow^k\tau/\sigma|}_{i(\downarrow^k\tau/\sigma)}}\norme{\sigma}_\mcZ^*,\qquad \sigma\prec\tau, k\geq0    \Big\},
$$
and 
\begin{align*}
	\norme{\tau}_{\mcZ:\mcZ'}^{*,-}  = \max\Big\{  &  \big\lVert{Z_{\downarrow^k\tau/\sigma} - Z'_{\downarrow^k\tau/\sigma}\big\lVert}_{B^{|\downarrow^k\tau/\sigma|}_{i(\downarrow^k\tau/\sigma)}} \norme{\sigma}_{\mcZ}^*,  
	\quad \big\lVert { Z'_{\downarrow^k\tau/\sigma}\big\lVert}_{B^{|\downarrow^k\tau/\sigma|}_{i(\downarrow^k\tau/\sigma)}} \norme{\sigma}_{\mcZ:\mcZ'}^*, \\ & \sigma\prec\tau,k\geq0  \Big\}.
\end{align*}
The next Lemma follows from the same argument as in the proof of Lemma \ref{lem_reconstructionarg_Ztau} in Besov spaces.


\begin{lemma}\label{cor_reconstruction_convergence}
	For stochastic data $\mcZ,\mcZ'$ and $\tau\in \mcT_{\leq j}\sqcup\dot\mcT_{\leq j}$ with $|\tau|>0$, we have 
$$
\norme{Z_\tau}_{B^{|\tau|}_{i(\tau)}} \lesssim \lVert D^0\tau \lVert_\mcZ^* + \norme{\tau}_\mcZ^{*,-}. 
$$
and 
$$
\norme{Z_\tau-Z'_\tau}_{B^{|\tau|}_{i(\tau)}} \lesssim \lVert D^0\tau \lVert_{\mcZ,\mcZ'}^* + \norme{\tau}_{\mcZ,\mcZ'}^{*,-}.
$$
\end{lemma}

\begin{proof}
We have from Proposition \ref{Prop_CommutDerivativePsi} 
\begin{align*}
	\PPi^\mcZ\big(\Psi(D^0\tau)\big) - g^\mcZ\big(D^0\overline{\Psi}(\tau)\big) = \PPi^\mcZ\big(\Psi(D^0\tau)\big) - g^\mcZ\big(D^0{\Psi}(\tau)\big) + Z_\tau
	=Z_\tau.
\end{align*}
The Proposition \ref{prop_modelwords_besov} extension of  \cite[Theorem 1]{LocalExpansionsParacSystems}  enables to lift both $\PPi^\mcZ\big(\Psi(D^0\tau)\big) $ and $g^\mcZ\big(D^0\overline{\Psi}(\tau)\big)$ as modelled distributions ${\bm h}^1_\tau,{\bm h}^2_\tau$ in $\mcD^{|\tau|}_{i(\tau)}$. The same computations as in the proof of Lemma \ref{lem_reconstructionarg_Ztau} show that the difference ${\bm h}^1_\tau-{\bm h}^2_\tau$ takes value in the polynomial sector, then  $Z_\tau=\PPi^\mcZ\big(\Psi(D^0\tau)\big) - g^\mcZ\big(D^0\overline{\Psi}(\tau)\big)\in B^{|\tau|}_{i(\tau)}$ and  
\begin{align*}
	\norme{ Z_\tau }_{B^{|\tau|}_{i(\tau)}} \lesssim \norme{{\bm h}_\tau^1-{\bm h}_\tau^2}_{\mcD^{|\tau|}_{i(\tau)}}\lesssim \norme{{\bm h}_\tau^1}_{\mcD^{|\tau|}_{i(\tau)}}+\norme{{\bm h}_\tau^2}_{\mcD^{|\tau|}_{i(\tau)}}
	\lesssim  \lVert {D^0\tau} \lVert_\mcZ^*+ \norme{\tau}^{*,-}_\mcZ.
\end{align*}
The second inequality follows from the same arguments.
\end{proof}

\paragraph{Spectral gap argument}

\begin{lemma}\label{lem_estimate_deltaZtau}
For stochastic data $\mcZ,\mcZ'$, for any tree $\tau\in\mcT$ with $|\tau|<0$, we have $$\norme{\delta Z_\tau}_{C^{|\tau|}}\lesssim \norme{D_\Xi\tau}^*_\mcZ$$
and 
$$\norme{\delta Z_\tau-\delta Z_\tau'}_{C^{|\tau|}}\lesssim \norme{D_\Xi\tau}^*_{\mcZ:\mcZ'}.$$
\end{lemma}
\begin{proof}
		We use Theorem \ref{prop_alg_identity_convergence} and Proposition \ref{prop_regularity_Y_Besov}
	$$
	\norme{\delta Z_\tau}_{C^{|\tau|}}\lesssim \norme{j_*\Psi(D_\Xi\tau)}_{ (\PPi,g)}^*.
	$$
	Then \cite[Theorem 1]{LocalExpansionsParacSystems} ensures that
	$$
	\norme{j_*\Psi(D_\Xi\tau)}_{ (\PPi,g)}^* \lesssim \norme{D_\Xi\tau}^*_{j(\mcZ)}
	$$
	and the Besov injection \eqref{eq_besovIneq_normeTauStar} gives finally
	$$
	\norme{\delta Z_\tau}_{C^{|\tau|}} \lesssim \norme{D_\Xi\tau}^*_{\mcZ}.
	$$
	The second continuity estimate is obtained from the linearity of the $Y$ operator and the multilinearity of the iterated paraproduct.
\end{proof}

\begin{lemma}\label{lem_spectralgap_argument}
	Let random stochastic data $\mcZ,\mcZ'$ that are invariant by translation in law, for any $\tau\in\mcT$ with $|\tau|<0$ and any $\eps>0$, we have for  big enough $q\geq 1$
	$$
	\bbE\Big[\norme{Z_\tau}_{C^{|\tau|-\eps}}^q\Big] \lesssim  \bbE\Big[\big(\norme{D_\Xi\tau}^*_\mcZ\big)^q\Big] +  \Big|\bbE\big[Z_\tau\big]\Big|^q,
	$$
	and 
	$$
	\bbE\Big[\norme{Z_\tau-Z_\tau'}_{C^{|\tau|-\eps}}^q\Big] \lesssim  \bbE\Big[\big(\norme{D_\Xi\tau}^*_{\mcZ:\mcZ'}\big)^q\Big] + \Big|  \bbE\big[Z_\tau-Z_\tau'\big]   \Big|^q.
	$$
\end{lemma}

\begin{proof}
For any $i\geq-1$, the function $\Delta_i Z_\tau$ is a multilinear functions of the noises $\xi_\ell$. Then for any $q\in [1,\infty)$ the spectral gap inequality gives
$$
\mathbb{E}\Big[|\Delta_iZ_\tau|^q\Big] \lesssim  \Big| \bbE[\Delta_i Z_\tau]\Big|^q  + \mathbb{E}\Big[|\delta\Delta_iZ_\tau|^q\Big].
$$
As the noise is invariant by translation in law the function $\bbE[Z_\tau]$ is constant and for $i\ne-1$ we have $\bbE[\Delta_iZ_\tau] =\Delta_i\bbE[Z_\tau]=0$. 
Integrating over the space-time variable and using Theorem \ref{prop_alg_identity_convergence} gives 
\begin{align*}
\mathbb{E}\Big[\lVert\Delta_iZ_\tau\lVert_{L^q_x}^q\Big] &\lesssim \big| \bbE[\Delta_i Z_\tau]\big|^q +  \mathbb{E}\Big[\lVert\delta\Delta_iZ_\tau\lVert_{L^q_x}^q\Big]
\\&\lesssim  \big|\bbE[Z_\tau ]\big|^q \one_{i=-1} +
 2^{-iq|\tau|} \bbE\Big[\norme{\delta Z_\tau}_{C^{|\tau|}}^q\Big]  .
\end{align*}
Then for any $\eps>0$
 $$
 \mathbb{E}\Big[\norme{Z_\tau}^q_{B^{|\tau|-\frac{\eps}{2}}_{q,q}}  \Big] = \mathbb{E}\Big[\sum_{i\geq -1} 2^{iq(|\tau|-\frac{\eps}{2})} \lVert\Delta_iZ_\tau\lVert_{L^q_x}^q \Big]\lesssim \bbE\Big[\norme{\delta Z_\tau}_{C^{|\tau|}}^q\Big] + \big|\bbE[Z_\tau]\big|^q.
 $$
We use the Besov injection $B^{|\tau|-\frac{\eps}{2}}_{q,q} \hookrightarrow C^{|\tau|-\frac{\eps}{2}-\frac{|\frak s|}{q}}$ and Lemma \ref{lem_estimate_deltaZtau} to get 
$$
\mathbb{E}\Big[\norme{Z_\tau}^q_{C^{|\tau|-\frac{\eps}{2}-\frac{|\frak s|}{q}}}\Big]\lesssim \bbE\Big[\big(\norme{D_\Xi\tau}^*_\mcZ\big)^q\Big] + \big|\bbE[Z_\tau]\big|^q.
$$ 
Choosing $q$ big enough gives the result. The second inequality is obtained via the same arguments.
\end{proof}

For $\kappa>0$ one sets a degree $|\cdot|^{(\kappa)}$ for any $\ell\in\{1,...,\ell_0\}$ 
$$
|\Xi_\ell|^{(\kappa)}= |\Xi_\ell| + \kappa, \qquad\qquad |\dot\Xi_\ell|^{(\kappa)}= |\dot\Xi_\ell|+\kappa,
$$
and extending it to all the trees in $\widehat\mcT$ via the same rule as in \eqref{degree_tree}.
For $\kappa_0$ sufficiently small and $\kappa\in[0,\kappa_0)$, the degree $|\cdot|^{(\kappa)}$ still verifies the assumptions of Theorem \ref{main_convergence} and the associated stochastic data does not depend on $\kappa\in(0,\kappa_0)$ as the coproduct does not either. 
We write $\norme{\tau}^{*,\kappa}_\mcZ$ when we take the degree $|\cdot|^{(\kappa)}$ instead of $|\cdot|$ for measuring the regularity of the $Z_\sigma$.
Theorem \ref{main_convergence} is a direct corollary of the following proposition

\begin{proposition}
Under the assumptions of Theorem \ref{main_convergence}, let $\varrho$ be some mollifier, let $\kappa\in(0,\kappa_0)$ and the BPHZ preparation maps $R^{m}$. 
For any $\tau\in\widehat\mcT$ the sequence of stochastic fields $(\widehat{Z}_{\tau}^m)_m $ converges in $L^q\big(\Omega,B_{i(\tau)}^{|\tau|^{(\kappa)}}\big)$ for any $ 1 \leq q < + \infty $. The limit is independent of the mollifier $\varrho$.
\end{proposition}

\begin{proof}
	We prove the proposition for trees in $\widehat{\mcT}_{\leq j}$ by induction on $j$. We have $\mcT_{\leq 0} = \{ \Xi_\ell, 1\leq\ell\leq\ell_0  \}$ and  $\dot\mcT_{\leq 0} = \{ \dot\Xi_\ell, 1\leq\ell\leq\ell_0  \}$, and $\widehat{Z}_{\Xi_\ell}^m = \rho_{1/m}*\xi_\ell $ and $\widehat{Z}_{\dot\Xi_\ell}^m = \rho_{1/m}*\delta\xi_\ell$, which gives the result for $j=0$.
	
	We suppose now that the result holds true for all the trees in $\widehat\mcT_{\leq j-1}$. 	
	We first take $\tau\in \mcT_j\sqcup \dot\mcT_j$ with  $|\tau|>0$. From Lemma \ref{cor_reconstruction_convergence} we see that $\widehat{Z}_\tau^m$ is bounded in $B_{i(\tau)}^{|\tau|^{(\kappa)}}$ as $m$ goes to $+ \infty$ and that for any $q\geq1$ 
	$$
	\bbE\Big[\norme{\widehat{Z}_\tau^{m_1}-\widehat{Z}_\tau^{m_2}}_{B^{|\tau|^{(\kappa)}}_{i(\tau)}}^q\Big] \lesssim \sup_{\sigma\in \widehat\mcT_{\leq j-1}} \bbE\Big[\big(\norme{\sigma}^{*,\kappa}_{\widehat{\mcZ}^{m_1}:\widehat{\mcZ}^{m_2}}\big)^q\Big] \to 0 
	$$ as $m_1,m_2\to + \infty$. Which proves the convergence of $(\widehat{Z}_\tau^m)_m$ in $L^q(\Omega,B_{i(\tau)}^{|\tau|})$ by completeness. We prove that the limit is independent of the mollifier by noting that the same Lemma \ref{cor_reconstruction_convergence} gives by induction  
		$$
	\bbE\Big[\big\lVert{\widehat{Z}_\tau^{m,\rho}-\widehat{Z}_\tau^{m,\rho'}\big\lVert}_{B^{|\tau|^{(\kappa)}}_{i(\tau)}}^q\Big] \lesssim \sup_{\sigma\in \widehat\mcT_{\leq j-1}} \bbE\Big[\big(\norme{\sigma}^{*,\kappa}_{\widehat{\mcZ}^{m,\rho}:\widehat{\mcZ}^{m,\rho'}}\big)^q\Big] \to 0 
	$$
	as $m\to\infty$, where $\widehat{Z}_\tau^{m,\rho}$ is the sequence of stochastic data obtained via the mollifier $\rho$.

	We now take $\tau\in\mcT_j$ with $|\tau|<0$. We know by induction that all the $\widehat{Z}_\sigma^m$ with $\sigma\in\widehat\mcT_{\leq j-1}$ converge in $B_{i(\tau)}^{|\tau|^{(\kappa)}}$ for a small $\kappa>0$. Then Lemma \ref{lem_spectralgap_argument} gives for some $\kappa'\in(\kappa,\kappa_0)$ and big enough $q\geq 1$.
	$$
	\bbE\Big[\norme{\widehat{Z}_\tau^{m_1}-\widehat{Z}_\tau^{m_2}}_{C^{|\tau|^{(\kappa)}}}^q\Big] \lesssim  \bbE\Big[\big(\norme{D_\Xi\tau}^{*,\kappa'}_{\widehat{\mcZ}^{m_1}:\widehat{\mcZ}^{m_2}}\big)^q\Big] + \Big|\bbE\big[ \widehat{Z}_\tau^{m_1} - \widehat{Z}_\tau^{m_2}\big]\Big|^q\to 0
	$$
	as $m_1,m_2\to +\infty$. The fact that $ \Big|\bbE\big[ \widehat{Z}_\tau^{m_1} - \widehat{Z}_\tau^{m_2}\big]\Big| $ tends to zero comes from an inductive argument. Indeed,  from the definition of $ 	\widehat{Z}^m_{\tau} $, one has
	 \begin{equs}
			\widehat{Z}^m_{\tau} = \widehat{\PPi}^m \tau + F(\widehat{Z}^m_{\sigma}, \sigma \in \mathcal{T} \cup \mathcal{T}^+) 
		\end{equs}
		where  the $ \sigma $ are smaller than $\tau$ in the inductive construction of the trees and $F$ is a continuous function in the $\widehat{Z}^m_{\sigma}$ (Sum of iterated paraproducts). Therefore, one has
		\begin{equs}
			\mathbb{E}(	\widehat{Z}^m_{\tau}) =  \mathbb{E}(F(\widehat{Z}^m_{\sigma}, \sigma \in \mathcal{T} \cup \mathcal{T}^+) ) = \mathbb{E}(F(\widehat{Z}^m_{\sigma}, \sigma \in \mathcal{T} \cup \mathcal{T}^+)(0) ) 
		\end{equs}
		which convergences due to the convergence of the $\widehat{Z}^m_{\sigma}$. We have also used above the translation invariance in law.
	  This proves the convergence of $(\widehat{Z}_\tau^m)_m$ in $L^q(\Omega,C^{|\tau|^{(\kappa)}})$ for any $q\geq1$. The same argument as above gives the independence from the mollifier for the limit of the sequence $(\widehat{Z}^m_\tau)_m$, and shows also that this limit depends only on the limit of the sequence $(\bbE[\widehat{Z}_\tau^m])_m$. 
\end{proof}

	\begin{appendix}
	
	\section{Local expansions for more general paraproducts}
	\label{appendix_A}
	In this appendix we describe how the results from \cite{LocalExpansionsParacSystems} generalise to paraproducts built from more general Littlewood-Paley projectors. The algebraic part of \cite{LocalExpansionsParacSystems} will remain unchanged, we only have to check the analytical aspects. We make the following assumptions on the Littlewood-Paley operators we will use. 
	
	\begin{assumption}
	Suppose we are given a family of Schwartzian kernels $(K_i)_{i\geq -1}$ on $\bfR^{d +1}$ satisfying the following three properties  
	\begin{itemize}
	\item We have the reproducing formula
	$
	\sum_{i \geq -1 } K_i = \delta_0.
	$
	\item For any $i\geq 0$ we have the scaling property 
	$$
	K_i = 2^{i (d+\beta)}K_0 \circ  \mathcal{S}^{2^i}.
	$$
	\item For any $k\in\bfN^{d+1}$ with $|k|\leq B$ and $i\geq0$, one has the cancellation property
	$$
	\int_{\bfR^{d+1}} K_i(x) x^k \, \ddd x = 0.
	$$
	\end{itemize}
\end{assumption}
From  this family of kernels one can build paraproducts 
	$$
	\P(f,g) \defeq \sum_{i\geq -1} K_{<i-1} f \,  K_i g.
	$$
as well as their iterates with polynomial weights $\P_\ell(h_1,\dots,h_n)$ and the star derivatives $\partial^k_*\P_\ell(h_1,\dots,h_n)$ as it was defined in Subsections  \ref{subsection_analyticalsetting} and \ref{subsect_theRSIteratedParap}. 
The main result from \cite{LocalExpansionsParacSystems} we aim at generalising reads here as the following proposition.
	\begin{proposition} \label{prop_generalisation_LocalDevParap}
		Suppose the $n-$uplet of real numbers $(\alpha_j)_{1\leq j\leq n}$ is such that $\sum_{i=1}^n |\alpha_i| < B$ and that $\scrT_\mcA$ satisfies Assumption \ref{assumption_notint}. For any uplet of distributions $(h_{a_j})_{1 \leq j \leq n} \in \prod_{j=1}^n C^{\alpha_j}$, one defines a model on the regularity structure $\scrT_\mcA$ setting by
		$$
		{ \PPi}(\lg a_1,...,a_n\rg_\ell ) = \P_\ell(h_{a_1},\dots,h_{a_n}) 
	,\quad 
		{ g}(\lg a_1,...,a_n\rg_\ell^k ) = \partial^k_*\P_\ell(h_{a_1},\dots,h_{a_n}).
		$$
		We have furthermore the quantitative bound
		$$
		(\!| \PPi,g |\!) \lesssim \prod_{j=1}^{n} \norme{h_j}_{C^{\alpha_j}}.
		$$
	\end{proposition}
		In \cite{LocalExpansionsParacSystems}, one defines a space of Littlewood-Paley blocks denoted $\mcC^{\alpha , (B)}$ as spaces of functions sequences $(f_i)_{i\geq -1}$ where each $f_i$ has its Fourier transform supported in some ball $2^i B$, endowed with the norm $\norme{(f_i)}_\alpha=\sup_i 2^{i\alpha}\norme{f_i}_{L^\infty}$. In our case the Littlewood-Paley block does not typically have its Fourier transform supported in a ball, so that we have to find another definition for the space $\mcC^{\alpha}$.   
	
	The  interest of having functions with Fourier transform supported in balls is for using Bernstein estimates for bounding the derivatives of the blocks $f_i$. The idea here is to force these estimates by incorporating them in the norms we will be using. We let the space $\mcC^{\alpha , (B)}$ to be the vector space consisting of families of functions $(f_i)_{i\geq -1}$ such that
	$$
	\norme{(f_i)}_{\alpha,(B)} \defeq \sup_{|k| \leq B} \sup_{i\geq -1} 2^{i(\alpha-|k|)}\lVert \partial^k f_i \lVert_{L^\infty} < \infty, 
	$$ 
	the quantity $\norme{(f_i)}_{\alpha,(B)}$ defines a norm on this vector space. We will often omit the parameter $B$ as it will remain fixed.

	One checks that the results from \cite[Section 2]{LocalExpansionsParacSystems} still hold true for the simplified paraproducts $\P_<$ evaluated on these family of blocks. This follows from the same computations, by noting that we have the Bernstein estimate $$	\big\lVert{(\partial^kf_i)}\big\lVert_{\alpha-|k|,(B-|k|)}\leq 	\norme{(f_i)}_{\alpha,(B)}.$$ Then from \cite[Lemma 3]{LocalExpansionsParacSystems}, it generalises to 
	$$
	\sup_i 2^{i(\alpha-\theta)} \big\lVert{R_h^\theta f_i}\big\lVert_{L^\infty} \lesssim 	\norme{(f_i)}_{\alpha,(B)} |h|^\theta
	$$
	 for $0<\theta<B$, and where $R_h^\theta f=f(\cdot+h) - \sum_{|k|<\theta} \frac{h^k}{k!} \, \partial^kf $ is the Taylor remainder.  
	
	The second part of the proof of Proposition \ref{prop_generalisation_LocalDevParap} is an induction on the number of arguments $n$, where one writes an iterated paraproduct as a sum of simplified paraproducts.
	$$
	 \P_\ell(h_{a_1},\dots,h_{a_n}) = \sum_{\ell_{n-1}} { \ell \choose \ell_{n-1}  } \sum_{\tau_q\prec\dots\prec\tau_0 =\tau_0(\ell,\ell_{n-1}) } \P_< \left( [\tau_0/\tau_1 ] ,\dots,[\tau_q] , [\langle a_n \rangle]  \right)
	$$
	where $$\tau_0(\ell,\ell_{n-1}) = \lg a_1,...,a_{n-1}\rg_{\ell-\ell_{n-1}}X^{\ell_{n-1}}. $$ 
	The blocks $[\tau_s/\tau_{s+1}]$ are built by integrating the remainder of the expansion corresponding to the symbol $\tau_s/\tau_{s+1}$, and are meant to be in our space of generalised Littlewood-Paley blocks of order $|\tau_s/\tau_{s+1}|$.
	Here again the algebraic part of the proof is the same and we only have to check the analytical bounds on the blocks $[\sigma]$.
	 As the definition of the space of Littlewood-Paley blocks was changed to incorporate bounds on the derivatives, we have to check the adequate bound on the $\partial ^k[\tau_s/\tau_{s+1}]$.
The bracket can be defined via the formula
	$$
	[\tau]_i = (\P \tau)_i + \sum_{\sigma < \tau} ({\bf Q}(\tau/\sigma))_i \, [\sigma]_i,
	$$
	with 
	\begin{align*}
	(\P \tau)_i(z) &= \int_{(\bfR^{d+1})^2} K_{<i-1}(z-x)K_i(z-y) \, {\Pi}_x\tau(y) \,  \ddd y \ddd x,
	\\
		(\mathbf{Q} \tau)_i(z) &= \int_{(\bfR^{d+1})^2} K_{<i-1}(z-x)K_{<i-1}(z-y) \, {\ \gamma}_{yx}(\tau) \, \ddd y \ddd x.
\end{align*}
It was proven in \cite[Proposition 8]{BailleulHoshinoRS1}, that if $(\PPi, { g})$ is a model, then we have the bounds \begin{equs}
	|(\P \tau)_i|\lesssim \Vert\tau\Vert_{ (\PPi,g)}^* 2^{-i|\tau|}, \quad |(\mathbf{Q} \tau)_i|\lesssim \Vert\tau\Vert_{ (\PPi,g)}^* 2^{-i|\tau|}.
	\end{equs} The proof relies solely on the scaling property of the kernels $K_i$, and for any $k\in\bfN^{d+1}$ the derivatives $\partial^k K_i$ satisfy the scaling property 
	$$
		\partial^kK_i = 2^{i(d +\beta+|k|)}(\partial^kK_0)\circ \mcS^{2^i},
	$$ then setting 
	\begin{align*}
	(\partial^k\P  \tau)_i(z) &=\sum_{k=k_1+k_2} \binom{k}{k_1} \int_{(\bfR^{d+1})^2} (\partial^{k_1}K_{<i-1})(z-x)(\partial^{k_2}K_i)(z-y) {\Pi}_x\tau(y) \ddd y \ddd x
	\\
	(\partial^k\mathbf{Q}  \tau)_i(z) &= \sum_{k=k_1+k_2} \binom{k}{k_1}  \int_{(\bfR^{d+1})^2} (\partial^{k_1}K_{<i-1})(z-x)(\partial^{k_2}K_{<i-1})(z-y) { \gamma}_{yx}(\tau) \ddd y \ddd x
\end{align*}
we have this time the bounds $$|(\partial^k\P \tau)_i|\lesssim \Vert\tau\Vert_{ (\PPi,g)}^* 2^{-i(|\tau|-|k|)}, \quad  |(\partial^k\mathbf{Q} \tau)_i|\lesssim \Vert\tau\Vert_{ (\PPi,g)}^* 2^{-i(|\tau|-|k|)}.$$ Then, writing for any $k\in\bfN^{d+1} $ 
		$$
	\partial^k[\tau]_i= (\partial^k\P \tau)_i + \sum_{k=k_1+k_2}\sum_{\sigma\prec\tau} \binom{k}{k_1} (\partial^{k_1}{\bf Q}\, \tau/\sigma)_i \, \partial^{k_2}[\sigma]_i.
	$$
	we obtain from the same induction as in the proof of \cite[Proposition 18]{LocalExpansionsParacSystems} the desired bound $\partial^k[\tau]_i \lesssim \Vert\tau\Vert_{ (\PPi,g)}^* 2^{-i(|\tau|-|k|)}.$
	\section{Local expansions on Besov spaces}
	\label{Appendix_B}
	This section is devoted to the proof of Proposition \ref{prop_modelwords_besov}. As in the proof of Proposition \ref{prop_generalisation_LocalDevParap}, the argument follows the same algebraic steps, so it remains only to verify the analytical estimates.
	
The additional feature in the present setting is the bigrading $(r,\frac{1}{i})$. The exponent $r$ replaces the usual degree and plays the same role in the algebraic part of the proof. Whereas, the exponent $\frac{1}{i}$ does not enter the algebraic arguments; it only determines the integrability exponent appearing in the analytical estimates. 
The only essential difference from the proof of Proposition \ref{prop_generalisation_LocalDevParap} is that the estimates for products now rely on Hölder's inequality.

	The new space of Littlewood-Paley block is denoted $\mcB_p^{\alpha , (B)}$, and it is defined as the space of functions sequences $(f_i)_{i\geq -1}$ such that
		$$
		\big\lVert{(f_i)}\big\lVert_{\alpha,p,(B)} \defeq \sup_{|k| \leq B} \sup_{i\geq -1} 2^{i(\alpha-|k|)}\lVert \partial^k f_i \lVert_{L^p} < \infty,
		$$ 
		the quantity $\norme{(f_i)}_{\alpha,p,(B)}$ defines a norm. We will still omit the parameter $B$ as it will remain fixed. 
	We have the Bernstein estimate 
	$$	\big\lVert{(\partial^kf_i)}\big\lVert_{\alpha-|k|,p,(B-|k|)}\leq 	\norme{(f_i)}_{\alpha,p,(B)}.
	$$
	Then \cite[Lemma 3]{LocalExpansionsParacSystems} generalises to 
	$$
	\sup_i 2^{i(\alpha-|\theta|)} \big\lVert{R_h^\theta f_i}\big\lVert_{L^p} \lesssim 	\norme{(f_i)}_{\alpha,p,(B)} |h|^{\theta}
	$$
	for $\theta<B$, where $R_h^\theta f=f(\cdot+h) - \sum_{|k|<\theta} \frac{h^k}{k!} \, \partial^kf $ is the Taylor remainder.  
	We then check that the results from \cite[Section 2]{LocalExpansionsParacSystems} hold true for the simplified paraproducts and integrability exponents that can be different from $+\infty$. The only difference is that one uses the Hölder inequality, this is where one uses the second part of identity \eqref{eq_degree_besov_words}. 
	
	The second part of the proof of Proposition \ref{prop_generalisation_LocalDevParap} is the same induction on the number of arguments $n$, where one writes an iterated paraproduct as a sum of simplified paraproducts. 
	$$
	\P_\ell(h_{a_1},\dots,h_{a_n}) = \sum_{\ell_{n-1}} { \ell_{n-1} \choose \ell  } \sum_{\tau_q\prec\dots\prec\tau_0 =\tau_0(\ell,\ell_{n-1}) } \P_< \left( [\tau_0/\tau_1 ] ,\dots,[\tau_q] , [\langle a_n \rangle]  \right)
	$$
	where the blocks $[\mu]$ are in $ \mcB_{i(\mu)}^{r(\mu) , (B)}$.
	
The regularity of these brackets is established from the same computations as in the proof of Lemma \ref{lem_crochets_Besov}. The only difference is that one integrates against either $\partial^{k_1}K_{<i-1},\partial^{k_2}K_i$ or $\partial^{k_1}K_{<i-1},\partial^{k_2}K_{<i-1}$, where $|k_1|+|k_2|<B$. 
	
	\section{Symbolic index}
	\label{Appendix_C}
	In this appendix, we collect the most used symbols of the article, together
	with their meaning and the page where they were first introduced.
	\begin{center}
		\renewcommand{\arraystretch}{1.1}
		\begin{longtable}{lll}
			\toprule
			Symbol & Meaning & Page\\
			\midrule
			\endfirsthead
			\toprule
			Symbol & Meaning & Page\\
			\midrule
			\endhead
			\bottomrule
			\endfoot
			\bottomrule
			\endlastfoot
	$B^{\alpha}_p $	& Short hand notation $B^{\alpha}_p =B^\alpha_{p,\infty}$.  & \pageref{besov}	\\
		$	B^{\alpha}_{p,q} $ & Besov space & \pageref{besov} 
			\\
		$\mcB$ & Basis of $T$ &  \pageref{subsection_basicsRS}
			\\
				$\dot{\mcB}$ & Abstract Malliavin derivative of $ \mcB $ &  \pageref{subsect_Malliavin_Trees}
			\\
				$\mcB^+$ & Basis of $T^+$ &  \pageref{subsection_basicsRS}
			\\
			$\CB_X$ & Basis of $T_X$ &  \pageref{eq_sigmabackslashtau}
			\\
		$ \bfC_\ell $	& Corrector with weight $ \ell $   &  \pageref{eqdef_corrector}
			\\
		$ D^n $ & Linear operator $ 	D^n\tau = n! \,  \tau / X^n  $  &	\pageref{linear_operator}
		\\
		$D_{\Xi}$ & Abstract Malliavin derivative  & \pageref{D_Xi}
		\\ 
			$ \downarrow^n $  & Linear operator $ \downarrow^n \tau =  n!\, X^n \backslash \tau$ & \pageref{linear_operator}
			\\
	 $ \partial^k_* \P_\ell$ & Linear operator on $ \P_\ell $ &	\pageref{eqdef_starderivative}	
	 \\
	 $ \partial^k_* u$ & Linear operator on $u$ &	\pageref{eqdef_partialStarU}	
			\\
			$\delta$ & Malliavin derivative &  \pageref{Malliavin_derivative}
			\\
		$\Delta$	& Coaction for the recentering  & \pageref{eq_sigmabackslashtau}
			\\
				$\Delta^{\!+}$	& Coproduct for the recentering  & \pageref{eq_sigmabackslashtau}
			\\
			$\Delta_i^\ell$  & 	Littlewood-Paley projector with monomial weight $\ell$ & \pageref{eqdef_littlewoodprojectorL} \\ 
			$\mcI_{a}$ & Abstract integrator map $ \mathcal{T} \rightarrow \mathcal{T} $ &
		 \pageref{subsection_DecoratedTrees}
		 \\
	$\mcI_{a}^+$	& Abstract integrator map $ \mathcal{T} \rightarrow \mathcal{T}^+ $ & \pageref{I_+}
			\\
			$i$ & Integrability degree $\mcB\sqcup \mcB^+ \to [1,+\infty]$ & \pageref{integrability}
			\\
		$j$ &  Besov injection $B^{\alpha}_p(\bfR^{d+1}) \hookrightarrow C^{\alpha-\frac{|\mathfrak{s}|}{p}}(\bfR^{d+1})$ & \pageref{Besov_injection}	
		\\
		$ j_*  $ & Abstract Besov injection &   \pageref{inj_ab}
			\\
			 $K_i$ & Kernel associated with $ \Delta_i $ &  \pageref{eqdef_littlewoodprojectorL}
			\\
		  $  \lg w\rg_\ell X^p $	& Decorated words in $ \mcB $  &  \pageref{subsect_theRSIteratedParap}
		  \\
		  $  \lg w\rg_\ell^k X^p $	& Decorated words in $ \mcB^+ $  &  \pageref{subsect_theRSIteratedParap}
			\\
		$	\P_{\ell}$ & Iterated paraproducts with monomial weight $ \ell $ &  \pageref{eqdef_ParapPolWeight2}
		\\
	  $ (\PPi,g)$ & Model & \pageref{model}
	  \\
	  $(\widetilde{ \PPi},\widetilde{g})$ & Modified model, no extraction of monomials  & \pageref{eqdef_pitilde}
	  \\
	  $ (\PPi^\mathcal{A},g^\mathcal{A}) $ & Model over the alphabet $ \mathcal{A} $  &  \pageref{model_words} 
	  \\
	  $(\widehat{\PPi}^m, \widehat{g}^m)$ & BPHZ model & \pageref{BPHZ_model}
	  \\
	 $ \Psi$ & Hairer-Kelly map $ \mathcal{T} \rightarrow \mathcal{T}_{\mathcal{A}} $, $\mathcal{A}$ set of decorated trees  & \pageref{def_Psitau}
	 \\
	 $\Psi^I$ & Modified Hairer-Kelly map for the integration & \pageref{modified_Hairer_Kelly}
	 \\
	 $R$   & Strong preparation map & \pageref{eq_DeltaRTau} 
	  \\
	  $R^m$ & BPHZ preparation map & \pageref{BPHZ_preparation}
	  \\
	  ${\bm R}$ & Reconstruction operator  & \pageref{reconstrcution_op}
	  \\
	 $ \star $ & Grossman-Larson product on decorated trees  & \pageref{star_product}
	  \\
	 $r(\cdot)$ & Degree $|\cdot|$ for a regularity-integrability structure   & \pageref{integrability}
	  \\
$S(\tau)$	&  Symmetry factor of $\tau$  &  \pageref{symmetry_factor}
		\\
	$[\tau]_i$	& Generalised Littlewood-Paley block & \pageref{subsect_modifmodel}
	\\
	$\widehat{[\tau]}_i $ & Integrated generalised Littlewood-Paley block & \pageref{integrated_bracket}
	\\
 $\tau/\sigma $ & Term appearing in 	$\sigma \otimes  \tau/\sigma $ & \pageref{subsection_basicsRS}
		\\
		$\tau \backslash \sigma $ & Term appearing in 	$  \tau \backslash \sigma \otimes \sigma $ & \pageref{eq_sigmabackslashtau}
		\\
		 $(T,T^+)$ & Regularity Structure  & \pageref{subsection_basicsRS}
		 \\
	 $ (\mathcal{T},\mathcal{T}^+) $ & Regularity Structure on decorated trees  &	 \pageref{subsection_DecoratedTrees}
		 \\
		$ (\mathcal{T}_{\mathcal{A}}, \mathcal{T}_{\mathcal{A}}^+)$ & Regularity Structure over words on the alphabet $ \mathcal{A} $  & \pageref{def_model_words}
		 	\\
		 $T_X$ & Polynomial Regularity Structure  & \pageref{subsection_basicsRS}
		 \\
	$\mcT$	 & Decorated trees $T_{\frak e}^{\frak n, \frak f} $ & \pageref{subsection_DecoratedTrees}		
		 \\
		 $\mcT_o$	 & Decorated trees $T_{\frak e}^{\frak n, \frak f}$ with no noise at the root  & \pageref{subsection_DecoratedTrees}	
		 \\
		 $\dot{\mcT}$	 & Decorated trees with one Malliavin derivative  & \pageref{D_Xi}	
		 \\
		 $ u^{\#} $ & Paracontrolled remainder  & \pageref{eq_pamansatz}
		 \\
		 $ \Upsilon_F $ & Elementary differentials &  \pageref{subsec_subcriticality}
		 \\ $|\cdot|$ & Degree map & \pageref{subsection_basicsRS}
		 \\ $Y_{\tau}$ & Stochastic data associated with a modified model $ (\widetilde{\PPi}, \widetilde{g})$ & \pageref{eqdef_Ztau1}
		 \\
		 $ \widehat{Y}_{\tau} $ & Integration of the stochastic data $Y_{\tau}$ & \pageref{def_YItau}
		 \\ $\mathcal{Y}$ & Alphabet whose letters $ \tau$ are sent to $ Y_{\tau} $  & \pageref{eqdef_Ztau1}
		 \\ $ Z_{\tau} $ & Main stochastic data  & \pageref{eq_defZtau_Naif}
		 \\  $Z_{\tau}^\circ$ & Multiplicative stochastic data  & 
		 \pageref{eq_defZotau}
		 \\ $ \widehat{Z}_{\tau} $  & Renormalised stochastic data with  $R$ & \pageref{eq_Renorm_ZtauFromZotau}
		 \\  $ \widehat{Z}_{\tau}^m $  & Renormalised stochastic data with  $R^m$ & \pageref{BPHZ_model}
		 \\ $ \widehat{Z}_{\tau}^{\circ} $  & Renormalised multiplicative stochastic data & \pageref{eq_Renorm_ZtauFromZotau}
		 \\ $\mathcal{Z}$  & Alphabet whose letters $ \tau$ are sent to $ Z_{\tau} $  & \pageref{eq_defZtau_T+}
		  \\ $\mathcal{Z}^{\circ}$  & Alphabet whose letters $ \tau$ are sent to $ Z^{\circ}_{\tau} $  & \pageref{eq_defZotau}
		 \\   $\widehat{\mathcal{Z}}$  & Alphabet whose letters $ \tau$ are sent to $ \widehat{Z}_{\tau} $   & \pageref{eq_Renorm_ZtauFromZotau}
		 \\   $ \widehat{\mathcal{Z}}^m $  & Alphabet whose letters $ \tau$ are sent to $ \widehat{Z}_{\tau}^m $ & \pageref{BPHZ_model}
		 \\ $\widehat{\mathcal{Z}}^{\circ}$  & Alphabet whose letters $ \tau$ are sent to $ \widehat{Z}^{\circ}_{\tau} $  & \pageref{eq_Renorm_ZtauFromZotau}
		\end{longtable}
	\end{center}

	\end{appendix}

\end{document}